\documentclass[12pt]{article}
\usepackage[a4paper,margin=1.2cm]{geometry}
\usepackage{latexsym, amsmath, amssymb}
\usepackage{graphicx}
\usepackage{epstopdf}
\usepackage{amsmath, amsthm, amscd, amsfonts}
\usepackage{amsmath, xcolor}
\usepackage{array}
\usepackage{multirow}
\usepackage{stfloats}
\usepackage{amssymb}
\usepackage{ltablex}
\usepackage{float}
\usepackage{cite}
\usepackage{enumerate}
\usepackage{placeins}
\usepackage{comment}
 \usepackage{soul}
\usepackage{color}
\usepackage{lscape}
\usepackage{soul}
\usepackage{caption} 
\usepackage[T1]{fontenc}
\usepackage[colorlinks,urlcolor=blue]{hyperref}

\newtheorem{theorem}{Theorem}[section]
\newtheorem{proposition}{Proposition}[section]

\newtheorem{definition}{Definition}[section]

\makeatletter
\newcommand*{\rom}[1]{\expandafter\@slowromancap\romannumeral #1@}
\makeatother

\def\bege{\begin{equation}} \def\ende{\end{equation}} 

   \def\begr{\begin{eqnarray}}
\def\endr{\end{eqnarray}} 
 
\def\bege{\begin{equation}} \def\ende{\end{equation}}
\def\begr{\begin{eqnarray}} \def\endr{\end{eqnarray}} \def\bnum{\begin{enumerate}} \def\enum{\end{enumerate}}
\begin{document}
	
\begin{center}\Large
\textbf{Lee Distance of cyclic codes of length $2^\varsigma$ over $\mathbb{F}_{2^m} + u\mathbb{F}_{2^m} + u^2\mathbb{F}_{2^m}$}
\end{center}

%%%%%%%%%%%%%%%%%%%%%%%%%%%%%%%%%%%%%%%%%%%%%%%%%%%%%%%%%%%%%%%%%%%%%%
\begin{center}
Divya Acharya$^{1}$, Prasanna Poojary$^{1}$, Vadiraja Bhatta G R$^{2}$
\end{center}

\begin{center}
$^1$Department of Mathematics, Manipal Institute of Technology Bengaluru, Manipal Academy of Higher Education, Manipal, India\\ 
% {\it \textcolor{blue}{\underline{acharyadivya1998@gmail.com; poojary.prasanna@manipal.edu; poojaryprasanna34@gmail.com}}}\\
$^2$ Department of Mathematics, Manipal Institute of Technology, Manipal Academy of Higher Education, Manipal, India\\ 
{\it \textcolor{blue}{acharyadivya1998@gmail.com; }}{\it \textcolor{blue}{poojary.prasanna@manipal.edu;}}\\
{\it \textcolor{blue}{poojaryprasanna34@gmail.com;  vadiraja.bhatta@manipal.edu}}\\
\end{center}

\abstract{Let $p$ be a prime number and $\varsigma$ and $m$ be a positive integers. Let $\mathcal{R} = \mathbb{F}_{2^m} + u\mathbb{F}_{2^m} + u^2\mathbb{F}_{2^m}$ ($u^3 = 0$). Cyclic codes of length $2^\varsigma$ over $\mathcal{R}$ are precisely the ideals of the local ring  $\frac{\mathcal{R}[x]}{\langle x^{2^\varsigma}-1 \rangle}$. The Gray map from a code of Lee weight over $\mathbb{Z}_4$ to a code with Hamming weight over $\mathbb{F}_2$ is known to preserve weight. 
%Recently, researchers have been more interested in codes over rings with the Lee metric.
In this paper, we determine the Lee distance of cyclic codes of length $2^\varsigma$ over $\mathcal{R}$. }\\

%$^{\ast}$Corresponding: 
%vadiraja.bhatta@manipal.edu
%poojary.prasanna@manipal.edu; poojaryprasanna34@gmail.com	%%%%%%%%%%%%%%%%%%%%%%%%%%%%%%%%%%%%%%%%%%%%%%%%%%%%%%%%%%%%%%%%%%%

\section{Introduction}
Algebraic coding theory focuses on identifying codes that transmit quickly and can correct or detect numerous errors. The $\alpha$-constacyclic codes of length $n$ over the finite field $\mathbb{F}_q$ are identified as ideals $\langle \ell(x)\rangle$  of the ambient ring $\frac{\mathbb{F}_q[x]}{\langle x^n-\alpha \rangle}$, where $\ell(x)$ is a divisor of $x^n-\alpha$. The codes are called simple root codes when the length of the code $n$ is relatively prime to the characteristic of the finite field $\mathbb{F}_q$. Otherwise, they are known as repeated-root codes, which Berman \cite{berman1967semisimple} first investigated in 1967, followed by a series of papers \cite{massey2003polynomial, falkner1979existence, roth2003cyclic,van1991repeated, castagnoli1991repeated}. There is a lot of work being done on the structures of codes(see, for example, \cite{sualuagean2006repeated, ankur2020type, dinh2008linear, cao2015repeated, batoul2016some, zhao2018all, cao2018constacyclic, sidana2020repeated, li2022unique}) and Hamming distances (see, for example, \cite{dinh2021hamming,dinh2008linear,dinh2018hamming, dinh2020hamming}). 
%However, there are very few of these codes in terms of LD.

The Lee distance was first introduced in \cite{lee1958some}. Following the study by Hammons et al. \cite{hammons1994z}, there was a shift in the traditional algebraic coding theory framework of finite fields with Hamming distance. This study demonstrated that good nonlinear codes over $\mathbb{F}_2$ with the Hamming metric can be obtained through isometric, nonlinear Gray images of linear codes over $\mathbb{Z}_4$. Finite commutative rings were made popular as code alphabets by this landmark research. Further, it encouraged the research of other metrics, particularly the Lee metric for codes over rings.

%A landmark study by Hammons et al. \cite{hammons1994z} established the importance of the Lee distance for establishing an isometry between binary and quaternary codes.
 Dinh  \cite{dinh2007complete}  determined the Lee distance of all negacyclic codes of length $2^\varsigma$ over $\mathbb{Z}_{2^a}$ using their Hamming distance. Also, Kai et al. \cite{kai2010distances} determined the Lee distance of some $\mathbb{Z}_4$-cyclic codes of length $2^e$. In \cite{kim2017lee}, Kim and Lee calculated the minimum Lee weights of cyclic self-dual codes of length $p^k$ over a Galois ring $GR(p^2, m)$. After that, Dinh et al. \cite{dinh2021lee} determined the Lee distance of $(4z-1)$-constacyclic codes of length $2^\varsigma$ over the Galois ring $GR (2^a, m)$ and in \cite{dinh2022lee} they examined the Lee distance distribution of repeated-root constacyclic codes over $GR(2^e,m)$. Betsumiya et al. \cite{betsumiya2004type} used the concept of a trace-orthogonal basis of $\mathbb{F}_{2^m}$ to define the Lee weight over $\mathbb{F}_{2^m}$  and $\mathbb{F}_{2^m}+u\mathbb{F}_{2^m}$. Recently, in \cite{dinh2021lee1}, the Lee distance of cyclic codes of length $2^\varsigma$ over $\mathbb{F}_{2^m}$, cyclic codes of length $2^\varsigma$ over $\mathbb{F}_{2^m}+u\mathbb{F}_{2^m}$
and  $(1 +u\gamma)$-constacyclic codes of length $2^\varsigma$ over $\mathbb{F}_{2^m}+u\mathbb{F}_{2^m}$, where $\gamma$ is a non-zero element of $\mathbb{F}_{2^m}$ are determined. Motivated by these works, in this paper, we determine the Lee distance of cyclic codes of length $2^\varsigma$ over $\mathbb{F}_{2^m} + u\mathbb{F}_{2^m} + u^2\mathbb{F}_{2^m}$ with $u^3 = 0$.

The paper is structured in the following manner. Section \ref{sec2} provides a summary of preliminary notations and results.
%Based on the generators of ideals of $\mathcal{S}$, cyclic codes of length $2^\varsigma$ over $\mathcal{R}$ are classified into eight different types. 
In Section \ref{sec3}, we calculate the Lee distance of cyclic codes of length $2^\varsigma$ over $\mathcal{R}$ of seven such types. 
%Section \ref{sec4} concludes the paper.

\section{Preliminaries}\label{sec2}
Let $p$ be a prime number and $m$ be a positive integer. Let $\mathcal{R} = \mathbb{F}_{2^m} + u\mathbb{F}_{2^m} + u^2\mathbb{F}_{2^m}$ ($u^3 = 0$)
 and $\mathcal{S}=\frac{\mathcal{R}[x]}{\langle x^{2^\varsigma}-1 \rangle}$. Clearly, $\mathcal{R}$ is a local ring
with maximal ideal $ \langle u\rangle = u\mathbb{F}_{2^m}$. Also, a polynomial $f(x)$  of degree less than $n$ in $\mathcal{R}$ can be uniquely represent as $f(x)=\sum \limits_{\ell=0}^{n-1}a_{\ell}(x+1)^\ell+u\sum \limits_{\ell=0}^{n-1}b_\ell(x+1)^\ell+u^2\sum \limits_{\ell=0}^{n-1}c_\ell(x+1)^\ell$, where $a_\ell,b_\ell$ and $c_\ell \in \mathbb{F}_{2^m}$.

A code $\mathcal{C}$ of length $n$ over $\mathcal{R}$ is a nonempty subset of $\mathcal{R}^n$. An element of $\mathcal{C}$ is called a codeword. $\mathcal{C}$ is called a  linear code over $\mathcal{R}$ if $\mathcal{C}$ is an $\mathcal{R}$-submodule of $\mathcal{R}^n$. Let $\alpha$ be a unit of $\mathcal{R}$.  The $\alpha$-constacyclic shift $\sigma_{\alpha}$ on $\mathcal{R}^n$ is  defined by $\sigma_{\alpha}(c_0,  c_1,\ldots,  c_{n-1}) = (\alpha c_{n-1},  c_0,\ldots, c_{n-2}).$ A code $\mathcal{C}$ is said to be $\alpha$-constacyclic if $\mathcal{C}$ is closed under the operator $\sigma_{\alpha}$. If  $\alpha$ is equal to 1(or -1), then the $\alpha$-constacyclic codes are referred to as cyclic (or negacyclic) codes. It is well-known that constacyclic codes are precisely the ideals in a quotient ring   \cite{macwilliams1977theory,huffman2010fundamentals}.
    \begin{proposition}\label{1}
        A linear code $\mathcal{C}$ of length $n$ over $\mathcal{R}$ is an $\alpha$-constacyclic if and only if $\mathcal{C}$ is an ideal of $\frac{\mathcal{R}[x]}{\langle x^n-\alpha \rangle}$.
    \end{proposition}

\begin{definition}\cite{lee1958some}
    The Lee weight, denoted by $w_L$, over $\mathbb{F}_{2}$ is defined as $wt_L(0) =0$ and $wt_L(1) =1$.
\end{definition}
\begin{definition}\cite{lidl1997finite}
   For $x \in \mathbb{F}_{2^m}$, the trace $Tr(x)$ 
of $x$ over $\mathbb{F}_{2}$ is defined by $Tr(x) =x +x^2+x^{2^2}+\cdots+x^{2^{m-1}}$. A basis $\mathcal{B}=\{\zeta_1, \zeta_2,\ldots, \zeta_m\}$ of $\mathbb{F}_{2^m}$ over $\mathbb{F}_{2}$ is called a trace orthogonal basis (TOB) if
\begin{equation*}
    Tr(\zeta_{i}\zeta_j)=
    \begin{cases}
    1&\text{if}\quad {i}=j,\\
    0&\text{if}\quad {i}\neq j.
\end{cases}
\end{equation*}
    
\end{definition}  
\begin{theorem}\cite{lempel1975matrix}
   $\mathbb{F}_{2^m}$ has a trace orthogonal basis over $\mathbb{F}_{2}$.
\end{theorem}

%\cite{betsumiya2004type}
Let $\mathcal{B}=\{\zeta_1, \zeta_2,\ldots, \zeta_m\}$ be a trace orthogonal basis of $\mathbb{F}_{2^m}$ over $\mathbb{F}_{2}$. Any element $x\in \mathbb{F}_{2^m}$ can be uniquely written as $x=\sum\limits_{i=1}^{m}x_i\zeta_i$, where $x_i\in \mathbb{F}_{2}$ for all $i$. The Lee weight of an element $x$ with respect to $\mathcal{B}$ is given by $x=\sum\limits_{i=1}^{m}wt_L(x_i)$. The Lee weight $wt^{\mathcal{B}}_L(v)$ of a vector $v \in \mathbb{F}^n_{2^m}$ with respect to $\mathcal{B}$ is defined as the sum of Lee weights of its components. The Lee distance $d^{\mathcal{B}}_L(\mathcal{C})$ of a non-zero linear code $\mathcal{C}$ over $\mathbb{F}_{2^m}$ with respect to $\mathcal{B}$ is defined as the minimum of Lee weights of non-zero elements of $\mathcal{C}$ with respect to $\mathcal{B}$. For the zero code, it is defined as zero.

Any element of $\mathcal{R}$ is of the form $a +ub+u^2c$, where $a, b, c \in\mathbb{F}_{2^m}$. The Lee weight $wt^{\mathcal{B}}_L(a +ub+u^2c)$ with respect to trace orthogonal basis $\mathcal{B}$ of $\mathbb{F}_{2^m}$ is defined as $wt^{\mathcal{B}}_L(a +ub+u^2c) =wt^{\mathcal{B}}_L(a+b+c, b+c, b) =wt^{\mathcal{B}}_L(a +b+c)+wt^{\mathcal{B}}_L(b+c) +wt^{\mathcal{B}}_L(b) $. In the same way as above, we define the Lee weight $wt^{\mathcal{B}}_L(v)$ of a vector $v\in \mathcal{R}^n$ and the Lee distance $d^{\mathcal{B}}_L(\mathcal{C})$ of a linear code $\mathcal{C}$ of over $\mathcal{R}$ with respect to $\mathcal{B}$. Any element $\wp(x)$ in $ \mathcal{R}[x]$ of degree less than $n$ can be uniquely written as $\wp(x) =a_0+a_1x +\cdots+a_{n-1}x^{n-1}$ for some $a_0, a_1,\ldots, a_{n-1}\in \mathcal{R}$. We define $wt_L(\wp(x)) =wt_L(a_0) +wt_L(a_1) +\cdots+wt_L(a_{n-1})$.
     
From \cite{dinh2021lee1,dinh2008linear}, recall the structure, Hamming distance and Lee distances of cyclic codes of length $2^\varsigma$ over $\mathbb{F}_{2^m}$.

\begin{theorem}\label{thm2}\cite{dinh2008linear} 
    Cyclic codes of length $2^\varsigma$ over $\mathbb{F}_{2^m}$, i.e., the ideals of the ring $\frac{\mathbb{F}_{2^m}[x]}{\langle  x^{2^\varsigma}-1 \rangle}$ are $\langle  (x+1)^\ell \rangle)$, where $0\leq \ell \leq 2^\varsigma$.
\end{theorem}
\begin{theorem}\label{thm3}\cite{dinh2008linear} 
    The Hamming distance of cyclic codes of length $2^\varsigma$ over $\mathbb{F}_{2^m}$ is given by
    \begin{center}
			$d_H(\langle  (x+1)^\ell \rangle)=$
			$\begin{cases}
				1  & \text{if}\quad \ell=0,\\
				2  &\text{if}\quad 1\leq \ell \leq 2^{\varsigma-1}, \\
                2^{\gamma+1} &\text{if}\quad 2^\varsigma-2^{\varsigma-\gamma}+1 \leq \ell \leq 2^\varsigma-2^{\varsigma-\gamma}+2^{\varsigma-\gamma-1},\quad \text{where}\quad 1\leq \gamma \leq \varsigma-1,\\
                0  & \text{if}\quad \ell=2^\varsigma.
			\end{cases}$
		\end{center}
\end{theorem}
\begin{theorem}\label{thm4}\cite{dinh2021lee1}
    The Lee distance of cyclic codes of length $2^\varsigma$ over $\mathbb{F}_{2^m}$ is given by
    \begin{center}
	$d_L(\langle  (x+1)^\ell \rangle)=$
		$\begin{cases}
			1  & \text{if}\quad \ell=0,\\
			2  &\text{if}\quad 1\leq \ell \leq 
                2^{\varsigma-1}, \\
                2^{\gamma+1} &\text{if}\quad 2^\varsigma-2^{\varsigma-\gamma}+1 \leq \ell \leq 2^\varsigma-2^{\varsigma-\gamma}+2^{\varsigma-\gamma-1},\quad \text{where}\quad 1\leq \gamma \leq \varsigma-1,\\
                0  & \text{if}\quad \ell=2^\varsigma.
			\end{cases}$
		\end{center}
\end{theorem}
\section{Lee distance of cyclic codes of length $2^\varsigma$ over $\mathcal{R}$}\label{sec3}
We start by reviewing the cyclic codes of length $2^\varsigma$ over $\mathcal{R}$ and their structures from \cite{laaouine2021complete}.

\begin{theorem}\cite{laaouine2021complete}\label{thm1}
Cyclic codes of length $2^\varsigma$ over $\mathcal{R}$, i.e ideals of the ring $\mathcal{S}$ are
    \begin{enumerate}
       \item \textbf{Type 1: }
		%\begin{center}
		$\langle 0 \rangle $,  $\langle 1 \rangle $.
		%\end{center}
	\item \textbf{Type 2:}
            %\begin{center}
			$\mathcal{C}_2=\langle u^2(x+1)^\ell \rangle $,
		%\end{center}
  where $0\leq  \ell \leq 2^\varsigma-1$.
	\item \textbf{Type 3:}
		%\begin{center}
		$\mathcal{C}_3=\langle u(x+1)^\ell +u^2(x+1)^t z(x) \rangle $,
		%\end{center}
        where $0\leq \mathcal{L}\leq \ell \leq 2^\varsigma-1$,  $0\leq  t < \mathcal{L} $ and either $z(x)$ is 0 or $z(x)$ is a unit in $\mathcal{S}$ which can be represented as $z(x)=\sum\limits_{\kappa=0}^{\mathcal{L}-t-1}z_\kappa(x+1)^{\kappa}$ with $z_\kappa \in \mathbb{F}_{2^m}$ and $z_0\neq 0$. Here ${\mathcal{L}}$ being the smallest integer such that $u^2(x+1)^{\mathcal{L}} \in \mathcal{C}_3$ given by
            \begin{center}
			${\mathcal{L}}=$
			$\begin{cases}
				\ell  & \text{if}\quad z(x)=0,\\
				min\{\ell, 2^\varsigma +t-\ell\}  &\text{if}\quad z(x)\neq 0.
			\end{cases}$
		\end{center}
	\item \textbf{Type 4: }
		%\begin{center}
		$\mathcal{C}_4=\langle u(x+1)^\ell +u^2(x+1)^t z(x), u^2 (x+1)^{\mu} \rangle $,
		%\end{center}
        where $0\leq  \mu < \mathcal{L}\leq \ell \leq 2^\varsigma-1$, $0\leq  t < \mu $ and either $z(x)$ is 0 or $z(x)$ is a unit in $\mathcal{S}$ which can be represented as $z(x)=\sum\limits_{\kappa=0}^{\mu-t-1}z_\kappa(x+1)^{\kappa}$ with $z_\kappa \in \mathbb{F}_{2^m}$ and $z_0\neq 0$. Here ${\mathcal{L}}$ being the smallest integer such that $u^2(x+1)^{\mathcal{L}} \in \mathcal{C}_3$.
        \item \textbf{Type 5:}
            %\begin{center}
			$\mathcal{C}_5=\langle (x+1)^{\alpha}+u(x+1)^{\mathfrak{T}_1} z_1(x)+u^2(x+1)^{\mathfrak{T}_2}z_2(x) \rangle $,
		%\end{center}
         where $0<  \mathcal{V} \leq \mathcal{U}\leq \alpha \leq 2^\varsigma-1$, $0\leq  \mathfrak{T}_1 < \mathcal{U} $, $0\leq  \mathfrak{T}_2 < \mathcal{V} $ and  $z_1(x)$ is either 0 or a unit in $\mathcal{S}$ which can be represented as $z_1(x)=\sum\limits_{\kappa=0}^{\mathcal{U}-\mathfrak{T}_1-1}a_\kappa(x+1)^{\kappa}$ with $a_\kappa \in \mathbb{F}_{2^m}$ and $a_0\neq 0$ and $z_2(x)$ is either 0 or a unit in $\mathcal{S}$ which can be represented as $z_2(x)=\sum\limits_{\kappa=0}^{\mathcal{V}-\mathfrak{T}_2-1}b_\kappa(x+1)^{\kappa}$ with $b_\kappa \in \mathbb{F}_{2^m}$ and $b_0\neq 0$. Here $\mathcal{U}$ is the smallest integer such that $u(x+1)^\mathcal{U}+u^2g(x) \in \mathcal{C}_5$, for some $g(x) \in \mathcal{S}$ given by
             \begin{center}
			$\mathcal{U}=$
			$\begin{cases}
				\alpha  & \text{if}\quad z_1(x)=0,\\
				min\{\alpha, 2^\varsigma +\mathfrak{T}_1-\alpha\}  &\text{if}\quad z_1(x)\neq 0.
			\end{cases}$
		\end{center}
  and $\mathcal{V}$ is the smallest integer such that $u^2(x+1)^{\mathcal{V}} \in \mathcal{C}_5$ given by
  \begin{center}
			$\mathcal{V}=$
			$\begin{cases}
				\alpha  & \text{if}\quad z_1(x)=z_2(x)=0,\\
				min\{\alpha, 2^\varsigma +\mathfrak{T}_2-\alpha\}  &\text{if}\quad z_1(x)=0\quad \text{and}\quad z_2(x)\neq 0 ,\\
                min\{\alpha, 2^\varsigma +\mathfrak{T}_1-\alpha\}  &\text{if}\quad z_1(x)\neq 0.
			\end{cases}$
		\end{center}
        \item \textbf{Type 6:} $\mathcal{C}_6=\langle (x+1)^{\alpha}+u(x+1)^{\mathfrak{T}_1} z_1(x)+u^2(x+1)^{\mathfrak{T}_2}z_2(x), u^2 (x+1)^{\omega}  \rangle $, where $0\leq \omega < \mathcal{V} \leq \mathcal{U}\leq \alpha \leq 2^\varsigma-1$, $0\leq  \mathfrak{T}_1 < \mathcal{U} $, $0\leq  \mathfrak{T}_2 < \omega $ and  $z_1(x)$ is either 0 or a unit in $\mathcal{S}$ which can be represented as $z_1(x)=\sum\limits_{\kappa=0}^{\mathcal{U}-\mathfrak{T}_1-1}a_\kappa(x+1)^{\kappa}$ with $a_\kappa \in \mathbb{F}_{2^m}$ and $a_0\neq 0$ and $z_2(x)$ is either 0 or a unit in $\mathcal{S}$ which can be represented as $z_2(x)=\sum\limits_{\kappa=0}^{\omega-\mathfrak{T}_2-1}b_\kappa(x+1)^{\kappa}$ with $b_\kappa \in \mathbb{F}_{2^m}$ and $b_0\neq 0$. Here $\mathcal{U}$ is the smallest integer such that $u(x+1)^\mathcal{U}+u^2g(x) \in \mathcal{C}_5$, for some $g(x) \in \mathcal{S}$ and $\mathcal{V}$ is the smallest integer such that $u^2(x+1)^{\mathcal{V}} \in \mathcal{C}_5$.
        \item \textbf{Type 7:} $\mathcal{C}_7=\langle (x+1)^{\alpha}+u(x+1)^{\mathfrak{T}_1} z_1(x)+u^2(x+1)^{\mathfrak{T}_2}z_2(x), u(x+1)^{\beta}+u^2 (x+1)^{\mathfrak{T}_3}  z_3(x) \rangle $, where $0\leq \mathcal{W} \leq \beta <  \mathcal{U} \leq \alpha \leq 2^\varsigma-1$, $0\leq  \mathfrak{T}_1 < \beta $, $0\leq  \mathfrak{T}_2 < \mathcal{W} $, $0\leq  \mathfrak{T}_3 < \mathcal{W} $ and  $z_1(x)$ is either 0 or a unit in $\mathcal{S}$ which can be represented as $z_1(x)=\sum\limits_{\kappa=0}^{\beta-\mathfrak{T}_1-1}a_\kappa(x+1)^{\kappa}$ with $a_\kappa \in \mathbb{F}_{2^m}$ and $a_0\neq 0$, $z_2(x)$ is either 0 or a unit in $\mathcal{S}$ which can be represented as $z_2(x)=\sum\limits_{\kappa=0}^{\mathcal{W}-\mathfrak{T}_2-1}b_\kappa(x+1)^{\kappa}$ with $b_\kappa \in \mathbb{F}_{2^m}$ and $b_0\neq 0$ and $z_2(x)$ is either 0 or a unit in $\mathcal{S}$ which can be represented as $z_3(x)=\sum\limits_{\kappa=0}^{\mathcal{W}-\mathfrak{T}_3-1}c_\kappa(x+1)^{\kappa}$ with $c_\kappa \in \mathbb{F}_{2^m}$ and $c_0\neq 0$. Here $\mathcal{U}$ is the smallest integer such that $u(x+1)^\mathcal{U}+u^2g (x) \in \mathcal{C}_5$, for some $g(x) \in \mathcal{S}$ and  $\mathcal{W}$ is the smallest integer such that $u^2(x+1)^\mathcal{W} \in \mathcal{C}_7$ given by
            \begin{center}
			$\mathcal{W}=$
			$\begin{cases}
				\beta  & \text{if}\quad z_1(x)=z_2(x)=z_3(x)=0 ,\\
                & \qquad \text{or}\quad  z_1(x)\neq 0 \quad \text{and} \quad z_3(x)=0,\\
				min\{\beta, 2^\varsigma +\mathfrak{T}_2-\alpha\}  &\text{if}\quad z_1(x)=z_3(x)=0\quad \text{and}\quad z_2(x)\neq 0 ,\\
                min\{\beta, 2^\varsigma +\mathfrak{T}_3-\beta\}  &\text{if}\quad z_1(x)=z_2(x)=0,z_3(x)\neq 0 ,\\
                & \qquad \text{or}\quad  z_1(x)\neq 0 \quad \text{and}\quad  z_3(x)\neq 0,\\
                 min\{\beta, 2^\varsigma +\mathfrak{T}_2-\alpha,  2^\varsigma +\mathfrak{T}_3-\beta\}  &\text{if}\quad z_1(x)=0,z_2(x)\neq 0, z_3(x)\neq 0.
			\end{cases}$
		\end{center}

        \item \textbf{Type 8:} $\mathcal{C}_8=\langle (x+1)^{\alpha}+u(x+1)^{\mathfrak{T}_1} z_1(x)+u^2(x+1)^{\mathfrak{T}_2}z_2(x), u(x+1)^{\beta}+u^2 (x+1)^{\mathfrak{T}_3}  z_3(x),u^2(x+1)^{\omega} \rangle $, where $0\leq \omega< \mathcal{W} \leq \mathcal{L}_1 \leq \beta <  \mathcal{U} \leq \alpha \leq 2^\varsigma-1$, $0\leq  \mathfrak{T}_1 < \beta $, $0\leq  \mathfrak{T}_2 < \omega $, $0\leq  \mathfrak{T}_3 < \omega $ and  $z_1(x)$ is either 0 or a unit in $\mathcal{S}$ which can be represented as $z_1(x)=\sum\limits_{\kappa=0}^{\beta-\mathfrak{T}_1-1}a_\kappa(x+1)^{\kappa}$ with $a_\kappa \in \mathbb{F}_{2^m}$ and $a_0\neq 0$, $z_2(x)$ is either 0 or a unit in $\mathcal{S}$ which can be represented as $z_2(x)=\sum\limits_{\kappa=0}^{\omega-\mathfrak{T}_2-1}b_\kappa(x+1)^{\kappa}$ with $b_\kappa \in \mathbb{F}_{2^m}$ and $b_0\neq 0$ and $z_2(x)$ is either 0 or a unit in $\mathcal{S}$ which can be represented as $z_3(x)=\sum\limits_{\kappa=0}^{\omega-\mathfrak{T}_3-1}c_\kappa(x+1)^{\kappa}$ with $c_\kappa \in \mathbb{F}_{2^m}$ and $c_0\neq 0$. Here $\mathcal{U}$ is the smallest integer such that $u(x+1)^\mathcal{U}+u^2g (x) \in \mathcal{C}_5$, for some $g(x) \in \mathcal{S}$  and $\mathcal{W}$ is the smallest integer such that $u^2(x+1)^{\mathcal{W}} \in \mathcal{C}_7$ and ${\mathcal{L}}_1$ is the smallest integer such that $u^2(x+1)^{\mathcal{L}_1} \in \langle u(x+1)^{\beta} +u^2(x+1)^{\mathfrak{T}_3} z_3(x) \rangle $ given by 
            \begin{center}
			${\mathcal{L}}_1=$
			$\begin{cases}
				\beta & \text{if}\quad z_3(x)=0,\\
				min\{\beta, 2^\varsigma +\mathfrak{T}_3-\beta\}  &\text{if}\quad z_3(x)\neq 0.
			\end{cases}$
		\end{center}
    \end{enumerate}

\end{theorem}
By considering notations as in Theorem \ref{thm1}, now we will compute the Lee distances of the cyclic codes of length $2^\varsigma$ over $\mathcal{R}$. 
\subsection{Type 1:}
For Type 1 ideals, we have 
$d_L(\langle 0 \rangle)=0$ and $d_L(\langle 1 \rangle)=1$.

\subsection{Type 2:}
\begin{theorem}\label{thm5}  
    Let $\mathcal{C}_2=\langle u^2(x+1)^\ell \rangle $, where $0\leq  \ell \leq 2^\varsigma-1$. Then
    \begin{center}
	$d_L(\mathcal{C}_2)=$
	$\begin{cases}
		2  & \text{if}\quad \ell=0,\\
		4  &\text{if}\quad 1\leq \ell \leq 2^{\varsigma-1},\\
            2^{\gamma+2} &\text{if}\quad 2^\varsigma-2^{\varsigma-\gamma}+1 \leq \ell \leq 2^\varsigma-2^{\varsigma-\gamma}+2^{\varsigma-\gamma-1},\quad \text{where}\quad 1\leq \gamma \leq \varsigma-1
	\end{cases}$
    \end{center}
    \end{theorem}
\begin{proof}
    Let us fix a TOB B of $\mathbb{F}_{2^m}$ over $\mathbb{F}_{2}$. Let $\langle (x+1)^\ell \rangle$ be ideals of $\frac{\mathbb{F}_{2^m}[x]}{\langle x^{2^\varsigma}-1 \rangle}$, where $0\leq \ell \leq2^\varsigma-1$. Let $\wp(x)\in\langle (x+1)^\ell \rangle$. Then $wt^{\mathcal{B}}_L(u^2\wp(x))=wt^{\mathcal{B}}_L(\wp(x))+wt^{\mathcal{B}}_L(\wp(x))=2wt^{\mathcal{B}}_L(\wp(x))$. Therefore $d^{\mathcal{B}}_L(\mathcal{C}_2)=2d_L(\langle (x+1)^\ell \rangle)$. Proof follows from Theorem \ref{thm4}. It is clear that the Lee distance of $\mathcal{C}_2$ is independent of the choice of a TOB.
\end{proof}

\subsection{Type 3:}
\begin{theorem}\cite{dinh2021hamming}\label{thm6} 
Let  $\mathcal{C}_3=\langle u(x+1)^\ell +u^2(x+1)^t z(x) \rangle $, where $0\leq \mathcal{L}\leq \ell \leq 2^\varsigma-1$,  $0\leq  t < \mathcal{L} $ and either $z(x)$ is 0 or $z(x)$ is a unit in $\mathcal{S}$.  Then
     %\begin{center}
       $d_H(\mathcal{C}_3)= d_H(\langle (x+1)^{\mathcal{L}} \rangle)$.
       % =
       %  $\begin{cases}
       %      1  &\text{if}\quad 0\leq  L \leq 2^\varsigma,\\
       %      2^{\gamma +1}& \text{if} \quad 2p^s-p^{s-\gamma }+(\Gamma-1)\omega+1\leq \mathcal{L} \leq 2p^s-p^{s-\gamma}+\Gamma \omega, \\
       %  \end{cases}$
   % \end{center}
    %where $\omega=p^{s-\gamma-1}$, $ 1\leq \Gamma \leq p-1 $ and $ 0\leq \gamma \leq s-1 $.
\end{theorem}

\begin{proposition}\label{prop1}
    Let $\mathcal{C}_3$ be a cyclic code of length $2^\varsigma$ over $\mathcal{R}$ and ${\mathcal{L}}$ be the smallest integer such that $u^2(x+1)^{\mathcal{L}} \in \mathcal{C}_3$. Then  $d_H(\mathcal{C}_3)\leq d_L(\mathcal{C}_3)\leq 2 d_H(\langle (x+1)^{\mathcal{L}}\rangle)$, where $\langle (x+1)^{\mathcal{L}}\rangle$ is an ideal of $\frac{\mathbb{F}_{2^m}[x]}{\langle x^{2^\varsigma}-1 \rangle}$.
\end{proposition}
\begin{proof}
    $d_H(\mathcal{C}_3)\leq d_L(\mathcal{C}_3)$ is obvious. We have $\langle u^2(x+1)^{\mathcal{L}} \rangle \subseteq \mathcal{C}_3$. Then $d_L(\mathcal{C}_3)\leq d_L(\langle u^2(x+1)^{\mathcal{L}} \rangle)$. The result follows from Theorem \ref{thm5}.
\end{proof}

\subsubsection{If z(x)=0}%
\begin{theorem}\label{thm7} 
     Let $\mathcal{C}^1_3=\langle u(x+1)^\ell  \rangle $, where $0\leq  \ell \leq 2^\varsigma-1$. Then
    \begin{center}
	$d_L(\mathcal{C}^1_3)=$
	$\begin{cases}
            3&\text{if}\quad \ell =0,\\
		6  &\text{if}\quad 1\leq \ell \leq 2^{\varsigma-1},\\
            3 \cdot 2^{\gamma+1} &\text{if}\quad 2^\varsigma-2^{\varsigma-\gamma}+1 \leq \ell  \leq 2^\varsigma-2^{\varsigma-\gamma}+2^{\varsigma-\gamma-1},\quad \text{where}\quad 1\leq \gamma \leq \varsigma-1.
	\end{cases}$
    \end{center}
\begin{proof}
   Let $\langle (x+1)^\ell \rangle$ be ideals of $\frac{\mathbb{F}_{2^m}[x]}{\langle x^{2^\varsigma}-1 \rangle}$, where $0\leq \ell \leq2^\varsigma-1$. Let $\wp(x)\in\langle (x+1)^\ell \rangle$. Then $wt^{\mathcal{B}}_L(u\wp(x))=wt^{\mathcal{B}}_L(\wp(x))+wt^{\mathcal{B}}_L(\wp(x))+wt^{\mathcal{B}}_L(\wp(x))=3wt^{\mathcal{B}}_L(\wp(x))$. Therefore $d^{\mathcal{B}}_L(\mathcal{C}^1_3)=3d_L(\langle (x+1)^\ell \rangle)$. Proof follows from Theorem \ref{thm4}.
\end{proof}

\end{theorem}

\subsection{If $z(x) \neq 0$ and $t \neq 0$}
\begin{theorem}\label{thm8} 
    Let $\mathcal{C}^2_3=\langle u(x+1)^\ell +u^2(x+1)^t z(x) \rangle $,
    %where $1 < \ell \leq2^\varsigma-1$. 
    where $0< \mathcal{L}\leq \ell \leq 2^\varsigma-1$,  $0<  t < \mathcal{L} $ and either $z(x)$ is 0 or $z(x)$ is a unit in $\mathcal{S}$. Then
\begin{center}
	$d_L(\mathcal{C}^2_3)=$
	$\begin{cases}
            4&\text{if}\quad 1< \ell \leq 2^{\varsigma-1},\\
		4  &\text{if}\quad 2^{\varsigma-1}+1\leq \ell \leq 
            2^{\varsigma}-1 \quad \text{with} \quad \ell \geq 2^{\varsigma-1}+t.
            % 2^{\gamma+2} &\text{if}\quad 2^\varsigma-2^{\varsigma-\gamma}+1 \leq \ell  \leq 2^\varsigma-2^{\varsigma-\gamma}+2^{\varsigma-\gamma-1}\quad  \text{with} \quad  \ell \leq 2^{\varsigma-1}+\frac{t}{2},\\
            % &\quad\quad \text{where}\quad 1\leq \gamma \leq \varsigma-1.
            % \geq 2^{\gamma+1} &\text{if}\quad 2^\varsigma-2^{\varsigma-\gamma}+1 \leq \ell  \leq 2^\varsigma-2^{\varsigma-\gamma}+2^{\varsigma-\gamma-1},\\
            % &\qquad \text{where}\quad 1\leq \gamma \leq \varsigma-1.
	\end{cases}$
    \end{center}
And if  $ 2^\varsigma-2^{\varsigma-\gamma}+1 \leq \ell  \leq 2^\varsigma-2^{\varsigma-\gamma}+2^{\varsigma-\gamma-1}$ with $\ell \leq 2^{\varsigma-1}+\frac{t}{2}$ then  $2^{\gamma+1}\leq d_L(\mathcal{C}^2_3)\leq 2^{\gamma+2}$, where $1\leq \gamma \leq \varsigma-1$.   
\end{theorem}
\begin{proof}
Let $\mathcal{B}=\{ \zeta_1,\zeta_2,\ldots, \zeta_m\}$ be a TOB of $\mathbb{F}_{2^m}$ over $\mathbb{F}_2.$
    \begin{enumerate}
        \item \textbf{Case 1:} Let $1< \ell \leq 2^{\varsigma-1}$. By Theorem \ref{thm1}, ${\mathcal{L}}=\ell$, By Theorem \ref{thm6} and Theorem \ref{thm3}, $d_H(\mathcal{C}^2_3)=2$. Hence $2\leq d_L(\mathcal{C}^2_3)$.

        First, we show that there is no codeword in $\mathcal{C}^2_3$ of Lee weight 2. Let $\chi (x) \in \mathcal{C}^2_3$ with $wt^{\mathcal{B}}_L(\chi (x))=2$. Since $wt^{\mathcal{B}}_L(\chi (x)) \geq wt_H(\chi(x))$ and $d_H(\mathcal{C}^2_3)=2$, $wt_H(\chi (x))=2$. Suppose $\chi (x)=\lambda_1x^i+\lambda_2x^j$, where $\lambda_1, \lambda_2 \in \mathcal{R} \textbackslash \{0\}$, $0\leq i<j$. Since $\mathcal{S}$ is a local ring with maximal ideal $\langle x-1,u \rangle$, $a(x)$ is not a unit in $\mathcal{S}$ if and only if it is mapped to $0$ under the natural reduction mod $\langle x-1,u \rangle$. Thus, $x^i,x^j$ are units in $\mathcal{S}$.
        \begin{enumerate}
            \item If $\lambda_1$ is unit  and $\lambda_2$ is non-unit in $\mathcal{R}$ then $\lambda_1 x^i$ is unit and $\lambda_2x^j$ is non-unit in $\mathcal{S}$. Since $\mathcal{S}$ is a local ring, $\chi (x)$ is a unit in $\mathcal{S}$, which is not possible.
            \item If $\lambda_1$ and $\lambda_2$ are non-units in $\mathcal{R}$ and since $\mathcal{R}$ is a local ring with the maximal ideal $\langle u \rangle$, $\lambda_1, \lambda_2 \in \langle u \rangle$. Then $wt^{\mathcal{B}}_L(\lambda_1),wt^{\mathcal{B}}_L(\lambda_2)\geq 3$ and $wt^{\mathcal{B}}_L(\chi (x))\geq 6$, which is not possible.
            \item Let both $\lambda_1$ and $\lambda_2$ are units in $\mathcal{R}$. We have $\lambda x^i(1+\lambda_1 ^{-1}\lambda_2 x^{j-i})\in \mathcal{C}^2_3$. Since $\lambda_1 x^i$ is a unit in $\mathcal{S}$, $(1+\lambda_1 ^{-1}\lambda_2 x^{j-i})\in \mathcal{C}^2_3$. Therefore, we can write 
            \begin{equation}\label{eqn1}
                (1+\lambda_1 ^{-1}\lambda_2 x^{j-i})=\big[ u(x+1)^\ell +u^2(x+1)^t z(x) \big] \phi(x)
            \end{equation}
             for some $\phi(x)\in \mathcal{S}$. As $t \geq 1$, by substituting $x=1$ in Equation \ref{eqn1}, we get $1+\lambda_1 ^{-1}\lambda_2 =0$, that is, $\lambda_1 ^{-1}\lambda_2=1$. Therefore $(1+ x^{j-i})\in \mathcal{C}^2_3$. We can write $i-j=2^wr$, where $1\leq w \leq \varsigma-1$ and $r$ is odd. Then 
             \begin{equation*}
                 (1+ x^{2^wr})=(1+ x^{2^w})\Big[1+x^{2^w}+(x^{2^w})^2+\cdots +(x^{2^w})^{r-1}\Big].
             \end{equation*}
             Since $\Big[1+x^{2^w}+(x^{2^w})^2+\cdots +(x^{2^w})^{r-1}\Big]$ maps to $1\in \mathbb{F}_{2^m}$ under the natural reduction mod $\langle x-1,u \rangle$, $\Big[1+x^{2^w}+(x^{2^w})^2+\cdots +(x^{2^w})^{r-1}\Big]$ is a unit in $\mathcal{S}$. Therefore $(1+ x^{2^w}) \in \mathcal{C}^2_3$. Also, $(1+ x)^{2^{\varsigma-1}}\in \langle (1+ x)^{2^{w}}\rangle \subseteq \mathcal{C}^2_3$. Thus, 
             \begin{align*}
                  (1+ x)^{2^{\varsigma-1}}=&\Big[u(x+1)^\ell +u^2(x+1)^t z(x)\Big]\Big[\varphi_1(x)+u\varphi_2(x)+u^2\varphi_3(x)\Big]\\
                  =&u(x+1)^\ell \varphi_1(x)+u^2\Big[(x+1)^t z(x)\varphi_1(x)+(x+1)^\ell \varphi_2(x)\Big]
             \end{align*}
             for some $\varphi_1(x),\varphi_2(x), \varphi_3(x) \in \frac{\mathbb{F}_{p^m}[x]}{\langle x^{2^\varsigma}-1 \rangle}$. Then $(1+ x)^{2^{\varsigma-1}}=0$, which is not possible. Thus, there is no codeword in $\mathcal{C}^2_3$ of Lee weight 2.

        \end{enumerate}
Now we show that there is no codeword in $\mathcal{C}^2_3$ of Lee weight 3. Let $\chi (x)\in \mathcal{C}^2_3$ with $wt^{\mathcal{B}}_L(\chi (x))=3$. From the above discussion $\chi (x)=\lambda_1 x^{k_1}+\lambda_2 x^{k_2}+\lambda_3 x^{k_3}$, where $\lambda_1, \lambda_2, \lambda_3 \in \mathcal{R} \textbackslash \{0\}$, $0\leq k_1<k_2<k_3$. Then we must have $wt^{\mathcal{B}}_L(\lambda_i)=1$ for all $i=1,2$ and $3$. That is $\lambda_i=\zeta_j$, where $\zeta_j\in \mathcal{B}$. As $\chi (x)$ is a non-unit in $\mathcal{S}$, under the natural reduction mod $\langle x-1, u \rangle$, we have $\zeta_1+\zeta_2+\zeta_3=0$. This is not possible as $\zeta_1, \zeta_2$ and $\zeta_3$ are basis elements. Thus, there is no codeword in $\mathcal{C}^2_3$ of Lee weight 3. 
        
        A codeword $\wp(x)=\zeta_1\Big[u(x+1)^\ell +u^2(x+1)^t z(x) \Big]u(x+1)^{2^{\varsigma-1}-\ell}=\zeta_1u^2(x+1)^{2^{\varsigma-1}}\in \mathcal{C}^2_3$ with $wt^{\mathcal{B}}_L(\wp(x))=4$. Thus, $d_L(\mathcal{C}^2_3)=4$.
        \item \textbf{Case 2:} Let $2^{\varsigma-1}+1\leq \ell \leq 2^\varsigma-1$. %with $\ell \geq 2^{\varsigma-1}+t$.
        
        \begin{enumerate}
            \item \textbf{Subcase i:} If $\ell \geq 2^{\varsigma-1}+t$ we have $2^\varsigma-\ell +t \leq2^{\varsigma-1}$ and ${\mathcal{L}}=2^\varsigma-\ell +t$. By Theorem \ref{thm3} and Theorem \ref{thm6}, $d_H(\mathcal{C}^2_3)=2$. Thus, $2\leq d_L(\mathcal{C}^2_3)$. Following as in the above case, there exist no codewords of the form $\lambda_1x^i+\lambda_2x^j$ in $\mathcal{C}^2_3$ with $\lambda_1$ or $\lambda_2$ non-unit in $\mathcal{R}$. If $\lambda x^i+\lambda_2 x^j \in \mathcal{C}^2_3$ with $\lambda$ and $\lambda_2$ are units in $\mathcal{R}$, following as in the above case, we get $(1+ x)^{2^{\varsigma-1}}\in \mathcal{C}^2_3$. Thus,
             \begin{align*}
                  (1+ x)^{2^{\varsigma-1}}=&\Big[u(x+1)^\ell +u^2(x+1)^t z(x)\Big]\Big[f^{\prime}_1(x)+uf^{\prime}_2(x)+u^2f^{\prime}_3(x)\Big]\\
                  =&u(x+1)^\ell f^{\prime}_1(x)+u^2\Big[(x+1)^t z(x)f^{\prime}_1(x)+(x+1)^\ell f^{\prime}_2(x)\Big]
             \end{align*}
             for some $f^{\prime}_1(x),f^{\prime}_2(x), f^{\prime}_3(x) \in \frac{\mathbb{F}_{p^m}[x]}{\langle x^{2^\varsigma}-1 \rangle}$. Then %$f^{\prime}_1(x)=f^{\prime}_2(x)=f^{\prime}_3(x)=0$. i.e., 
             $(1+ x)^{2^{\varsigma-1}}=0$, which is not possible. Thus, there exists no codeword of Lee weight 2. Also, following as in the above case there exists no codewords of the form $\chi (x)=\lambda_1 x^{k_1}+\lambda_2 x^{k_2}+\lambda_3 x^{k_3} \in \mathcal{C}^2_3$, where $\lambda_1, \lambda_2, \lambda_3 \in \mathcal{R} \textbackslash \{0\}$, $0\leq k_1<k_2<k_3$ with $wt^{\mathcal{B}}_L(\chi (x))=3$. Thus, $\mathcal{C}^2_3$ has no codeword of Lee weight 3.

             A codeword $\wp(x)=u^2\zeta_1(x^{2^{\varsigma-1}}+1)=u^2\zeta_1(x+1)^{2^{\varsigma-1}}\in \langle u^2(x+1)^{2^\varsigma-\ell +t}\rangle \subseteq \mathcal{C}^2_3$ with $wt^{\mathcal{B}}_L(\wp(x))=4$. Thus, $d_L(\mathcal{C}^2_3)=4$.
             
             \item \textbf{Subcase ii:} Let $\ell \leq 2^{\varsigma-1}+t$. If $\ell \leq 2^{\varsigma-1}+\frac{t}{2}$ then ${\mathcal{L}}=\ell$.  If $2^\varsigma-2^{\varsigma-\gamma}+1 \leq \ell  \leq 2^\varsigma-2^{\varsigma-\gamma}+2^{\varsigma-\gamma-1},$ where $ 1\leq \gamma \leq \varsigma-1$, by Theorem \ref{thm3} and Theorem \ref{thm6}, $d_H(\langle (x+1)^{\mathcal{L}}\rangle)=2^{\gamma+1}$. Thus, $2^{\gamma+1}\leq d_L(\mathcal{C}^2_3)\leq 2^{\gamma+2}$.

        \end{enumerate}
    \end{enumerate} 
\end{proof}
%\newpage
\subsection{If $z(x) \neq 0$ and $t=0$}  
\begin{theorem}\label{thm9} 
    Let $\mathcal{C}^3_3=\langle u(x+1)^\ell +u^2 z(x) \rangle $, where $1 \leq \ell \leq2^\varsigma-1$. Then 
    %\begin{equation*}
        $d_L(\mathcal{C}^3_3)=4.$
    %\end{equation*}
\end{theorem}
\begin{proof}
    Let $\mathcal{B}=\{ \zeta_1,\zeta_2,\ldots, \zeta_m\}$ be a TOB of $\mathbb{F}_{2^m}$ over $\mathbb{F}_2.$ Let ${\mathcal{L}}$ be the smallest integer such that $u^2(x+1)^{\mathcal{L}} \in \mathcal{C}^3_3$. By Theorem \ref{thm1}, ${\mathcal{L}}=min\{\ell, 2^\varsigma-\ell\}$. Then $1\leq \mathcal{L} \leq 2^{\varsigma-1}$. By Theorem \ref{thm3} and Theorem \ref{thm6}, $d_H(\mathcal{C}^3_3)=2$. Thus, $2\leq d_L(\mathcal{C}^3_3)\leq 4$.
   
    Following as in Theorem \ref{thm8}, there exist no codewords of the form $\lambda_1 x^{k_1}+\lambda_2 x^{k_2}$ in $\mathcal{C}^3_3$ with $\lambda_1$ or $\lambda_2$ non-unit in $\mathcal{R}$, where $\lambda_1, \lambda_2 \in \mathcal{R} \textbackslash \{0\}$, $0\leq {k_1}<{k_2}$. Let $\chi (x)=\lambda_1 x^{k_1}+\lambda_2 x^{k_2} \in \mathcal{C}^3_3$ with $\lambda_1$ and $\lambda_2$ are units in $\mathcal{R}$. Then we must have $wt^{\mathcal{B}}_L(\lambda_i)=1$ for all $i=1$ and $2$. That is $\lambda_i=\zeta_j$, where $\zeta_j\in \mathcal{B}$. As $\chi (x)$ is a non-unit in $\mathcal{S}$, under the natural reduction mod $\langle x-1, u \rangle$, we have $\zeta_1+\zeta_2=0$. Since $\zeta_1$ and $\zeta_2$ are basis elements, we get a contradiction if $\zeta_1 \neq \zeta_2$. If $\zeta_1=\zeta_2$ we get $1+x^{k_2-k_1} \in \mathcal{C}^3_3$. We can write $k_1-k_2=2^wr$, where $1\leq w \leq \varsigma-1$ and $r$ is odd. By following the same line of arguments as in case 1 of Theorem \ref{thm8}, we get that there exists no codewords of Lee weight 2. Also, following Theorem \ref{thm8}, there exist no codewords of the form $\chi (x)=\lambda_1 x^{k_1}+\lambda_2 x^{k_2}+\lambda_3 x^{k_3} \in \mathcal{C}^3_3$, where $\lambda_1, \lambda_2, \lambda_3 \in \mathcal{R} \textbackslash \{0\}$, $0\leq k_1<k_2<k_3$ with $wt^{\mathcal{B}}_L(\chi (x))=3$. Thus, $\mathcal{C}^3_3$ has no codeword of Lee weight 3. Hence $d_L(\mathcal{C}^3_3)=4$.
\end{proof}

%\newpage
\subsection{Type 4:}
\begin{theorem}\cite{dinh2021hamming}\label{thm10}
Let  $\mathcal{C}_4=\langle u(x+1)^\ell +u^2(x+1)^t z(x), u^2 (x+1)^{\mu} \rangle $, where $0\leq  \mu < \mathcal{L}\leq \ell \leq 2^\varsigma-1$, $0\leq  t < \mu $ and either $z(x)$ is 0 or $z(x)$ is a unit in $\mathcal{S}$. Then
     %\begin{center}
       $d_H(\mathcal{C}_4)= d_H(\langle (x+1)^{\mu} \rangle)$.
    %\end{center}
    
\end{theorem}

\subsection{If $z(x) = 0$}
\begin{theorem}\label{thm11} 
Let  $\mathcal{C}^1_4=\langle u(x+1)^\ell , u^2 (x+1)^{\mu} \rangle $, where $0\leq  \mu < \mathcal{L}\leq \ell \leq 2^\varsigma-1$. Then
\begin{equation*}
    d_L(\mathcal{C}^1_4)=
    \begin{cases}
        2&\text{if}\quad 1\leq \ell \leq 2^{\varsigma-1}\quad \text{with} \quad\mu=0,\\
        4&\text{if}\quad 1\leq \mu<\ell \leq 2^{\varsigma-1},\\
      %  \quad \text{with}        \quad 1\leq \mu \leq 2^{\varsigma-1},\\
        2 &\text{if}\quad 2^{\varsigma-1}+1\leq \ell \leq 
        2^{\varsigma}-1 \quad \text{with} \quad\mu=0,\\
	4 &\text{if}\quad 2^{\varsigma-1}+1\leq \ell \leq 2^{\varsigma}-1 \quad \text{with}        \quad 1\leq \mu \leq 2^{\varsigma-1}.
        % \geq 2^{\gamma+1} &\text{if}\quad 2^\varsigma-2^{\varsigma-\gamma}+1 \leq \mu < \ell  \leq 2^\varsigma-2^{\varsigma-\gamma}+2^{\varsigma-\gamma-1}\quad \text{where}\quad 1\leq \gamma \leq \varsigma-1.
	\end{cases}
    \end{equation*}
    And if  $ 2^\varsigma-2^{\varsigma-\gamma}+1 \leq\mu< \ell  \leq 2^\varsigma-2^{\varsigma-\gamma}+2^{\varsigma-\gamma-1}$ then  $2^{\gamma+1}\leq d_L(\mathcal{C}^1_4)\leq 2^{\gamma+2}$, where $1\leq \gamma \leq \varsigma-1$.  
\end{theorem}
\begin{proof}
    Let $\mathcal{B}=\{ \zeta_1,\zeta_2,\ldots, \zeta_m\}$ be a TOB of $\mathbb{F}_{2^m}$ over $\mathbb{F}_2.$ From Theorem \ref{thm10}, $d_H(\mathcal{C}^1_4)= d_H(\langle (x+1)^{\mu} \rangle)$. Following as in Theorem \ref{thm8}, we get $d_H(\langle (x+1)^{\mu}\rangle)\leq d_L(\mathcal{C}^1_4)\leq 2 d_H(\langle (x+1)^{\mu}\rangle)$. Also, since $\langle u(x+1)^{\ell}\rangle \subseteq \mathcal{C}^1_4$, $d_L(\mathcal{C}^1_4)\leq d_L(\langle u(x+1)^{\ell}\rangle)$.
    Also, since $\langle u^2 (x+1)^{\mu}\rangle \subseteq \mathcal{C}^1_4$, $d_L(\mathcal{C}^1_4)\leq d_L(\langle u^2 (x+1)^{\mu}\rangle)$.

    % \begin{enumerate}
    %     \item \textbf{Case 1:} Let $1\leq \ell \leq 2^{\varsigma-1}$. From Theorem \ref{thm7}, $d_L(\mathcal{C}^1_3)\leq 6$.
    %     \begin{enumerate}
    %         \item \textbf{Subcase i:} If $\mu \geq 0$ then by Theorem \ref{thm3}, $d_H(\langle (x+1)^{\mu}\rangle)\geq 2$. Hence $2 \leq d_L(\mathcal{C}^1_3)\leq 6$.
    %         Following Theorem \ref{thm9}, we can provethere exists no codewords of Lee weights 2 and 3. Get \varkappa (x) of Lee weight 4.}
    %         \colorbox{pink}{Need to analyze about the upper bound for Lee distance (4 or 6)}
    %         \item \textbf{Subcase ii:} If $\mu = 0$ then by Theorem \ref{thm3}, $d_H(\langle (x+1)^{\mu}\rangle)= 1$. Hence $1 \leq d_L(\mathcal{C}^1_3)\leq 6$.

    %     \end{enumerate}
    % \end{enumerate}
 \begin{enumerate}
         \item \textbf{Case 1:} Let $1\leq \ell \leq 2^{\varsigma-1}$. 
         %From Theorem \ref{thm15}, $d_L(\mathcal{C}^1_6)\leq 2$.
         \begin{enumerate}
             \item Let $\mu=0$. 
             From Theorem \ref{thm3} and Theorem \ref{thm5}, $1\leq d_L(\mathcal{C}^1_4)\leq 2$. Suppose $\chi (x)=\lambda x^j \in \mathcal{C}^1_4$, $\lambda \in \mathcal{R}$ with $wt^{\mathcal{B}}_L(\chi (x))=1$. 
             \begin{enumerate}
                 \item if $\lambda$ is a unit in $\mathcal{R}$ then $\lambda x^j$ is a unit. This is not possible.
                 \item if $\lambda$ is non-unit in $\mathcal{R}$ then $\lambda \in \langle u \rangle$ and $wt^{\mathcal{B}}_L(\lambda)\geq 3 $. Again, this is not possible.
             \end{enumerate}
             Hence $d_L(\mathcal{C}^1_4)=2$.
             \item If $1\leq \mu\leq 2^{\varsigma-1}$, by Theorem \ref{thm3} and Theorem \ref{thm5}, $2\leq d_L(\mathcal{C}^1_4)\leq 4$.
             Following as in Theorem \ref{thm9}, we get $\mathcal{C}^1_4$ has no codeword of Lee weights 2 and 3.  Hence $d_L(\mathcal{C}^1_4)=4$.
        \end{enumerate}
        \item \textbf{Case 2:} Let $2^{\varsigma-1}+1\leq \ell \leq 2^\varsigma-1$.
        \begin{enumerate}
            \item \textbf{Subcase i:} Let $\mu=0$. Then $\chi (x)=\zeta_1u^2 \in \mathcal{C}^1_4$ with $wt^{\mathcal{B}}_L(\chi (x))=2$. Hence $d_L(\mathcal{C}^1_4)=2$.
            
            \item \textbf{Subcase ii:} Let $1\leq \mu \leq 2^{\varsigma-1}$. 
            Following as in Theorem \ref{thm9}, $\mathcal{C}^1_4$ has no codeword of Lee weights 2 and 3. Hence $d_L(\mathcal{C}^1_4)=4$.
            
            \item \textbf{Subcase iii:} Let $ 2^\varsigma-2^{\varsigma-\gamma}+1 \leq \mu \leq 2^\varsigma-2^{\varsigma-\gamma}+2^{\varsigma-\gamma-1}$, where $1\leq \gamma \leq \varsigma-1$. By Theorem \ref{thm3} and Theorem \ref{thm5}, 
            $2^{\gamma+1}\leq d_L(\mathcal{C}^1_4)\leq 2^{\gamma+2}$.
            
        %     Suppose there exists a codeword of Lee weight $2^{\gamma+1}$.
        %     %%%%%%% Analyse
        %     %Following Theorem \ref{thm9}, we can provethere is no codeword in $\mathcal{C}^1_4$ of Lee weight $2^{\gamma+1}, 2^{\gamma+1}+1, 2^{\gamma+1}+2,\ldots , 2^{\gamma+2}-1$ as in Theorem } \ref{thm8}.
        
        % Hence $d_L(\mathcal{C}^1_4)=2^{\gamma+2}$.
            
        \end{enumerate}
     \end{enumerate}

\end{proof}

%\newpage
\subsection{If $z(x) \neq 0$ and $t\neq 0$}
\begin{theorem}\label{thm12} 
Let  $\mathcal{C}^2_4=\langle u(x+1)^\ell +u^2(x+1)^t z(x), u^2 (x+1)^{\mu} \rangle  $, where $1<  \mu < \mathcal{L}\leq \ell \leq 2^\varsigma-1$, $0<  t < \mu $ and $z(x)$ a unit in $\mathcal{S}$ . Then
\begin{center}
	$d_L(\mathcal{C}^2_4)=$
	$\begin{cases}
            4&\text{if}\quad 1< \mu < \ell \leq 2^{\varsigma-1},\\
            4  &\text{if}\quad 2^{\varsigma-1}+1\leq \ell \leq 
            2^{\varsigma}-1 \quad \text{with} \quad 1 < \mu 
            \leq 2^{\varsigma-1},\\
		4  &\text{if}\quad 2^{\varsigma-1}+1\leq\mu< \ell \leq 
            2^{\varsigma}-1 \quad \text{with} \quad \ell \geq 2^{\varsigma-1}+t.
            %  2^{\gamma+2} &\text{if}\quad 2^\varsigma-2^{\varsigma-\gamma}+1 \leq \mu < \ell  \leq 2^\varsigma-2^{\varsigma-\gamma}+2^{\varsigma-\gamma-1}\quad  \quad \text{with} \quad \ell \leq 2^{\varsigma-1}+\frac{t}{2},\\
            % &\quad \text{where}\quad 1\leq \gamma \leq \varsigma-1.
            %\geq 2^{\gamma+1} &\text{if}\quad 2^\varsigma-2^{\varsigma-\gamma}+1 \leq \mu < \ell  \leq 2^\varsigma-2^{\varsigma-\gamma}+2^{\varsigma-\gamma-1},\\
            %&\qquad \text{where}\quad 1\leq \gamma \leq \varsigma-1.
	\end{cases}$
    \end{center}
    And if  $ 2^\varsigma-2^{\varsigma-\gamma}+1 \leq \mu<\ell  \leq 2^\varsigma-2^{\varsigma-\gamma}+2^{\varsigma-\gamma-1}$ then  $2^{\gamma+1}\leq d_L(\mathcal{C}^2_4)\leq 2^{\gamma+2}$, where $1\leq \gamma \leq \varsigma-1$.  
\end{theorem}
\begin{proof}
    Let $\mathcal{B}=\{ \zeta_1,\zeta_2,\ldots, \zeta_m\}$ be a TOB of $\mathbb{F}_{2^m}$ over $\mathbb{F}_2.$  Following as in Theorem \ref{thm8}, we get $d_H(\langle (x+1)^{\mu}\rangle)\leq d_L(\mathcal{C}^2_4)\leq 2 d_H(\langle (x+1)^{\mu}\rangle)$. Also, since $\langle  u(x+1)^\ell +u^2(x+1)^t z(x)\rangle \subseteq \mathcal{C}^2_4$, $d_L(\mathcal{C}^2_4)\leq d_L(\langle u(x+1)^\ell +u^2(x+1)^t z(x)\rangle)$.
    \begin{enumerate}
         \item \textbf{Case 1:} Let $1<\ell \leq 2^{\varsigma-1}$. Since $1< \mu <\ell \leq 2^{\varsigma-1}$, by Theorem \ref{thm3}, $d_H(\langle (x+1)^{\mu}\rangle)=2$, Thus, $2\leq d_L(\mathcal{C}^2_4)\leq 4$.  
         %Let $\chi (x)=\lambda_1 x^{k_1}+\lambda_2 x^{k_2} \in \mathcal{C}^3_3$ with $\lambda_1$ and $\lambda_2$ are units in $\mathcal{R}$. 
         By following the same line of the arguments as in case 1 of Theorem \ref{thm8}, we get
         $(x+1)^{2^{\varsigma-1}} \in \mathcal{C}^2_4$. Then 
             \begin{align*}
                  (1+ x)^{2^{\varsigma-1}}=&\Big[u(x+1)^\ell +u^2(x+1)^t z(x)\Big]\Big[\varphi_1(x)+u\varphi_2(x)+u^2\varphi_3(x)\Big]\\&+\Big[u^2(x+1)^{\mu}\Big]\Big[\varkappa_1(x)+u\varkappa_2(x)+u^2\varkappa_3(x)\Big]\\
                  =&u(x+1)^\ell \varphi_1(x)+u^2\Big[(x+1)^t z(x)\varphi_1(x)+(x+1)^\ell \varphi_2(x)+(x+1)^{\mu} \varkappa_1(x)\Big]
             \end{align*}
             for some $\varphi_1(x),\varphi_2(x), \varphi_3(x),\varkappa_1(x),\varkappa_2(x), \varkappa_3(x) \in \frac{\mathbb{F}_{p^m}[x]}{\langle x^{2^\varsigma}-1 \rangle}$. Then $(1+ x)^{2^{\varsigma-1}}=0$, which is not possible. Thus, 
        there is no codeword in $\mathcal{C}^2_4$ of Lee weights 2. Also, following Theorem \ref{thm8}, there exist no codewords of Lee weight 3.
        % Also, following Theorem \ref{thm8}, there exist no codewords of the form $\chi (x)=\lambda_1 x^{k_1}+\lambda_2 x^{k_2}+\lambda_3 x^{k_3} \in \mathcal{C}^4_3$, where $\lambda_1, \lambda_2, \lambda_3 \in \mathcal{R} \textbackslash \{0\}$, $0\leq k_1<k_2<k_3$ with $wt^{\mathcal{B}}_L(\chi (x))=3$. Thus, $\mathcal{C}^2_4$ has no codeword of Lee weight 3. 
        Hence $d_L(\mathcal{C}^2_4)=4$.
            \item \textbf{Case 2:} Let $2^{\varsigma-1}+1\leq \ell \leq 2^\varsigma-1$.
            \begin{enumerate}
                \item \textbf{Subcase i:} Let $1< \mu \leq 2^{\varsigma-1}$.
                By Theorem \ref{thm3}, $d_H(\langle (x+1)^{\mu}\rangle)=2$, Thus, $2\leq d_L(\mathcal{C}^2_4)\leq 4$. As in case 2 of Theorem \ref{thm8}, we get $\mathcal{C}^2_4$ has no codeword of Lee weights 2 and 3.  Hence $d_L(\mathcal{C}^2_4)=4$.
                \item \textbf{Subcase ii:} Let $2^{\varsigma-1}+1\leq \mu \leq 2^{\varsigma}-1 $ and $\ell \geq 2^{\varsigma-1}+t$.  By Theorem \ref{thm3}, $d_L(\mathcal{C}^2_4)\geq 4$.  From Theorem \ref{thm8}, $d_L(\langle u(x+1)^\ell +u^2(x+1)^t z(x) \rangle )=4$. Then $d_L(\mathcal{C}^2_4)\leq 4$. Hence $d_L(\mathcal{C}^2_4)=4$.
                % \item \textbf{Subcase iii:} Let $2^\varsigma-2^{\varsigma-\gamma}+1 \leq \ell  \leq 2^\varsigma-2^{\varsigma-\gamma}+2^{\varsigma-\gamma-1}$ and $\ell \leq 2^{\varsigma-1}+\frac{t}{2}$, where $ 1\leq \gamma \leq \varsigma-1$. From Theorem \ref{thm3}, $d_H(\langle (x+1)^{\mu}\rangle)=2^{\gamma+1}$. Then $d_L(\mathcal{C}^2_4)\geq 2^{\gamma+1}$. From Theorem \ref{thm8}, $d_L(\langle u(x+1)^\ell +u^2(x+1)^t z(x) \rangle )=2^{\gamma+2}$. Then $2^{\gamma+1}\leq d_L(\mathcal{C}^2_4)\leq 2^{\gamma+2}$. 
                % %Following Theorem \ref{thm9}, we can provethere is no codeword in $\mathcal{C}^2_4$ of Lee weight $2^{\gamma+1}, 2^{\gamma+1}+1, 2^{\gamma+1}+2,\ldots , 2^{\gamma+2}-1$ as in Theorem } \ref{thm8}.
                % Hence $d_L(\mathcal{C}^2_4)=2^{\gamma+2}$.
                \item \textbf{Subcase iii:} Let $2^\varsigma-2^{\varsigma-\gamma}+1 \leq\mu< \ell  \leq 2^\varsigma-2^{\varsigma-\gamma}+2^{\varsigma-\gamma-1}$, where $ 1\leq \gamma \leq \varsigma-1$.  By Theorem \ref{thm3} and Theorem \ref{thm5}, $2^{\gamma+1}\leq d_L(\mathcal{C}^1_4)\leq 2^{\gamma+2}$.
                % From Theorem \ref{thm3}, $d_H(\langle (x+1)^{\mu}\rangle)=2^{\gamma+1}$.
                %Then $d_L(\mathcal{C}^2_4)\geq 2^{\gamma+1}$. 
                %From Theorem \ref{thm8}, $d_L(\langle u(x+1)^\ell +u^2(x+1)^t z(x) \rangle )=2^{\gamma+2}$. Then $2^{\gamma+1}\leq d_L(\mathcal{C}^2_4)\leq 2^{\gamma+2}$. 
                %Following Theorem \ref{thm9}, we can provethere is no codeword in $\mathcal{C}^2_4$ of Lee weight $2^{\gamma+1}, 2^{\gamma+1}+1, 2^{\gamma+1}+2,\ldots , 2^{\gamma+2}-1$ as in Theorem } \ref{thm8}.
                % Hence $d_L(\mathcal{C}^2_4)=2^{\gamma+2}$
            \end{enumerate}  
    \end{enumerate}
\end{proof}

%\newpage
\subsection{If $z(x) \neq 0$ and $t=0$}
\begin{theorem}\label{thm13} 
Let  $\mathcal{C}^3_4=\langle u(x+1)^\ell +u^2 z(x), u^2 (x+1)^{\mu} \rangle$, where $0<  \mu < \mathcal{L}\leq \ell \leq 2^\varsigma-1$ and $z(x)$ is a unit in $\mathcal{S}$. Then 
%\begin{center}
    $d_L(\mathcal{C}^3_4)=4$.
%\end{center}
\end{theorem}
\begin{proof}
     Let $\mathcal{B}=\{ \zeta_1,\zeta_2,\ldots, \zeta_m\}$ be a TOB of $\mathbb{F}_{2^m}$ over $\mathbb{F}_2.$ Let ${\mathcal{L}}$ be the smallest integer such that $u^2(x+1)^{\mathcal{L}} \in \mathcal{C}^3_4$. By Theorem \ref{thm1}, ${\mathcal{L}}=min\{\ell, 2^\varsigma-\ell\}$. Then $1\leq \mathcal{L} \leq 2^{\varsigma-1}$. Since $0<\mu < \mathcal{L} \leq 2^{\varsigma-1}$ and by Theorem \ref{thm3} and Theorem \ref{thm6}, $d_H(\mathcal{C}^3_4)=2$. Thus, $2\leq d_L(\mathcal{C}^3_4)\leq 4$. Let $\mathcal{C}^3_4$ have a codeword of Lee weight 2. 
     %following Theorem \ref{thm9}, we get $1+x^{k_2-k_1} \in \mathcal{C}^3_4$. 
     %We can write $k_1-k_2=2^wr$, where $1\leq w \leq s-1$ and $r$ is odd.  
     Following Theorem \ref{thm9}, $(1+ x)^{2^{\varsigma-1}}\in \mathcal{C}^3_4$. Thus,
             \begin{align*}
                  (1+ x)^{2^{\varsigma-1}}=&\Big[u(x+1)^\ell +u^2(x+1)^t z(x)\Big]\Big[f_1(x)+uf_2(x)+u^2f_3(x)\Big]+\Big[ u^2 (x+1)^{\mu} \Big]g(x)\\
                  =&u(x+1)^\ell f_1(x)+u^2\Big[(x+1)^t z(x)f_1(x)+(x+1)^\ell f_2(x)+ (x+1)^{\mu}g(x) \Big]
             \end{align*}
             for some $f_1(x),f_2(x), f_3(x), g(x) \in \frac{\mathbb{F}_{p^m}[x]}{\langle x^{2^\varsigma}-1 \rangle}$. Then $(1+ x)^{2^{\varsigma-1}}=0$, which is not possible. Thus, there exists no codeword of Lee weight 2. Also, following Theorem \ref{thm8}, there exist no codewords of Lee weight 3. Hence $d_L(\mathcal{C}^3_4)=4$.
\end{proof}

%\newpage
\subsection{Type 5:}
\begin{theorem}\cite{dinh2021hamming}\label{thm14} 
  Let $\mathcal{C}_5=\langle (x+1)^{\alpha}+u(x+1)^{\mathfrak{T}_1}z_1(x)+u^2(x+1)^{\mathfrak{T}_2}z_2(x) \rangle $,
  where $0<  \mathcal{V} \leq \mathcal{U}\leq \alpha \leq 2^\varsigma-1$, $0\leq  \mathfrak{T}_1 < \mathcal{U} $, $0\leq  \mathfrak{T}_2 < \mathcal{V} $ and  $z_1(x)$ and $z_1(x)$ are either 0 or a unit in $\mathcal{S}$. Then
     %\begin{center}
       $d_H(\mathcal{C}_5)= d_H(\langle (x+1)^{\mathcal{V}} \rangle)$.
    %\end{center}
\end{theorem}

%%\newpage
    \subsection{If $z_1(x)=0$ and $z_2(x)=0$}
\begin{theorem}\label{thm15}
   Let $\mathcal{C}^1_5=\langle (x+1)^{\alpha} \rangle $,
  where $1\leq \alpha \leq 2^\varsigma-1$. Then 
  \begin{center}
	$d_L(  \mathcal{C}^1_5 )=$
		$\begin{cases}
			2  &\text{if}\quad 1\leq \alpha \leq 
                2^{\varsigma-1}, \\
                2^{\gamma+1} &\text{if}\quad 2^\varsigma-2^{\varsigma-\gamma}+1 \leq \alpha \leq 2^\varsigma-2^{\varsigma-\gamma}+2^{\varsigma-\gamma-1},\quad \text{where}\quad 1\leq \gamma \leq \varsigma-1.
			\end{cases}$
		\end{center}
\end{theorem}
\begin{proof}
   Let $\langle (x+1)^{\alpha} \rangle$ be ideals of $\frac{\mathbb{F}_{2^m}[x]}{\langle x^{2^\varsigma}-1 \rangle}$, where $0\leq \alpha \leq2^\varsigma-1$. From Theorem \ref{thm14} and Theorem \ref{thm4}, $d_H(\mathcal{C}^1_5)=d_H(\langle  (x+1)^{\alpha} \rangle)=d_L(\langle  (x+1)^{\alpha} \rangle)$. We have $wt_L(\chi(x))\geq wt_H(\chi (x))$ for $\chi (x)\in \mathcal{C}^1_5$. Thus, $d_L(\mathcal{C}^1_5)=d_H(\langle  (x+1)^{\alpha} \rangle)$. Thus, the theorem follows from Theorem \ref{thm3}.
\end{proof}

%\newpage
\subsection{If $z_1(x)=0$ and $z_2(x)\neq 0$ and $\mathfrak{T}_2=0$}
\begin{theorem}\label{thm16}
   Let $\mathcal{C}^2_5=\langle (x+1)^{\alpha}+u^2z_2(x) \rangle $,
  where $0<  \mathcal{V} \leq \mathcal{U}\leq \alpha \leq 2^\varsigma-1$ and  $z_2(x)$ is a unit in $\mathcal{S}$. Then 
  \begin{center}
	$d_L(\mathcal{C}^2_5)=$
	$\begin{cases}
            2&\text{if}\quad 1\leq \alpha \leq 2^{\varsigma-2},\\
		2  &\text{if}\quad z_2(x)=1 \quad\text{and} \quad\alpha=2^{\varsigma-1},\\
            4& \text{otherwise}.
	\end{cases}$
    \end{center}
\end{theorem}
\begin{proof}
Let $\mathcal{B}=\{ \zeta_1,\zeta_2,\ldots, \zeta_m\}$ be a TOB of $\mathbb{F}_{2^m}$ over $\mathbb{F}_2.$ Let $\mathcal{V}$ be the smallest integer such that $u^2(x+1)^{\mathcal{V}} \in \mathcal{C}^2_5$. By Theorem \ref{thm1}, $\mathcal{V}=min\{\alpha, 2^\varsigma-\alpha\}$. Then $1\leq \mathcal{V} \leq 2^{\varsigma-1}$. By Theorem \ref{thm3} and Theorem \ref{thm6}, $d_H(\mathcal{C}^2_5)=2$. Thus, $2\leq d_L(\mathcal{C}^2_5)\leq 4$.
    \begin{enumerate}
        \item  \textbf{Case 1:} Let $1\leq \alpha \leq 2^{\varsigma-2}$. We have $\chi (x)=\zeta_1(x^{2^{\varsigma-1}}+1)=\zeta_1(x+1)^{2^{\varsigma-1}}=\zeta_1[(x+1)^{\alpha}+u^2z_2(x)][(x+1)^{2^{\varsigma-1}-\alpha}+u^2(x+1)^{2^{\varsigma-1}-2\alpha}z_2(x)] \in \mathcal{C}^2_5$. Since  $wt^{\mathcal{B}}_L(\chi (x))=2$, $d_L(\mathcal{C}^2_5)=2$.
        \item  \textbf{Case 2:} Let $2^{\varsigma-2}<\alpha \leq 2^\varsigma-1$
        \begin{enumerate}
            \item \textbf{Subcase i:} If $z_2(x)=1$ and $\alpha=2^{\varsigma-1}$, we have  $\chi (x)=\zeta_1((x+1)^{2^{\varsigma-1}}+u^2)=\zeta_1(x^{2^{\varsigma-1}}+1+u^2)\in \mathcal{C}^2_5$. Since $wt^{\mathcal{B}}_L(\chi (x))=2$, we have $d_L(\mathcal{C}^2_5)=2$.
            \item \textbf{Subcase ii:} Let either $z_2(x)\neq1$ or $\alpha \neq2^{\varsigma-1}$.  
            %Following as in Theorem \ref{thm8}, there exist no codewords of the form $\lambda_1 x^{k_1}+\lambda_2 x^{k_2}$ in $\mathcal{C}^3_3$ with $\lambda_1$ or $\lambda_2$ non-unit in $\mathcal{R}$, where $\lambda_1, \lambda_2 \in \mathcal{R} \textbackslash \{0\}$, $0\leq {k_1}<{k_2}$. Let $\chi (x)=\lambda_1 x^{k_1}+\lambda_2 x^{k_2} \in \mathcal{C}^3_3$ with $\lambda_1$ and $\lambda_2$ are units in $\mathcal{R}$. Then we must have $wt^{\mathcal{B}}_L(\lambda_i)=1$ for all $i=1$ and $2$. That is $\lambda_i=\zeta_j$, where $a_j\in \mathcal{B}$. As $\chi (x)$ is a non-unit in $\mathcal{S}$, under the natural reduction mod $\langle x-1, u \rangle$, we have $a_1+a_2=0$. Since $a_1$ and $a_2$ are basis elements we get a contradiction if $a_1 \neq a_2$. If $a_1=\zeta_2$ we get $1+x^{k_2-k_1} \in \mathcal{C}^3_3$. We can write $k_1-k_2=2^wr$, where $1\leq w \leq s-1$ and $r$ is odd. 
            Following Theorem \ref{thm9}, $(1+ x)^{2^{\varsigma-1}}\in \mathcal{C}^2_5$.  Then
            \begin{align*}
                   (1+ x)^{2^{\varsigma-1}}=&\Big[(x+1)^{\alpha} +u^2z_2(x)\Big]\Big[\varphi_1(x)+u\varphi_2(x)+u^2\varphi_3(x)\Big]\\
                  =&(x+1)^{\alpha} \varphi_1(x)+u(x+1)^{\alpha} \varphi_2(x)+u^2\Big[ z_2(x)\varphi_1(x)+(x+1)^{\alpha} \varphi_3(x)\Big]
             \end{align*}
             for some $\varphi_1(x),\varphi_2(x), \varphi_3(x) \in \frac{\mathbb{F}_{p^m}[x]}{\langle x^{2^\varsigma}-1 \rangle}$. Then $\varphi_1(x)=(x+1)^{2^{\varsigma-1}-\alpha}$, $\varphi_2(x)=0$ and $\varphi_3(x)=(x+1)^{2^{\varsigma-1}-2\alpha}z_2(x)$. Since  $2^{\varsigma-2}<\alpha$, we get a contradiction. Thus, there exists no codeword of Lee weight 2. Also, following Theorem \ref{thm8}, $\mathcal{C}^2_5$ has no codeword of Lee weight 3. Hence $d_L(\mathcal{C}^2_5)=4$.
        \end{enumerate}
    \end{enumerate}
\end{proof}

%\newpage
\subsection{If $z_1(x)=0$ and $z_2(x)\neq 0$ and $\mathfrak{T}_2\neq0$}
\begin{theorem}\label{thm17}
    Let $\mathcal{C}^3_5=\langle (x+1)^{\alpha}+u^2(x+1)^{\mathfrak{T}_2}z_2(x) \rangle $, where $1<  \mathcal{V} \leq \mathcal{U}\leq \alpha \leq 2^\varsigma-1$, $0<  \mathfrak{T}_2 < \mathcal{V} $ and  $z_2(x)$ is a unit in $\mathcal{S}$. Then 
    \begin{center}
	$d_L(\mathcal{C}^3_5)=$
	$\begin{cases}
            2&\text{if}\quad 1< \alpha \leq 2^{\varsigma-1} \quad\text{with}\quad \alpha \leq2^{\varsigma-2}+\frac{\mathfrak{T}_2}{2},\\
            4&\text{if}\quad 1< \alpha \leq 2^{\varsigma-1} \quad\text{with} \quad  \alpha > 2^{\varsigma-2}+\frac{\mathfrak{T}_2}{2},\\
		4  &\text{if}\quad 2^{\varsigma-      1}+1\leq \alpha \leq 
            2^{\varsigma}-1 \quad \text{with} \quad \alpha \geq 2^{\varsigma-1}+\mathfrak{T}_2,\\
            2^{\gamma+1} &\text{if}\quad 2^\varsigma-2^{\varsigma-\gamma}+1 \leq \alpha  \leq 2^\varsigma-2^{\varsigma-\gamma}+2^{\varsigma-\gamma-1},\\
            &\qquad   \text{with} \quad \alpha \leq 2^{\varsigma-1}-2^{\varsigma-\gamma-1}+2^{\varsigma-\gamma-2}+\frac{\mathfrak{T}_2}{2},
            \quad \text{where}\quad 1\leq \gamma \leq \varsigma-1.
	\end{cases}$
    \end{center}
\end{theorem}

\begin{proof}
Let $\mathcal{B}=\{ \zeta_1,\zeta_2,\ldots, \zeta_m\}$ be a TOB of $\mathbb{F}_{2^m}$ over $\mathbb{F}_2.$
    \begin{enumerate}
        \item \textbf{Case 1:} Let $1< \alpha \leq 2^{\varsigma-1}$ then $\mathcal{V}=\alpha$. By Theorem \ref{thm14} and Theorem \ref{thm3} $d_H(\mathcal{C}^3_5)=2$. Thus, $2\leq d_L(\mathcal{C}^3_5)\leq 4$.
        \begin{enumerate}
            \item \textbf{Subcase i:} Let $\alpha \leq2^{\varsigma-2}+\frac{\mathfrak{T}_2}{2}$. We have $\chi (x)=\zeta_1(x^{2^{\varsigma-1}}+1)=\zeta_1(x+1)^{2^{\varsigma-1}}=\zeta_1[(x+1)^{\alpha}+u^2(x+1)^{\mathfrak{T}_2}z_2(x)][(x+1)^{2^{\varsigma-1}-\alpha}+u^2(x+1)^{2^{\varsigma-1}-2\alpha+\mathfrak{T}_2}z_2(x)] \in \mathcal{C}^3_5$. Since  $wt^{\mathcal{B}}_L(\chi (x))=2$, $d_L(\mathcal{C}^3_5)=2$.
            \item  \textbf{Subcase ii:} Let $ \alpha >2^{\varsigma-2}+\frac{\mathfrak{T}_2}{2}$. Following the same steps as in Theorem \ref{thm8}, we get $(1+ x)^{2^{\varsigma-1}}\in \mathcal{C}^3_5$. Then
            \begin{align*}
                 (1+ x)^{2^{\varsigma-1}}=&\Big[(x+1)^{\alpha} +u^2(x+1)^{\mathfrak{T}_2}z_2(x)\Big]\Big[\varphi_1(x)+u\varphi_2(x)+u^2\varphi_3(x)\Big]\\
                  =&(x+1)^{\alpha} \varphi_1(x)+u(x+1)^{\alpha} \varphi_2(x)+u^2\Big[(x+1)^{\mathfrak{T}_2} z_2(x)\varphi_1(x)+(x+1)^{\alpha} \varphi_3(x)\Big]
             \end{align*}
             for some $\varphi_1(x),\varphi_2(x), \varphi_3(x) \in \frac{\mathbb{F}_{p^m}[x]}{\langle x^{2^\varsigma}-1 \rangle}$. Then $\varphi_1(x)=(x+1)^{2^{\varsigma-1}-\alpha}$, $\varphi_2(x)=0$ and $\varphi_3(x)=(x+1)^{2^{\varsigma-1}+\mathfrak{T}_2-2\alpha}z_2(x)$. Since  $ \alpha >2^{\varsigma-2}+\frac{\mathfrak{T}_2}{2}$, we get a contradiction. Thus, there exists no codeword of Lee weight 2. Also, following Theorem \ref{thm8}, there exist no codewords of Lee weight 3. Hence $d_L(\mathcal{C}^3_5)=4$.
        \end{enumerate}
        \item \textbf{Case 2:} Let $2^{\varsigma-1}+1\leq \alpha \leq 2^\varsigma-1$.
         \begin{enumerate}
            \item \textbf{Subcase i:} If $\alpha \geq 2^{\varsigma-1}+\mathfrak{T}_2$ then $2^\varsigma-\alpha +\mathfrak{T}_2 \leq2^{\varsigma-1}$ and $\mathcal{V}=2^\varsigma-\alpha +\mathfrak{T}_2$. By Theorem \ref{thm3} and Theorem \ref{thm6}, $d_H(\mathcal{C}^3_5)=2$. Thus, $2\leq d_L(\mathcal{C}^3_5)$. Following as in the above case, we get $(1+ x)^{2^{\varsigma-1}}\in \mathcal{C}^3_5$. 
             Since $\alpha > 2^{\varsigma-1}$, this is not possible. Thus, there exists no codeword of Lee weight 2. Also, following Theorem \ref{thm8}, there exists no codewords of Lee weight 3.
             A codeword $\wp(x)=u^2\zeta_1(x^{2^{\varsigma-1}}+1)=u^2\zeta_1(x+1)^{2^{\varsigma-1}}\in \langle u^2(x+1)^{2^\varsigma-\alpha +\mathfrak{T}_2}\rangle \subseteq \mathcal{C}^3_5$ with $wt^{\mathcal{B}}_L(\wp(x))=4$. Thus, $d_L(\mathcal{C}^3_5)=4$.
             \item \textbf{Subcase ii:} Let $\alpha \leq 2^{\varsigma-1}+\mathfrak{T}_2$. If $\alpha \leq 2^{\varsigma-1}+\frac{\mathfrak{T}_2}{2}$ then $\mathcal{V}=\alpha$.  If $2^\varsigma-2^{\varsigma-\gamma}+1 \leq \alpha  \leq 2^\varsigma-2^{\varsigma-\gamma}+2^{\varsigma-\gamma-1},$ where $ 1\leq \gamma \leq \varsigma-1$ then by Theorem \ref{thm3} and Theorem \ref{thm6}, $d_H(\langle (x+1)^{\mathcal{V}}\rangle)=2^{\gamma+1}$. Thus, $2^{\gamma+1}\leq d_L(\mathcal{C}^3_5)$. If $\alpha \leq 2^{\varsigma-1}-2^{\varsigma-\gamma-1}+2^{\varsigma-\gamma-2}+\frac{\mathfrak{T}_2}{2}$, we have
             \begin{align*}
                \prod \limits_{\alpha=1}^{\gamma+1}(x^{2^{\varsigma-\alpha}}+1)&= (x+1)^{2^{\varsigma-1}+2^{\varsigma-2}+\cdots+2^{\varsigma-\gamma-1}}\\
                 &=(x+1)^{2^{\varsigma}-2^{\varsigma-\gamma}+2^{\varsigma-\gamma-1}}\\
                 &=\Big[(x+1)^{\alpha} +u^2(x+1)^{\mathfrak{T}_2} z_2(x) \Big]\\
                 & \qquad\Big[(x+1)^{2^\varsigma-2^{\varsigma-\gamma}+2^{\varsigma-\gamma-1}-\alpha}+u^2(x+1)^{2^\varsigma-2^{\varsigma-\gamma}+2^{\varsigma-\gamma-1}-2\alpha+\mathfrak{T}_2}z_2(x)\Big] \in \mathcal{C}^3_5
             \end{align*}
             Let $f(x)=\zeta_1 \prod \limits_{\alpha=1}^{\gamma+1}(x^{2^{\varsigma-\alpha}}+1)$. Then $wt^{\mathcal{B}}_L(f(x))=2^{\gamma+1}$. Thus, $d_L(\mathcal{C}^3_5)=2^{\gamma+1}$.

        \end{enumerate}
        
    \end{enumerate}
\end{proof}

%\newpage
\subsection{If $z_1(x)\neq0$ and $\mathfrak{T}_1=0$ and $z_2(x)=0$}
\begin{theorem}\label{thm18}
   Let $\mathcal{C}^4_5=\langle (x+1)^{\alpha}+u z_1(x) \rangle $,
  where $0<  \mathcal{V} \leq \mathcal{U}\leq \alpha \leq 2^\varsigma-1$ and  $z_1(x)$ is a unit in $\mathcal{S}$. Then 
   \begin{center}
	$d_L(\mathcal{C}^4_5)=$
		$\begin{cases}
			2  &\text{if}\quad 2^{\varsigma-1}\geq 3\alpha, \\
                3 &\text{if}\quad z_1(x)=1 \quad \text{and}\quad\alpha=2^{\varsigma-1}, \\
                 4& \text{otherwise}.
			\end{cases}$
		\end{center}
\end{theorem}

\begin{proof}
    Let $\mathcal{B}=\{ \zeta_1,\zeta_2,\ldots, \zeta_m\}$ be a TOB of $\mathbb{F}_{2^m}$ over $\mathbb{F}_2.$ Let $\mathcal{V}$ be the smallest integer such that $u^2(x+1)^{\mathcal{V}} \in \mathcal{C}^4_5$. By Theorem \ref{thm1}, $\mathcal{V}=min\{\alpha, 2^\varsigma-\alpha\}$. Then $1\leq \mathcal{V} \leq 2^{\varsigma-1}$. By Theorem \ref{thm14} and Theorem \ref{thm3}, $d_H(\mathcal{C}^4_5)=2$. Thus, $2\leq d_L(\mathcal{C}^4_5)\leq 4$.
    \begin{enumerate}
        \item \textbf{\textbf{Case 1:} }If $2^{\varsigma-1}\geq 3\alpha$ we have $\chi (x)=\zeta_1(x+1)^{2^{\varsigma-1}}=\zeta_1\big[(x+1)^{\alpha}+u z_1(x)\big]\big[(x+1)^{2^{\varsigma-1}-\alpha}+u(x+1)^{2^{\varsigma-1}-2\alpha}z_1(x)+u^2(x+1)^{2^{\varsigma-1}-3\alpha}z_1(x)z_1(x)\big] \in \mathcal{C}^4_5$. Since $wt^{\mathcal{B}}_L(\chi (x))=2$, we have $d_L(\mathcal{C}^4_5)=2$.
        \item \textbf{\textbf{Case 2:}} Let $z_1(x)=1$ and $\alpha=2^{\varsigma-1}$. Following the same steps as in Theorem \ref{thm9}, we get $(1+ x)^{2^{\varsigma-1}}\in \mathcal{C}^4_5$. Then
        \begin{align*}
            (1+ x)^{2^{\varsigma-1}}=&\Big[(x+1)^{\alpha}+u z_1(x)\Big]\Big[\varphi_1(x)+u\varphi_2(x)+u^2\varphi_3(x)\Big]\\
            =&(x+1)^{\alpha} \varphi_1(x)+u\Big[\varphi_1(x)z_1(x)+(x+1)^{\alpha} \varphi_2(x)\Big]\\
            &+u^2\Big[\varphi_2(x)z_1(x)+(x+1)^{\alpha} \varphi_3(x)\Big]
        \end{align*}
        for some $\varphi_1(x),\varphi_2(x), \varphi_3(x) \in \frac{\mathbb{F}_{p^m}[x]}{\langle x^{2^\varsigma}-1 \rangle}$. Then $\varphi_1(x)=(x+1)^{2^{\varsigma-1}-\alpha}$, $\varphi_2(x)=(x+1)^{2^{\varsigma-1}-2\alpha}z_1(x)$ and $\varphi_3(x)=(x+1)^{2^{\varsigma-1}-3\alpha}z_1(x)z_1(x)$.
        As $\alpha=2^{\varsigma-1}$, we have $2^{\varsigma-1} < 3\alpha$. Then we obtain a contradiction. Thus, there exists no codeword of Lee weight 2. Also, we have $\chi (x)=\zeta_1\big[x^{2^{\varsigma-1}}+1+u \big]=\zeta_1\big[(x+1)^{2^{\varsigma-1}}+u \big]\in \mathcal{C}^4_5$. Since $wt^{\mathcal{B}}_L(\chi (x))=3$, we have $d_L(\mathcal{C}^4_5)=3$.
        \item \textbf{\textbf{Case 3:}} Let $2^{\varsigma-1}< 3\alpha$ and either $z_1(x)\neq 1$  or  $\alpha \neq 2^{\varsigma-1}$. Following as in the above case, there exists no codeword of Lee weight 2 as $2^{\varsigma-1} < 3\alpha$. Also, following Theorem \ref{thm8}, $\mathcal{C}^4_5$ has no codeword of Lee weight 3. Hence $d_L(\mathcal{C}^4_5)=4$.
    \end{enumerate}
\end{proof}

%\newpage
\subsection{If $z_1(x)\neq0$ and $\mathfrak{T}_1\neq0$ and $z_2(x)=0$}
\begin{theorem}\label{thm19}
   Let $\mathcal{C}^5_5=\langle (x+1)^{\alpha}+u(x+1)^{\mathfrak{T}_1} z_1(x) \rangle $,
  where $0<  \mathcal{V} \leq \mathcal{U}\leq \alpha \leq 2^\varsigma-1$, $0<  \mathfrak{T}_1 < \mathcal{U} $ and  $z_1(x)$ is a unit in $\mathcal{S}$. Then 
  \begin{center}
	$d_L(\mathcal{C}^5_5)=$
	$\begin{cases}
            2&\text{if}\quad 1<\alpha \leq 2^{\varsigma-1} \quad\text{with}\quad \alpha \leq2^{\varsigma-2}+\frac{\mathfrak{T}_1}{2} \quad\text{and}\quad 3\alpha \leq2^{\varsigma-1}+2\mathfrak{T}_1,\\
            4&\text{if}\quad 1<\alpha \leq 2^{\varsigma-1} \quad\text{with}\quad \alpha >2^{\varsigma-2}+\frac{\mathfrak{T}_1}{2}\quad\text{or}\quad 3\alpha >2^{\varsigma-1}+2\mathfrak{T}_1 ,\\
		4  &\text{if}\quad 2^{\varsigma-1}+1\leq \alpha \leq 
            2^{\varsigma}-1 \quad \text{with} \quad \alpha \geq 2^{\varsigma-1}+\mathfrak{T}_1,\\
            2^{\gamma+1} &\text{if}\quad 2^\varsigma-2^{\varsigma-\gamma}+1 \leq \alpha  \leq 2^\varsigma-2^{\varsigma-\gamma}+2^{\varsigma-\gamma-1}\\
            &\quad   \text{with} \quad \alpha \leq 2^{\varsigma-1}-2^{\varsigma-\gamma-1}+2^{\varsigma-\gamma-2}+\frac{\mathfrak{T}_1}{2}\\
            &\quad\text{and}\quad 3\alpha\leq 2^\varsigma-2^{\varsigma-\gamma}+2^{\varsigma-\gamma-1}+2\mathfrak{T}_1,
            \quad \text{where}\quad 1\leq \gamma \leq \varsigma-1.
	\end{cases}$
    \end{center}
\end{theorem}
\begin{proof}
Let $\mathcal{B}=\{ \zeta_1,\zeta_2,\ldots, \zeta_m\}$ be a TOB of $\mathbb{F}_{2^m}$ over $\mathbb{F}_2.$
    \begin{enumerate}
        \item \textbf{Case 1:} If $1<\alpha \leq 2^{\varsigma-1}$ then $\mathcal{V}=\alpha$. By Theorem \ref{thm14} and Theorem \ref{thm3}, $d_H(\mathcal{C}^5_5)=2$. Thus, $2\leq d_L(\mathcal{C}^5_5)\leq 4$.
        \begin{enumerate}
            \item \textbf{Subcase i:} Let $\alpha \leq2^{\varsigma-2}+\frac{\mathfrak{T}_1}{2}$ and $3\alpha \leq2^{\varsigma-1}+2\mathfrak{T}_1$. We have $\chi (x)=\zeta_1(x^{2^{\varsigma-1}}+1)=\zeta_1(x+1)^{2^{\varsigma-1}}=\zeta_1[(x+1)^{\alpha}+u(x+1)^{\mathfrak{T}_1}z_1(x)][(x+1)^{2^{\varsigma-1}-\alpha}+u(x+1)^{2^{\varsigma-1}-2\alpha+\mathfrak{T}_1}z_1(x)+u^2(x+1)^{2^{\varsigma-1}-3\alpha+2\mathfrak{T}_1}z_1(x)z_1(x)] \in \mathcal{C}^5_5$. Since  $wt^{\mathcal{B}}_L(\chi (x))=2$, $d_L(\mathcal{C}^5_5)=2$.
            \item  \textbf{Subcase ii:} Let $ \alpha >2^{\varsigma-2}+\frac{\mathfrak{T}_1}{2}$ or $3\alpha > 2^{\varsigma-1}+2\mathfrak{T}_1$. Following the same steps as in Theorem \ref{thm8}, we get $(1+ x)^{2^{\varsigma-1}}\in \mathcal{C}^5_5$. Then
            \begin{align*}
                 (1+ x)^{2^{\varsigma-1}}=&\Big[(x+1)^{\alpha} +u(x+1)^{\mathfrak{T}_1}z_1(x)\Big]\Big[\varphi_1(x)+u\varphi_2(x)+u^2\varphi_3(x)\Big]\\
                  =&(x+1)^{\alpha} \varphi_1(x)+u\Big[\varphi_1(x)(x+1)^{\mathfrak{T}_1}z_1(x)+(x+1)^{\alpha} \varphi_2(x)\Big]\\
                  &+u^2\Big[(x+1)^{\mathfrak{T}_1} z_1(x)\varphi_2(x)+(x+1)^{\alpha} \varphi_3(x)\Big]
             \end{align*}
             for some $\varphi_1(x),\varphi_2(x), \varphi_3(x) \in \frac{\mathbb{F}_{p^m}[x]}{\langle x^{2^\varsigma}-1 \rangle}$. Then $\varphi_1(x)=(x+1)^{2^{\varsigma-1}-\alpha}$, $\varphi_2(x)=(x+1)^{2^{\varsigma-1}-2\alpha+\mathfrak{T}_1}z_1(x)$ and $\varphi_3(x)=(x+1)^{2^{\varsigma-1}-3\alpha+2\mathfrak{T}_1}z_1(x)z_1(x)$. Since  $ \alpha >2^{\varsigma-2}+\frac{\mathfrak{T}_1}{2}$ or $3\alpha > 2^{\varsigma-1}+2\mathfrak{T}_1$, we get a contradiction. Thus, there exists no codeword of Lee weight 2. Also, following Theorem \ref{thm8}, there are no codewords of Lee weight 3. Hence $d_L(\mathcal{C}^5_5)=4$.
        \end{enumerate}
        \item \textbf{Case 2:} Let $2^{\varsigma-1}+1\leq \alpha \leq 2^\varsigma-1$.
         \begin{enumerate}
            \item \textbf{Subcase i:} If $\alpha \geq 2^{\varsigma-1}+\mathfrak{T}_1$ then $2^\varsigma-\alpha +\mathfrak{T}_1 \leq2^{\varsigma-1}$ and $\mathcal{V}=2^\varsigma-\alpha +\mathfrak{T}_1$. By Theorem \ref{thm3} and Theorem \ref{thm6}, $d_H(\mathcal{C}^5_5)=2$. Thus, $2\leq d_L(\mathcal{C}^5_5)$. 
%     Following Theorem \ref{thm9}, we can prove
         % Thus, $\mathcal{C}^5_5$ has no codeword of Lee weights 2 and 3.
             Following as in the above case, we get $(1+ x)^{2^{\varsigma-1}}\in \mathcal{C}^5_5$. 
             Since $\alpha > 2^{\varsigma-1}$, this is not possible. Thus, there exists no codeword of Lee weight 2. Also, following Theorem \ref{thm8}, there exist no codewords of Lee weight 3.
             A codeword $\wp(x)=u^2\zeta_1(x^{2^{\varsigma-1}}+1)=u^2\zeta_1(x+1)^{2^{\varsigma-1}}\in \langle u^2(x+1)^{2^\varsigma-\alpha +\mathfrak{T}_1}\rangle \subseteq \mathcal{C}^5_5$ with $wt^{\mathcal{B}}_L(\wp(x))=4$. Thus, $d_L(\mathcal{C}^5_5)=4$.
             \item \textbf{Subcase ii:} Let $\alpha \leq 2^{\varsigma-1}+\mathfrak{T}_1$. If $\alpha \leq 2^{\varsigma-1}+\frac{\mathfrak{T}_1}{2}$ then $\mathcal{V}=\alpha$.  If $2^\varsigma-2^{\varsigma-\gamma}+1 \leq \alpha  \leq 2^\varsigma-2^{\varsigma-\gamma}+2^{\varsigma-\gamma-1},$ where $\quad 1\leq \gamma \leq \varsigma-1$ then by Theorem \ref{thm3} and Theorem \ref{thm6}, $d_H(\langle (x+1)^{\mathcal{V}}\rangle)=2^{\gamma+1}$. Thus, $2^{\gamma+1}\leq d_L(\mathcal{C}^5_5)$. If $\alpha \leq 2^{\varsigma-1}-2^{\varsigma-\gamma-1}+2^{\varsigma-\gamma-2}+\frac{\mathfrak{T}_1}{2}$ and $3\alpha\leq 2^\varsigma-2^{\varsigma-\gamma}+2^{\varsigma-\gamma-1}+2\mathfrak{T}_1$, we have
             \begin{align*}
                \prod \limits_{\alpha=1}^{\gamma+1}(x^{2^{\varsigma-\alpha}}+1)=& (x+1)^{2^{\varsigma-1}+2^{\varsigma-2}+\cdots+2^{\varsigma-\gamma-1}},\\
                 =&(x+1)^{2^{\varsigma}-2^{\varsigma-\gamma}+2^{\varsigma-\gamma-1}},\\
                 =&\Big[(x+1)^{\alpha} +u(x+1)^{\mathfrak{T}_1} z_1(x) \Big] \Big[(x+1)^{2^\varsigma-2^{\varsigma-\gamma}+2^{\varsigma-\gamma-1}-\alpha},\\
                 &+u(x+1)^{2^\varsigma-2^{\varsigma-\gamma}+2^{\varsigma-\gamma-1}-2\alpha+\mathfrak{T}_1}z_1(x)
                 +u^2(x+1)^{2^\varsigma-2^{\varsigma-\gamma}+2^{\varsigma-\gamma-1}-3\alpha+2\mathfrak{T}_1}z_1(x)\Big] \in \mathcal{C}^5_5
             \end{align*}
             Let $f(x)=\zeta_1 \prod \limits_{\alpha=1}^{\gamma+1}(x^{2^{\varsigma-\alpha}}+1)$. Then $wt^{\mathcal{B}}_L(f(x))=2^{\gamma+1}$. Thus, $d_L(\mathcal{C}^5_5)=2^{\gamma+1}$.

        \end{enumerate}
        
    \end{enumerate}
\end{proof}

%\newpage
\subsection{If $z_1(x)\neq0$ and $\mathfrak{T}_1=0$ and $z_2(x)\neq0$ and $\mathfrak{T}_2=0$ }
\begin{theorem}\label{thm20}
   Let $\mathcal{C}^6_5=\langle (x+1)^{\alpha}+u z_1(x)+u^2z_2(x) \rangle $,
  where $0<  \mathcal{V} \leq \mathcal{U}\leq \alpha \leq 2^\varsigma-1$ and  $z_1(x)$ and $z_2(x)$ are units in $\mathcal{S}$. Then 
  \begin{center}
	$d_L(\mathcal{C}^6_5)=$
		$\begin{cases}
			2  &\text{if}\quad  3\alpha \leq 2^{\varsigma-1}, \\
                3 &\text{if}\quad z_1(x)=z_2(x)=1 \quad \text{and}\quad\alpha=2^{\varsigma-1}, \\
                 4& \text{otherwise}.
			\end{cases}$
		\end{center}
\end{theorem}
\begin{proof}
    Let $\mathcal{B}=\{ \zeta_1,\zeta_2,\ldots, \zeta_m\}$ be a TOB of $\mathbb{F}_{2^m}$ over $\mathbb{F}_2.$ Let $\mathcal{V}$ be the smallest integer such that $u^2(x+1)^{\mathcal{V}} \in \mathcal{C}^6_5$. By Theorem \ref{thm1}, $\mathcal{V}=min\{\alpha, 2^\varsigma-\alpha\}$. Then $1\leq \mathcal{V} \leq 2^{\varsigma-1}$. By Theorem \ref{thm14} and Theorem \ref{thm3}, $d_H(\mathcal{C}^6_5)=2$. Thus, $2\leq d_L(\mathcal{C}^6_5)\leq 4$.
    \begin{enumerate}
        \item \textbf{\textbf{Case 1:} }If $2^{\varsigma-1}\geq 3\alpha$ we have $\chi (x)=\zeta_1(x+1)^{2^{\varsigma-1}}=\zeta_1\big[(x+1)^{\alpha}+u z_1(x)+u^2z_2(x)\big]\big[(x+1)^{2^{\varsigma-1}-\alpha}+u(x+1)^{2^{\varsigma-1}-2\alpha}z_1(x)+u^2(x+1)^{2^{\varsigma-1}-2\alpha}z_2(x)+u^2(x+1)^{2^{\varsigma-1}-3\alpha}z_1(x)z_1(x)\big] \in \mathcal{C}^6_5$. Since $wt^{\mathcal{B}}_L(\chi (x))=2$, we have $d_L(\mathcal{C}^6_5)=2$.
        \item \textbf{\textbf{Case 2:}} Let $z_1(x)=z_2(x)=1$ and $\alpha=2^{\varsigma-1}$
        Following the same steps as in Theorem \ref{thm9}, we get $(1+ x)^{2^{\varsigma-1}}\in \mathcal{C}^6_5$. Then
        \begin{align*}
            (1+ x)^{2^{\varsigma-1}}=&\Big[(x+1)^{\alpha}+u z_1(x)+u^2z_2(x)\Big]\Big[\varphi_1(x)+u\varphi_2(x)+u^2\varphi_3(x)\Big]\\
            =&(x+1)^{\alpha} \varphi_1(x)+u\Big[\varphi_1(x)z_1(x)+(x+1)^{\alpha} \varphi_2(x)\Big]\\
            &+u^2\Big[\varphi_1(x)z_2(x)+\varphi_2(x)z_1(x)+(x+1)^{\alpha} \varphi_3(x)\Big]
        \end{align*}
        for some $\varphi_1(x),\varphi_2(x), \varphi_3(x) \in \frac{\mathbb{F}_{p^m}[x]}{\langle x^{2^\varsigma}-1 \rangle}$. Then $\varphi_1(x)=(x+1)^{2^{\varsigma-1}-\alpha}$, $\varphi_2(x)=(x+1)^{2^{\varsigma-1}-2\alpha}z_1(x)$ and $\varphi_3(x)=(x+1)^{2^{\varsigma-1}-2\alpha}z_2(x)+(x+1)^{2^{\varsigma-1}-3\alpha}z_1(x)z_1(x)$.
        As $\alpha=2^{\varsigma-1}$, we have $2^{\varsigma-1} < 3\alpha$. Then we obtain a contradiction. Thus, there exists no codeword of Lee weight 2. Also, we have $\chi (x)=\zeta_1\big[x^{2^{\varsigma-1}}+1+u +u^2\big]=\zeta_1\big[(x+1)^{2^{\varsigma-1}}+u +u^2\big]\in \mathcal{C}^6_5$. Since $wt^{\mathcal{B}}_L(\chi (x))=3$, we have $d_L(\mathcal{C}^6_5)=3$.
        \item \textbf{\textbf{Case 3:}} Let $2^{\varsigma-1}< 3\alpha$ and either $z_1(x)\neq 1$ or $z_2(x)\neq 1$ or  $\alpha \neq 2^{\varsigma-1}$.  Following as in the above case, there exists no codeword of Lee weight 2 as $2^{\varsigma-1} < 3\alpha$. Also, following Theorem \ref{thm8}, $\mathcal{C}^6_5$ has no codeword of Lee weight 3. Hence $d_L(\mathcal{C}^6_5)=4$.
    \end{enumerate}
\end{proof}

%\newpage
\subsection{If $z_1(x)\neq0$ and $\mathfrak{T}_1\neq0$ and $z_2(x)\neq0$ and $\mathfrak{T}_2=0$ }
\begin{theorem}\label{thm21}
   Let $\mathcal{C}^7_5=\langle (x+1)^{\alpha}+u(x+1)^{\mathfrak{T}_1} z_1(x)+u^2z_2(x) \rangle $,
  where $0<  \mathcal{V} \leq \mathcal{U}\leq \alpha \leq 2^\varsigma-1$, $0<  \mathfrak{T}_1 < \mathcal{U} $ and  $z_1(x)$ and $z_2(x)$ are units in $\mathcal{S}$. Then 

  \begin{center}
	$d_L(\mathcal{C}^7_5)=$
	$\begin{cases}
            2&\text{if}\quad 1< \alpha \leq 2^{\varsigma-1} \quad\text{with}\quad 3\alpha \leq2^{\varsigma-1}+2\mathfrak{T}_1,\\
            4&\text{if}\quad 1< \alpha \leq 2^{\varsigma-1} \quad\text{with} \quad 3 \alpha > 2^{\varsigma-1}+2\mathfrak{T}_1,\\
		4  &\text{if}\quad 2^{\varsigma-1}+1\leq \alpha \leq 
            2^{\varsigma}-1 \quad \text{with} \quad \alpha \geq 2^{\varsigma-1}+\mathfrak{T}_1,\\
            2^{\gamma+1} &\text{if}\quad 2^\varsigma-2^{\varsigma-\gamma}+1 \leq \alpha  \leq 2^\varsigma-2^{\varsigma-\gamma}+2^{\varsigma-\gamma-1}\\
            &\quad   \text{with} \quad 3\alpha \leq 2^{\varsigma}-2^{\varsigma-\gamma}+2^{\varsigma-\gamma-1}+2\mathfrak{T}_1 \quad\text{and}\\
            &\quad \alpha \leq 2^{\varsigma-1}+\frac{\mathfrak{T}_1}{2}
            \quad \text{where}\quad 1\leq \gamma \leq \varsigma-1.
	\end{cases}$
    \end{center}
\end{theorem}

\begin{proof}
Let $\mathcal{B}=\{ \zeta_1,\zeta_2,\ldots, \zeta_m\}$ be a TOB of $\mathbb{F}_{2^m}$ over $\mathbb{F}_2.$
    \begin{enumerate}
        \item \textbf{Case 1:} If $1<\alpha \leq 2^{\varsigma-1}$ then $\mathcal{V}=\alpha$. By Theorem \ref{thm14} and Theorem \ref{thm3} $d_H(\mathcal{C}^7_5)=2$. Thus, $2\leq d_L(\mathcal{C}^7_5)\leq 4$
        \begin{enumerate}
            \item \textbf{Subcase i:} Let $2^{\varsigma-1}+2\mathfrak{T}_1 \geq 3\alpha$. We have $\chi (x)=\zeta_1(x^{2^{\varsigma-1}}+1)=\zeta_1(x+1)^{2^{\varsigma-1}}=\zeta_1[(x+1)^{\alpha}+u(x+1)^{\mathfrak{T}_1} z_1(x)+u^2z_2(x)][(x+1)^{2^{\varsigma-1}-\alpha}+u(x+1)^{2^{\varsigma-1}-2\alpha+\mathfrak{T}_1}z_1(x)  +u^2(x+1)^{2^{\varsigma-1}-2\alpha}z_2(x)+u^2(x+1)^{2^{\varsigma-1}-3\alpha+2\mathfrak{T}_1}z_1(x)z_1(x)] \in \mathcal{C}^7_5$. Since  $wt^{\mathcal{B}}_L(\chi (x))=2$, $d_L(\mathcal{C}^7_5)=2$.
            \item  \textbf{Subcase ii:} Let $2^{\varsigma-1}+2\mathfrak{T}_1 < 3\alpha$. Following the same steps as in Theorem \ref{thm8}, we get $(1+ x)^{2^{\varsigma-1}}\in \mathcal{C}^7_5$. Then
            \begin{align*}
                 (1+ x)^{2^{\varsigma-1}}=&\Big[(x+1)^{\alpha}+u(x+1)^{\mathfrak{T}_1} z_1(x)+u^2z_2(x)\Big]\Big[\varphi_1(x)+u\varphi_2(x)+u^2\varphi_3(x)\Big]\\
                  =&(x+1)^{\alpha} \varphi_1(x)+u\Big[(x+1)^{\mathfrak{T}_1} \varphi_1(x)z_1(x)+(x+1)^{\alpha}\varphi_2(x)\Big]\\
                  &+u^2\Big[z_2(x)\varphi_1(x)+(x+1)^{\mathfrak{T}_1} z_1(x)\varphi_2(x)+(x+1)^{\alpha} \varphi_3(x)\Big]
             \end{align*}
             for some $\varphi_1(x),\varphi_2(x), \varphi_3(x) \in \frac{\mathbb{F}_{p^m}[x]}{\langle x^{2^\varsigma}-1 \rangle}$. Then $\varphi_1(x)=(x+1)^{2^{\varsigma-1}-\alpha}$, $\varphi_2(x)=(x+1)^{2^{\varsigma-1}-2\alpha+\mathfrak{T}_1}z_1(x)$ and $\varphi_3(x)=(x+1)^{2^{\varsigma-1}-2\alpha}z_2(x)+(x+1)^{2^{\varsigma-1}-3\alpha+2\mathfrak{T}_1}z_1(x)z_1(x)$. Since  $2^{\varsigma-1}+2\mathfrak{T}_1 < 3\alpha$, we get a contradiction. Thus, there exists no codeword of Lee weight 2. Also, following Theorem \ref{thm8}, $\mathcal{C}^7_5$ has no codeword of Lee weight 3. Hence $d_L(\mathcal{C}7_5)=4$.
        \end{enumerate}
        \item \textbf{Case 2:} Let $2^{\varsigma-1}+1\leq \alpha \leq 2^\varsigma-1$.
         \begin{enumerate}
            \item \textbf{Subcase i:} If $\alpha \geq 2^{\varsigma-1}+\mathfrak{T}_1$ then $2^\varsigma-\alpha +\mathfrak{T}_1 \leq2^{\varsigma-1}$ and $\mathcal{V}=2^\varsigma-\alpha +\mathfrak{T}_1$. By Theorem \ref{thm3} and Theorem \ref{thm6}, $d_H(\mathcal{C}^7_5)=2$. Thus, $2\leq d_L(\mathcal{C}^7_5)$. Following as in the above case, we get $(1+ x)^{2^{\varsigma-1}}\in \mathcal{C}^7_5$. 
             Since $\alpha > 2^{\varsigma-1}$, this is not possible. Thus, there exists no codeword of Lee weight 2. Also, following as in Theorem \ref{thm8}, $\mathcal{C}^7_5$ has no codeword of Lee weight 3. A codeword $\wp(x)=u^2\zeta_1(x^{2^{\varsigma-1}}+1)=u^2\zeta_1(x+1)^{2^{\varsigma-1}}\in \langle u^2(x+1)^{2^\varsigma-\alpha +\mathfrak{T}_1}\rangle \subseteq \mathcal{C}^7_5$ with $wt^{\mathcal{B}}_L(\wp(x))=4$. Thus, $d_L(\mathcal{C}^7_5)=4$.
             \item \textbf{Subcase ii:} Let $\alpha \leq 2^{\varsigma-1}+\mathfrak{T}_1$. If $\alpha \leq 2^{\varsigma-1}+\frac{\mathfrak{T}_1}{2}$ then $\mathcal{V}=\alpha$.  If $2^\varsigma-2^{\varsigma-\gamma}+1 \leq \alpha  \leq 2^\varsigma-2^{\varsigma-\gamma}+2^{\varsigma-\gamma-1},$ where $\quad 1\leq \gamma \leq \varsigma-1$ then by Theorem \ref{thm3} and Theorem \ref{thm14}, $d_H(\langle (x+1)^{\mathcal{V}}\rangle)=2^{\gamma+1}$. Thus, $2^{\gamma+1}\leq d_L(\mathcal{C}^7_5)$. If $3\alpha \leq 2^{\varsigma}-2^{\varsigma-\gamma}+2^{\varsigma-\gamma-1}+2\mathfrak{T}_1$, we have
             \begin{align*}
                \prod \limits_{\alpha=1}^{\gamma+1}(x^{2^{\varsigma-\alpha}}+1)=& (x+1)^{2^{\varsigma-1}+2^{\varsigma-2}+\cdots+2^{\varsigma-\gamma-1}},\\
                 =&(x+1)^{2^{\varsigma}-2^{\varsigma-\gamma}+2^{\varsigma-\gamma-1}},\\
                 =&\Big[(x+1)^{\alpha}+u(x+1)^{\mathfrak{T}_1} z_1(x)+u^2z_2(x) \Big]\Big[(x+1)^{2^\varsigma-2^{\varsigma-\gamma}+2^{\varsigma-\gamma-1}-\alpha},\\
                 &+u (x+1)^{2^\varsigma-2^{\varsigma-\gamma}+2^{\varsigma-\gamma-1}-2\alpha+\mathfrak{T}_1} z_1(x)+u^2(x+1)^{2^\varsigma-2^{\varsigma-\gamma}+2^{\varsigma-\gamma-1}-2\alpha}z_2(x),\\
                 &+u^2(x+1)^{2^\varsigma-2^{\varsigma-\gamma}+2^{\varsigma-\gamma-1}-3\alpha+2\mathfrak{T}_1}z_1(x) z_1(x)\Big] \in \mathcal{C}^7_5
             \end{align*}
             Let $f(x)=\zeta_1 \prod \limits_{\alpha=1}^{\gamma+1}(x^{2^{\varsigma-\alpha}}+1)$. Then $wt^{\mathcal{B}}_L(f(x))=2^{\gamma+1}$. Thus, $d_L(\mathcal{C}^7_5)=2^{\gamma+1}$.

        \end{enumerate}
        \end{enumerate}
\end{proof}

%\newpage
\subsection{If $z_1(x)\neq0$ and $\mathfrak{T}_1=0$ and $z_2(x)\neq0$ and $\mathfrak{T}_2\neq0$ }
\begin{theorem}\label{thm22}
   Let $\mathcal{C}^8_5=\langle (x+1)^{\alpha}+uz_1(x)+u^2(x+1)^{\mathfrak{T}_2}z_2(x) \rangle $,
  where $1<  \mathcal{V} \leq \mathcal{U}\leq \alpha \leq 2^\varsigma-1$, $0<  \mathfrak{T}_2 < \mathcal{V} $ and  $z_1(x)$ and $z_2(x)$ are units in $\mathcal{S}$. Then 
   \begin{center}
	$d_L(\mathcal{C}^8_5)=$
		$\begin{cases}
			2  &\text{if}\quad 2^{\varsigma-1}\geq 3\alpha \quad\text{and}\quad 2^{\varsigma-2}+\frac{\mathfrak{T}_2}{2}\geq \alpha, \\
                 4& \text{otherwise}.
			\end{cases}$
		\end{center}
\end{theorem}
\begin{proof}
    Let $\mathcal{B}=\{ \zeta_1,\zeta_2,\ldots, \zeta_m\}$ be a TOB of $\mathbb{F}_{2^m}$ over $\mathbb{F}_2.$ Let $\mathcal{V}$ be the smallest integer such that $u^2(x+1)^{\mathcal{V}} \in \mathcal{C}^8_5$. By Theorem \ref{thm1}, $\mathcal{V}=min\{\alpha, 2^\varsigma-\alpha\}$. Then $1<\mathcal{V} \leq 2^{\varsigma-1}$. By Theorem \ref{thm3} and Theorem \ref{thm14}, $d_H(\mathcal{C}^8_5)=2$. Thus, $2\leq d_L(\mathcal{C}^8_5)\leq 4$.
    \begin{enumerate}
        \item \textbf{Case 1:} Let $2^{\varsigma-1}\geq 3\alpha$ and $2^{\varsigma-2}+\frac{\mathfrak{T}_2}{2}\geq \alpha$. We have $\chi (x)=\zeta_1(x^{2^{\varsigma-1}}+1)=\zeta_1(x+1)^{2^{\varsigma-1}}=\zeta_1[(x+1)^{\alpha}+uz_1(x)+u^2(x+1)^{\mathfrak{T}_2}z_2(x)][(x+1)^{2^{\varsigma-1}-\alpha}+u(x+1)^{2^{\varsigma-1}-2\alpha}z_1(x)  +u^2(x+1)^{2^{\varsigma-1}-2\alpha+\mathfrak{T}_2}z_2(x)+u^2(x+1)^{2^{\varsigma-1}-3\alpha}z_1(x)z_1(x)] \in \mathcal{C}^8_5$. Since  $wt^{\mathcal{B}}_L(\chi (x))=2$, $d_L(\mathcal{C}^8_5)=2$.
        \item \textbf{Case 2:} Let either $2^{\varsigma-1}< 3\alpha$ or $2^{\varsigma-2}+\frac{\mathfrak{T}_2}{2} < \alpha$. Following as in Theorem \ref{thm9},  $(1+ x)^{2^{\varsigma-1}}\in \mathcal{C}^8_5$. Then
        %We get a contradiction as in Theorem \ref{thm8}.
            \begin{align*}
                (1+ x)^{2^{\varsigma-1}}=&\Big[(x+1)^{\alpha}+uz_1(x)+u^2(x+1)^{\mathfrak{T}_2}z_2(x)\Big]\Big[\varphi_1(x)+u\varphi_2(x)+u^2\varphi_3(x)\Big]\\
                =&(x+1)^{\alpha} \varphi_1(x)+u\Big[ \varphi_1(x)z_1(x)+(x+1)^{\alpha}\varphi_2(x)\Big]\\
                &+u^2\Big[(x+1)^{\mathfrak{T}_2}z_2(x)\varphi_1(x)+ z_1(x)\varphi_2(x)+(x+1)^{\alpha} \varphi_3(x)\Big]
             \end{align*}
             for some $\varphi_1(x),\varphi_2(x), \varphi_3(x) \in \frac{\mathbb{F}_{p^m}[x]}{\langle x^{2^\varsigma}-1 \rangle}$. Then $\varphi_1(x)=(x+1)^{2^{\varsigma-1}-\alpha}$, $\varphi_2(x)=(x+1)^{2^{\varsigma-1}-2\alpha}z_1(x)$ and $\varphi_3(x)=(x+1)^{2^{\varsigma-1}-2\alpha+\mathfrak{T}_2}z_2(x)+(x+1)^{2^{\varsigma-1}-3\alpha}z_1(x)z_1(x)$. Since  $2^{\varsigma-1}< 3\alpha$ or $2^{\varsigma-2}+\frac{\mathfrak{T}_2}{2} < \alpha$, we get a contradiction. Thus, there exists no codeword of Lee weight 2. Also, following Theorem \ref{thm8}, $\mathcal{C}^8_5$ has no codeword of Lee weight 3. Hence $d_L(\mathcal{C}^8_5)=4$.
    \end{enumerate}
\end{proof}

%\newpage
\subsection{If $z_1(x)\neq0$ and $\mathfrak{T}_1\neq0$ and $z_2(x)\neq0$ and $\mathfrak{T}_2\neq0$ }
\begin{theorem}\label{thm23}
   Let $\mathcal{C}^9_5=\langle (x+1)^{\alpha}+u(x+1)^{\mathfrak{T}_1} z_1(x)+u^2(x+1)^{\mathfrak{T}_2}z_2(x) \rangle $,
  where $1<  \mathcal{V} \leq \mathcal{U}\leq \alpha \leq 2^\varsigma-1$, $0<  \mathfrak{T}_1 < \mathcal{U} $, $0<  \mathfrak{T}_2 < \mathcal{V} $ and  $z_1(x)$ and $z_2(x)$ are units in $\mathcal{S}$. Then 

  \begin{center}
	$d_L(\mathcal{C}^9_5)=$
	$\begin{cases}
            2&\text{if}\quad 1< \alpha \leq 2^{\varsigma-1} \quad\text{with}\quad 3\alpha \leq2^{\varsigma-1}+2\mathfrak{T}_1 \quad\text{and}\quad \alpha \leq 2^{\varsigma-2}+\frac{\mathfrak{T}_2}{2},\\
            4&\text{if}\quad 1< \alpha \leq 2^{\varsigma-1} \quad\text{with either}\quad 2^{\varsigma-1}+2\mathfrak{T}_1 < 3\alpha\quad\text{or}\quad  2^{\varsigma-2}+\frac{\mathfrak{T}_2}{2} < \alpha,\\
		4  &\text{if}\quad 2^{\varsigma-1}+1\leq \alpha \leq 
            2^{\varsigma}-1 \quad \text{with} \quad \alpha \geq 2^{\varsigma-1}+\mathfrak{T}_1,\\
            2^{\gamma+1} &\text{if}\quad 2^\varsigma-2^{\varsigma-\gamma}+1 \leq \alpha  \leq 2^\varsigma-2^{\varsigma-\gamma}+2^{\varsigma-\gamma-1}\\
            &\qquad   \text{with} \quad 3\alpha \leq 2^{\varsigma}-2^{\varsigma-\gamma}+2^{\varsigma-\gamma-1}+2\mathfrak{T}_1, \quad 2\alpha \leq 2^{\varsigma}-2^{\varsigma-\gamma}+2^{\varsigma-\gamma-1}+\mathfrak{T}_1,\\
            &\qquad \alpha \leq 2^{\varsigma-1}-2^{\varsigma-\gamma-1}+2^{\varsigma-\gamma-2}+\frac{\mathfrak{T}_2}{2} \\
            &\qquad\text{and}
            \quad \alpha \leq 2^{\varsigma-1}+\frac{\mathfrak{T}_1}{2},
            \quad \text{where}\quad 1\leq \gamma \leq \varsigma-1.
	\end{cases}$
    \end{center}
   %\hl{ Try to combine these two conditions into a single condition$\alpha \leq 2^{\varsigma-1}+\frac{\mathfrak{T}_1}{2}$ and $3\alpha \leq 2^{\varsigma}-2^{\varsigma-\gamma}+2^{\varsigma-\gamma-1}+2\mathfrak{T}_1$}
\end{theorem}
\begin{proof}
Let $\mathcal{B}=\{ \zeta_1,\zeta_2,\ldots, \zeta_m\}$ be a TOB of $\mathbb{F}_{2^m}$ over $\mathbb{F}_2.$
    \begin{enumerate}
        \item {\textbf{Case 1:}}Let $1<\alpha \leq 2^{\varsigma-1}$. By Theorem \ref{thm1}, $\mathcal{V}=\alpha$, By Thoerem \ref{thm14} and Theorem \ref{thm4}, $d_H(\mathcal{C}^9_5)=2$. Hence $2\leq d_L(\mathcal{C}^9_5)\leq 4$.
        \begin{enumerate}
            \item \textbf{Subcase i:} %\hl{Think about $3\alpha \leq2^{\varsigma-1}+2\mathfrak{T}_1 $ and $2\alpha \leq2^{\varsigma-1}+\mathfrak{T}_1 $ how to combine}
            Let $2^{\varsigma-1}+2\mathfrak{T}_1\geq 3\alpha$ and $2^{\varsigma-2}+\frac{\mathfrak{T}_2}{2} \geq \alpha$. We have $\chi (x)=\zeta_1(x^{2^{\varsigma-1}}+1)=\zeta_1(x+1)^{2^{\varsigma-1}}=\zeta_1[(x+1)^{\alpha}+u(x+1)^{\mathfrak{T}_1} z_1(x)+u^2(x+1)^{\mathfrak{T}_2} z_2(x)][(x+1)^{2^{\varsigma-1}-\alpha}+u(x+1)^{2^{\varsigma-1}-2\alpha+\mathfrak{T}_1}z_1(x)  +u^2(x+1)^{2^{\varsigma-1}-2\alpha+\mathfrak{T}_2}z_2(x)+u^2(x+1)^{2^{\varsigma-1}-3\alpha+2\mathfrak{T}_1}z_1(x)z_1(x)] \in \mathcal{C}^9_5$. Since  $wt^{\mathcal{B}}_L(\chi (x))=2$, $d_L(\mathcal{C}^9_5)=2$.
            \item  \textbf{Subcase ii:} Let either $2^{\varsigma-1}+2\mathfrak{T}_1 < 3\alpha$ or $2^{\varsigma-2}+\frac{\mathfrak{T}_2}{2} < \alpha$.
        Following as in Theorem \ref{thm8}, we get $(1+ x)^{2^{\varsigma-1}}\in\mathcal{C}^9_5$. Then  
             \begin{align*}
                  (1+ x)^{2^{\varsigma-1}}=&\Big[(x+1)^{\alpha}+u(x+1)^{\mathfrak{T}_1} z_1(x)+u^2(x+1)^{\mathfrak{T}_2}z_2(x)\Big]\Big[\varphi_1(x)+u\varphi_2(x)+u^2\varphi_3(x)\Big]\\
                  =&(x+1)^{\alpha} \varphi_1(x)+u\Big[(x+1)^{\mathfrak{T}_1} \varphi_1(x)z_1(x)+(x+1)^{\alpha}\varphi_2(x)\Big]\\
                &+u^2\Big[(x+1)^{\mathfrak{T}_2}z_2(x)\varphi_1(x)+(x+1)^{\mathfrak{T}_1} z_1(x)\varphi_2(x)+(x+1)^{\alpha} \varphi_3(x)\Big]
             \end{align*}
             for some $\varphi_1(x),\varphi_2(x), \varphi_3(x) \in \frac{\mathbb{F}_{p^m}[x]}{\langle x^{2^\varsigma}-1 \rangle}$. Then $\varphi_1(x)=(x+1)^{2^{\varsigma-1}-\alpha}$ and $\varphi_2(x)=(x+1)^{2^{\varsigma-1}-2\alpha+\mathfrak{T}_1}z_1(x)$ and $\varphi_3(x)=(x+1)^{2^{\varsigma-1}-2\alpha+\mathfrak{T}_2}z_2(x)+(x+1)^{2^{\varsigma-1}-3\alpha+2\mathfrak{T}_1}z_1(x)z_1(x)$. Since  either $2^{\varsigma-1}+2\mathfrak{T}_1 < 3\alpha$ or $2^{\varsigma-2}+\frac{\mathfrak{T}_2}{2} < \alpha$, we get a contradiction.
             Thus, there is no codeword in $\mathcal{C}^9_5$ of Lee weight 2. 
             % Now we show that there is no codeword in $\mathcal{C}^9_5$ of Lee weight 3. Let $\chi (x)\in \mathcal{C}^9_5$ with $wt^{\mathcal{B}}_L(\chi (x))=3$. From the above discussion $\chi (x)=\lambda_1 x^{k_1}+\lambda_2 x^{k_2}+\lambda_3 x^{k_3}$, where $\lambda_1, \lambda_2, \lambda_3 \in \mathcal{R} \textbackslash \{0\}$, $0\leq k_1<k_2<k_3$. Then we must have $wt^{\mathcal{B}}_L(\lambda_i)=1$ for all $i=1,2$ and $3$. That is $\lambda_i=\zeta_j$, where $a_j\in \mathcal{B}$. As $\chi (x)$ is a non-unit in $\mathcal{S}$, under the natural reduction mod $\langle x-1, u \rangle$, we have $a_1+a_2+a_3=0$. This is not possible as $a_1, a_2$ and $a_3$ are basis elements. Thus, there is no codeword in $\mathcal{C}^9_5$ of Lee weight 3. 
             Following as in Theorem \ref{thm8}, we get $\mathcal{C}^9_5$ has no codeword of Lee weight 3. Hence $d_L(\mathcal{C}^9_5)=4$.
        \end{enumerate}
        % A codeword $\wp(x)=\Big[u(x+1)^{\alpha} +u^2(x+1)^t z(x) \Big]u(x+1)^{2^{\varsigma-1}-\alpha}=u^2(x+1)^{2^{\varsigma-1}}\in \mathcal{C}^9_5$ with $wt^{\mathcal{B}}_L(\wp(x))=4$. Thus, $d_L(\mathcal{C}^9_5)=4$.
        \end{enumerate}
        \item \textbf{Case 2:} Let $2^{\varsigma-1}+1\leq \alpha \leq 2^\varsigma-1$.
        \begin{enumerate}
            \item \textbf{Subcase i:} If $\alpha \geq 2^{\varsigma-1}+\mathfrak{T}_1$ then $2^\varsigma-\alpha +\mathfrak{T}_1 \leq2^{\varsigma-1}$ and $\mathcal{V}=2^\varsigma-\alpha +\mathfrak{T}_1$. By Theorem \ref{thm3} and Theorem \ref{thm14}, $d_H(\mathcal{C}^9_5)=2$. Thus, $2\leq d_L(\mathcal{C}^9_5)$. Following as in the above case, we get $(1+ x)^{2^{\varsigma-1}}\in \mathcal{C}^9_5$. Since $\alpha > 2^{\varsigma-1}$, this is not possible. Thus, there exists no codeword of Lee weight 2. Also, following Theorem \ref{thm8}, $\mathcal{C}^9_5$ has no codeword of Lee weight 3. A codeword $\wp(x)=ua_1(x^{2^{\varsigma-1}}+1)=ua_1(x+1)^{2^{\varsigma-1}}\in \langle u(x+1)^{2^\varsigma-\alpha +\mathfrak{T}_1}\rangle \subseteq\mathcal{C}^9_5$ with $wt^{\mathcal{B}}_L(\wp(x))=4$. Thus, $d_L(\mathcal{C}^9_5)=4$.
             \item \textbf{Subcase ii:} Let $\alpha \leq 2^{\varsigma-1}+\mathfrak{T}_1$. If $\alpha \leq 2^{\varsigma-1}+\frac{\mathfrak{T}_1}{2}$ then $\mathcal{V}=\alpha$.  If $2^\varsigma-2^{\varsigma-\gamma}+1 \leq \alpha  \leq 2^\varsigma-2^{\varsigma-\gamma}+2^{\varsigma-\gamma-1},$ where $ 1\leq \gamma \leq \varsigma-1$ by Theorem \ref{thm3} and Theorem \ref{thm14}, $d_H(\langle (x+1)^{\mathcal{V}}\rangle)=2^{\gamma+1}$. Thus, $2^{\gamma+1}\leq d_L(\mathcal{C}^9_5)$. If $3\alpha \leq 2^{\varsigma}-2^{\varsigma-\gamma}+2^{\varsigma-\gamma-1}+2\mathfrak{T}_1$, $2\alpha \leq 2^{\varsigma}-2^{\varsigma-\gamma}+2^{\varsigma-\gamma-1}+\mathfrak{T}_1$  and $\alpha \leq 2^{\varsigma-1}-2^{\varsigma-\gamma-1}+2^{\varsigma-\gamma-2}+\frac{\mathfrak{T}_2}{2}$, we have
             \begin{align*}
                \prod \limits_{\alpha=1}^{\gamma+1}(x^{2^{\varsigma-\alpha}}+1)=& (x+1)^{2^{\varsigma-1}+2^{\varsigma-2}+\cdots+2^{\varsigma-\gamma-1}}\\
                 =&(x+1)^{2^{\varsigma}-2^{\varsigma-\gamma}+2^{\varsigma-\gamma-1}}\\
                 =&\Big[(x+1)^{\alpha}+u(x+1)^{\mathfrak{T}_1} z_1(x)+u^2(x+1)^{\mathfrak{T}_2}z_2(x) \Big]\Big[(x+1)^{2^\varsigma-2^{\varsigma-\gamma}+2^{\varsigma-\gamma-1}-\alpha}\\
                 &+u (x+1)^{2^\varsigma-2^{\varsigma-\gamma}+2^{\varsigma-\gamma-1}-2\alpha+\mathfrak{T}_1} z_1(x)+u^2(x+1)^{2^\varsigma-2^{\varsigma-\gamma}+2^{\varsigma-\gamma-1}-2\alpha+\mathfrak{T}_2}z_2(x)\\
                 &+u^2(x+1)^{2^\varsigma-2^{\varsigma-\gamma}+2^{\varsigma-\gamma-1}-3\alpha+2\mathfrak{T}_1}z_1(x) z_1(x)\Big] \in \mathcal{C}^9_5
             \end{align*}
             Let $f(x)=\zeta_1 \prod \limits_{\alpha=1}^{\gamma+1}(x^{2^{\varsigma-\alpha}}+1)$. Then $wt^{\mathcal{B}}_L(f(x))=2^{\gamma+1}$. Thus, $d_L(\mathcal{C}^9_5)=2^{\gamma+1}$.

        %\end{enumerate}
    \end{enumerate} 
\end{proof}
%\newpage
\subsection{Type 6:}
\begin{theorem}\cite{dinh2021hamming}\label{thm24} 
  Let $\mathcal{C}_6=\langle ((x+1)^{\alpha}+u(x+1)^{\mathfrak{T}_1} z_1(x)+u^2(x+1)^{\mathfrak{T}_2}z_2(x), u^2 (x+1)^{\omega}   \rangle $,
   where $0\leq \omega < \mathcal{V} \leq \mathcal{U}\leq \alpha \leq 2^\varsigma-1$, $0\leq  \mathfrak{T}_1 < \mathcal{U} $, $0\leq  \mathfrak{T}_2 < \omega $ and  $z_1(x)$ and $z_2(x)$ are either 0 or a unit in $\mathcal{S}$ . Then
     %\begin{center}
       $d_H(\mathcal{C}_6)= d_H(\langle (x+1)^{\omega} \rangle)$.
    %\end{center}
\end{theorem}

\begin{proposition}\label{prop6}
    Let $\mathcal{C}_6$ be a cyclic code of length $2^\varsigma$ over $\mathcal{R}$ and $\omega$ be the smallest integer such that $u^2(x+1)^{\omega} \in \mathcal{C}_6$. Then  $d_H(\mathcal{C}_6)\leq d_L(\mathcal{C}_6)\leq 2 d_H(\langle (x+1)^{\omega}\rangle)$, where $\langle (x+1)^{\omega}\rangle$ is an ideal of $\frac{\mathbb{F}_{2^m}[x]}{\langle x^{2^\varsigma}-1 \rangle}$.
\end{proposition}
\begin{proof}
    $d_H(\mathcal{C}_6)\leq d_L(\mathcal{C}_6)$ is obvious. We have $\langle u^2(x+1)^{\omega} \rangle \subseteq \mathcal{C}_6$. Then $d_L(\mathcal{C}_6)\leq d_L(\langle u^2(x+1)^{\omega} \rangle)$. The result follows from Theorem \ref{thm5}.
\end{proof}

%\newpage
\subsection{If $z_1(x)=0$ and $z_2(x)=0$}
\begin{theorem}\label{thm25}
Let $\mathcal{C}^1_6=\langle (x+1)^{\alpha} , u^2 (x+1)^{\omega}\rangle$, where $0\leq \omega < \mathcal{V} \leq \mathcal{U}\leq \alpha \leq 2^\varsigma-1$. Then
\begin{center}
    $d_L(\mathcal{C}^1_6)=$
    $\begin{cases}
        2&\text{if}\quad 1\leq \alpha \leq 2^{\varsigma-1},\\
        2 &\text{if}\quad 2^{\varsigma-1}+1\leq \alpha \leq 
        2^{\varsigma}-1 \quad \text{with} \quad\omega=0,\\
	4 &\text{if}\quad 2^{\varsigma-1}+1\leq \alpha\leq 2^{\varsigma}-1 \quad \text{with}        \quad 1\leq \omega \leq 2^{\varsigma-1},\\
        2^{\gamma+1} &\text{if}\quad 2^\varsigma-2^{\varsigma-\gamma}+1 \leq \omega < \alpha  \leq 2^\varsigma-2^{\varsigma-\gamma}+2^{\varsigma-\gamma-1}, \quad \text{where}\quad 1\leq \gamma \leq \varsigma-1.
	\end{cases}$
    \end{center}
\end{theorem}
\begin{proof}
     Let $\mathcal{B}=\{ \zeta_1,\zeta_2,\ldots, \zeta_m\}$ be a TOB of $\mathbb{F}_{2^m}$ over $\mathbb{F}_2.$ From Theorem \ref{thm24}, $d_H(\mathcal{C}^1_6)= d_H(\langle (x+1)^{\omega} \rangle)$. Thus, $d_H(\langle (x+1)^{\omega} \rangle)\leq d_L(\mathcal{C}^1_6)$. 
     %\hl{Following as in Theorem }\ref{thm}, we get $d_H(\langle (x+1)^{\omega} \rangle)\leq  d_L(\mathcal{C}^1_6) \leq 2d_H(\langle (x+1)^{\omega} \rangle) $. 
     Since $\langle (x+1)^{\alpha} \rangle\subseteq \mathcal{C}^1_6$, $ d_L(\mathcal{C}^1_6)\leq d_L(\langle (x+1)^{\alpha} \rangle)$.
     \begin{enumerate}
         \item \textbf{Case 1:} Let $1\leq \alpha \leq 2^{\varsigma-1}$. From Theorem \ref{thm15}, $d_L(\mathcal{C}^1_6)\leq 2$.
         \begin{enumerate}
             \item If $\omega>0$, by Theorem \ref{thm3}, $d_H(\langle (x+1)^{\omega} \rangle)\geq 2$. Hence $d_L(\mathcal{C}^1_6)=2$.
             \item Let $\omega=0$. Suppose $\chi (x)=\lambda x^j \in \mathcal{C}^1_6$, $\lambda \in \mathcal{R}$ with $wt^{\mathcal{B}}_L(\chi (x))=1$.
             \begin{enumerate}
                 \item If $\lambda$ is a unit in $\mathcal{R}$ then $\lambda x^j$ is a unit. This is not possible.
                 \item If $\lambda$ is non-unit in $\mathcal{R}$ then $\lambda \in \langle u \rangle$ and $wt^{\mathcal{B}}_L(\lambda)\geq 3 $. Again this is not possible.
             \end{enumerate}
             Hence $d_L(\mathcal{C}^1_6)=2$.
        \end{enumerate}
        \item \textbf{Case 2:} Let $2^{\varsigma-1}+1\leq \alpha \leq 2^\varsigma-1$.
        \begin{enumerate}
            \item \textbf{Subcase i:} Let $\omega=0$. Then $1\leq d_L(\mathcal{C}^1_6)\leq 2$.  As in the above case, $\mathcal{C}^1_6$ has no codeword of Lee weights 1.  Then $\chi (x)=\zeta_1u^2 \in \mathcal{C}^1_6$ with $wt^{\mathcal{B}}_L(\chi (x))=2$. Hence $d_L(\mathcal{C}^1_6)=2$.
            
            \item \textbf{Subcase ii:} Let $1\leq \omega \leq 2^{\varsigma-1}$. Then $2\leq d_L(\mathcal{C}^1_6)\leq 4$. Following as in Theorem \ref{thm9}, we get $(x+1)^{2^{\varsigma-1}} \in \mathcal{C}^1_6$. Then 
            \begin{align*}
                 (1+ x)^{2^{\varsigma-1}}=&\Big[(x+1)^{\alpha} \Big]\Big[\varphi_1(x)+u\varphi_2(x)+u^2\varphi_3(x)\Big]+\Big[u^2 (x+1)^{\omega}\Big]\varkappa (x)\\
                  =&(x+1)^{\alpha} \varphi_1(x)+u(x+1)^{\alpha} \varphi_2(x)+u^2\Big[ (x+1)^{\alpha} \varphi_3(x)+ (x+1)^{\omega}\varkappa (x)\Big]
             \end{align*}
            for some $\varphi_1(x),\varphi_2(x), \varphi_3(x), \varkappa (x) \in \frac{\mathbb{F}_{p^m}[x]}{\langle x^{2^\varsigma}-1 \rangle}$. Then $\varphi_1(x)=(x+1)^{2^{\varsigma-1}-\alpha}$, $\varphi_2(x)=0$ and $\chi (x)=(x+1)^{\alpha-\omega}\varphi_3(x)$. Since $ \alpha > 2^{\varsigma-1}$, we get a contradiction. Also, following Theorem \ref{thm8}. Thus, $\mathcal{C}^1_6$ has no codeword of Lee weights 3. A codeword $\wp(x)=u^2\zeta_1(x^{2^{\varsigma-1}}+1)=u^2\zeta_1(x+1)^{2^{\varsigma-1}}\in \langle u^2(x+1)^{\omega}\rangle \subseteq \mathcal{C}^1_6$ with $wt^{\mathcal{B}}_L(\wp(x))=4$. Thus, $d_L(\mathcal{C}^1_6)=4$.
            
            \item \textbf{Subcase iii:} Let $ 2^\varsigma-2^{\varsigma-\gamma}+1 \leq \omega <\alpha\leq 2^\varsigma-2^{\varsigma-\gamma}+2^{\varsigma-\gamma-1}$, where $1\leq \gamma \leq \varsigma-1$. By Theorem \ref{thm3} and Theorem \ref{thm15}, $d_H(\langle (x+1)^{\omega} \rangle)=d_L(\langle (x+1)^{\alpha} \rangle)=2^{\gamma+1}$. As $ d_H(\langle (x+1)^{\omega} \rangle)\leq d_L(\mathcal{C}^1_6) \leq d_L(\langle (x+1)^{\alpha} \rangle)$, $d_L(\mathcal{C}^1_6)=2^{\gamma+1}$.
        \end{enumerate}
     \end{enumerate}
     
\end{proof}

%%\newpage
\subsection{If $z_1(x)=0$ and $z_2(x)\neq 0$ and $\mathfrak{T}_2=0$}
\begin{theorem}\label{thm26} 
   Let $\mathcal{C}^2_6=\langle (x+1)^{\alpha}+u^2z_2(x) , u^2 (x+1)^{\omega}\rangle$,
  where $0< \omega < \mathcal{V} \leq \mathcal{U}\leq \alpha \leq 2^\varsigma-1$ and  $z_2(x)$ is a unit in $\mathcal{S}$. Then 
  \begin{center}
	$d_L(\mathcal{C}^2_6)=$
	$\begin{cases}
            2&\text{if}\quad \omega +\alpha\leq 2^{\varsigma-1},\\
		2  &\text{if}\quad z_2(x)=1 \quad\text{and} \quad\alpha=2^{\varsigma-1},\\
            4& \text{otherwise}.
	\end{cases}$
    \end{center}
\end{theorem}
\begin{proof}
   Let $\mathcal{B}=\{ \zeta_1,\zeta_2,\ldots, \zeta_m\}$ be a TOB of $\mathbb{F}_{2^m}$ over $\mathbb{F}_2.$ Let $\mathcal{V}$ be the smallest integer such that $u^2(x+1)^{\mathcal{V}} \in \mathcal{C}^2_5$. By Theorem \ref{thm1}, $\mathcal{V}=min\{\alpha, 2^\varsigma-\alpha\}$. Then $1< \mathcal{V} \leq 2^{\varsigma-1}$. Since $0< \omega< \mathcal{V} \leq 2^{\varsigma-1}$ and by Theorem \ref{thm3} and Theorem \ref{thm6}, $d_H(\mathcal{C}^2_6)=2$. Thus, $2\leq d_L(\mathcal{C}^2_6)\leq 4$. 
   \begin{enumerate}
       \item \textbf{Case 1:} Let $\omega+\alpha\leq 2^{\varsigma-1}$. Since $0<\omega <\alpha$, clearly $1<\alpha<2^{\varsigma-1}$. Let $\chi (x)=\zeta_1(x^{2^{\varsigma-1}}+1)=\zeta_1(x+1)^{2^{\varsigma-1}}=\zeta_1[(x+1)^{\alpha}+u^2z_2(x)][(x+1)^{2^{\varsigma-1}-\alpha} ]+[u^2 (x+1)^{\omega}][(x+1)^{2^{\varsigma-1}-\alpha-\omega}z_2(x)]\in \mathcal{C}^2_6$. Since  $wt^{\mathcal{B}}_L(\chi (x))=2$, $d_L(\mathcal{C}^2_6)=2$.
       \item \textbf{Case 2:} If $z_2(x)=1$ and $\alpha=2^{\varsigma-1}$, we have  $\chi (x)=\zeta_1((x+1)^{2^{\varsigma-1}}+u^2)=\zeta_1(x^{2^{\varsigma-1}}+1+u^2)\in \mathcal{C}^2_6$. Since $wt^{\mathcal{B}}_L(\chi (x))=2$, we have $d_L(\mathcal{C}^2_6)=2$.
       \item \textbf{Case 3:} Let $ \omega+ \alpha >2^{\varsigma-1}$ and either $z_2(x)\neq1$ or $\alpha \neq2^{\varsigma-1}$.   Following Theorem \ref{thm9}, we get $(1+ x)^{2^{\varsigma-1}}\in \mathcal{C}^2_6$. Then
           \begin{align*}
                 (1+ x)^{2^{\varsigma-1}}=&\Big[(x+1)^{\alpha} +u^2z_2(x)\Big]\Big[\varphi_1(x)+u\varphi_2(x)+u^2\varphi_3(x)\Big]+\Big[u^2 (x+1)^{\omega}\Big]\varkappa (x)\\
                  =&(x+1)^{\alpha} \varphi_1(x)+u(x+1)^{\alpha} \varphi_2(x)+u^2\Big[ z_2(x)\varphi_1(x)+(x+1)^{\alpha} \varphi_3(x)+u^2 (x+1)^{\omega}\varkappa (x)\Big]
             \end{align*}
             for some $\varphi_1(x),\varphi_2(x), \varphi_3(x), \varkappa (x) \in \frac{\mathbb{F}_{p^m}[x]}{\langle x^{2^\varsigma}-1 \rangle}$. Then $\varphi_1(x)=(x+1)^{2^{\varsigma-1}-\alpha}$, $\varphi_2(x)=0$ and $\chi (x)=(x+1)^{2^{\varsigma-1}-\alpha-\omega}z_2(x)+(x+1)^{\alpha-\omega}\varphi_3(x)$. Since  $ \omega+ \alpha >2^{\varsigma-1}$, we get a contradiction. Thus, there exists no codeword of Lee weight 2. Also, following Theorem \ref{thm8}, $\mathcal{C}^2_6$ has no codeword of Lee weight 3. Hence $d_L(\mathcal{C}^2_6)=4$.
   \end{enumerate}
\end{proof}

%\newpage
\subsection{If $z_1(x)=0$ and $z_2(x)\neq 0$ and $\mathfrak{T}_2\neq0$}
\begin{theorem}\label{thm27}
   Let $\mathcal{C}^3_6=\langle (x+1)^{\alpha}+u^2(x+1)^{\mathfrak{T}_2}z_2(x) , u^2 (x+1)^{\omega}\rangle$,
  where $1< \omega < \mathcal{V} \leq \mathcal{U}\leq \alpha \leq 2^\varsigma-1$, $0<  \mathfrak{T}_2 < \omega $ and  $z_2(x)$ is a unit in $\mathcal{S}$. Then 

  \begin{center}
	$d_L(\mathcal{C}^3_6)=$
	$\begin{cases}
            2&\text{if}\quad 1< \alpha \leq 2^{\varsigma-1} \quad\text{with}\quad \omega\leq 2^{\varsigma-1}-\alpha+\mathfrak{T}_2,\\
            4&\text{if}\quad 1< \alpha \leq 2^{\varsigma-1} \quad\text{with} \quad \omega>  2^{\varsigma-1}-\alpha+\mathfrak{T}_2,\\
            4  &\text{if}\quad 2^{\varsigma-1}+1\leq \alpha \leq 
            2^{\varsigma}-1 \quad \text{with} \quad   1 < \omega \leq 2^{\varsigma-1},\\
		4  &\text{if}\quad 2^{\varsigma-1}+1\leq \omega < \alpha \leq 
            2^{\varsigma}-1 \quad \text{with} \quad \alpha \geq 2^{\varsigma-1}+\mathfrak{T}_2,\\
            2^{\gamma+1} &\text{if}\quad 2^\varsigma-2^{\varsigma-\gamma}+1 \leq \omega <\alpha  \leq 2^\varsigma-2^{\varsigma-\gamma}+2^{\varsigma-\gamma-1},\\
            &\quad   \text{with} \quad \alpha \leq 2^{\varsigma-1}-2^{\varsigma-\gamma-1}+2^{\varsigma-\gamma-2}+\frac{\mathfrak{T}_2}{2},
            \quad \text{where}\quad 1\leq \gamma \leq \varsigma-1.
	\end{cases}$
    \end{center}
\end{theorem}
\begin{proof}
    Let $\mathcal{B}=\{ \zeta_1,\zeta_2,\ldots, \zeta_m\}$ be a TOB of $\mathbb{F}_{2^m}$ over $\mathbb{F}_2.$
    \begin{enumerate}
        \item \textbf{Case 1:} Let $1<\alpha \leq 2^{\varsigma-1}$. Since  $1<\omega <\alpha$, clearly $ 1<\omega\leq 2^{\varsigma-1}$, by Theorem \ref{thm3}, $d_H(\langle (x+1)^{\omega} \rangle)=2$. Thus, $2\leq d_L(\mathcal{C}^3_6)\leq 4$.
        \begin{enumerate}
                \item \textbf{Subcase i:} Let $\omega \leq 2^{\varsigma-1}-\alpha+\mathfrak{T}_2$. We have $\chi (x)=\zeta_1(x^{2^{\varsigma-1}}+1)=\zeta_1(x+1)^{2^{\varsigma-1}}=\zeta_1[(x+1)^{\alpha}+u^2(x+1)^{\mathfrak{T}_2}z_2(x)][(x+1)^{2^{\varsigma-1}-\alpha}]+[u^2(x+1)^{\omega}][(x+1)^{2^{\varsigma-1}-\alpha+\mathfrak{T}_2-\omega}z_2(x)] \in \mathcal{C}^3_6$. Since  $wt^{\mathcal{B}}_L(\chi (x))=2$, $d_L(\mathcal{C}^3_6)=2$.
            \item \textbf{Subcase ii:} Let $\omega > 2^{\varsigma-1}-\alpha+\mathfrak{T}_2$.   By following the same line of arguments as in Theorem \ref{thm8}, we get $(1+ x)^{2^{\varsigma-1}}\in \mathcal{C}^3_6$. Then
            \begin{align*}
                  (1+ x)^{2^{\varsigma-1}}=&\Big[(x+1)^{\alpha} +u^2(x+1)^{\mathfrak{T}_2}z_2(x)\Big]\Big[\varphi_1(x)+u\varphi_2(x)+u^2\varphi_3(x)\Big]+\Big[u^2 (x+1)^{\omega}\Big]\varkappa (x)\\
                  =&(x+1)^{\alpha} \varphi_1(x)+u(x+1)^{\alpha} \varphi_2(x),\\
                  &+u^2\Big[(x+1)^{\mathfrak{T}_2} z_2(x)\varphi_1(x)+(x+1)^{\alpha} \varphi_3(x)+u^2 (x+1)^{\omega}\varkappa (x)\Big]
             \end{align*}
             for some $\varphi_1(x),\varphi_2(x), \varphi_3(x), \varkappa (x) \in \frac{\mathbb{F}_{p^m}[x]}{\langle x^{2^\varsigma}-1 \rangle}$. Then $\varphi_1(x)=(x+1)^{2^{\varsigma-1}-\alpha}$, $\varphi_2(x)=0$ and $\chi (x)=(x+1)^{2^{\varsigma-1}-\alpha+\mathfrak{T}_2-\omega}z_2(x)+(x+1)^{\alpha-\omega}\varphi_3(x)$. Since  $\omega > 2^{\varsigma-1}-\alpha+\mathfrak{T}_2$, we get a contradiction. Thus, there exists no codeword of Lee weight 2. Also, following Theorem \ref{thm8}, there exist no codewords of Lee weight 3. Hence $d_L(\mathcal{C}^3_6)=4$. 
        \end{enumerate}
        \item \textbf{Case 2:} Let  $ 2^{\varsigma-1}+1\leq \alpha \leq 2^\varsigma-1$. 
        \begin{enumerate}
            \item \textbf{Subcase i:} Let $1<\omega \leq 2^{\varsigma-1}$. From Theorem \ref{thm3}, $d_H(\langle (x+1)^{\omega} \rangle)=2$. Thus, $2\leq d_L(\mathcal{C}^3_6)\leq 4$. Following Theorem \ref{thm8}, we get that there exists no codeword of Lee weight 2 or 3. Hence $d_L(\mathcal{C}^3_6)=4$.
            \item \textbf{Subcase ii:} Let $2^{\varsigma-1}+1\leq \omega <2^{\varsigma-1}$ and $\alpha \geq 2^{\varsigma-1}+\mathfrak{T}_2$. By Theorem \ref{thm3},  $d_H(\langle (x+1)^{\omega} \rangle)\geq 4$ and by Theorem \ref{thm17},  $d_L(\langle (x+1)^{\alpha}+u^2(x+1)^{\mathfrak{T}_2}z_2(x) \rangle)=4$. Thus, $d_L(\mathcal{C}^3_6)=4$.
            \item \textbf{Subcase iii:} Let $2^\varsigma-2^{\varsigma-\gamma}+1 \leq \omega <\alpha  \leq 2^\varsigma-2^{\varsigma-\gamma}+2^{\varsigma-\gamma-1}$
              and $\alpha \leq 2^{\varsigma-1}-2^{\varsigma-\gamma-1}+2^{\varsigma-\gamma-2}+\frac{\mathfrak{T}_2}{2}$, where $1\leq \gamma \leq \varsigma-1$. By Theorem \ref{thm3},  $d_H(\langle (x+1)^{\omega} \rangle)=2^{\gamma+1}$ and by Theorem \ref{thm17},  $d_L(\langle (x+1)^{\alpha}+u^2(x+1)^{\mathfrak{T}_2}z_2(x) \rangle)=2^{\gamma+1}$. Thus, $d_L(\mathcal{C}^3_6)=2^{\gamma+1}$.
        \end{enumerate}
    \end{enumerate}
\end{proof}

%\newpage
\subsection{If $z_1(x)\neq0$ and $\mathfrak{T}_1=0$ and $z_2(x)=0$}
\begin{theorem}\label{thm28}
   Let $\mathcal{C}^4_6=\langle (x+1)^{\alpha}+u z_1(x) , u^2 (x+1)^{\omega}\rangle$,
  where $0\leq \omega < \mathcal{V} \leq \mathcal{U}\leq \alpha \leq 2^\varsigma-1$ and  $z_1(x)$ is a unit in $\mathcal{S}$. Then 
   \begin{equation*}
       d_L(\mathcal{C}^4_6)=
	\begin{cases}
            2&\text{if}\quad  \omega=0,\\
            2  &\text{if} \quad 1\leq \omega \leq2^{\varsigma-1}\quad\text{with} \quad \omega+2\alpha \leq 2^{\varsigma-1},\\
		3  &\text{if} \quad 1\leq \omega      \leq2^{\varsigma-1}\quad\text{with} \quad z_1(x)=1 \quad\text{and} \quad\alpha=2^{\varsigma-1},\\
        4  &\text{if} \quad 1\leq \omega \leq2^{\varsigma-1}\quad\text{with} \quad \omega+2 \alpha >2^{\varsigma-1}\quad \text{and either} \quad z_1(x)\neq1 \quad \text{or}\quad  \alpha \neq2^{\varsigma-1}.
        \end{cases}
    \end{equation*}
 %   \begin{center}
	% $d_L(\mathcal{C}^4_6)=$
	% $\begin{cases}
 %            2&\text{if}\quad \omega+2\alpha \leq 2^{\varsigma-1} \quad\text{or} \quad \omega=0,\\
	% 	3  &\text{if}\quad z_1(x)=1 \quad\text{and} \quad\alpha=2^{\varsigma-1},\\
 %            4& \text{otherwise}.
	% \end{cases}$
 %    \end{center}
\end{theorem}
\begin{proof}
    Let $\mathcal{B}=\{ \zeta_1,\zeta_2,\ldots, \zeta_m\}$ be a TOB of $\mathbb{F}_{2^m}$ over $\mathbb{F}_2.$ Let $\mathcal{V}$ be the smallest integer such that $u^2(x+1)^{\mathcal{V}} \in \mathcal{C}^4_5$. By Theorem \ref{thm1}, $\mathcal{V}=min\{\alpha, 2^\varsigma-\alpha\}$. Then $1\leq \mathcal{V} \leq 2^{\varsigma-1}$. Since $0\leq\omega< \mathcal{V}$, clearly, $ 0\leq\omega\leq 2^{\varsigma-1}$ 
    %and by Theorem \ref{thm3} and Theorem \ref{thm24}, $d_H(\mathcal{C}^4_6)=2$. Thus, $2\leq d_L(\mathcal{C}^4_6)\leq 4$. 
    \begin{enumerate}
    \item  \textbf{Case 1:} Let $\omega=0$. Then by Theorem \ref{thm3} and Theorem \ref{thm24}, $d_H(\mathcal{C}^4_6)=1$. Thus, $1\leq d_L(\mathcal{C}^4_6)\leq 2$. Suppose $\chi (x)=\lambda x^j \in \mathcal{C}^4_6$, $\lambda \in \mathcal{R}$ with $wt^{\mathcal{B}}_L(\chi (x))=1$. If $\lambda$ is a unit in $\mathcal{R}$, then $\lambda x^j$ is a unit. This is not possible. If $\lambda$ is non-unit in $\mathcal{R}$ then $\lambda \in \langle u \rangle$ and $wt^{\mathcal{B}}_L(\lambda)\geq 3 $. Again, this is not possible. Hence $d_L(\mathcal{C}^4_6)=2$.

    \item  \textbf{Case 2:} Let $1\leq \omega \leq2^{\varsigma-1}$. Then by Theorem \ref{thm3} and Theorem \ref{thm24}, $d_H(\mathcal{C}^4_6)=2$. Thus, $2\leq d_L(\mathcal{C}^4_6)\leq 4$. 
    \begin{enumerate}
        \item  \textbf{Subcase i:} Let $\omega+2\alpha \leq 2^{\varsigma-1}$. Since $1\leq \omega <\alpha$, clearly $1<\alpha<2^{\varsigma-1}$. Let $\chi (x)=\zeta_1(x^{2^{\varsigma-1}}+1)=\zeta_1(x+1)^{2^{\varsigma-1}}=\zeta_1[(x+1)^{\alpha}+uz_1(x)][(x+1)^{2^{\varsigma-1}-\alpha}+(x+1)^{2^{\varsigma-1}-2\alpha} z_1(x)]+[u^2 (x+1)^{\omega}][(x+1)^{2^{\varsigma-1}-2\alpha-\omega}z_1(x)z_1(x)]\in \mathcal{C}^4_6$. Since  $wt^{\mathcal{B}}_L(\chi (x))=2$, $d_L(\mathcal{C}^4_6)=2$.
        
        \item \textbf{Subcase ii:} Let $z_1(x)=1$ and $\alpha=2^{\varsigma-1}$. Following Theorem \ref{thm9}, $(1+ x)^{2^{\varsigma-1}}\in \mathcal{C}^4_6$. Then
        \begin{align*}
            (1+ x)^{2^{\varsigma-1}}=&\Big[(x+1)^{\alpha} +uz_1(x)\Big]\Big[\varphi_1(x)+u\varphi_2(x)+u^2\varphi_3(x)\Big]+\Big[u^2 (x+1)^{\omega}\Big]\varkappa (x)\\
            =&(x+1)^{\alpha} \varphi_1(x)+u\Big[z_1(x)\varphi_1(x)+(x+1)^{\alpha} \varphi_2(x)\Big]\\
            &+u^2\Big[ z_1(x)\varphi_2(x)+(x+1)^{\alpha} \varphi_3(x)+ (x+1)^{\omega}\varkappa (x)\Big]
        \end{align*}
        for some $\varphi_1(x),\varphi_2(x), \varphi_3(x), \varkappa (x) \in \frac{\mathbb{F}_{p^m}[x]}{\langle x^{2^\varsigma}-1 \rangle}$. Then $\varphi_1(x)=(x+1)^{2^{\varsigma-1}-\alpha}$, $\varphi_2(x)=(x+1)^{2^{\varsigma-1}-2\alpha}z_1(x)$ and $\chi (x)=(x+1)^{2^{\varsigma-1}-2\alpha-\omega}z_1(x)z_1(x)+(x+1)^{\alpha-\omega}\varphi_3(x)$. As $\alpha=2^{\varsigma-1}$ and $\omega>0$, we have $ \omega+2 \alpha >2^{\varsigma-1}$. Then we obtain a contradiction. Thus, there exists no codeword of Lee weight 2. Also, we have  $\chi (x)=\zeta_1((x+1)^{2^{\varsigma-1}}+u)=\zeta_1(x^{2^{\varsigma-1}}+1+u)\in \mathcal{C}^4_6$. Since $wt^{\mathcal{B}}_L(\chi (x))=3$, we have $d_L(\mathcal{C}^4_6)=3$.
        
        \item \textbf{Subcase iii:} Let $ \omega+2 \alpha >2^{\varsigma-1}$ and either $z_1(x)\neq1$ or $\alpha \neq2^{\varsigma-1}$.  Following as in the above case, there exists no codeword of Lee weight 2 as $2^{\varsigma-1} < 3\alpha$. Also, following Theorem \ref{thm9}, there exists no codeword of Lee weight 3. Hence$d_L(\mathcal{C}^4_6)=4$.
    \end{enumerate}
    \end{enumerate}
\end{proof}

%\newpage
\subsection{If $z_1(x)\neq0$ and $\mathfrak{T}_1\neq0$ and $z_2(x)=0$}
\begin{theorem}\label{thm29} 
   Let $\mathcal{C}^5_6=\langle (x+1)^{\alpha}+u(x+1)^{\mathfrak{T}_1} z_1(x) , u^2 (x+1)^{\omega}\rangle$,
  where $0\leq \omega < \mathcal{V} \leq \mathcal{U}\leq \alpha \leq 2^\varsigma-1$, $0<  \mathfrak{T}_1 < \mathcal{U} $ and  $z_1(x)$ is a unit in $\mathcal{S}$. Then
  \begin{center}
	$d_L(\mathcal{C}^5_6)=$
	$\begin{cases}
    2&\text{if}\quad 1< \alpha \leq 2^{\varsigma-1} \quad\text{with}\quad \omega=0,\\
            2&\text{if}\quad 1\leq\omega< \alpha \leq 2^{\varsigma-1} \quad\text{with}\quad \omega\leq 2^{\varsigma-1}-2\alpha+2\mathfrak{T}_1\quad\text{and}\quad \alpha \leq2^{\varsigma-2}+\frac{\mathfrak{T}_1}{2},\\
            4&\text{if}\quad 1\leq\omega< \alpha \leq 2^{\varsigma-1} \quad\text{with} \quad \omega> 2^{\varsigma-1}-2\alpha+2\mathfrak{T}_1 \quad\text{or}\quad  \alpha > 2^{\varsigma-2}+\frac{\mathfrak{T}_1}{2},\\
            2 &\text{if}\quad 2^{\varsigma-1}+1\leq \alpha \leq 
            2^{\varsigma}-1 \quad \text{with} \quad   \omega =0,\\
            4  &\text{if}\quad 2^{\varsigma-1}+1\leq \alpha \leq 
            2^{\varsigma}-1 \quad \text{with} \quad   1 \leq \omega \leq 2^{\varsigma-1},\\
		4  &\text{if}\quad 2^{\varsigma-1}+1\leq \omega < \alpha \leq 
            2^{\varsigma}-1 \quad \text{with} \quad \alpha \geq 2^{\varsigma-1}+\mathfrak{T}_1,\\
            2^{\gamma+1} &\text{if}\quad 2^\varsigma-2^{\varsigma-\gamma}+1 \leq \omega <\alpha  \leq 2^\varsigma-2^{\varsigma-\gamma}+2^{\varsigma-\gamma-1},\\
            &\qquad   \text{with} \quad \alpha \leq 2^{\varsigma-1}-2^{\varsigma-\gamma-1}+2^{\varsigma-\gamma-2}+\frac{\mathfrak{T}_2}{2}, \\
            &\qquad   \text{and} \quad
            3\alpha\leq 2^\varsigma-2^{\varsigma-\gamma}+2^{\varsigma-\gamma-1}+2\mathfrak{T}_1, \quad \text{where}\quad 1\leq \gamma \leq \varsigma-1.
	\end{cases}$
    \end{center}
\end{theorem}
\begin{proof}
    Let $\mathcal{B}=\{ \zeta_1,\zeta_2,\ldots, \zeta_m\}$ be a TOB of $\mathbb{F}_{2^m}$ over $\mathbb{F}_2.$
    \begin{enumerate}
        \item \textbf{Case 1:} Let $1\leq\alpha \leq 2^{\varsigma-1}$.
        \begin{enumerate}
        \item \textbf{Subcase i:} Let $\omega=0$. 
        Then by Theorem \ref{thm3} and Theorem \ref{thm24}, $d_H(\mathcal{C}^5_6)=1$. Thus, $1\leq d_L(\mathcal{C}^5_6)\leq 2$. Suppose $\chi (x)=\lambda x^j \in \mathcal{C}^5_6$, $\lambda \in \mathcal{R}$ with $wt^{\mathcal{B}}_L(\chi (x))=1$. If $\lambda$ is a unit in $\mathcal{R}$, then $\lambda x^j$ is a unit. This is not possible. If $\lambda$ is non-unit in $\mathcal{R}$ then $\lambda \in \langle u \rangle$ and $wt^{\mathcal{B}}_L(\lambda)\geq 3 $. Again, this is not possible. Hence $d_L(\mathcal{C}^5_6)=2$.
         % Then $\chi (x)=\zeta_1u^2 \in \mathcal{C}^5_6$ with $wt^{\mathcal{B}}_L(\chi (x))=2$. Hence $d_L(\mathcal{C}^5_6)=2$.
        \item \textbf{Subcase ii:} Let $1\leq \omega <\alpha \leq 2^{\varsigma-1}$. By Theorem \ref{thm3}, $d_H(\langle (x+1)^{\omega} \rangle)=2$. Thus, $2\leq d_L(\mathcal{C}^5_6)\leq 4$.
        \begin{itemize}
            \item  Let $\omega \leq 2^{\varsigma-1}-2\alpha+2\mathfrak{T}_1$ and $2^{\varsigma-2}+\frac{\mathfrak{T}_1}{2}\geq \alpha$. We have $\chi (x)=\zeta_1(x^{2^{\varsigma-1}}+1)=\zeta_1(x+1)^{2^{\varsigma-1}}=\zeta_1[(x+1)^{\alpha}+u(x+1)^{\mathfrak{T}_1} z_1(x)][(x+1)^{2^{\varsigma-1}-\alpha}+u(x+1)^{2^{\varsigma-1}-2\alpha+\mathfrak{T}_1}z_1(x)] +[u^2 (x+1)^{\omega}][(x+1)^{2^{\varsigma-1}-2\alpha+2\mathfrak{T}_1-\omega}z_1(x)z_1(x)]\in \mathcal{C}^5_6$. Since  $wt^{\mathcal{B}}_L(\chi (x))=2$, $d_L(\mathcal{C}^5_6)=2$.
            \item  Let $\omega > 2^{\varsigma-1}-2\alpha+2\mathfrak{T}_1$ or $2^{\varsigma-2}+\frac{\mathfrak{T}_1}{2}< \alpha$.  Let $\chi (x)=\lambda_1 x^{k_1}+\lambda_2 x^{k_2} \in \mathcal{C}^5_6$ with $wt^{\mathcal{B}}_L(\chi (x))=2$, where $\lambda_1$ and $\lambda_2 \in \mathcal{R}\textbackslash \{0\}$. By following the same line of arguments as in Theorem \ref{thm8}, we get $(1+ x)^{2^{\varsigma-1}}\in \mathcal{C}^5_6$. Then
            \begin{align*}
                  (1+ x)^{2^{\varsigma-1}}=&\Big[(x+1)^{\alpha}+u(x+1)^{\mathfrak{T}_1} z_1(x)\Big]\Big[\varphi_1(x)+u\varphi_2(x)+u^2\varphi_3(x)\Big]+\Big[u^2 (x+1)^{\omega}\Big]\varkappa (x)\\
                  =&(x+1)^{\alpha} \varphi_1(x)+u\Big[(x+1)^{\mathfrak{T}_1} z_1(x)\varphi_1(x)+(x+1)^{\alpha} \varphi_2(x)\Big]\\
                  &+u^2\Big[(x+1)^{\mathfrak{T}_1} z_1(x)\varphi_2(x)+(x+1)^{\alpha} \varphi_3(x)+ (x+1)^{\omega}\varkappa (x)\Big]
             \end{align*}
             for some $\varphi_1(x),\varphi_2(x), \varphi_3(x), \varkappa (x) \in \frac{\mathbb{F}_{p^m}[x]}{\langle x^{2^\varsigma}-1 \rangle}$. Then $\varphi_1(x)=(x+1)^{2^{\varsigma-1}-\alpha}$, $\varphi_2(x)=(x+1)^{2^{\varsigma-1}-2\alpha+\mathfrak{T}_1}z_1(x)$ and $\chi (x)=(x+1)^{2^{\varsigma-1}-2\alpha+2\mathfrak{T}_1-\omega}z_1(x)z_1(x)+(x+1)^{\alpha-\omega}\varphi_3(x)$. Since  $\omega > 2^{\varsigma-1}-2\alpha+2\mathfrak{T}_1$ or $2^{\varsigma-2}+\frac{\mathfrak{T}_1}{2}< \alpha$, we get a contradiction. Thus, there exists no codeword of Lee weight 2. Also, following Theorem \ref{thm8}, there exist no codewords of Lee weight 3. Hence $d_L(\mathcal{C}^5_6)=4$. 
             \end{itemize}
        \end{enumerate}
        \item \textbf{Case 2:} Let $ 2^{\varsigma-1}+1\leq \alpha \leq 2^\varsigma-1$. 
        \begin{enumerate}
        \item \textbf{Subcase i:} Let $\omega=0$. As in the above case, $\mathcal{C}^5_6$ has no codeword of Lee weights 1. Then $\chi (x)=\zeta_1u^2 \in \mathcal{C}^5_6$ with $wt^{\mathcal{B}}_L(\chi (x))=2$. Hence $d_L(\mathcal{C}^5_6)=2$.
            \item \textbf{Subcase ii:} Let $1 \leq \omega \leq 2^{\varsigma-1}$. From Theorem \ref{thm3}, $d_H(\langle (x+1)^{\omega} \rangle)=2$. Thus, $2\leq d_L(\mathcal{C}^5_6)\leq 4$. Following Theorem \ref{thm8}, we get that there exists no codeword of Lee weight 2 or 3. Hence $d_L(\mathcal{C}^5_6)=4$.
            \item \textbf{Subcase iii:} Let $2^{\varsigma-1}+1\leq \omega <2^{\varsigma-1}$ and $\alpha \geq 2^{\varsigma-1}+\mathfrak{T}_1$. By Theorem \ref{thm3},  $d_H(\langle (x+1)^{\omega} \rangle)\geq 4$ and by Theorem \ref{thm19},  $d_L(\langle (x+1)^{\alpha}+u(x+1)^{\mathfrak{T}_1} z_1(x) \rangle)=4$. Thus, $d_L(\mathcal{C}^5_6)=4$.
            \item \textbf{Subcase iv:} Let $2^\varsigma-2^{\varsigma-\gamma}+1 \leq \omega <\alpha  \leq 2^\varsigma-2^{\varsigma-\gamma}+2^{\varsigma-\gamma-1}$
              and $\alpha \leq 2^{\varsigma-1}-2^{\varsigma-\gamma-1}+2^{\varsigma-\gamma-2}+\frac{\mathfrak{T}_1}{2}$ and $3\alpha\leq 2^\varsigma-2^{\varsigma-\gamma}+2^{\varsigma-\gamma-1}+2\mathfrak{T}_1$, where $1\leq \gamma \leq \varsigma-1$. By Theorem \ref{thm3},  $d_H(\langle (x+1)^{\omega} \rangle)=2^{\gamma+1}$ and by Theorem \ref{thm19},  $d_L(\langle (x+1)^{\alpha}+u(x+1)^{\mathfrak{T}_1} z_1(x) \rangle)=2^{\gamma+1}$. Thus, $d_L(\mathcal{C}^5_6)=2^{\gamma+1}$.
        \end{enumerate}
    \end{enumerate}
\end{proof}

%\newpage
\subsection{If $z_1(x)\neq0$ and $\mathfrak{T}_1=0$ and $z_2(x)\neq0$ and $\mathfrak{T}_2=0$ }
\begin{theorem}\label{thm30} 
   Let $\mathcal{C}^6_6=\langle (x+1)^{\alpha}+u z_1(x)+u^2z_2(x) , u^2 (x+1)^{\omega}\rangle$,
  where $0< \omega < \mathcal{V} \leq \mathcal{U}\leq \alpha \leq 2^\varsigma-1$,  and  $z_1(x)$ and $z_2(x)$ are units in $\mathcal{S}$. Then 
  \begin{center}
	$d_L(\mathcal{C}^6_6)=$
		$\begin{cases}
			2  &\text{if}\quad \omega+2\alpha \leq 2^{\varsigma-1}, \\
            3  &\text{if}\quad z_1(x)=z_2(x)=1 \quad\text{and} \quad\alpha=2^{\varsigma-1},\\
                 4& \text{otherwise}.
			\end{cases}$
		\end{center}
\end{theorem}
\begin{proof}
   Let $\mathcal{B}=\{ \zeta_1,\zeta_2,\ldots, \zeta_m\}$ be a TOB of $\mathbb{F}_{2^m}$ over $\mathbb{F}_2.$ Let $\mathcal{V}$ be the smallest integer such that $u^2(x+1)^{\mathcal{V}} \in \mathcal{C}^6_5$. By Theorem \ref{thm1}, $\mathcal{V}=min\{\alpha, 2^\varsigma-\alpha\}$. Then $1< \mathcal{V} \leq 2^{\varsigma-1}$. Since $0< \omega< \mathcal{V} \leq 2^{\varsigma-1}$ and by Theorem \ref{thm3} and Theorem \ref{thm24}, $d_H(\mathcal{C}^6_6)=2$. Thus, $2\leq d_L(\mathcal{C}^6_6)\leq 4$. 
   \begin{enumerate}
       \item \textbf{Case 1:} Let $\omega+2\alpha \leq 2^{\varsigma-1}$. Since $0<\omega <\alpha$, clearly $1<\alpha<2^{\varsigma-1}$. Let $\chi (x)=\zeta_1(x^{2^{\varsigma-1}}+1)=\zeta_1(x+1)^{2^{\varsigma-1}}=\zeta_1[(x+1)^{\alpha}+u z_1(x)+u^2z_2(x)][(x+1)^{2^{\varsigma-1}-\alpha} +u(x+1)^{2^{\varsigma-1}-2\alpha}z_1(x) ]+[u^2 (x+1)^{\omega}][(x+1)^{2^{\varsigma-1}-\alpha-\omega}z_2(x)+(x+1)^{2^{\varsigma-1}-2\alpha-\omega}z_1(x)z_1(x)]\in \mathcal{C}^6_6$. Since  $wt^{\mathcal{B}}_L(\chi (x))=2$, $d_L(\mathcal{C}^6_6)=2$.
       \item \textbf{Case 2:} Let $z_1(x)=z_2(x)=1$ and $\alpha=2^{\varsigma-1}$.
       Following as in Theorem \ref{thm9}, we get  $(1+ x)^{2^{\varsigma-1}}\in \mathcal{C}^6_6$. Then 
            \begin{align*}
                   (1+ x)^{2^{\varsigma-1}}=&\Big[(x+1)^{\alpha}+u z_1(x) +u^2z_2(x)\Big]\Big[\varphi_1(x)+u\varphi_2(x)+u^2\varphi_3(x)\Big]+\Big[u^2 (x+1)^{\omega}\Big]\varkappa (x)\\
                  =&(x+1)^{\alpha} \varphi_1(x)+u\Big[\varphi_1(x)z_1(x)+(x+1)^{\alpha} \varphi_2(x)\Big]+u^2\Big[ z_2(x)\varphi_1(x)+\varphi_2(x)z_1(x)\\
                  &+(x+1)^{\alpha} \varphi_3(x)+(x+1)^{\omega}\varkappa (x)\Big]
             \end{align*}
             for some $\varphi_1(x),\varphi_2(x), \varphi_3(x), \varkappa (x) \in \frac{\mathbb{F}_{p^m}[x]}{\langle x^{2^\varsigma}-1 \rangle}$. Then $\varphi_1(x)=(x+1)^{2^{\varsigma-1}-\alpha}$, $\varphi_2(x)=(x+1)^{2^{\varsigma-1}-2\alpha}z_1(x)$ and $\chi (x)=(x+1)^{2^{\varsigma-1}-\alpha-\omega}z_2(x)+(x+1)^{2^{\varsigma-1}-2\alpha-\omega}z_1(x)z_1(x)+(x+1)^{\alpha-\omega}\varphi_3(x)$. As $\alpha=2^{\varsigma-1}$ and $\omega>0$, we have $ \omega+ \alpha >2^{\varsigma-1}$. Then we obtain a contradiction. Thus, there exists no codeword of Lee weight 2. Also, we have  $\chi (x)=\zeta_1((x+1)^{2^{\varsigma-1}}+u+u^2)=\zeta_1(x^{2^{\varsigma-1}}+1+u+u^2)\in \mathcal{C}^6_6$. Since $wt^{\mathcal{B}}_L(\chi (x))=2$, we have $d_L(\mathcal{C}^6_6)=3$.
       \item \textbf{Case 3:} Let $ \omega+ \alpha >2^{\varsigma-1}$ and either $z_2(x)\neq1$ or $\alpha \neq2^{\varsigma-1}$. Following as in the above case, there exists no codeword of Lee weight 2 as $2^{\varsigma-1} < 3\alpha$. Also, following Theorem \ref{thm8}, there exists no codeword of Lee weight 3. Hence $d_L(\mathcal{C}^6_6)=4$.
   \end{enumerate}
\end{proof}

%\newpage
\subsection{If $z_1(x)\neq0$ and $\mathfrak{T}_1\neq0$ and $z_2(x)\neq0$ and $\mathfrak{T}_2=0$ }
\begin{theorem}\label{thm31} 
   Let $\mathcal{C}^7_6=\langle (x+1)^{\alpha}+u(x+1)^{\mathfrak{T}_1} z_1(x)+u^2z_2(x) , u^2 (x+1)^{\omega}\rangle$,
  where $0< \omega < \mathcal{V} \leq \mathcal{U}\leq \alpha \leq 2^\varsigma-1$, $0< \mathfrak{T}_1 < \mathcal{U} $ and  $z_1(x)$ and $z_2(x)$ are units in $\mathcal{S}$. Then 
\begin{center}
    $d_L(\mathcal{C}^7_6)=$
    $\begin{cases}
        2&\text{if}\quad 1\leq \omega<\alpha \leq 2^{\varsigma-1}\quad \text{with} \quad 2\alpha \leq 2^{\varsigma-1}+\mathfrak{T}_1,\quad \alpha+\omega \leq 2^{\varsigma-1},\\
        &\qquad \text{and} \quad 2\alpha+\omega \leq 2^{\varsigma-1}+2\mathfrak{T}_1,\\
        4&\text{if}\quad 1\leq \omega<\alpha \leq 2^{\varsigma-1}\quad \text{either with} \quad 2\alpha > 2^{\varsigma-1}+\mathfrak{T}_1,\quad \alpha+\omega > 2^{\varsigma-1},\\
        &\qquad \text{or} \quad 2\alpha+\omega > 2^{\varsigma-1}+2\mathfrak{T}_1,\\
	4 &\text{if}\quad 2^{\varsigma-1}+1\leq \alpha\leq             2^{\varsigma}-1\quad\text{with} \qquad 1\leq \omega \leq 2^{\varsigma-1},\\
        2^{\gamma+1} &\text{if}\quad 2^\varsigma-2^{\varsigma-\gamma}+1 \leq \omega < \alpha  \leq 2^\varsigma-2^{\varsigma-\gamma}+2^{\varsigma-\gamma-1} ,\\
        &\quad   \text{with} \quad 3\alpha \leq 2^{\varsigma}-2^{\varsigma-\gamma}+2^{\varsigma-\gamma-1}+2\mathfrak{T}_1 \quad\text{and}\\
            &\quad \alpha \leq 2^{\varsigma-1}+\frac{\mathfrak{T}_1}{2}
            \quad \text{where}\quad 1\leq \gamma \leq \varsigma-1.
	\end{cases}$
    \end{center}
\end{theorem}
\begin{proof}
     Let $\mathcal{B}=\{ \zeta_1,\zeta_2,\ldots, \zeta_m\}$ be a TOB of $\mathbb{F}_{2^m}$ over $\mathbb{F}_2.$ From Theorem \ref{thm24}, $d_H(\mathcal{C}^7_6)= d_H(\langle (x+1)^{\omega} \rangle)$. Thus, $d_H(\langle (x+1)^{\omega} \rangle)\leq d_L(\mathcal{C}^7_6)$.
     \begin{enumerate}
         \item \textbf{Case 1:} Let $1\leq \omega<\alpha \leq 2^{\varsigma-1}$. By Theorem \ref{thm3}, we have $d_H(\langle (x+1)^{\omega} \rangle)= 2$. Hence by Theorem \ref{thm24}, $2\leq d_L(\mathcal{C}^7_6)\leq 4$.
         %From Theorem \ref{thm15}, $d_L(\mathcal{C}^7_6)\leq 2$.
         \begin{enumerate}
                 \item \textbf{Subcase i:} If $2\alpha \leq 2^{\varsigma-1}+\mathfrak{T}_1$, $\alpha+\omega \leq 2^{\varsigma-1}$ and $2\alpha+\omega \leq 2^{\varsigma-1}+2\mathfrak{T}_1$.  We have $\chi (x)=\zeta_1(x^{2^{\varsigma-1}}+1)=\zeta_1(x+1)^{2^{\varsigma-1}}=\zeta_1\Big[(x+1)^{\alpha}+u (x+1)^{\mathfrak{T}_1}z_1(x)+u^2z_2(x)\Big]\Big[(x+1)^{2^{\varsigma-1}-\alpha}+u(x+1)^{2^{\varsigma-1}-2\alpha+\mathfrak{T}_1}z_1(x)\Big] +\Big[u^2(x+1)^{\omega}\Big]\Big[(x+1)^{2^{\varsigma-1}-\alpha-\omega}z_2(x)+(x+1)^{2^{\varsigma-1}-2\alpha-\omega+2\mathfrak{T}_1}z_1(x)z_1(x)\Big]\in \mathcal{C}^7_6$. Since  $wt^{\mathcal{B}}_L(\chi (x))=2$, $d_L(\mathcal{C}^7_6)=2$.
                 \item \textbf{Subcase ii:} Let either $2\alpha > 2^{\varsigma-1}+\mathfrak{T}_1$ or $\alpha+\omega > 2^{\varsigma-1}$ or $2\alpha+\omega > 2^{\varsigma-1}+2\mathfrak{T}_1$.
                Following as in Theorem \ref{thm9}, we get $(x+1)^{2^{\varsigma-1}} \in \mathcal{C}^7_6$.
                Then 
            \begin{align*}
                   (1+ x)^{2^{\varsigma-1}}=&\Big[(x+1)^{\alpha}+u(x+1)^{\mathfrak{T}_1} z_1(x) +u^2z_2(x)\Big]\Big[\varphi_1(x)+u\varphi_2(x)+u^2\varphi_3(x)\Big]\\
                   &+\Big[u^2 (x+1)^{\omega}\Big]\varkappa (x)\\
                  =&(x+1)^{\alpha} \varphi_1(x)+u\Big[(x+1)^{\mathfrak{T}_1}\varphi_1(x)z_1(x)+(x+1)^{\alpha} \varphi_2(x)\Big]\\
                  &+u^2\Big[ z_2(x)\varphi_1(x)+(x+1)^{\mathfrak{T}_1}\varphi_2(x)z_1(x)+(x+1)^{\alpha} \varphi_3(x)+(x+1)^{\omega}\varkappa (x)\Big]
             \end{align*}
             for some $\varphi_1(x),\varphi_2(x), \varphi_3(x), \varkappa (x) \in \frac{\mathbb{F}_{p^m}[x]}{\langle x^{2^\varsigma}-1 \rangle}$. Then $\varphi_1(x)=(x+1)^{2^{\varsigma-1}-\alpha}$, $\varphi_2(x)=(x+1)^{2^{\varsigma-1}-2\alpha+\mathfrak{T}_1}z_1(x)$ and $\chi (x)=(x+1)^{2^{\varsigma-1}-\alpha-\omega}z_2(x)+(x+1)^{2^{\varsigma-1}-2\alpha-\omega+2\mathfrak{T}_1}z_1(x)z_1(x)+(x+1)^{\alpha-\omega}\varphi_3(x)$. 
                Since either $2\alpha > 2^{\varsigma-1}+\mathfrak{T}_1$ or $\alpha+\omega > 2^{\varsigma-1}$ or $2\alpha+\omega > 2^{\varsigma-1}+2\mathfrak{T}_1$, we get a contradiction. Thus $\mathcal{C}^7_6$ has no codeword of Lee weights 2. Also, following Theorem \ref{thm8}, $\mathcal{C}^7_6$ has no codeword of Lee weights 3.
                 
             \end{enumerate}
        \item \textbf{Case 2:} Let $2^{\varsigma-1}+1\leq \alpha \leq 2^\varsigma-1$.
        \begin{enumerate}
            \item \textbf{Subcase i:} Let $1\leq \omega \leq 2^{\varsigma-1}$. By Theorem \ref{thm3}, we have $d_H(\langle (x+1)^{\omega} \rangle)= 2$. Hence by Theorem \ref{thm24}, $2\leq d_L(\mathcal{C}^7_6)\leq 4$. Following as in the above case, we get $(x+1)^{2^{\varsigma-1}} \in \mathcal{C}^7_6$. Since $ \alpha > 2^{\varsigma-1}$, we get a contradiction.
            Thus, there exists no codeword of Lee weight 2. Also, following Theorem \ref{thm8}, $\mathcal{C}^7_6$ has no codeword of Lee weight 3. Hence $d_L(\mathcal{C}^7_6)=4$.

            % Additionally, by Theorem \ref{thm8}, $\mathcal{C}^7_6$ does not contain any codewords with Lee weights 2 and 3. A codeword $\wp(x)=u^2\zeta_1(x^{2^{\varsigma-1}}+1)=u^2\zeta_1(x+1)^{2^{\varsigma-1}}\in \langle u^2(x+1)^{\omega}\rangle \subseteq \mathcal{C}^7_6$ with $wt^{\mathcal{B}}_L(\wp(x))=4$. Thus, $d_L(\mathcal{C}^7_6)=4$.
            
            \item \textbf{Subcase ii:} Let $ 2^\varsigma-2^{\varsigma-\gamma}+1 \leq \omega<\alpha \leq 2^\varsigma-2^{\varsigma-\gamma}+2^{\varsigma-\gamma-1}$, where $1\leq \gamma \leq \varsigma-1$. From Theorem \ref{thm3}, $d_H(\langle (x+1)^{\omega}\rangle)=2^{\gamma+1}$.  Then $d_L(\mathcal{C}^7_6)\geq 2^{\gamma+1}$. 
            %From Theorem \ref{thm8}, $d_L(\langle u(x+1)^{\alpha} +u^2(x+1)^t z(x) \rangle )=2^{\gamma+1}$. Then $d_L(\mathcal{C}^7_6)\leq 2^{\gamma+1}$ if $\alpha \leq 2^{\varsigma-1}-2^{\varsigma-\gamma-1}+2^{\varsigma-\gamma-2}+\frac{\mathfrak{T}_2}{2}$. Hence $d_L(\mathcal{C}^7_6)=2^{\gamma+1}$.
            %If $\alpha \leq 2^{\varsigma-1}-2^{\varsigma-\gamma-1}+2^{\varsigma-\gamma-2}+\frac{\mathfrak{T}_1}{2}$, $\alpha+\omega \leq 2^{\varsigma}-2^{\varsigma-\gamma}+2^{\varsigma-\gamma-1}$ and $2\alpha+\omega \leq 2^{\varsigma}-2^{\varsigma-\gamma}+2^{\varsigma-\gamma-1}+2\mathfrak{T}_1$ we have
            And by Theorem \ref{thm21},  $d_L(\langle (x+1)^{\alpha}+u(x+1)^{\mathfrak{T}_1} z_1(x)+u^2z_2(x) \rangle)=2^{\gamma+1}$ if $\quad 3\alpha \leq 2^{\varsigma}-2^{\varsigma-\gamma}+2^{\varsigma-\gamma-1}+2\mathfrak{T}_1$ and $\alpha \leq 2^{\varsigma-1}+\frac{\mathfrak{T}_1}{2}$. Thus, $d_L(\mathcal{C}^7_6)=2^{\gamma+1}$.

             % \begin{align*}
             %    \prod \limits_{\alpha=1}^{\gamma+1}(x^{2^{\varsigma-\alpha}}+1)=& (x+1)^{2^{\varsigma-1}+2^{\varsigma-2}+\cdots+2^{\varsigma-\gamma-1}},\\
             %     =&(x+1)^{2^{\varsigma}-2^{\varsigma-\gamma}+2^{\varsigma-\gamma-1}},\\
             %     =&\Big[(x+1)^{\alpha}+u(x+1)^{\mathfrak{T}_1}z_1(x)+u^2z_2(x) \Big]\Big[(x+1)^{2^\varsigma-2^{\varsigma-\gamma}+2^{\varsigma-\gamma-1}-\alpha},\\
             %     &+u(x+1)^{2^\varsigma-2^{\varsigma-\gamma}+2^{\varsigma-\gamma-1}-2\alpha+\mathfrak{T}_1}z_2(x)\Big] 
             %     +\Big[ u^2(x+1)^{\omega}\Big]\\
             %     &\Big[ u\Big((x+1)^{2^\varsigma-2^{\varsigma-\gamma}+2^{\varsigma-\gamma-1}-\alpha-\omega}z_2(x)+(x+1)^{2^\varsigma-2^{\varsigma-\gamma}+2^{\varsigma-\gamma-1}-2\alpha-\omega+2\mathfrak{T}_1}z_1(x)z_1(x)\Big)\Big]\in \mathcal{C}^7_6
             % \end{align*}
             % Let $f(x)=\zeta_1 \prod \limits_{\alpha=1}^{\gamma+1}(x^{2^{\varsigma-\alpha}}+1)$. Then $wt^{\mathcal{B}}_L(f(x))=2^{\gamma+1}$. Thus, $d_L(\mathcal{C}^7_6)=2^{\gamma+1}$.
        \end{enumerate}
     \end{enumerate}
     
\end{proof}

%\newpage
\subsection{If $z_1(x)\neq0$ and $\mathfrak{T}_1=0$ and $z_2(x)\neq0$ and $\mathfrak{T}_2\neq0$ }
\begin{theorem}\label{thm32} 
   Let $\mathcal{C}^8_6=\langle (x+1)^{\alpha}+uz_1(x)+u^2(x+1)^{\mathfrak{T}_2}z_2(x) , u^2 (x+1)^{\omega}\rangle$,
  where $1< \omega < \mathcal{V} \leq \mathcal{U}\leq \alpha \leq 2^\varsigma-1$, $0<  \mathfrak{T}_2 < \omega $ and  $z_1(x)$ and $z_2(x)$ are units in $\mathcal{S}$.  Then 
  \begin{center}
	$d_L(\mathcal{C}^8_6)=$
		$\begin{cases}
			2  &\text{if}\quad \omega+2\alpha \leq 2^{\varsigma-1} \quad \text{and} \quad\omega+\alpha \leq 2^{\varsigma-1}+\mathfrak{T}_2,\\
                 4& \text{otherwise}.
			\end{cases}$
		\end{center}
\end{theorem}
\begin{proof}
     Let $\mathcal{B}=\{ \zeta_1,\zeta_2,\ldots, \zeta_m\}$ be a TOB of $\mathbb{F}_{2^m}$ over $\mathbb{F}_2.$ Let $\mathcal{V}$ be the smallest integer such that $u^2(x+1)^{\mathcal{V}} \in \mathcal{C}^8_5$. By Theorem \ref{thm1}, $\mathcal{V}=min\{\alpha, 2^\varsigma-\alpha\}$. Then $1< \mathcal{V} \leq 2^{\varsigma-1}$. Since $1< \omega< \mathcal{V} \leq 2^{\varsigma-1}$ and by Theorem \ref{thm3} and Theorem \ref{thm24}, $d_H(\mathcal{C}^8_6)=2$. Thus, $2\leq d_L(\mathcal{C}^8_6)\leq 4$. 
    \begin{enumerate}
        \item \textbf{Case 1:} Let $\omega +2\alpha\leq 2^{\varsigma-1}$ and $\omega\leq 2^{\varsigma-1}-\alpha+\mathfrak{T}_2$. Since $0<\omega <\alpha$, clearly $\alpha<2^{\varsigma-1}$. Let $\chi (x)=\zeta_1(x^{2^{\varsigma-1}}+1)=\zeta_1(x+1)^{2^{\varsigma-1}}=\zeta_1[(x+1)^{\alpha}+uz_1(x)+u^2(x+1)^{\mathfrak{T}_2}z_2(x)][(x+1)^{2^{\varsigma-1}-\alpha} +u(x+1)^{2^{\varsigma-1}-2\alpha}z_1(x) ]+[u^2 (x+1)^{\omega}][(x+1)^{2^{\varsigma-1}-\alpha+\mathfrak{T}_2-\omega}z_2(x)+(x+1)^{2^{\varsigma-1}-2\alpha-\omega}z_1(x)z_1(x)]\in \mathcal{C}^8_6$. Since  $wt^{\mathcal{B}}_L(\chi (x))=2$, $d_L(\mathcal{C}^8_6)=2$.
        \item \textbf{Case 2:} Let $\omega+2\alpha > 2^{\varsigma-1}$ or $\omega> 2^{\varsigma-1}-\alpha+\mathfrak{T}_2$. 
        % Following as in Theorem \ref{thm8}, there exist no codewords of the form $\lambda_1 x^{k_1}+\lambda_2 x^{k_2}$ in $\mathcal{C}^8_6$ with $\lambda_1$ or $\lambda_2$ non-unit in $\mathcal{R}$, where $\lambda_1, \lambda_2 \in \mathcal{R} \textbackslash \{0\}$, $0\leq {k_1}<{k_2}$. Let $\chi (x)=\lambda_1 x^{k_1}+\lambda_2 x^{k_2} \in \mathcal{C}^8_6$ with $\lambda_1$ and $\lambda_2$ are units in $\mathcal{R}$. Then we must have $wt^{\mathcal{B}}_L(\lambda_i)=1$ for all $i=1$ and $2$. That is $\lambda_i=\zeta_j$, where $a_j\in \mathcal{B}$. As $\chi (x)$ is a non-unit in $\mathcal{S}$, under the natural reduction mod $\langle x-1, u \rangle$, we have $a_1+a_2=0$. Since $a_1$ and $a_2$ are basis elements we get a contradiction if $a_1 \neq a_2$. If $a_1=\zeta_2$ we get $1+x^{k_2-k_1} \in \mathcal{C}^8_6$. We can write $k_1-k_2=2^wr$, where $1\leq w \leq s-1$ and $r$ is odd.
        Following Theorem \ref{thm9}, $(1+ x)^{2^{\varsigma-1}}\in \mathcal{C}^8_6$. Then
        %We get a contradiction as in Theorem \ref{thm8}.
            \begin{align*}
                (1+ x)^{2^{\varsigma-1}}=&\Big[(x+1)^{\alpha}+uz_1(x)+u^2(x+1)^{\mathfrak{T}_2}z_2(x)\Big]\Big[\varphi_1(x)+u\varphi_2(x)+u^2\varphi_3(x)\Big]\\
                &+\Big[u^2 (x+1)^{\omega}\Big]\varkappa (x)\\
                =&(x+1)^{\alpha} \varphi_1(x)+u\Big[ \varphi_1(x)z_1(x)+(x+1)^{\alpha}\varphi_2(x)\Big]\\
                &+u^2\Big[(x+1)^{\mathfrak{T}_2}z_2(x)\varphi_1(x)+ z_1(x)\varphi_2(x)+(x+1)^{\alpha} \varphi_3(x)+(x+1)^{\omega}\varkappa (x)\Big]
             \end{align*}
             for some $\varphi_1(x),\varphi_2(x), \varphi_3(x) \in \frac{\mathbb{F}_{p^m}[x]}{\langle x^{2^\varsigma}-1 \rangle}$. Then $\varphi_1(x)=(x+1)^{2^{\varsigma-1}-\alpha}$, $\varphi_2(x)=(x+1)^{2^{\varsigma-1}-2\alpha}z_1(x)$ and $\chi (x)=(x+1)^{2^{\varsigma-1}-\alpha+\mathfrak{T}_2-\omega}z_2(x)+(x+1)^{2^{\varsigma-1}-2\alpha-\omega}z_1(x)z_1(x)+(x+1)^{\alpha-\omega}\varphi_3(x)$. Since  $\omega+2\alpha > 2^{\varsigma-1}$ or $\omega> 2^{\varsigma-1}-\alpha+\mathfrak{T}_2$, we get a contradiction. Thus, there exists no codeword of Lee weight 2. Also, following Theorem \ref{thm8},  $\mathcal{C}^8_6$ has no codeword of Lee weight 3. Hence $d_L(\mathcal{C}^8_6)=4$.
    \end{enumerate}
\end{proof}

%\newpage
\subsection{If $z_1(x)\neq0$ and $\mathfrak{T}_1\neq0$ and $z_2(x)\neq0$ and $\mathfrak{T}_2\neq0$ }
\begin{theorem}\label{thm33} 
   Let $\mathcal{C}^9_6=\langle (x+1)^{\alpha}+u(x+1)^{\mathfrak{T}_1} z_1(x)+u^2(x+1)^{\mathfrak{T}_2}z_2(x) , u^2 (x+1)^{\omega}\rangle$,
  where $1<\omega < \mathcal{V} \leq \mathcal{U}\leq \alpha \leq 2^\varsigma-1$, $0< \mathfrak{T}_1 < \mathcal{U} $, $0<  \mathfrak{T}_2 < \omega $ and  $z_1(x)$ and $z_1(x)$ are units in $\mathcal{S}$. Then 
   \begin{center}
	$d_L(\mathcal{C}^9_6)=$
	$\begin{cases}
            2&\text{if}\quad 1< \alpha \leq 2^{\varsigma-1} \quad\text{with}\quad \omega \leq 2^{\varsigma-1}-\alpha+2\mathfrak{T}_2,\quad 2\alpha \leq 2^{\varsigma-1}+\mathfrak{T}_1\\
            &\qquad\text{and}\quad \omega \leq 2^{\varsigma-1}-2\alpha+2\mathfrak{T}_1,\\
            4&\text{if}\quad 1< \alpha \leq 2^{\varsigma-1} \quad\text{with} \quad \omega > 2^{\varsigma-1}-\alpha+2\mathfrak{T}_2 \quad\text{or}\quad2\alpha \leq 2^{\varsigma-1}+\mathfrak{T}_1 \\
            &\qquad\text{or}\quad \omega > 2^{\varsigma-1}-2\alpha+2\mathfrak{T}_1,\\
            4  &\text{if}\quad 2^{\varsigma-1}+1\leq \alpha \leq 
            2^{\varsigma}-1 \quad \text{with} \quad   1 < \omega \leq 2^{\varsigma-1},\\
		4  &\text{if}\quad 2^{\varsigma-1}+1\leq \omega < \alpha \leq 
            2^{\varsigma}-1 \quad \text{with} \quad \alpha \geq 2^{\varsigma-1}+\mathfrak{T}_1,\\
            2^{\gamma+1} &\text{if}\quad 2^\varsigma-2^{\varsigma-\gamma}+1 \leq \omega <\alpha  \leq 2^\varsigma-2^{\varsigma-\gamma}+2^{\varsigma-\gamma-1}\\
            &\quad   \text{with} \quad 3\alpha \leq 2^{\varsigma}-2^{\varsigma-\gamma}+2^{\varsigma-\gamma-1}+2\mathfrak{T}_1, \quad \alpha \leq 2^{\varsigma-1}-2^{\varsigma-\gamma-1}+2^{\varsigma-\gamma-2}+\frac{\mathfrak{T}_2}{2} \\
            &\quad\text{and}\quad \alpha \leq 2^{\varsigma-1}+\frac{\mathfrak{T}_1}{2}
            \quad \text{where}\quad 1\leq \gamma \leq \varsigma-1.
	\end{cases}$
    \end{center}
\end{theorem}
\begin{proof}
    Let $\mathcal{B}=\{ \zeta_1,\zeta_2,\ldots, \zeta_m\}$ be a TOB of $\mathbb{F}_{2^m}$ over $\mathbb{F}_2.$
    \begin{enumerate}
        \item \textbf{Case 1:} If $1<\alpha \leq 2^{\varsigma-1}$. Since  $1<\omega <\alpha \leq 2^{\varsigma-1}$, by Theorem \ref{thm3}, $d_H(\langle (x+1)^{\omega} \rangle)=2$. Thus, $2\leq d_L(\mathcal{C}^9_6)\leq 4$.
        \begin{enumerate}
            \item \textbf{Subcase i:} Let $\omega \leq 2^{\varsigma-1}-\alpha+\mathfrak{T}_2$, $2\alpha \leq 2^{\varsigma-1}+\mathfrak{T}_1$ and $\omega \leq 2^{\varsigma-1}-2\alpha+2\mathfrak{T}_1$. We have $\chi (x)=\zeta_1(x^{2^{\varsigma-1}}+1)=\zeta_1(x+1)^{2^{\varsigma-1}}=\zeta_1[(x+1)^{\alpha}+u(x+1)^{\mathfrak{T}_1} z_1(x)+u^2(x+1)^{\mathfrak{T}_2}z_2(x)][(x+1)^{2^{\varsigma-1}-\alpha}+u(x+1)^{2^{\varsigma-1}-2\alpha+\mathfrak{T}_1}z_1(x)] +[u^2 (x+1)^{\omega}][(x+1)^{2^{\varsigma-1}-\alpha+\mathfrak{T}_2-\omega}z_2(x)+(x+1)^{2^{\varsigma-1}-2\alpha+2\mathfrak{T}_1-\omega}z_1(x)z_1(x)]\in \mathcal{C}^9_6$. Since  $wt^{\mathcal{B}}_L(\chi (x))=2$, $d_L(\mathcal{C}^9_6)=2$.
            \item \textbf{Subcase ii:} Let $\omega > 2^{\varsigma-1}-\alpha+\mathfrak{T}_2$ or $2\alpha \leq 2^{\varsigma-1}+\mathfrak{T}_1$ or $\omega > 2^{\varsigma-1}-2\alpha+2\mathfrak{T}_1$. By following the same line of arguments as in Theorem \ref{thm8}, we get $(1+ x)^{2^{\varsigma-1}}\in \mathcal{C}^9_6$. Then
            \begin{align*}
                  (1+ x)^{2^{\varsigma-1}}=&\Big[(x+1)^{\alpha}+u(x+1)^{\mathfrak{T}_1} z_1(x)+u^2(x+1)^{\mathfrak{T}_2}z_2(x)\Big]\Big[\varphi_1(x)+u\varphi_2(x)+u^2\varphi_3(x)\Big]\\
                  &+\Big[u^2 (x+1)^{\omega}\Big]\varkappa (x)\\
                  =&(x+1)^{\alpha} \varphi_1(x)+u\Big[(x+1)^{\mathfrak{T}_1} z_1(x)\varphi_1(x)+(x+1)^{\alpha} \varphi_2(x)\Big]\\
                  &+u^2\Big[(x+1)^{\mathfrak{T}_2} z_2(x)\varphi_1(x)+(x+1)^{\mathfrak{T}_1} z_1(x)\varphi_2(x)+(x+1)^{\alpha} \varphi_3(x)+ (x+1)^{\omega}\varkappa (x)\Big]
             \end{align*}
             for some $\varphi_1(x),\varphi_2(x), \varphi_3(x), \varkappa (x) \in \frac{\mathbb{F}_{p^m}[x]}{\langle x^{2^\varsigma}-1 \rangle}$. Then $\varphi_1(x)=(x+1)^{2^{\varsigma-1}-\alpha}$, $\varphi_2(x)=(x+1)^{2^{\varsigma-1}-2\alpha+\mathfrak{T}_1}z_1(x)$ and $\chi (x)=(x+1)^{2^{\varsigma-1}-\alpha+\mathfrak{T}_2-\omega}z_2(x)+(x+1)^{2^{\varsigma-1}-2\alpha+2\mathfrak{T}_1-\omega}z_1(x)z_1(x)+(x+1)^{\alpha-\omega}\varphi_3(x)$. Since  $\omega > 2^{\varsigma-1}-\alpha+2\mathfrak{T}_2$ or $2\alpha \leq 2^{\varsigma-1}+\mathfrak{T}_1$ or $\omega > 2^{\varsigma-1}-2\alpha+2\mathfrak{T}_1$, we get a contradiction. Thus, there exists no codeword of Lee weight 2. Also, following Theorem \ref{thm8}, there exist no codewords of Lee weight 3. Hence $d_L(\mathcal{C}^9_6)=4$. 
        \end{enumerate}
        \item \textbf{Case 2:} Let  $ 2^{\varsigma-1}+1\leq \alpha \leq 2^\varsigma-1$. 
        \begin{enumerate}
            \item \textbf{Subcase i:} Let  $1 < \omega \leq 2^{\varsigma-1}$. From Theorem \ref{thm3}, $d_H(\langle (x+1)^{\omega} \rangle)=2$. Thus, $2\leq d_L(\mathcal{C}^9_6)\leq 4$. Following Theorem \ref{thm8}, we get that there exists no codeword of Lee weight 2 or 3. Hence $d_L(\mathcal{C}^9_6)=4$.
            \item \textbf{Subcase ii:} Let $2^{\varsigma-1}+1\leq \omega <2^{\varsigma-1}$ and $\alpha \geq 2^{\varsigma-1}+\mathfrak{T}_1$. By Theorem \ref{thm3},  $d_H(\langle (x+1)^{\omega} \rangle)\geq 4$ and by Theorem \ref{thm23},  $d_L(\langle (x+1)^{\alpha}+u(x+1)^{\mathfrak{T}_1} z_1(x)+u^2(x+1)^{\mathfrak{T}_2}z_2(x) \rangle)=4$. Thus, $d_L(\mathcal{C}^9_6)=4$.
            \item \textbf{Subcase iii:} Let $2^\varsigma-2^{\varsigma-\gamma}+1 \leq \omega <\alpha  \leq 2^\varsigma-2^{\varsigma-\gamma}+2^{\varsigma-\gamma-1}$
              and $ 3\alpha \leq 2^{\varsigma}-2^{\varsigma-\gamma}+2^{\varsigma-\gamma-1}+2\mathfrak{T}_1$,$\alpha \leq 2^{\varsigma-1}-2^{\varsigma-\gamma-1}+2^{\varsigma-\gamma-2}+\frac{\mathfrak{T}_2}{2}$ and $\alpha \leq 2^{\varsigma-1}+\frac{\mathfrak{T}_1}{2}$,
            where $1\leq \gamma \leq \varsigma-1$. By Theorem \ref{thm3},  $d_H(\langle (x+1)^{\omega} \rangle)=2^{\gamma+1}$ and by Theorem \ref{thm23},  $d_L(\langle (x+1)^{\alpha}+u(x+1)^{\mathfrak{T}_1} z_1(x)+u^2(x+1)^{\mathfrak{T}_2}z_2(x) \rangle)=2^{\gamma+1}$. Thus, $d_L(\mathcal{C}^9_6)=2^{\gamma+1}$.
        \end{enumerate}
    \end{enumerate}
\end{proof}

%\newpage
\subsection{Type 7:}
\begin{theorem}\cite{dinh2021hamming}\label{thm34} 
Let  $\mathcal{C}_7=\langle (x+1)^{\alpha}+u(x+1)^{\mathfrak{T}_1} z_1(x)+u^2(x+1)^{\mathfrak{T}_2}z_2(x), u(x+1)^{\beta}+u^2 (x+1)^{\mathfrak{T}_3}  z_3(x) \rangle $,  where $0\leq \mathcal{W} \leq \beta <  \mathcal{U} \leq \alpha \leq 2^\varsigma-1$, $0\leq  \mathfrak{T}_1 < \beta $, $0\leq  \mathfrak{T}_2 < \mathcal{W} $, $0\leq  \mathfrak{T}_3 < \mathcal{W} $ and  $z_1(x)$, $z_2(x)$ and $z_3(x)$ are either 0 or a unit in $\mathcal{S}$. Then
    % \begin{center}
       $d_H(\mathcal{C}_7)= d_H(\langle (x+1)^{\mathcal{W}} \rangle)$.
       %\end{center}
\end{theorem}
\begin{proposition}\label{prop7}
    Let $\mathcal{C}_7$ be a cyclic code of length $2^\varsigma$ over $\mathcal{R}$ and $\mathcal{W}$ be the smallest integer such that $u^2(x+1)^{\mathcal{W}} \in \mathcal{C}_7$. Then  $d_H(\mathcal{C}_7)\leq d_L(\mathcal{C}_7)\leq 2 d_H(\langle (x+1)^{\mathcal{W}}\rangle)$, where $\langle (x+1)^{\mathcal{W}}\rangle$ is an ideal of $\frac{\mathbb{F}_{2^m}[x]}{\langle x^{2^\varsigma}-1 \rangle}$.
\end{proposition}
\begin{proof}
    $d_H(\mathcal{C}_7)\leq d_L(\mathcal{C}_7)$ is obvious. We have $\langle u^2(x+1)^{\mathcal{W}} \rangle \subseteq \mathcal{C}_7$. Then $d_L(\mathcal{C}_7)\leq d_L(\langle u^2(x+1)^{\mathcal{W}} \rangle)$. The result follows from Theorem \ref{thm5}.
\end{proof}
%\newpage
\subsection{If $z_1(x)=0$, $z_2(x)=0$ and $z_3(x)=0$ }
\begin{theorem}\label{thm35} 
   Let $\mathcal{C}^1_7=\langle (x+1)^{\alpha}, u(x+1)^{\beta} \rangle$,
  where $0\leq \mathcal{W} \leq \beta <  \mathcal{U} \leq \alpha \leq 2^\varsigma-1$. Then
  \begin{equation*}
    d_L(\mathcal{C}^1_7)=
    \begin{cases}
        2&\text{if}\quad 1\leq \alpha \leq 2^{\varsigma-1},\\
        2 &\text{if}\quad 2^{\varsigma-1}+1\leq \alpha \leq 
        2^{\varsigma}-1 \quad \text{with} \quad\beta=0,\\
	4 &\text{if}\quad 2^{\varsigma-1}+1\leq \alpha \leq 2^{\varsigma}-1 \quad \text{with} \quad 1\leq \beta \leq 2^{\varsigma-1},\\
        2^{\gamma+1} &\text{if}\quad 2^\varsigma-2^{\varsigma-\gamma}+1 \leq \beta < \alpha  \leq 2^\varsigma-2^{\varsigma-\gamma}+2^{\varsigma-\gamma-1},\quad \text{where}\quad 1\leq \gamma \leq \varsigma-1.
	\end{cases}
    \end{equation*}
\end{theorem}
\begin{proof}
     Let $\mathcal{B}=\{ \zeta_1,\zeta_2,\ldots, \zeta_m\}$ be a TOB of $\mathbb{F}_{2^m}$ over $\mathbb{F}_2.$ By Theorem \ref{thm1}, $\mathcal{W}=\beta$. From Theorem \ref{thm34}, $d_H(\mathcal{C}^1_7)= d_H(\langle (x+1)^{\beta} \rangle)$. Thus, $d_H(\langle (x+1)^{\beta} \rangle)\leq d_L(\mathcal{C}^1_7)$. By Proposition \ref{prop7}, we get $d_H(\langle (x+1)^{\beta} \rangle)\leq  d_L(\mathcal{C}^1_7) \leq 2d_H(\langle (x+1)^{\beta} \rangle) $. Since $\langle (x+1)^{\alpha} \rangle\subseteq \mathcal{C}^1_7$, $ d_L(\mathcal{C}^1_7)\leq d_L(\langle (x+1)^{\alpha} \rangle)$.
     \begin{enumerate}
         \item \textbf{Case 1:} Let $1\leq \alpha \leq 2^{\varsigma-1}$. From Theorem \ref{thm15}, $d_L(\mathcal{C}^1_7)\leq 2$.
         \begin{enumerate}
             \item If $\beta>0$, by Theorem \ref{thm3}, $d_H(\langle (x+1)^{\beta} \rangle)\geq 2$. Hence $d_L(\mathcal{C}^1_7)=2$.
             \item Let $\beta=0$. Suppose $\chi (x)=\lambda x^j \in \mathcal{C}^1_7$, $\lambda \in \mathcal{R}$ with $wt^{\mathcal{B}}_L(\chi (x))=1$
             \begin{enumerate}
                 \item if $\lambda$ is a unit in $\mathcal{R}$ then $\lambda x^j$ is a unit. This is not possible.
                 \item if $\lambda$ is non-unit in $\mathcal{R}$ then $\lambda \in \langle u \rangle$ and $wt^{\mathcal{B}}_L(\lambda)\geq 3 $. Again, this is not possible.
             \end{enumerate}
             Hence $d_L(\mathcal{C}^1_7)=2$.
        \end{enumerate}
        \item \textbf{Case 2:} Let $2^{\varsigma-1}+1\leq \alpha \leq 2^\varsigma-1$.
        \begin{enumerate}
            \item \textbf{Subcase i:} Let $\beta=0$. Then $1\leq d_L(\mathcal{C}^1_7)\leq 2$.  As in the above case, $\mathcal{C}^1_7$ has no codeword of Lee weights 1. Hence $d_L(\mathcal{C}^1_7)=2$.
            
            % Then $\chi (x)=\zeta_1u \in \mathcal{C}^1_7$ with $wt^{\mathcal{B}}_L(\chi (x))=3$. Hence $d_L(\mathcal{C}^1_7)=2$.
            
            \item \textbf{Subcase ii:} Let $1\leq \beta \leq 2^{\varsigma-1}$ then $2\leq d_L(\mathcal{C}^1_7)\leq 4$. Following as in Theorem \ref{thm9}, we get $(x+1)^{2^{\varsigma-1}} \in \mathcal{C}^1_7$.
            Then 
            \begin{align*}
                 (1+ x)^{2^{\varsigma-1}}=&\Big[(x+1)^{\alpha} \Big]\Big[\varphi_1(x)+u\varphi_2(x)+u^2\varphi_3(x)\Big]+\Big[u (x+1)^{\beta}\Big]\Big[\varkappa_1 (x)+u\varkappa_2(x)\Big]\\
                  =&(x+1)^{\alpha} \varphi_1(x)+u\Big[(x+1)^{\alpha} \varphi_2(x)+(x+1)^{\beta}\varkappa_1 (x)\Big]\\
                  &+u^2\Big[ (x+1)^{\alpha} \varphi_3(x)+(x+1)^{\beta}\varkappa_2 (x) \Big]
             \end{align*}
            for some $\varphi_1(x),\varphi_2(x), \varphi_3(x), \varkappa_1 (x), \varkappa_2(x) \in \frac{\mathbb{F}_{p^m}[x]}{\langle x^{2^\varsigma}-1 \rangle}$. Then $\varphi_1(x)=(x+1)^{2^{\varsigma-1}-\alpha}$, $\varphi_2(x)=(x+1)^{\beta-\alpha}\varkappa (x)$ and $\varkappa_2 (x)=(x+1)^{\alpha-\beta} \varphi_3(x)$.
            Since $ \alpha > 2^{\varsigma-1}$, we get a contradiction. Also, following Theorem \ref{thm8}, $\mathcal{C}^1_7$ has no codeword of Lee weights 3. 
            % A codeword $\wp(x)=u^2\zeta_1(x^{2^{\varsigma-1}}+1)=u^2\zeta_1(x+1)^{2^{\varsigma-1}}\in \langle u^2(x+1)^{\beta}\rangle \subseteq\mathcal{C}^1_7$ with $wt^{\mathcal{B}}_L(\wp(x))=4$.
            Thus, $d_L(\mathcal{C}^1_7)=4$.
            
            \item \textbf{Subcase iii:} Let $ 2^\varsigma-2^{\varsigma-\gamma}+1 \leq \beta \leq 2^\varsigma-2^{\varsigma-\gamma}+2^{\varsigma-\gamma-1}$, where $1\leq \gamma \leq \varsigma-1$. By Theorem \ref{thm3} and Theorem \ref{thm15}, $d_H(\langle (x+1)^{\beta} \rangle)=d_L(\langle (x+1)^{\alpha} \rangle)=2^{\gamma+1}$. As $ d_H(\langle (x+1)^{\beta} \rangle)\leq d_L(\mathcal{C}^1_7) \leq d_L(\langle (x+1)^{\alpha} \rangle)$, $d_L(\mathcal{C}^1_7)=2^{\gamma+1}$.
        \end{enumerate}
     \end{enumerate}
\end{proof}

%\newpage
\subsection{If $z_1(x)\neq 0$, $\mathfrak{T}_1=0$ $z_2(x)=0$ and $z_3(x)=0$ }
\begin{theorem}\label{thm36} 
   Let $\mathcal{C}^2_7=\langle (x+1)^{\alpha}+u z_1(x), u(x+1)^{\beta} \rangle$,
  where $0\leq \mathcal{W} \leq \beta <  \mathcal{U} \leq \alpha \leq 2^\varsigma-1$, $0 < \beta $ and  $z_1(x)$ a unit in $\mathcal{S}$. Then
   \begin{center}
	$d_L(\mathcal{C}^2_7)=$
	$\begin{cases}
            2&\text{if}\quad 1<\alpha \leq 2^{\varsigma-1} \quad \text{with} \quad \alpha + \beta\leq2^{\varsigma-1} ,\\
            3&\text{if}\quad 1<\alpha \leq 2^{\varsigma-1} \quad \text{with} \quad z_1(x)=1 \quad\text{and} \quad\alpha=2^{\varsigma-1},\\
            4&\text{if}\quad 1<\alpha \leq 2^{\varsigma-1} \quad \text{with} \quad \alpha + \beta>2^{\varsigma-1}\quad\text{and either} \quad z_1(x)\neq 1 \quad\text{or} \quad  \alpha\neq2^{\varsigma-1},\\
            % 3  &\text{if}\quad 2^{\varsigma-1}+1\leq \alpha \leq 
            % 2^{\varsigma}-1 \quad \text{with} \quad \beta=0 ,\\
            % &\quad \text{or} \quad     z_1(x)\neq 1 \quad\text{and} \quad\alpha=2^{\varsigma-1}\quad \text{with} \quad \beta=0,\\
		4  &\text{if}\quad 2^{\varsigma-1}+1\leq \alpha \leq 
            2^{\varsigma}-1 \quad \text{with} \quad 1\leq\beta\leq 2^{\varsigma-1} ,\\
            % &\quad \text{or} \quad     z_1(x)\neq 1 \quad\text{and} \quad\alpha=2^{\varsigma-1}\quad \text{with} \quad 1\leq\beta\leq 2^{\varsigma-1},\\
            4 &\text{if}\quad 2^\varsigma-2^{\varsigma-\gamma}+1 \leq \beta <\alpha  \leq 2^\varsigma-2^{\varsigma-\gamma}+2^{\varsigma-\gamma-1}, \quad\text{where}\quad 1\leq \gamma \leq \varsigma-1.\\
            % &\quad \text{or} \quad     z_1(x)\neq 1 \quad\text{and} \quad\alpha=2^{\varsigma-1} ,\\
            % &\quad \text{with} \quad 2^\varsigma-2^{\varsigma-\gamma}+1 \leq \beta <\alpha  \leq 2^\varsigma-2^{\varsigma-\gamma}+2^{\varsigma-\gamma-1} 2^{\varsigma-1}
	\end{cases}$
    \end{center}
\end{theorem}
\begin{proof}
    Let $\mathcal{B}=\{ \zeta_1,\zeta_2,\ldots, \zeta_m\}$ be a TOB of $\mathbb{F}_{2^m}$ over $\mathbb{F}_2.$ By Theorem \ref{thm1}, $\mathcal{W}=\beta$. 
    \begin{enumerate}
    \item {\textbf{Case 1:}} Let $1< \alpha\leq2^{\varsigma-1}$. By Thoerem \ref{thm34} and Theorem \ref{thm3}, $d_H(\mathcal{C}^2_7)=2$. Hence $2\leq d_L(\mathcal{C}^2_7)\leq 4$.
    % \hl{If $1\leq \alpha\leq2^{\varsigma-1}$ with $\alpha + \beta\leq 2^{\varsigma-1}$ and $z_1(x)\neq 1$}. 
     \begin{enumerate}
         \item \textbf{Subcase i:} Let $\alpha + \beta\leq2^{\varsigma-1}$.  We have $\chi (x)=\zeta_1(x^{2^{\varsigma-1}}+1)=\zeta_1(x+1)^{2^{\varsigma-1}}=\zeta_1[(x+1)^{\alpha}+uz_1(x)]+[(x+1)^{2^{\varsigma-1}-\alpha}]+u(x+1)^{\beta}[(x+1)^{2^{\varsigma-1}-\alpha-\beta}z_1(x)] \in \mathcal{C}^2_7$. Since  $wt^{\mathcal{B}}_L(\chi (x))=2$, $d_L(\mathcal{C}^2_7)=2$.

        \item \textbf{Subcase ii:}  Let $z_1(x)=1$ and $\alpha=2^{\varsigma-1}$. Following as in Theorem \ref{thm9}, we get $(x+1)^{2^{\varsigma-1}} \in \mathcal{C}^2_7$.
            Then 
            \begin{align*}
                 (1+ x)^{2^{\varsigma-1}}=&\Big[(x+1)^{\alpha}+u z_1(x) \Big]\Big[\varphi_1(x)+u\varphi_2(x)+u^2\varphi_3(x)\Big]\\
                 &+\Big[u (x+1)^{\beta}\Big]\Big[\varkappa_1 (x)+u\varkappa_2(x)\Big]\\
                  =&(x+1)^{\alpha} \varphi_1(x)+u\Big[z_1(x)\varphi_1(x)+(x+1)^{\alpha} \varphi_2(x)+(x+1)^{\beta}\varkappa_1 (x)\Big]\\
                  &+u^2\Big[ z_1(x)\varphi_2(x)+(x+1)^{\alpha} \varphi_3(x)+(x+1)^{\beta}\varkappa_2 (x) \Big]
             \end{align*}
            for some $\varphi_1(x),\varphi_2(x), \varphi_3(x), \varkappa_1 (x), \varkappa_2(x) \in \frac{\mathbb{F}_{p^m}[x]}{\langle x^{2^\varsigma}-1 \rangle}$. Then $\varphi_1(x)=(x+1)^{2^{\varsigma-1}-\alpha}$, $\varphi_2(x)=(x+1)^{2^{\varsigma-1}-2\alpha}z_1(x)+(x+1)^{\beta-\alpha}\varkappa_1 (x)$ and $\varkappa_2 (x)=(x+1)^{2^{\varsigma-1}-2\alpha-\beta}z_1(x)z_1(x)+(x+1)^{-\alpha}\varkappa_1 (x)z_1(x)+(x+1)^{\alpha-\beta} \varphi_3(x)$.
            As $\alpha=2^{\varsigma-1}$ and $\beta>0$, we have $ \alpha+\beta >2^{\varsigma-1}$. Then we obtain a contradiction. Thus, there exists no codeword of Lee weight 2. We have  $\chi (x)=\zeta_1((x+1)^{2^{\varsigma-1}}+u)=\zeta_1(x^{2^{\varsigma-1}}+1+u)\in \mathcal{C}^2_7$. Since $wt^{\mathcal{B}}_L(\chi (x))=3$, we have $d_L(\mathcal{C}^2_7)=3$.
        
        \item \textbf{Subcase iii:}  Let $\alpha + \beta>2^{\varsigma-1}$ and either $z_1(x)\neq 1$ or $\alpha\neq2^{\varsigma-1}$. Following as in the above case, there exists no codeword of Lee weight 2 as $\alpha + \beta>2^{\varsigma-1}$. Also, following Theorem \ref{thm9}, there exists no codeword of Lee weight 3.Hence $d_L(\mathcal{C}^2_7)=4$.
     \end{enumerate}

    \item \textbf{Case 2:} Let $2^{\varsigma-1}+1\leq \alpha \leq 2^\varsigma-1$.
    \begin{enumerate}
        % \item \textbf{Subcase i:} Let $\beta=0$. Then $\chi (x)=\zeta_1u \in \mathcal{C}^2_7$ with $wt^{\mathcal{B}}_L(\chi (x))=3$. Hence $d_L(\mathcal{C}^2_7)=3$.
        \item \textbf{Subcase i:} Let $1\leq \beta \leq 2^{\varsigma-1}$. Since $1\leq \beta \leq 2^{\varsigma-1}$. By Theorem \ref{thm3} and Theorem \ref{thm34}, $d_H(\mathcal{C}^2_7)=2$. Thus, $2\leq d_L(\mathcal{C}^2_7)\leq 4$. Following as in \ref{thm9}, we can prove $\mathcal{C}^2_7$ has no codeword of Lee weights 2 and 3. Thus, $d_L(\mathcal{C}^2_7)=4$.
            
        \item \textbf{Subcase ii:} Let $ 2^\varsigma-2^{\varsigma-\gamma}+1 \leq \beta \leq 2^\varsigma-2^{\varsigma-\gamma}+2^{\varsigma-\gamma-1}$, where $1\leq \gamma \leq \varsigma-1$. By Theorem \ref{thm3}, $d_L(\mathcal{C}^2_7)\geq 4$.  From Theorem \ref{thm18}, $d_L(\langle (x+1)^{\alpha}+u z_1(x) \rangle )=4$. Then $d_L(\mathcal{C}^2_7)\leq 4$. Hence $d_L(\mathcal{C}^2_7)=4$.
        \end{enumerate}
    \end{enumerate}
    
\end{proof}

%\newpage
\subsection{If $z_1(x)\neq 0$, $\mathfrak{T}_1\neq0$, $z_2(x)=0$ and $z_3(x)=0$ }
\begin{theorem}\label{thm37}
   Let $\mathcal{C}^3_7=\langle (x+1)^{\alpha}+u(x+1)^{\mathfrak{T}_1} z_1(x), u(x+1)^{\beta}\rangle$, where $0\leq \mathcal{W} \leq \beta <  \mathcal{U} \leq \alpha \leq 2^\varsigma-1$, $0<  \mathfrak{T}_1 < \beta $ and  $z_1(x)$ is a unit in $\mathcal{S}$. Then 
  \begin{center}
	$d_L(\mathcal{C}^3_7)=$
	$\begin{cases}
            2&\text{if}\quad 1< \alpha \leq 2^{\varsigma-1} \quad\text{with}\quad \beta\leq 2^{\varsigma-1}-\alpha+\mathfrak{T}_1,\\
            4&\text{if}\quad 1< \alpha \leq 2^{\varsigma-1} \quad\text{with}\quad\beta > 2^{\varsigma-1}-\alpha+\mathfrak{T}_1,\\
		4  &\text{if}\quad 2^{\varsigma-1}+1\leq \alpha \leq 
            2^{\varsigma}-1 \quad \text{with} \quad 1 < \beta \leq 2^{\varsigma-1},\\
            4  &\text{if}\quad 2^{\varsigma-1}+1\leq \beta<\alpha \leq 
            2^{\varsigma}-1 \quad \text{with} \quad \alpha \geq 2^{\varsigma-1}+\mathfrak{T}_1,\\
            2^{\gamma+1} &\text{if}\quad 2^\varsigma-2^{\varsigma-\gamma}+1 \leq \beta < \alpha  \leq 2^\varsigma-2^{\varsigma-\gamma}+2^{\varsigma-\gamma-1},\\
            &\qquad   \text{with} \quad \alpha \leq 2^{\varsigma-1}-2^{\varsigma-\gamma-1}+2^{\varsigma-\gamma-2}+\frac{\mathfrak{T}_1}{2}\\
            &\qquad \text{and} \quad 3\alpha\leq 2^\varsigma-2^{\varsigma-\gamma}+2^{\varsigma-\gamma-1}+2\mathfrak{T}_1,
            \qquad \text{where}\quad 1\leq \gamma \leq \varsigma-1.
	\end{cases}$
    \end{center}
\end{theorem}
\begin{proof}
    Let $\mathcal{B}=\{ \zeta_1,\zeta_2,\ldots, \zeta_m\}$ be a TOB of $\mathbb{F}_{2^m}$ over $\mathbb{F}_2.$ By Theorem \ref{thm1}, $\mathcal{W}=\beta$.
    \begin{enumerate}
        \item \textbf{Case 1:} Let $1<\alpha \leq 2^{\varsigma-1}$. Since  $1<\beta <\alpha$, clearly $1<\beta\leq 2^{\varsigma-1}$, by Theorem \ref{thm3}, $d_H(\langle (x+1)^{\beta} \rangle)=2$. Thus, $2\leq d_L(\mathcal{C}^3_7)\leq 4$.
        \begin{enumerate}
            \item \textbf{Subcase i:} Let $\beta \leq 2^{\varsigma-1}-\alpha+\mathfrak{T}_1$. We have $\chi (x)=\zeta_1(x^{2^{\varsigma-1}}+1)=\zeta_1(x+1)^{2^{\varsigma-1}}=\zeta_1[(x+1)^{\alpha}+u(x+1)^{\mathfrak{T}_1} z_1(x)][(x+1)^{2^{\varsigma-1}-\alpha}] +[u (x+1)^{\beta}][(x+1)^{2^{\varsigma-1}-\alpha+\mathfrak{T}_1-\beta}z_1(x)]\in \mathcal{C}^3_7$. Since  $wt^{\mathcal{B}}_L(\chi (x))=2$, $d_L(\mathcal{C}^3_7)=2$.
            
            \item \textbf{Subcase ii:} Let $\beta > 2^{\varsigma-1}-\alpha+\mathfrak{T}_1$. Following Theorem \ref{thm9}, we can prove $\mathcal{C}^3_7$ has no codeword of Lee weights 2 and 3. Hence $d_L(\mathcal{C}^3_7)=4$. 
        \end{enumerate}
        
        \item \textbf{Case 2:} Let  $ 2^{\varsigma-1}+1\leq \alpha \leq 2^\varsigma-1$. 
        \begin{enumerate}
            \item \textbf{Subcase i:} Let $1 < \beta \leq 2^{\varsigma-1}$. From Theorem \ref{thm3}, $d_H(\langle (x+1)^{\beta} \rangle)=2$. Thus, $2\leq d_L(\mathcal{C}^3_7)\leq 4$. Following Theorem \ref{thm9}, we can prove $\mathcal{C}^3_7$ has no codeword of Lee weights 2 and 3. Hence $d_L(\mathcal{C}^3_7)=4$.
            
            \item \textbf{Subcase ii:} Let $2^{\varsigma-1}+1\leq \beta \leq 2^{\varsigma-1}-1$ and $\alpha \geq 2^{\varsigma-1}+\mathfrak{T}_1$. By Theorem \ref{thm3},  $d_H(\langle (x+1)^{\beta} \rangle)\geq 4$ and by Theorem \ref{thm19},  $d_L(\langle (x+1)^{\alpha}+u(x+1)^{\mathfrak{T}_1} z_1(x) \rangle)=4$. Thus, $d_L(\mathcal{C}^3_7)=4$.
            
            \item \textbf{Subcase iii:} Let $2^\varsigma-2^{\varsigma-\gamma}+1 \leq \beta <\alpha  \leq 2^\varsigma-2^{\varsigma-\gamma}+2^{\varsigma-\gamma-1}$,\quad $\alpha \leq 2^{\varsigma-1}-2^{\varsigma-\gamma-1}+2^{\varsigma-\gamma-2}+\frac{\mathfrak{T}_1}{2}$ and $3\alpha\leq 2^\varsigma-2^{\varsigma-\gamma}+2^{\varsigma-\gamma-1}+2\mathfrak{T}_1$, where $1\leq \gamma \leq \varsigma-1$. By Theorem \ref{thm3},  $d_H(\langle (x+1)^{\beta} \rangle)=2^{\gamma+1}$ and by Theorem \ref{thm19},  $d_L(\langle (x+1)^{\alpha}+u(x+1)^{\mathfrak{T}_1} z_1(x) \rangle)=2^{\gamma+1}$. Thus, $d_L(\mathcal{C}^3_7)=2^{\gamma+1}$.
        \end{enumerate}
    \end{enumerate}
\end{proof}

%\newpage
\subsection{If $z_1(x)=0$, $z_2(x)\neq 0$, $\mathfrak{T}_2=0$ and $z_3(x)=0$ }
\begin{theorem}\label{thm38}
   Let $\mathcal{C}^4_7=\langle (x+1)^{\alpha}+u^2z_2(x), u(x+1)^{\beta} \rangle$,
  where $0< \mathcal{W} \leq \beta <  \mathcal{U} \leq \alpha \leq 2^\varsigma-1$  and  $z_2(x)$ is a unit in $\mathcal{S}$. Then
   \begin{center}
	$d_L(\mathcal{C}^4_7)=$
	$\begin{cases}
            2&\text{if}\quad \beta +\alpha\leq 2^{\varsigma-1},\\
		2  &\text{if}\quad z_2(x)=1 \quad\text{and} \quad\alpha=2^{\varsigma-1},\\
            4& \text{otherwise}.
	\end{cases}$
    \end{center}
\end{theorem}
% \hl{This code  $\mathcal{C}^4_7$ is similar to $\mathcal{C}^4_7$. $\omega$ is replaced by $\beta$. 
% Proof is similar to Theorem }\ref{thm26}

\begin{proof}
 Let $\mathcal{B}=\{ \zeta_1,\zeta_2,\ldots, \zeta_m\}$ be a TOB of $\mathbb{F}_{2^m}$ over $\mathbb{F}_2.$ Let $\mathcal{W}$ be the smallest integer such that $u^2(x+1)^{\mathcal{W}} \in \mathcal{C}^4_7$. By Theorem \ref{thm1}, $\mathcal{W}=min\{\beta, 2^\varsigma-\alpha\}$. Then $1\leq \mathcal{W} \leq 2^{\varsigma-1}$. By Theorem \ref{thm3} and Theorem \ref{thm34}, $d_H(\mathcal{C}^4_7)=2$. Thus, $2\leq d_L(\mathcal{C}^4_7)\leq 4$. 
   \begin{enumerate}
       \item \textbf{Case 1:} Let $\beta+\alpha\leq 2^{\varsigma-1}$. Since $0<\beta <\alpha$, clearly $1<\alpha<2^{\varsigma-1}$. Let $\chi (x)=\zeta_1(x^{2^{\varsigma-1}}+1)=\zeta_1(x+1)^{2^{\varsigma-1}}=\zeta_1[(x+1)^{\alpha}+u^2z_2(x)][(x+1)^{2^{\varsigma-1}-\alpha} ]+[u (x+1)^{\beta}][u(x+1)^{2^{\varsigma-1}-\alpha-\beta}z_2(x)]\in \mathcal{C}^4_7$. Since  $wt^{\mathcal{B}}_L(\chi (x))=2$, $d_L(\mathcal{C}^4_7)=2$.
       
       \item \textbf{Case 2:} If $z_2(x)=1$ and $\alpha=2^{\varsigma-1}$, we have  $\chi (x)=\zeta_1((x+1)^{2^{\varsigma-1}}+u^2)=\zeta_1(x^{2^{\varsigma-1}}+1+u^2)\in \mathcal{C}^4_7$. Since $wt^{\mathcal{B}}_L(\chi (x))=2$, we have $d_L(\mathcal{C}^4_7)=2$.
       
       \item \textbf{Case 3:} Let $ \beta+ \alpha >2^{\varsigma-1}$ and either $z_2(x)\neq1$ or $\alpha \neq2^{\varsigma-1}$. Following Theorem \ref{thm9}, we can prove $\mathcal{C}^4_7$ has no codeword of Lee weights 2 and 3. Hence $d_L(\mathcal{C}^4_7)=4$.
   \end{enumerate} 
\end{proof}

%\newpage
\subsection{If $z_1(x)=0$, $z_2(x)\neq 0$, $\mathfrak{T}_2\neq 0$ and $z_3(x)=0$ }
\begin{theorem}\label{thm39} 
   Let $\mathcal{C}^5_7=\langle (x+1)^{\alpha}+u^2(x+1)^{\mathfrak{T}_2}z_2(x), u(x+1)^{\beta}\rangle$,
  where $1< \mathcal{W} \leq \beta <  \mathcal{U} \leq \alpha \leq 2^\varsigma-1$, $0< \mathfrak{T}_2 < \mathcal{W} $, and  $z_2(x)$ is a unit in $\mathcal{S}$. Then
 \begin{center}
	$d_L(\mathcal{C}^5_7)=$
	$\begin{cases}
            2&\text{if}\quad 1< \alpha \leq 2^{\varsigma-1} \quad \text{with} \quad \alpha \leq2^{\varsigma-2}+\frac{\mathfrak{T}_2}{2},\\
            4&\text{if}\quad 1< \alpha \leq 2^{\varsigma-1}\quad \text{with} \quad \alpha > 2^{\varsigma-2}+\frac{\mathfrak{T}_2}{2},\\
		4  &\text{if}\quad 2^{\varsigma-1}+1\leq \alpha \leq 
            2^{\varsigma}-1 \quad \text{with} \quad 1< \mathcal{W} \leq \beta \leq 2^{\varsigma-1},\\
            4  &\text{if}\quad 2^{\varsigma-1}+1\leq \beta <\alpha \leq 
            2^{\varsigma}-1 \quad \text{with} \quad 1< \mathcal{W} \leq 2^{\varsigma-1},\\
            4  &\text{if}\quad 2^{\varsigma-1}+1\leq \mathcal{W} \leq  \beta <\alpha \leq 
            2^{\varsigma}-1 \quad \text{with} \quad \alpha \geq 2^{\varsigma-1}+\mathfrak{T}_2,\\
            2^{\gamma+1} &\text{if}\quad 2^\varsigma-2^{\varsigma-\gamma}+1 \leq \mathcal{W}\leq \beta<\alpha  \leq 2^\varsigma-2^{\varsigma-\gamma}+2^{\varsigma-\gamma-1}\quad  \quad \text{with} \quad \alpha \leq 2^\varsigma-2^{\varsigma-\gamma}\\
            &\qquad \text{and} \quad  \alpha \leq 2^{\varsigma-1}+\frac{\mathfrak{T}_2}{2},\quad \text{where}\quad 1\leq \gamma \leq \varsigma-1.
	\end{cases}$
    
    \end{center}
\end{theorem}
%\hl{This code  $\mathcal{C}^5_7$ is similar to $\mathcal{C}^3_6$. $u^2(x+1)^{\beta}$ is replaced by $u(x+1)^{\beta}$. 
%Proof is similar to Theorem }\ref{thm27} 

\begin{proof}
    Let $\mathcal{B}=\{ \zeta_1,\zeta_2,\ldots, \zeta_m\}$ be a TOB of $\mathbb{F}_{2^m}$ over $\mathbb{F}_2.$
    \begin{enumerate}
        \item \textbf{Case 1:} Let $1<\alpha \leq 2^{\varsigma-1}$. Since  $1< \mathcal{W} <\alpha$, clearly $1< \mathcal{W} \leq 2^{\varsigma-1}$, by Theorem \ref{thm3}, $d_H(\langle (x+1)^{\mathcal{W}} \rangle)=2$. Thus, $2\leq d_L(\mathcal{C}^5_7)\leq 4$.
        \begin{enumerate}
            \item \textbf{Subcase i:} Let $\alpha \leq2^{\varsigma-2}+\frac{\mathfrak{T}_2}{2}$. By Theorem \ref{thm17},  $d_L(\langle (x+1)^{\alpha}+u^2(x+1)^{\mathfrak{T}_2}z_2(x) \rangle)=2$. Then $d_L(\mathcal{C}^5_7)\leq 2$.Hence $d_L(\mathcal{C}^5_7)=2$.
            
            \item \textbf{Subcase ii:} Let $\alpha > 2^{\varsigma-2}+\frac{\mathfrak{T}_2}{2}$. 
            %By Theorem \ref{thm17},  $d_L(\langle (x+1)^{\alpha}+u^2(x+1)^{\mathfrak{T}_2}z_2(x) \rangle)=4$. Then $2\leq d_L(\mathcal{C}^5_7)\leq 4$
            Following Theorem \ref{thm8}, we can prove $\mathcal{C}^5_7$ has no codeword of Lee weights 2 and 3. Hence $d_L(\mathcal{C}^5_7)=4$. 
        \end{enumerate}
        
        \item \textbf{Case 2:} Let  $ 2^{\varsigma-1}+1\leq \alpha \leq 2^\varsigma-1$. 
        \begin{enumerate}
            \item \textbf{Subcase i:} Let $1<\mathcal{W} \leq \beta \leq 2^{\varsigma-1}$
            By Theorem \ref{thm3}, $d_H(\langle (x+1)^{\mathcal{W}}\rangle)=2$, Thus, $2\leq d_L(\mathcal{C}^5_7)\leq 4$. Following Theorem \ref{thm9}, we can prove  $\mathcal{C}^5_7$ has no codeword of Lee weights 2 and 3. Hence $d_L(\mathcal{C}^5_7)=4$.
            \item \textbf{Subcase ii:} Let $2^{\varsigma-1}+1\leq \beta\leq 2^{\varsigma}-1 $. 
            \begin{itemize}
                \item  Let $1<\mathcal{W} \leq  2^{\varsigma-1}$. By Theorem \ref{thm3}, $d_H(\langle (x+1)^{\mathcal{W}}\rangle)=2$, Thus, $2\leq d_L(\mathcal{C}^5_7)\leq 4$. Following Theorem \ref{thm9}, we can prove  $\mathcal{C}^5_7$ has no codeword of Lee weights 2 and 3. Hence $d_L(\mathcal{C}^5_7)=4$.
                    
                \item  Let $2^{\varsigma-1}+1\leq \mathcal{W} \leq 2^{\varsigma}-1 $.
                By Theorem \ref{thm3},  $d_H(\langle (x+1)^{\mathcal{W}} \rangle)\geq 4$ and by Theorem \ref{thm17},  $d_L(\langle (x+1)^\alpha+u^2(x+1)^{\mathfrak{T}_2} z_3(x) \rangle)=4$ if $ \alpha \geq 2^{\varsigma-1}+\mathfrak{T}_2$. Thus, $d_L(\mathcal{C}^5_7)=4$.
                
                \item  Let $2^\varsigma-2^{\varsigma-\gamma}+1 \leq \mathcal{W}\leq \beta <\alpha \leq 2^\varsigma-2^{\varsigma-\gamma}+2^{\varsigma-\gamma-1}$ and $\alpha \leq 2^{\varsigma-1}-2^{\varsigma-\gamma-1}+2^{\varsigma-\gamma-2}+\frac{\mathfrak{T}_2}{2}$, where $ 1\leq \gamma \leq \varsigma-1$.  From Theorem \ref{thm3}, $d_H(\langle (x+1)^{\mathcal{W}}\rangle)=2^{\gamma+1}$.  Then $d_L(\mathcal{C}^5_7)\geq 2^{\gamma+1}$. From Theorem \ref{thm17}, $d_L(\langle (x+1)^{\alpha} +u^2(x+1)^{\mathfrak{T}_2} z(x) \rangle )=2^{\gamma+1}$. Then $d_L(\mathcal{C}^5_7)\leq 2^{\gamma+1}$ if $\alpha \leq 2^{\varsigma-1}-2^{\varsigma-\gamma-1}+2^{\varsigma-\gamma-2}+\frac{\mathfrak{T}_2}{2}$. Hence $d_L(\mathcal{C}^5_7)=2^{\gamma+1}$.
                \end{itemize}
            \end{enumerate}  
    \end{enumerate}
\end{proof}

%\newpage
\subsection{If $z_1(x)=0$, $z_2(x)=0$ and $z_3(x)\neq0$, $\mathfrak{T}_3=0$ }
\begin{theorem}\label{thm40} 
    Let $\mathcal{C}^6_7=\langle (x+1)^{\alpha}, u(x+1)^{\beta}+u^2 z_3(x)\rangle$, where $0< \mathcal{W} \leq \beta <  \mathcal{U} \leq \alpha \leq 2^\varsigma-1$ and  $z_3(x)$ is a unit in $\mathcal{S}$. Then
    \begin{center}
	$d_L(\mathcal{C}^6_7)=$
	$\begin{cases}
            2&\text{if}\quad 1 < \alpha \leq 2^{\varsigma-1},\\
            4& \text{if}\quad 2^{\varsigma-1}+1 \leq \alpha\leq 2^{\varsigma}-1.
	\end{cases}$
    \end{center}
\end{theorem}

\begin{proof}
 Let $\mathcal{B}=\{ \zeta_1,\zeta_2,\ldots, \zeta_m\}$ be a TOB of $\mathbb{F}_{2^m}$ over $\mathbb{F}_2.$ Let $\mathcal{W}$ be the smallest integer such that $u^2(x+1)^{\mathcal{W}} \in \mathcal{C}^6_7$. By Theorem \ref{thm1}, $\mathcal{W}=min\{\beta, 2^\varsigma-\beta\}$. Then $1\leq \mathcal{W} \leq 2^{\varsigma-1}$. By Theorem \ref{thm34} and Theorem \ref{thm6}, $d_H(\mathcal{C}^6_7)=2$. Thus, $2\leq d_L(\mathcal{C}^6_7)\leq 4$. 
\begin{enumerate}
    \item \textbf{Case 1:} Let $1 < \alpha \leq 2^{\varsigma-1}$. Let $\chi (x)=\zeta_1(x+1)^{\alpha} \in \mathcal{C}^6_7$. Since  $wt^{\mathcal{B}}_L(\chi (x))=2$, $d_L(\mathcal{C}^6_7)=2$.
       
    \item \textbf{Case 2:} Let $ 2^{\varsigma-1}+1 \leq \alpha\leq 2^{\varsigma}-1$. Following Theorem \ref{thm9}, we can prove that $\mathcal{C}^6_7$ has no codeword of Lee weights 2 and 3. Hence $d_L(\mathcal{C}^6_7)=4$.
   \end{enumerate} 
\end{proof}

%\newpage
\subsection{If $z_1(x)=0$, $z_2(x)=0$ and $z_3(x)\neq0$, $\mathfrak{T}_3\neq 0$ }
\begin{theorem}\label{thm41} 
   Let $\mathcal{C}^7_7=\langle (x+1)^{\alpha}, u(x+1)^{\beta}+u^2 (x+1)^{\mathfrak{T}_3}  z_3(x) \rangle$,
  where $1<\mathcal{W} \leq \beta <  \mathcal{U} \leq \alpha \leq 2^\varsigma-1$, $0<  \mathfrak{T}_3 < \mathcal{W} $ and  $z_3(x)$ is a unit in $\mathcal{S}$. Then
  \begin{center}
	$d_L(\mathcal{C}^7_7)=$
	$\begin{cases}
            2&\text{if}\quad 1< \alpha \leq 2^{\varsigma-1},\\
		4  &\text{if}\quad 2^{\varsigma-1}+1\leq \alpha \leq 
            2^{\varsigma}-1 \quad \text{with} \quad 1< \mathcal{W} \leq \beta \leq 2^{\varsigma-1},\\
            4  &\text{if}\quad 2^{\varsigma-1}+1\leq \beta <\alpha \leq 
            2^{\varsigma}-1 \quad \text{with} \quad 1< \mathcal{W} \leq 2^{\varsigma-1},\\
            4  &\text{if}\quad 2^{\varsigma-1}+1\leq \mathcal{W} \leq  \beta <\alpha \leq 
            2^{\varsigma}-1 \quad \text{with} \quad \beta \geq 2^{\varsigma-1}+\mathfrak{T}_3,\\
            2^{\gamma+1} &\text{if}\quad 2^\varsigma-2^{\varsigma-\gamma}+1 \leq \mathcal{W}\leq \beta<\alpha  \leq 2^\varsigma-2^{\varsigma-\gamma}+2^{\varsigma-\gamma-1},\\
            &\qquad \text{where}\quad 1\leq \gamma \leq \varsigma-1.
	\end{cases}$
    \end{center}
\end{theorem}

\begin{proof}
    Let $\mathcal{B}=\{ \zeta_1,\zeta_2,\ldots, \zeta_m\}$ be a TOB of $\mathbb{F}_{2^m}$ over $\mathbb{F}_2.$  Following as in Theorem \ref{thm8}, we get $d_H(\langle (x+1)^{\mathcal{W}}\rangle)\leq d_L(\mathcal{C}^7_7)\leq 2 d_H(\langle (x+1)^{\mathcal{W}}\rangle)$. Also, since $\langle  u(x+1)^{\beta}+u^2 (x+1)^{\mathfrak{T}_3}  z_3(x)\rangle \subseteq \mathcal{C}^7_7$, $d_L(\mathcal{C}^7_7)\leq d_L(\langle u(x+1)^{\beta}+u^2 (x+1)^{\mathfrak{T}_3}  z_3(x)\rangle)$.
    \begin{enumerate}
         \item \textbf{Case 1:} Let  $1<\alpha \leq 2^{\varsigma-1}$. Since $1< \mathcal{W} <\alpha \leq 2^{\varsigma-1}$, by Theorem \ref{thm3}, $d_H(\langle (x+1)^{\mathcal{W}}\rangle)=2$, Thus, $2\leq d_L(\mathcal{C}^7_7)\leq 4$. we have $\chi (x)=\zeta_1(x+1)^{\alpha} \in \mathcal{C}^7_7$. Since $wt^{\mathcal{B}}_L(\chi (x))=2$, we have $d_L(\mathcal{C}^7_7)=2$.
        \item \textbf{Case 2:} Let $2^{\varsigma-1}+1\leq \alpha \leq 2^\varsigma-1$.
        \begin{enumerate}
            \item \textbf{Subcase i:} Let $1< \mathcal{W} \leq \beta \leq 2^{\varsigma-1}$
            By Theorem \ref{thm3}, $d_H(\langle (x+1)^{\mathcal{W}}\rangle)=2$, Thus, $2\leq d_L(\mathcal{C}^7_7)\leq 4$. Following Theorem \ref{thm8}, we can prove that $\mathcal{C}^7_7$ has no codeword of Lee weights 2 and 3. Hence $d_L(\mathcal{C}^7_7)=4$.
            \item \textbf{Subcase ii:} Let $2^{\varsigma-1}+1\leq \beta\leq 2^{\varsigma}-1 $. 
            \begin{itemize}
                \item  Let $1< \mathcal{W} \leq  2^{\varsigma-1}$. By Theorem \ref{thm3}, $d_H(\langle (x+1)^{\mathcal{W}}\rangle)=2$, Thus, $2\leq d_L(\mathcal{C}^7_7)\leq 4$. Following Theorem \ref{thm8}, we can prove that $\mathcal{C}^7_7$ has no codeword of Lee weights 2 and 3. Hence $d_L(\mathcal{C}^7_7)=4$.
                    
                \item  Let $2^{\varsigma-1}+1\leq \mathcal{W} \leq 2^{\varsigma}-1 $
                By Theorem \ref{thm3},  $d_H(\langle (x+1)^{\mathcal{W}} \rangle)\geq 4$ and by Theorem \ref{thm8},  $d_L(\langle u(x+1)^\beta+u^2(x+1)^{\mathfrak{T}_3} z_3(x) \rangle)=4$ if $ \beta \geq 2^{\varsigma-1}+\mathfrak{T}_3$. Thus, $d_L(\mathcal{C}^7_7)=4$.
                
                \item  Let $2^\varsigma-2^{\varsigma-\gamma}+1 \leq \mathcal{W}\leq \beta<\alpha  \leq 2^\varsigma-2^{\varsigma-\gamma}+2^{\varsigma-\gamma-1}$, where $ 1\leq \gamma \leq \varsigma-1$.  From Theorem \ref{thm3}, $d_H(\langle (x+1)^{\mathcal{W}}\rangle)=2^{\gamma+1}$.  Then $d_L(\mathcal{C}^7_7)\geq 2^{\gamma+1}$. From Theorem \ref{thm15}, $d_L(\langle (x+1)^{\alpha}  \rangle )=2^{\gamma+1}$. Then $d_L(\mathcal{C}^7_7)\leq 2^{\gamma+1}$. Hence $d_L(\mathcal{C}^7_7)=2^{\gamma+1}$.
                \end{itemize}
            \end{enumerate}  
    \end{enumerate}
\end{proof}

%\newpage
\subsection{If $z_1(x)\neq0$, $\mathfrak{T}_1=0$, $z_2(x)\neq0$, $\mathfrak{T}_2=0$ and $z_3(x)=0$ }
\begin{theorem}\label{thm42} 
   Let $\mathcal{C}^8_7=\langle (x+1)^{\alpha}+uz_1(x)+u^2z_2(x), u(x+1)^{\beta} \rangle$,
  where $0< \mathcal{W} \leq \beta <  \mathcal{U} \leq \alpha \leq 2^\varsigma-1$, $0 < \beta$ and  $z_1(x)$ and $z_2(x)$ are units in $\mathcal{S}$. \begin{center}
    $d_L(\mathcal{C}^8_7)=$
    $\begin{cases}
        2&\text{if}\quad 1\leq \beta<\alpha < 2^{\varsigma-1}\quad \text{with} \quad 2\alpha \leq 2^{\varsigma-1}+\mathfrak{T}_1,\quad \alpha+\beta \leq 2^{\varsigma-1}\\
        &\qquad \text{and} \quad 2\alpha+\beta \leq 2^{\varsigma-1}+2\mathfrak{T}_1,\\
        3 &\text{if}\quad 1\leq \beta<\alpha < 2^{\varsigma-1}\quad \text{with} \quad z_1(x)=z_2(x)=1 \quad \text{and}\quad\alpha=2^{\varsigma-1}, \\
        4&\text{if}\quad 1\leq \beta <\alpha < 2^{\varsigma-1}\quad \text{either with} \quad 2\alpha > 2^{\varsigma-1}+\mathfrak{T}_1 \quad \text{or} \quad \alpha+\beta > 2^{\varsigma-1}\\& \qquad \text{or} \quad 2\alpha+\beta > 2^{\varsigma-1}+2\mathfrak{T}_1
        \quad \text{and either} \quad z_1(x)\neq 1 \quad \text{or} \quad z_2(x)\neq1\\& \qquad \text{or} \quad \alpha\neq2^{\varsigma-1},\\
        4 &\text{if}\quad 2^{\varsigma-1}+1\leq \alpha\leq             2^{\varsigma}-1\quad\text{with} \quad 1\leq \beta \leq 2^{\varsigma-1},\\
        4 &\text{if}\quad 2^{\varsigma-1}+1\leq \beta<\alpha \leq 2^{\varsigma}-1.
	\end{cases}$
    \end{center}
\end{theorem}
\begin{proof}
     Let $\mathcal{B}=\{ \zeta_1,\zeta_2,\ldots, \zeta_m\}$ be a TOB of $\mathbb{F}_{2^m}$ over $\mathbb{F}_2.$ By Theorem \ref{thm1}, $\mathcal{W}=\beta$.  From Theorem \ref{thm24}, $d_H(\mathcal{C}^8_7)= d_H(\langle (x+1)^{\beta} \rangle)$. Thus, $d_H(\langle (x+1)^{\beta} \rangle)\leq d_L(\mathcal{C}^8_7)$. 
     % \hl{Following as in Theorem }\ref{thm}, we get $d_H(\langle (x+1)^{\beta} \rangle)\leq  d_L(\mathcal{C}^8_7) \leq 2d_H(\langle (x+1)^{\beta} \rangle) $. 
     %Since $\langle (x+1)^{\alpha} \rangle\subseteq \mathcal{C}^7_6$, $ d_L(\mathcal{C}^7_6)\leq d_L(\langle (x+1)^{\alpha} \rangle)$.
     \begin{enumerate}
         \item \textbf{Case 1:} Let $1\leq \beta<\alpha < 2^{\varsigma-1}$.  By Theorem \ref{thm3}, $d_H(\langle (x+1)^{\beta} \rangle)= 2$. Hence by Proposition \ref{prop7}, $2\leq d_L(\mathcal{C}^8_7)\leq 4$.
         %From Theorem \ref{thm15}, $d_L(\mathcal{C}^7_6)\leq 2$.
         % \begin{enumerate}
              % \item \textbf{Subcase i:} Let $\beta=0$. Suppose $\chi (x)=\lambda x^j \in \mathcal{C}^8_7$, $\lambda \in \mathcal{R}$ with $wt^{\mathcal{B}}_L(\chi (x))=1$
             % \item \textbf{Subcase ii:}If $1\leq \beta\leq 2^{\varsigma-1}$,
             
             \begin{enumerate}
                 \item \textbf{Subcase i:} Let $2\alpha \leq 2^{\varsigma-1}$, $\alpha+\beta \leq 2^{\varsigma-1}$ and $2\alpha+\beta \leq 2^{\varsigma-1}$.  We have $\chi (x)=\zeta_1(x^{2^{\varsigma-1}}+1)=\zeta_1(x+1)^{2^{\varsigma-1}}=\zeta_1\Big[(x+1)^{\alpha}+uz_1(x)+u^2z_2(x)\Big]\Big[(x+1)^{2^{\varsigma-1}-\alpha}+u(x+1)^{2^{\varsigma-1}-2\alpha}z_1(x)\Big] +\Big[u(x+1)^{\beta}\Big]\Big[(x+1)^{2^{\varsigma-1}-\alpha-\beta}z_2(x)+(x+1)^{2^{\varsigma-1}-2\alpha-\beta}z_1(x)z_1(x)\Big]\in \mathcal{C}^8_7$. Since  $wt^{\mathcal{B}}_L(\chi (x))=2$, $d_L(\mathcal{C}^8_7)=2$.

                  \item \textbf{Subcase ii:} Let $z_1(x)=z_2(x)=1$ and $\alpha=2^{\varsigma-1}$.
                  Following as in Theorem \ref{thm9}, we can prove $\mathcal{C}^8_7$ has no codeword of Lee weights 2 as $\alpha+\beta > 2^{\varsigma-1}$. Also, we have $\chi (x)=\zeta_1\big[x^{2^{\varsigma-1}}+1+u +u^2\big]=\zeta_1\big[(x+1)^{2^{\varsigma-1}}+u +u^2\big]\in \mathcal{C}^8_7$. Since $wt^{\mathcal{B}}_L(\chi (x))=3$, we have $d_L(\mathcal{C}^8_7)=3$.
                  
                 \item \textbf{Subcase iii:} Consider either $2\alpha > 2^{\varsigma-1}$ or $\alpha+\beta > 2^{\varsigma-1}$ or $2\alpha+\beta > 2^{\varsigma-1}$ and either $z_1(x)\neq 1$ or $z_2(x)\neq1$ or $\alpha\neq2^{\varsigma-1}$. Following Theorem \ref{thm9}, we can prove that $\mathcal{C}^8_7$ has no codeword of Lee weights 2 and 3.  Hence $d_L(\mathcal{C}^8_7)=4$.
                 
             % \end{enumerate}
        \end{enumerate}
        
        \item \textbf{Case 2:} Let $2^{\varsigma-1}+1\leq \alpha \leq 2^\varsigma-1$.
        \begin{enumerate}
            \item \textbf{Subcase i:} Let $1\leq \beta \leq 2^{\varsigma-1}$. By Proposition \ref{prop7}, $2\leq d_L(\mathcal{C}^8_7)\leq 4$. Following as in Theorem \ref{thm9}, we get $(x+1)^{2^{\varsigma-1}} \in \mathcal{C}^8_7$. Since $ \alpha > 2^{\varsigma-1}$, we get a contradiction. Also, following Theorem \ref{thm8},  $\mathcal{C}^8_7$ has no codeword of Lee weights 3.
            %A codeword $\wp(x)=u^2\zeta_1(x^{2^{\varsigma-1}}+1)=u^2\zeta_1(x+1)^{2^{\varsigma-1}}\in \langle u^2(x+1)^{\beta}\rangle \subseteq \mathcal{C}^8_7$ with $wt^{\mathcal{B}}_L(\wp(x))=4$. 
            Thus, $d_L(\mathcal{C}^8_7)=4$.

             \item \textbf{Subcase ii:} Let $2^{\varsigma-1}+1\leq \beta<\alpha \leq 2^{\varsigma}-1$. By Theorem \ref{thm3}, $d_H(\langle (x+1)^{\beta}\rangle)\geq4$. Then $d_L(\mathcal{C}^8_7)\geq4$. And by Theorem \ref{thm20}, $d_L(\langle (x+1)^{\alpha}+u z_1(x)+u^2z_2(x) \rangle)=4$. Hence $d_L(\mathcal{C}^8_7)=4$.
        \end{enumerate}
     \end{enumerate}
     
\end{proof}

%\newpage
\subsection{If $z_1(x)\neq0$, $\mathfrak{T}_1\neq0$, $z_2(x)\neq0$, $\mathfrak{T}_2=0$ and $z_3(x)=0$  }
\begin{theorem}\label{thm43} 
   Let $\mathcal{C}^9_7=\langle (x+1)^{\alpha}+u(x+1)^{\mathfrak{T}_1} z_1(x)+u^2z_2(x), u(x+1)^{\beta} \rangle$,
  where $0< \mathcal{W} \leq \beta <  \mathcal{U} \leq \alpha \leq 2^\varsigma-1$, $0<  \mathfrak{T}_1 < \beta $ and  $z_1(x)$ and $z_2(x)$ are units in $\mathcal{S}$. Then
   \begin{center}
	$d_L(\mathcal{C}^9_7)=$
	$\begin{cases}
            2&\text{if}\quad 1< \alpha \leq 2^{\varsigma-1} \quad\text{with}\quad \alpha+\beta \leq 2^{\varsigma-1}, \alpha \leq 2^{\varsigma-2}+\frac{\mathfrak{T}_1}{2} \\& \qquad\text{and} \quad \alpha+\frac{\beta}{2} \leq 2^{\varsigma-2}+\mathfrak{T}_1 ,\\
            4&\text{if}\quad 1< \alpha \leq 2^{\varsigma-1} \quad\text{with} \quad \alpha+\beta > 2^{\varsigma-1} \quad\text{or} \quad \alpha > 2^{\varsigma-2}+\frac{\mathfrak{T}_1}{2}  \\&\qquad\text{or} \quad \alpha+\frac{\beta}{2} >2^{\varsigma-2}+\mathfrak{T}_1,\\
            4  &\text{if}\quad 2^{\varsigma-1}+1\leq \alpha \leq 
            2^{\varsigma}-1 \quad \text{with} \quad   1 < \beta \leq 2^{\varsigma-1},\\
            2^{\gamma+1} &\text{if}\quad 2^\varsigma-2^{\varsigma-\gamma}+1 \leq \beta <\alpha  \leq 2^\varsigma-2^{\varsigma-\gamma}+2^{\varsigma-\gamma-1},\\
           &\qquad   \text{with} \quad 3\alpha \leq 2^{\varsigma}-2^{\varsigma-\gamma}+2^{\varsigma-\gamma-1}+2\mathfrak{T}_1 \\
           &\qquad\text{and}\quad \alpha \leq 2^{\varsigma-1}+\frac{\mathfrak{T}_1}{2},
            \quad \text{where}\quad 1\leq \gamma \leq \varsigma-1.
	\end{cases}$
    \end{center}
\end{theorem}

\begin{proof}
    Let $\mathcal{B}=\{ \zeta_1,\zeta_2,\ldots, \zeta_m\}$ be a TOB of $\mathbb{F}_{2^m}$ over $\mathbb{F}_2.$ By Theorem \ref{thm1}, $\mathcal{W}=\beta$. 
    \begin{enumerate}
        \item \textbf{Case 1:} Let $1<\alpha \leq 2^{\varsigma-1}$. Since  $1<\beta<\alpha$, clearly $ 1<\beta\leq 2^{\varsigma-1}$, by Theorem \ref{thm3}, $d_H(\langle (x+1)^{\beta} \rangle)=2$. Thus, $2\leq d_L(\mathcal{C}^9_7)\leq 4$.
        \begin{enumerate}
            \item \textbf{Subcase i:} Let $\alpha+\beta \leq 2^{\varsigma-1}$, $\alpha \leq 2^{\varsigma-2}+\frac{\mathfrak{T}_1}{2}$  and $\alpha+\frac{\beta}{2} \leq 2^{\varsigma-2}+\mathfrak{T}_1$ . We have $\chi (x)=\zeta_1(x^{2^{\varsigma-1}}+1)=\zeta_1(x+1)^{2^{\varsigma-1}}=\zeta_1\Big[(x+1)^{\alpha}+u(x+1)^{\mathfrak{T}_1} z_1(x)+u^2z_2(x)\Big]\Big[(x+1)^{2^{\varsigma-1}-\alpha}+u(x+1)^{2^{\varsigma-1}-2\alpha+\mathfrak{T}_1}z_1(x)\Big] +\Big[u(x+1)^{\beta}\Big]\Big[u\Big((x+1)^{2^{\varsigma-1}-\alpha-\beta}z_2(x)+(x+1)^{2^{\varsigma-1}-2\alpha+2\mathfrak{T}_1-\beta}z_1(x)z_1(x)\Big)\Big]\in \mathcal{C}^9_7$. Since  $wt^{\mathcal{B}}_L(\chi (x))=2$, $d_L(\mathcal{C}^9_7)=2$.
            
            \item \textbf{Subcase ii:} Let $\alpha+\beta > 2^{\varsigma-1}$ or $\alpha > 2^{\varsigma-2}+\frac{\mathfrak{T}_1}{2}$  or $\alpha+\frac{\beta}{2} >2^{\varsigma-2}+\mathfrak{T}_1$. By following the same line of arguments as in Theorem \ref{thm9}, we get $(1+ x)^{2^{\varsigma-1}}\in \mathcal{C}^9_7$. Then
            \begin{align*}
                  (1+ x)^{2^{\varsigma-1}}=&\Big[(x+1)^{\alpha}+u(x+1)^{\mathfrak{T}_1} z_1(x)+u^2z_2(x)\Big]\Big[\varphi_1(x)+u\varphi_2(x)+u^2\varphi_3(x)\Big]\\
                  &+\Big[u (x+1)^{\beta}\Big]\Big[\varkappa_1(x)+u\varkappa_2(x)\Big]\\
                  =&(x+1)^{\alpha} \varphi_1(x)+u\Big[(x+1)^{\mathfrak{T}_1} z_1(x)\varphi_1(x)+(x+1)^{\alpha} \varphi_2(x)+\varkappa_1(x)(x+1)^{\beta}\Big]\\
                  &+u^2\Big[\varphi_1(x)z_2(x)+(x+1)^{\mathfrak{T}_1} z_1(x)\varphi_2(x)+(x+1)^{\alpha} \varphi_3(x)+ (x+1)^{\beta}\varkappa_2(x)\Big]
             \end{align*}
             for some $\varphi_1(x),\varphi_2(x), \varphi_3(x), \varkappa_1(x), \varkappa_2(x) \in \frac{\mathbb{F}_{p^m}[x]}{\langle x^{2^\varsigma}-1 \rangle}$. Then $\varphi_1(x)=(x+1)^{2^{\varsigma-1}-\alpha}$, $\varphi_2(x)=(x+1)^{2^{\varsigma-1}-2\alpha+\mathfrak{T}_1}z_1(x)+(x+1)^{\beta-\alpha}\varkappa_1(x)$ and $\varkappa_2(x)=(x+1)^{2^{\varsigma-1}-\alpha-\beta}z_2(x)+(x+1)^{2^{\varsigma-1}-2\alpha+2\mathfrak{T}_1-\beta}z_1(x)z_1(x)+(x+1)^{\alpha-\beta}\varphi_3(x)$. Since  $\alpha+\beta > 2^{\varsigma-1}$ or $\alpha > 2^{\varsigma-2}+\frac{\mathfrak{T}_1}{2}$  or $\alpha+\frac{\beta}{2} >2^{\varsigma-2}+\mathfrak{T}_1$, we get a contradiction. Thus, there exists no codeword of Lee weight 2. Also, following Theorem \ref{thm8}, there exist no codewords of Lee weight 3. 
             %A codeword $\wp(x)=u^2\zeta_1(x^{2^{\varsigma-1}}+1)=u^2\zeta_1(x+1)^{2^{\varsigma-1}}\in \langle u^2(x+1)^{\beta}\rangle \subseteq \mathcal{C}^9_7$ with $wt^{\mathcal{B}}_L(\wp(x))=4$.
             Hence $d_L(\mathcal{C}^9_7)=4$. 
        \end{enumerate}
        \item \textbf{Case 2:} Let $ 2^{\varsigma-1}+1\leq \alpha \leq 2^\varsigma-1$. 
        \begin{enumerate}
            \item \textbf{Subcase i:} Let $1 < \beta \leq 2^{\varsigma-1}$. From Theorem \ref{thm3}, $d_H(\langle (x+1)^{\beta} \rangle)=2$. Thus, $2\leq d_L(\mathcal{C}^9_7)\leq 4$. Following Theorem \ref{thm8}, we get that there exists no codeword of Lee weight 2 or 3. Hence $d_L(\mathcal{C}^9_7)=4$.
            
            % \item \textbf{Subcase ii:} Let $2^\varsigma-2^{\varsigma-\gamma}+1 \leq \beta <\alpha  \leq 2^\varsigma-2^{\varsigma-\gamma}+2^{\varsigma-\gamma-1}$
            %   and $ 2\alpha \leq 2^{\varsigma}-2^{\varsigma-\gamma}+2^{\varsigma-\gamma-1}$, $2\alpha +\beta \leq 2^{\varsigma}-2^{\varsigma-\gamma}+2^{\varsigma-\gamma-1}+2\mathfrak{T}_1$, where $1\leq \gamma \leq \varsigma-1$. By Theorem \ref{thm3},  $d_H(\langle (x+1)^{\beta} \rangle)=2^{\gamma+1}$. Then $d_L(\mathcal{C}^9_7)\geq2^{\gamma+1}$.
            %   we have
            %  \begin{align*}
            %     \prod \limits_{\alpha=1}^{\gamma+1}(x^{2^{\varsigma-\alpha}}+1)=& (x+1)^{2^{\varsigma-1}+2^{\varsigma-2}+\cdots+2^{\varsigma-\gamma-1}},\\
            %      =&(x+1)^{2^{\varsigma}-2^{\varsigma-\gamma}+2^{\varsigma-\gamma-1}},\\
            %      =&\Big[(x+1)^{\alpha}+u(x+1)^{\mathfrak{T}_1} z_1(x)+u^2z_2(x)\Big]\Big[(x+1)^{2^\varsigma-2^{\varsigma-\gamma}+2^{\varsigma-\gamma-1}-\alpha},\\
            %      &+u(x+1)^{2^\varsigma-2^{\varsigma-\gamma}+2^{\varsigma-\gamma-1}-2\alpha}z_1(x)\Big] +\Big[u(x+1)^{\beta}\Big]\\
            %      &u\Big[(x+1)^{2^\varsigma-2^{\varsigma-\gamma}+2^{\varsigma-\gamma-1}-\alpha-\beta}z_2(x)+(x+1)^{2^\varsigma-2^{\varsigma-\gamma}+2^{\varsigma-\gamma-1}-2\alpha+2\mathfrak{T}_1-\beta}z_1(x)z_1(x)\Big]\in \mathcal{C}^9_7
            %  \end{align*}
            %  Let $f(x)=\zeta_1 \prod \limits_{\alpha=1}^{\gamma+1}(x^{2^{\varsigma-\alpha}}+1)$. Then $wt^{\mathcal{B}}_L(f(x))=2^{\gamma+1}$. Thus, $d_L(\mathcal{C}^9_7)=2^{\gamma+1}$.

             \item \textbf{Subcase ii:} Let $2^\varsigma-2^{\varsigma-\gamma}+1 \leq \beta <\alpha  \leq 2^\varsigma-2^{\varsigma-\gamma}+2^{\varsigma-\gamma-1}$, where $ 1\leq \gamma \leq \varsigma-1.$  By Theorem \ref{thm3},  $d_H(\langle (x+1)^{\beta} \rangle)=2^{\gamma+1}$. Then $d_L(\mathcal{C}^9_7)\geq2^{\gamma+1}$. Also if $3\alpha \leq 2^{\varsigma}-2^{\varsigma-\gamma}+2^{\varsigma-\gamma-1}+2\mathfrak{T}_1 $ and $\alpha \leq 2^{\varsigma-1}+\frac{\mathfrak{T}_1}{2}$, where $ 1\leq \gamma \leq \varsigma-1.$ By Theorem \ref{thm21}, $d_L(\langle (x+1)^{\alpha}+u(x+1)^{\mathfrak{T}_1} z_1(x)+u^2z_2(x) \rangle)=2^{\gamma+1} $.  Thus, $d_L(\mathcal{C}^9_7)=2^{\gamma+1}$.
             
        \end{enumerate}
    \end{enumerate}
\end{proof}

%\newpage
\subsection{If $z_1(x)\neq0$, $\mathfrak{T}_1=0$, $z_2(x)\neq0$, $\mathfrak{T}_2\neq0$ and $z_3(x)=0$ }
\begin{theorem}\label{thm44} 
   Let $\mathcal{C}^{10}_7=\langle (x+1)^{\alpha}+u z_1(x)+u^2(x+1)^{\mathfrak{T}_2}z_2(x), u(x+1)^{\beta} \rangle$,
  where $1< \mathcal{W} \leq \beta <  \mathcal{U} \leq \alpha \leq 2^\varsigma-1$, $0 < \beta $, $0< \mathfrak{T}_2 < \mathcal{W} $ and  $z_1(x)$ and $z_2(x)$ are units in $\mathcal{S}$. Then
   \begin{center}
	$d_L(\mathcal{C}^{10}_7)=$
	$\begin{cases}
            2&\text{if}\quad 1<\beta<\alpha \leq 2^{\varsigma-1} \quad\text{with}\quad 2\alpha\leq 2^{s-1},\quad \alpha+\beta \leq 2^{\varsigma-1}+\mathfrak{T}_2\\
            &\qquad\text{and} \quad 2\alpha+\beta\leq 2^{s-1},\\
            4&\text{if}\quad 1<\beta< \alpha \leq 2^{\varsigma-1} \quad\text{with} \quad \quad 2\alpha> 2^{s-1},\quad\text{or} \quad \alpha+\beta > 2^{\varsigma-1}+\mathfrak{T}_2 \\
            &\qquad\text{or} \quad 2\alpha+\beta>2^{s-1},\\
            4  &\text{if}\quad 2^{\varsigma-1}+1\leq \alpha \leq 
            2^{\varsigma}-1 \quad \text{with} \quad   1 < \beta \leq 2^{\varsigma-1},\\
            4 &\text{if}\quad 2^{\varsigma-1}+1\leq \beta<\alpha \leq 2^{\varsigma}-1.
	\end{cases}$
    \end{center}
\end{theorem}
\begin{proof}
    Let $\mathcal{B}=\{ \zeta_1,\zeta_2,\ldots, \zeta_m\}$ be a TOB of $\mathbb{F}_{2^m}$ over $\mathbb{F}_2.$ By Theorem \ref{thm1}, $\mathcal{W}=\beta$. 
    \begin{enumerate}
        \item \textbf{Case 1:} Let $1<\alpha \leq 2^{\varsigma-1}$. Since  $1<\beta<\alpha$ , cleraly $1<\beta<\alpha \leq 2^{\varsigma-1}$. By Theorem \ref{thm3}, $d_H(\langle (x+1)^{\beta} \rangle)=2$. Thus, $2\leq d_L(\mathcal{C}^{10}_7)\leq 4$.
        \begin{enumerate}
            \item \textbf{Subcase i:} Let $2\alpha\leq 2^{\varsigma-1}$, $\alpha+\beta \leq 2^{\varsigma-1}+\mathfrak{T}_2$  and $2\alpha+\beta \leq 2^{\varsigma-1}$ . We have $\chi (x)=\zeta_1(x^{2^{\varsigma-1}}+1)=\zeta_1(x+1)^{2^{\varsigma-1}}=\zeta_1\Big[(x+1)^{\alpha}+u z_1(x)+u^2(x+1)^{\mathfrak{T}_2}z_2(x)\Big]\Big[(x+1)^{2^{\varsigma-1}-\alpha}+u(x+1)^{2^{\varsigma-1}-2\alpha}z_1(x)\Big] +\Big[u(x+1)^{\beta}\Big]\Big[u\Big((x+1)^{2^{\varsigma-1}-\alpha-\beta+\mathfrak{T}_2}z_2(x)+(x+1)^{2^{\varsigma-1}-2\alpha-\beta}z_1(x)z_1(x)\Big)\Big]\in \mathcal{C}^{10}_7$. Since  $wt^{\mathcal{B}}_L(\chi (x))=2$, $d_L(\mathcal{C}^{10}_7)=2$.
            \item \textbf{Subcase ii:} Let $2\alpha> 2^{\varsigma-1}$ or $\alpha+\beta >2^{\varsigma-1}+\mathfrak{T}_2$  or $2\alpha+\beta > 2^{\varsigma-1}$.  Let $\chi (x)=\lambda_1 x^{k_1}+\lambda_2 x^{k_2} \in \mathcal{C}^{10}_7$ with $wt^{\mathcal{B}}_L(\chi (x))=2$, where $\lambda_1$ and $\lambda_2 \in \mathcal{R}\textbackslash \{0\}$. By following the same line of arguments as in Theorem \ref{thm8}, we get $(1+ x)^{2^{\varsigma-1}}\in \mathcal{C}^{10}_7$. Then
            \begin{align*}
                  (1+ x)^{2^{\varsigma-1}}=&\Big[(x+1)^{\alpha}+u z_1(x)+u^2(x+1)^{\mathfrak{T}_2}z_2(x)\Big]\Big[\varphi_1(x)+u\varphi_2(x)+u^2\varphi_3(x)\Big]\\
                  &+\Big[u (x+1)^{\beta}\Big]\Big[\varkappa_1(x)+u\varkappa_2(x)\Big]\\
                  =&(x+1)^{\alpha} \varphi_1(x)+u\Big[z_1(x)\varphi_1(x)+(x+1)^{\alpha} \varphi_2(x)+\varkappa_1(x)(x+1)^{\beta}\Big]\\
                  &+u^2\Big[(x+1)^{\mathfrak{T}_2}\varphi_1(x)z_2(x)+ z_1(x)\varphi_2(x)+(x+1)^{\alpha} \varphi_3(x)+ (x+1)^{\beta}\varkappa_2(x)\Big]
             \end{align*}
             for some $\varphi_1(x),\varphi_2(x), \varphi_3(x), \varkappa_1(x), \varkappa_2(x) \in \frac{\mathbb{F}_{p^m}[x]}{\langle x^{2^\varsigma}-1 \rangle}$. Then $\varphi_1(x)=(x+1)^{2^{\varsigma-1}-\alpha}$, $\varphi_2(x)=(x+1)^{2^{\varsigma-1}-2\alpha}z_1(x)+(x+1)^{\beta-\alpha}\varkappa_1(x)$ and $\varkappa_2(x)=(x+1)^{2^{\varsigma-1}-\alpha-\beta+\mathfrak{T}_2}z_2(x)+(x+1)^{2^{\varsigma-1}-2\alpha-\beta}z_1(x)z_1(x)+(x+1)^{-\alpha}\varkappa_1(x)z_1(x)+(x+1)^{\alpha-\beta}\varphi_3(x)$. Since $2\alpha> 2^{s-1}$ or $\alpha+\beta > 2^{\varsigma-1}+\mathfrak{T}_2$ or $\alpha+\beta>2^{s-1}$ , we get a contradiction. Thus, there exists no codeword of Lee weight 2. Also, following Theorem \ref{thm8}, there exist no codewords of Lee weight 3. 
             % A codeword $\wp(x)=u^2\zeta_1(x^{2^{\varsigma-1}}+1)=u^2\zeta_1(x+1)^{2^{\varsigma-1}}\in \langle u^2(x+1)^{\beta}\rangle \subseteq \mathcal{C}^{10}_7$ with $wt^{\mathcal{B}}_L(\wp(x))=4$.
             Hence $d_L(\mathcal{C}^{10}_7)=4$. 
        \end{enumerate}
        \item \textbf{Case 2:} Let $ 2^{\varsigma-1}+1\leq \alpha \leq 2^\varsigma-1$. 
        \begin{enumerate}
            \item \textbf{Subcase i:} Let $1< \beta \leq 2^{\varsigma-1}$. From Theorem \ref{thm3}, $d_H(\langle (x+1)^{\beta} \rangle)=2$. Thus, $2\leq d_L(\mathcal{C}^{10}_7)\leq 4$. Following Theorem \ref{thm8}, we get that there exists no codeword of Lee weight 2 or 3. Hence $d_L(\mathcal{C}^{10}_7)=4$.
            
            % \item \textbf{Subcase ii:} Let $2^\varsigma-2^{\varsigma-\gamma}+1 \leq \beta <\alpha  \leq 2^\varsigma-2^{\varsigma-\gamma}+2^{\varsigma-\gamma-1}$
            %   and $ 2\alpha \leq 2^{\varsigma}-2^{\varsigma-\gamma}+2^{\varsigma-\gamma-1}$, $2\alpha +\beta \leq 2^{\varsigma}-2^{\varsigma-\gamma}+2^{\varsigma-\gamma-1}$ and $\alpha +\beta \leq 2^{\varsigma}-2^{\varsigma-\gamma}+2^{\varsigma-\gamma-1}+\mathfrak{T}_2$ where $1\leq \gamma \leq \varsigma-1$. By Theorem \ref{thm3},  $d_H(\langle (x+1)^{\beta} \rangle)=2^{\gamma+1}$. Then $d_L(\mathcal{C}^{10}_7)\geq2^{\gamma+1}$.
            %   we have
            %  \begin{align*}
            %     \prod \limits_{\alpha=1}^{\gamma+1}(x^{2^{\varsigma-\alpha}}+1)=& (x+1)^{2^{\varsigma-1}+2^{\varsigma-2}+\cdots+2^{\varsigma-\gamma-1}},\\
            %      =&(x+1)^{2^{\varsigma}-2^{\varsigma-\gamma}+2^{\varsigma-\gamma-1}},\\
            %      =&\Big[(x+1)^{\alpha}+u(x+1)^{\mathfrak{T}_1} z_1(x)+u^2z_2(x)\Big]\Big[(x+1)^{2^\varsigma-2^{\varsigma-\gamma}+2^{\varsigma-\gamma-1}-\alpha},\\
            %      &+u(x+1)^{2^\varsigma-2^{\varsigma-\gamma}+2^{\varsigma-\gamma-1}-2\alpha}z_1(x)\Big] +\Big[u(x+1)^{\beta}\Big]\\
            %      &\Big[u\Big((x+1)^{2^\varsigma-2^{\varsigma-\gamma}+2^{\varsigma-\gamma-1}-\alpha-\beta+\mathfrak{T}_2}z_2(x)+(x+1)^{2^\varsigma-2^{\varsigma-\gamma}+2^{\varsigma-\gamma-1}-2\alpha-\beta}z_1(x)z_1(x)\Big)\Big]\in \mathcal{C}^{10}_7
            %  \end{align*}
            %  Let $f(x)=\zeta_1 \prod \limits_{\alpha=1}^{\gamma+1}(x^{2^{\varsigma-\alpha}}+1)$. Then $wt^{\mathcal{B}}_L(f(x))=2^{\gamma+1}$. Thus, $d_L(\mathcal{C}^{10}_7)=2^{\gamma+1}$.
            \item \textbf{Subcase ii:} Let $2^{\varsigma-1}+1 \leq \beta <\alpha  \leq 2^\varsigma-1$. By Theorem \ref{thm3},  $d_H(\langle (x+1)^{\beta} \rangle)\geq 4$. Then $d_L(\mathcal{C}^{10}_7)\geq 4$.
            Also by Theorem \ref{thm22}, $d_L(\langle (x+1)^{\alpha}+uz_1(x)+u^2(x+1)^{\mathfrak{T}_2}z_2(x)\rangle)=4 $. Thus, $d_L(\mathcal{C}^{10}_7)=4$.
        \end{enumerate}
    \end{enumerate}
\end{proof}

%\newpage
\subsection{If $z_1(x)\neq0$, $\mathfrak{T}_1\neq0$, $z_2(x)\neq0$, $\mathfrak{T}_2\neq0$ and $z_3(x)=0$ }
\begin{theorem}\label{thm45}
   Let $\mathcal{C}^{11}_7=\langle (x+1)^{\alpha}+u(x+1)^{\mathfrak{T}_1} z_1(x)+u^2(x+1)^{\mathfrak{T}_2}z_2(x), u(x+1)^{\beta} \rangle$,
  where $0< \mathcal{W} \leq \beta <  \mathcal{U} \leq \alpha \leq 2^\varsigma-1$, $0<  \mathfrak{T}_1 < \beta $, $0< \mathfrak{T}_2 < \mathcal{W} $ and  $z_1(x)$ and $z_2(x)$ are units in $\mathcal{S}$. Then
  \begin{center}
	$d_L(\mathcal{C}^{11}_7)=$
	$\begin{cases}
            2&\text{if}\quad 1<\beta< \alpha \leq 2^{\varsigma-1} \quad\text{with}\quad 2\alpha \leq 2^{\varsigma-1}+\mathfrak{T}_1, \alpha+\beta \leq 2^{\varsigma-1}+\mathfrak{T}_2  ,\\
            &\qquad\text{and} \quad 2\alpha+\beta \leq 2^{\varsigma-1}+2\mathfrak{T}_1 ,\\
            4&\text{if}\quad 1<\beta< \alpha \leq 2^{\varsigma-1} \quad\text{with} \quad 2\alpha > 2^{\varsigma-1}+\mathfrak{T}_1 \quad\text{or} \quad \alpha+\beta > 2^{\varsigma-1}+\mathfrak{T}_2   \\&\qquad\text{or} \quad 2\alpha+\beta > 2^{\varsigma-1}+2\mathfrak{T}_1,\\
            4  &\text{if}\quad 2^{\varsigma-1}+1\leq \alpha \leq 
            2^{\varsigma}-1 \quad \text{with} \quad   1 < \beta \leq 2^{\varsigma-1},\\
            2^{\gamma+1} &\text{if}\quad 2^\varsigma-2^{\varsigma-\gamma}+1 \leq \alpha  \leq 2^\varsigma-2^{\varsigma-\gamma}+2^{\varsigma-\gamma-1},\\
            &\qquad   \text{with} \quad 3\alpha \leq 2^{\varsigma}-2^{\varsigma-\gamma}+2^{\varsigma-\gamma-1}+2\mathfrak{T}_1, \\
            &\qquad \alpha \leq 2^{\varsigma-1}-2^{\varsigma-\gamma-1}+2^{\varsigma-\gamma-2}+\frac{\mathfrak{T}_2}{2} \quad\text{and}\\
            &\qquad \alpha \leq 2^{\varsigma-1}+\frac{\mathfrak{T}_1}{2}
            \qquad \text{where}\quad 1\leq \gamma \leq \varsigma-1.
	\end{cases}$
    \end{center}
\end{theorem}
\begin{proof}
    Let $\mathcal{B}=\{ \zeta_1,\zeta_2,\ldots, \zeta_m\}$ be a TOB of $\mathbb{F}_{2^m}$ over $\mathbb{F}_2.$  By Theorem \ref{thm1}, $\mathcal{W}=\beta$.
    \begin{enumerate}
        \item \textbf{Case 1:} Let $1<\alpha \leq 2^{\varsigma-1}$. Since  $1< \beta<\alpha $, clearly $1< \beta\leq 2^{\varsigma-1}$, by Theorem \ref{thm3}, $d_H(\langle (x+1)^{\beta} \rangle)=2$. Thus, $2\leq d_L(\mathcal{C}^{11}_7)\leq 4$.
        \begin{enumerate}
            \item \textbf{Subcase i:} Let $2\alpha\leq 2^{\varsigma-1}+\mathfrak{T}_1$, $\alpha+\beta \leq 2^{\varsigma-1}+\mathfrak{T}_2$  and $2\alpha+\beta \leq 2^{\varsigma-1}+2\mathfrak{T}_1$ . We have $\chi (x)=\zeta_1(x^{2^{\varsigma-1}}+1)=\zeta_1(x+1)^{2^{\varsigma-1}}=\zeta_1\Big[(x+1)^{\alpha}+u (x+1)^{\mathfrak{T}_1}z_1(x)+u^2(x+1)^{\mathfrak{T}_2}z_2(x)\Big]\Big[(x+1)^{2^{\varsigma-1}-\alpha}+u(x+1)^{2^{\varsigma-1}-2\alpha+\mathfrak{T}_1}z_1(x)\Big] +\Big[u(x+1)^{\beta}\Big]\Big[u\Big((x+1)^{2^{\varsigma-1}-\alpha-\beta+\mathfrak{T}_2}z_2(x)+(x+1)^{2^{\varsigma-1}-2\alpha+2\mathfrak{T}_1-\beta}z_1(x)z_1(x)\Big)\Big]\in \mathcal{C}^{11}_7$. Since  $wt^{\mathcal{B}}_L(\chi (x))=2$, $d_L(\mathcal{C}^{11}_7)=2$.
            \item \textbf{Subcase ii:} Let $2\alpha> 2^{\varsigma-1}+\mathfrak{T}_1$ or $\alpha+\beta > 2^{\varsigma-1}+\mathfrak{T}_2$  or $2\alpha+\beta > 2^{\varsigma-1}+2\mathfrak{T}_1$. By following the same line of arguments as in Theorem \ref{thm8}, we get $(1+ x)^{2^{\varsigma-1}}\in \mathcal{C}^{11}_7$. Then
            \begin{align*}
                  (1+ x)^{2^{\varsigma-1}}=&\Big[(x+1)^{\alpha}+u (x+1)^{\mathfrak{T}_1}z_1(x)+u^2(x+1)^{\mathfrak{T}_2}z_2(x)\Big]\Big[\varphi_1(x)+u\varphi_2(x)+u^2\varphi_3(x)\Big]\\
                  &+\Big[u (x+1)^{\beta}\Big]\Big[\varkappa_1(x)+u\varkappa_2(x)\Big]\\
                  =&(x+1)^{\alpha} \varphi_1(x)+u\Big[(x+1)^{\mathfrak{T}_1}z_1(x)\varphi_1(x)+(x+1)^{\alpha} \varphi_2(x)+\varkappa_1(x)(x+1)^{\beta}\Big]\\
                  &+u^2\Big[(x+1)^{\mathfrak{T}_2}\varphi_1(x)z_2(x)+ (x+1)^{\mathfrak{T}_1}z_1(x)\varphi_2(x)+(x+1)^{\alpha} \varphi_3(x)+ (x+1)^{\beta}\varkappa_2(x)\Big]
             \end{align*}
             for some $\varphi_1(x),\varphi_2(x), \varphi_3(x), \varkappa_1(x), \varkappa_2(x) \in \frac{\mathbb{F}_{p^m}[x]}{\langle x^{2^\varsigma}-1 \rangle}$. Then $\varphi_1(x)=(x+1)^{2^{\varsigma-1}-\alpha}$, $\varphi_2(x)=(x+1)^{2^{\varsigma-1}-2\alpha+\mathfrak{T}_1}z_1(x)+(x+1)^{\beta-\alpha}\varkappa_1(x)$ and $\varkappa_2(x)=(x+1)^{2^{\varsigma-1}-\alpha-\beta+\mathfrak{T}_2}z_2(x)+(x+1)^{2^{\varsigma-1}-2\alpha-\beta+2\mathfrak{T}_1}z_1(x)z_1(x)+(x+1)^{\mathfrak{T}_1-\alpha}\varkappa_1(x)z_1(x)+(x+1)^{\alpha-\beta}\varphi_3(x)$. Since  $2\alpha> 2^{\varsigma-1}+\mathfrak{T}_1$ or $\alpha+\beta > 2^{\varsigma-1}+\mathfrak{T}_2$  or $2\alpha+\beta > 2^{\varsigma-1}+2\mathfrak{T}_1$, we get a contradiction. Thus, there exists no codeword of Lee weight 2. Also, following Theorem \ref{thm8}, there exist no codewords of Lee weight 3. 
             % A codeword $\wp(x)=u^2\zeta_1(x^{2^{\varsigma-1}}+1)=u^2\zeta_1(x+1)^{2^{\varsigma-1}}\in \langle u^2(x+1)^{\beta}\rangle \subseteq \mathcal{C}^{11}_7$ with $wt^{\mathcal{B}}_L(\wp(x))=4$.
             Hence $d_L(\mathcal{C}^{11}_7)=4$. 
        \end{enumerate}
        \item \textbf{Case 2:} Let  $ 2^{\varsigma-1}+1\leq \alpha \leq 2^\varsigma-1$. 
        \begin{enumerate}
            \item \textbf{Subcase i:} Let $1<\beta \leq 2^{\varsigma-1}$. From Theorem \ref{thm3}, $d_H(\langle (x+1)^{\beta} \rangle)=2$. Thus, $2\leq d_L(\mathcal{C}^{11}_7)\leq 4$. Following Theorem \ref{thm8}, we get that there exists no codeword of Lee weight 2 or 3. Hence $d_L(\mathcal{C}^{11}_7)=4$.

            \item \textbf{Subcase ii:} Let $2^\varsigma-2^{\varsigma-\gamma}+1 \leq \beta <\alpha  \leq 2^\varsigma-2^{\varsigma-\gamma}+2^{\varsigma-\gamma-1}$
              and $ 2\alpha \leq 2^{\varsigma}-2^{\varsigma-\gamma}+2^{\varsigma-\gamma-1}+\mathfrak{T}_1$, $2\alpha +\beta \leq 2^{\varsigma}-2^{\varsigma-\gamma}+2^{\varsigma-\gamma-1}+2\mathfrak{T}_1$ and $\alpha +\beta \leq 2^{\varsigma}-2^{\varsigma-\gamma}+2^{\varsigma-\gamma-1}+\mathfrak{T}_2$ where $1\leq \gamma \leq \varsigma-1$. By Theorem \ref{thm3},  $d_H(\langle (x+1)^{\beta} \rangle)=2^{\gamma+1}$. Then $d_L(\mathcal{C}^{11}_7)\geq2^{\gamma+1}$. Also if $3\alpha \leq 2^{\varsigma}-2^{\varsigma-\gamma}+2^{\varsigma-\gamma-1}+2\mathfrak{T}_1 $,\quad $\alpha \leq 2^{\varsigma-1}-2^{\varsigma-\gamma-1}+2^{\varsigma-\gamma-2}+\frac{\mathfrak{T}_2}{2}$ and $\alpha \leq 2^{\varsigma-1}+\frac{\mathfrak{T}_1}{2}$, where $ 1\leq \gamma \leq \varsigma-1.$ By Theorem \ref{thm23}, $d_L(\langle (x+1)^{\alpha}+u(x+1)^{\mathfrak{T}_1} z_1(x)+u^2(x+1)^{\mathfrak{T}_2}z_2(x)  \rangle)=2^{\gamma+1} $. Thus, $d_L(\mathcal{C}^{11}_7)=2^{\gamma+1}$.
             
        \end{enumerate}
    \end{enumerate}
\end{proof}

%\newpage
\subsection{If $z_1(x)=0$, $z_2(x)\neq0$, $\mathfrak{T}_2=0$ and $z_3(x)\neq0$, $\mathfrak{T}_3=0$ }
\begin{theorem}\label{thm46} 
   Let $\mathcal{C}^{12}_7=\langle (x+1)^{\alpha}+u^2z_2(x), u(x+1)^{\beta}+u^2   z_3(x) \rangle$,
  where $0< \mathcal{W} \leq \beta <  \mathcal{U} \leq \alpha \leq 2^\varsigma-1$ and  $z_2(x)$ and $z_3(x)$ are units in $\mathcal{S}$. Then
 \begin{center}
	$d_L(\mathcal{C}^{12}_7)=$
	$\begin{cases}
            2&\text{if}\quad \beta +\alpha\leq 2^{\varsigma-1},\\
		2  &\text{if}\quad z_2(x)=1 \quad\text{and} \quad\alpha=2^{\varsigma-1},\\
            4& \text{otherwise}.
	\end{cases}$
    \end{center}
\end{theorem}

\begin{proof}
 Let $\mathcal{B}=\{ \zeta_1,\zeta_2,\ldots, \zeta_m\}$ be a TOB of $\mathbb{F}_{2^m}$ over $\mathbb{F}_2.$ Let $\mathcal{W}$ be the smallest integer such that $u^2(x+1)^{\mathcal{W}} \in \mathcal{C}^{12}_7$. By Theorem \ref{thm1}, $\mathcal{W}=min\{\beta, 2^\varsigma-\alpha, 2^\varsigma-\beta\}$. Then $1\leq \mathcal{W} \leq 2^{\varsigma-1}$. By Theorem \ref{thm3} and Theorem \ref{thm34}, $d_H(\mathcal{C}^{12}_7)=2$. Thus, $2\leq d_L(\mathcal{C}^{12}_7)\leq 4$. 
   \begin{enumerate}
       \item \textbf{Case 1:} Let $\beta+\alpha\leq 2^{\varsigma-1}$. Since $0<\beta <\alpha$, clearly $1<\alpha<2^{\varsigma-1}$. Let $\chi (x)=\zeta_1(x^{2^{\varsigma-1}}+1)=\zeta_1(x+1)^{2^{\varsigma-1}}=\zeta_1\Big[(x+1)^{\alpha}+u^2z_2(x)\Big]\Big[(x+1)^{2^{\varsigma-1}-\alpha} \Big]+\Big[u(x+1)^{\beta}+u^2   z_3(x)\Big]\Big[u(x+1)^{2^{\varsigma-1}-\alpha-\beta}z_2(x)\Big]\in \mathcal{C}^{12}_7$. Since  $wt^{\mathcal{B}}_L(\chi (x))=2$, $d_L(\mathcal{C}^{12}_7)=2$.
       \item \textbf{Case 2:} If $z_2(x)=1$ and $\alpha=2^{\varsigma-1}$, we have  $\chi (x)=\zeta_1((x+1)^{2^{\varsigma-1}}+u^2)=\zeta_1(x^{2^{\varsigma-1}}+1+u^2)\in \mathcal{C}^{12}_7$. Since $wt^{\mathcal{B}}_L(\chi (x))=2$, we have $d_L(\mathcal{C}^{12}_7)=2$.
       \item \textbf{Case 3:} Let $ \beta+ \alpha >2^{\varsigma-1}$ and either $z_2(x)\neq1$ or $\alpha \neq2^{\varsigma-1}$. Following Theorem \ref{thm9}, we can prove that $\mathcal{C}^{12}_7$ has no codeword of Lee weights 2 and 3.
       Hence $d_L(\mathcal{C}^{12}_7)=4$.
   \end{enumerate} 
\end{proof}

%\newpage
\subsection{If $z_1(x)=0$, $z_2(x)\neq0$, $\mathfrak{T}_2\neq0$ and $z_3(x)\neq0$, $\mathfrak{T}_3=0$}
\begin{theorem}\label{thm47} 
   Let $\mathcal{C}^{13}_7=\langle (x+1)^{\alpha}+u^2(x+1)^{\mathfrak{T}_2}z_2(x), u(x+1)^{\beta}+u^2  z_3(x) \rangle$,
  where $1< \mathcal{W} \leq \beta <  \mathcal{U} \leq \alpha \leq 2^\varsigma-1$, $0<  \mathfrak{T}_2 < \mathcal{W} $ and  $z_2(x)$ and $z_3(x)$ are units in $\mathcal{S}$. Then
  
     \begin{equation*}
	d_L(\mathcal{C}^{13}_7)=
	\begin{cases}
            2&\text{if}\quad 1< \alpha \leq 2^{\varsigma-1} \text{with} \quad \alpha+\beta \leq 2^{\varsigma-1}+\mathfrak{T}_2,\\
            4&\text{if}\quad 1< \alpha \leq 2^{\varsigma-1} \text{with} \quad \alpha+\beta > 2^{\varsigma-1}+\mathfrak{T}_2,\\
		4  &\text{if}\quad 2^{\varsigma-1}+1\leq \alpha \leq 
            2^{\varsigma}-1 \quad \text{with} \quad 1< \mathcal{W} \leq \beta \leq 2^{\varsigma-1},\\
            4  &\text{if}\quad 2^{\varsigma-1}+1\leq \beta <\alpha \leq 
            2^{\varsigma}-1 \quad \text{with} \quad 1< \mathcal{W} \leq 2^{\varsigma-1},\\
            4  &\text{if}\quad 2^{\varsigma-1}+1\leq \mathcal{W} \leq  \beta <\alpha \leq 
            2^{\varsigma}-1 \quad \text{with} \quad \alpha \geq 2^{\varsigma-1}+\mathfrak{T}_2,\\
            2^{\gamma+1} &\text{if}\quad 2^\varsigma-2^{\varsigma-\gamma}+1 \leq \mathcal{W}\leq \beta<\alpha  \leq 2^\varsigma-2^{\varsigma-\gamma}+2^{\varsigma-\gamma-1}  \\
            &\text{with}\quad  \alpha \leq 2^{\varsigma-1}-2^{\varsigma-\gamma-1}+2^{\varsigma-\gamma-2}+\frac{\mathfrak{T}_2}{2}\quad \text{where}\quad 1\leq \gamma \leq \varsigma-1.
	\end{cases}
    \end{equation*}
\end{theorem}

\begin{proof}
    Let $\mathcal{B}=\{ \zeta_1,\zeta_2,\ldots, \zeta_m\}$ be a TOB of $\mathbb{F}_{2^m}$ over $\mathbb{F}_2.$
    \begin{enumerate}
        \item \textbf{Case 1:} Let $1< \mathcal{W}\leq\beta<\alpha \leq 2^{\varsigma-1}$.  By Theorem \ref{thm3}, $d_H(\langle (x+1)^{\mathcal{W}} \rangle)=2$. Thus, $2\leq d_L(\mathcal{C}^{13}_7)\leq 4$.
        \begin{enumerate}
            \item \textbf{Subcase i:} Let $\alpha+\beta \leq 2^{\varsigma-1}+\mathfrak{T}_2$. We have $\chi (x)=\zeta_1(x^{2^{\varsigma-1}}+1)=\zeta_1(x+1)^{2^{\varsigma-1}}=\zeta_1\Big[(x+1)^{\alpha}+u^2(x+1)^{\mathfrak{T}_2}z_2(x)\Big]\Big[(x+1)^{2^{\varsigma-1}-\alpha}\Big] +\Big[u(x+1)^{\beta}+u^2z_3(x)\Big]\Big[u(x+1)^{2^{\varsigma-1}-\alpha-\beta+\mathfrak{T}_2}z_2(x)\Big]\in \mathcal{C}^{13}_7$. Since  $wt^{\mathcal{B}}_L(\chi (x))=2$, $d_L(\mathcal{C}^{13}_7)=2$.
            \item \textbf{Subcase ii:} Let $\alpha+\beta > 2^{\varsigma-1}+\mathfrak{T}_2$. Following Theorem \ref{thm9}, we can prove that $\mathcal{C}^{13}_7$ has no codeword of Lee weights 2 and 3. Hence $d_L(\mathcal{C}^{13}_7)=4$.
            
        \end{enumerate}
        \item \textbf{Case 2:} Let  $ 2^{\varsigma-1}+1\leq \alpha \leq 2^\varsigma-1$. 
        \begin{enumerate}
            \item \textbf{Subcase i:} Let $1< \mathcal{W} \leq \beta \leq 2^{\varsigma-1}$. From Theorem \ref{thm3}, $d_H(\langle (x+1)^{\mathcal{W}}\rangle)=2$. Thus, $2\leq d_L(\mathcal{C}^{13}_7)\leq 4$. Following Theorem \ref{thm8}, we get that there exists no codeword of Lee weight 2 or 3. Hence $d_L(\mathcal{C}^{13}_7)=4$.
            
             \item \textbf{Subcase ii:} Let $2^{\varsigma-1}+1\leq \beta\leq 2^{\varsigma}-1 $. 
            \begin{itemize}
                \item  Let $1< \mathcal{W} \leq  2^{\varsigma-1}$. By Theorem \ref{thm3}, $d_H(\langle (x+1)^{\mathcal{W}}\rangle)=2$, Thus, $2\leq d_L(\mathcal{C}^{13}_7)\leq 4$. Following Theorem \ref{thm9}, we can prove that $\mathcal{C}^{13}_7$ has no codeword of Lee weights 2 and 3. Hence $d_L(\mathcal{C}^{13}_7)=4$.
                    
                \item  Let $2^{\varsigma-1}+1\leq \mathcal{W} \leq 2^{\varsigma}-1 $
                By Theorem \ref{thm3},  $d_H(\langle (x+1)^{\mathcal{W}} \rangle)\geq 4$ and by Theorem \ref{thm17},  $d_L(\langle (x+1)^\alpha+u(x+1)^{\mathfrak{T}_2} z_2(x) \rangle)=4$ if $ \alpha \geq 2^{\varsigma-1}+\mathfrak{T}_2$. Thus, $d_L(\mathcal{C}^{13}_7)=4$.

                \item  Let $2^\varsigma-2^{\varsigma-\gamma}+1 \leq \mathcal{W}\leq \beta  \leq 2^\varsigma-2^{\varsigma-\gamma}+2^{\varsigma-\gamma-1}$ and $\alpha \leq 2^{\varsigma-1}-2^{\varsigma-\gamma-1}+2^{\varsigma-\gamma-2}+\frac{\mathfrak{T}_2}{2}$, where $ 1\leq \gamma \leq \varsigma-1$.  From Theorem \ref{thm3}, $d_H(\langle (x+1)^{\mathcal{W}}\rangle)=2^{\gamma+1}$.  Then $d_L(\mathcal{C}^{13}_7)\geq 2^{\gamma+1}$. From Theorem \ref{thm17}, $d_L(\langle u(x+1)^{\alpha} +u^2(x+1)^t z(x) \rangle )=2^{\gamma+1}$. Then $d_L(\mathcal{C}^{13}_7)\leq 2^{\gamma+1}$ if $\alpha \leq 2^{\varsigma-1}-2^{\varsigma-\gamma-1}+2^{\varsigma-\gamma-2}+\frac{\mathfrak{T}_2}{2}$. Hence $d_L(\mathcal{C}^{13}_7)=2^{\gamma+1}$.
                \end{itemize}
        \end{enumerate}
    \end{enumerate}
\end{proof}

%\newpage
\subsection{If $z_1(x)=0$, $z_2(x)\neq0$, $\mathfrak{T}_2=0$ and $z_3(x)\neq0$, $\mathfrak{T}_3\neq0$ }
\begin{theorem}\label{thm48} 
   Let $\mathcal{C}^{14}_7=\langle (x+1)^{\alpha}+u^2z_2(x), u(x+1)^{\beta}+u^2 (x+1)^{\mathfrak{T}_3}  z_3(x) \rangle$,
  where $1<\mathcal{W} \leq \beta <  \mathcal{U} \leq \alpha \leq 2^\varsigma-1$, $0<  \mathfrak{T}_3 < \mathcal{W} $ and  $z_2(x)$ and $z_3(x)$ are units in $\mathcal{S}$. Then
   \begin{center}
	$d_L(\mathcal{C}^{14}_7)=$
	$\begin{cases}
            2&\text{if}\quad \beta +\alpha\leq 2^{\varsigma-1},\\
		2  &\text{if}\quad z_2(x)=1 \quad\text{and} \quad\alpha=2^{\varsigma-1},\\
            4& \text{otherwise}.
	\end{cases}$
    \end{center}
\end{theorem}

\begin{proof}
 Let $\mathcal{B}=\{ \zeta_1,\zeta_2,\ldots, \zeta_m\}$ be a TOB of $\mathbb{F}_{2^m}$ over $\mathbb{F}_2.$ Let $\mathcal{W}$ be the smallest integer such that $u^2(x+1)^{\mathcal{W}} \in \mathcal{C}^{14}_7$. By Theorem \ref{thm1}, $\mathcal{W}=min\{\beta, 2^\varsigma-\alpha, 2^\varsigma-\beta\}$. Then $1< \mathcal{W} \leq 2^{\varsigma-1}$. By Theorem \ref{thm3} and Theorem \ref{thm34}, $d_H(\mathcal{C}^{14}_7)=2$. Thus, $2\leq d_L(\mathcal{C}^{14}_7)\leq 4$. 
   \begin{enumerate}
       \item \textbf{Case 1:} Let $\beta+\alpha\leq 2^{\varsigma-1}$. Since $1\leq\beta <\alpha$ clearly $1\leq\alpha<2^{\varsigma-1}$. Let $\chi (x)=\zeta_1(x^{2^{\varsigma-1}}+1)=\zeta_1(x+1)^{2^{\varsigma-1}}=\zeta_1[(x+1)^{\alpha}+u^2z_2(x)][(x+1)^{2^{\varsigma-1}-\alpha} ]+[u(x+1)^{\beta}+u^2 (x+1)^{\mathfrak{T}_3}  z_3(x)][u(x+1)^{2^{\varsigma-1}-\alpha-\beta}z_2(x)]\in \mathcal{C}^{14}_7$. Since  $wt^{\mathcal{B}}_L(\chi (x))=2$, $d_L(\mathcal{C}^{14}_7)=2$.
       \item \textbf{Case 2:} If $z_2(x)=1$ and $\alpha=2^{\varsigma-1}$, we have  $\chi (x)=\zeta_1((x+1)^{2^{\varsigma-1}}+u^2)=\zeta_1(x^{2^{\varsigma-1}}+1+u^2)\in \mathcal{C}^{14}_7$. Since $wt^{\mathcal{B}}_L(\chi (x))=2$, we have $d_L(\mathcal{C}^{14}_7)=2$.
       \item \textbf{Case 3:} Let $ \beta+ \alpha >2^{\varsigma-1}$ and either $z_2(x)\neq1$ or $\alpha \neq2^{\varsigma-1}$. Following Theorem \ref{thm9}, we can prove that $\mathcal{C}^{14}_7$ has no codeword of Lee weights 2 and 3. Hence $d_L(\mathcal{C}^{14}_7)=4$.
   \end{enumerate} 
\end{proof}

%\newpage
\subsection{If $z_1(x)=0$, $z_2(x)\neq0$, $\mathfrak{T}_2\neq0$ and $z_3(x)\neq0$, $\mathfrak{T}_3\neq0$ }
\begin{theorem}\label{thm49} 
   Let $\mathcal{C}^{15}_7=\langle (x+1)^{\alpha}+u^2(x+1)^{\mathfrak{T}_2}z_2(x), u(x+1)^{\beta}+u^2 (x+1)^{\mathfrak{T}_3}  z_3(x) \rangle$,
  where $1<\mathcal{W} \leq \beta <  \mathcal{U} \leq \alpha \leq 2^\varsigma-1$, $0< \mathfrak{T}_2 < \mathcal{W} $, $0<  \mathfrak{T}_3 < \mathcal{W} $ and  $z_2(x)$ and $z_3(x)$ are units in $\mathcal{S}$. Then
  \begin{equation*}
	d_L(\mathcal{C}^{15}_7)=
	\begin{cases}
            2&\text{if}\quad 1<\alpha \leq 2^{\varsigma-1} \text{with} \quad \alpha+\beta \leq 2^{\varsigma-1}+\mathfrak{T}_2,\\
            4&\text{if}\quad 1<\alpha \leq 2^{\varsigma-1} \text{with} \quad \alpha+\beta > 2^{\varsigma-1}+\mathfrak{T}_2,\\
		4  &\text{if}\quad 2^{\varsigma-1}+1\leq \alpha \leq 
            2^{\varsigma}-1 \quad \text{with} \quad 1< \mathcal{W} \leq \beta \leq 2^{\varsigma-1},\\
            4  &\text{if}\quad 2^{\varsigma-1}+1\leq \beta <\alpha \leq 
            2^{\varsigma}-1 \quad \text{with} \quad 1< \mathcal{W} \leq 2^{\varsigma-1},\\
            4  &\text{if}\quad 2^{\varsigma-1}+1\leq \mathcal{W} \leq  \beta <\alpha \leq 
            2^{\varsigma}-1 \quad \text{with} \quad \alpha \geq 2^{\varsigma-1}+\mathfrak{T}_2,\\
            2^{\gamma+1} &\text{if}\quad 2^\varsigma-2^{\varsigma-\gamma}+1 \leq \mathcal{W}\leq \beta<\alpha  \leq 2^\varsigma-2^{\varsigma-\gamma}+2^{\varsigma-\gamma-1}  \\
            &\text{with}\quad  \alpha \leq 2^{\varsigma-1}-2^{\varsigma-\gamma-1}+2^{\varsigma-\gamma-2}+\frac{\mathfrak{T}_2}{2}\quad \text{where}\quad 1\leq \gamma \leq \varsigma-1.
	\end{cases}
    \end{equation*}
\end{theorem}

\begin{proof}
    Let $\mathcal{B}=\{ \zeta_1,\zeta_2,\ldots, \zeta_m\}$ be a TOB of $\mathbb{F}_{2^m}$ over $\mathbb{F}_2.$
    \begin{enumerate}
        \item \textbf{Case 1:} Let $1< \mathcal{W}\leq\beta<\alpha \leq 2^{\varsigma-1}$.  By Theorem \ref{thm3}, $d_H(\langle (x+1)^{\mathcal{W}} \rangle)=2$. Thus, $2\leq d_L(\mathcal{C}^{15}_7)\leq 4$.
        \begin{enumerate}
            \item \textbf{Subcase i:} Let $\alpha+\beta \leq 2^{\varsigma-1}+\mathfrak{T}_2$. We have $\chi (x)=\zeta_1(x^{2^{\varsigma-1}}+1)=\zeta_1(x+1)^{2^{\varsigma-1}}=\zeta_1\Big[(x+1)^{\alpha}+u^2(x+1)^{\mathfrak{T}_2}z_2(x)\Big]\Big[(x+1)^{2^{\varsigma-1}-\alpha}\Big] +\Big[u(x+1)^{\beta}+u^2(x+1)^{\mathfrak{T}_3}z_3(x)\Big]\Big[u(x+1)^{2^{\varsigma-1}-\alpha-\beta+\mathfrak{T}_2}z_2(x)\Big]\in \mathcal{C}^{15}_7$. Since  $wt^{\mathcal{B}}_L(\chi (x))=2$, $d_L(\mathcal{C}^{15}_7)=2$.
            \item \textbf{Subcase ii:} Let $\alpha+\beta > 2^{\varsigma-1}+\mathfrak{T}_2$. Following Theorem \ref{thm9}, we can prove that $\mathcal{C}^{15}_7$ has no codeword of Lee weights 2 and 3. Hence $d_L(\mathcal{C}^{15}_7)=4$.
        \end{enumerate}
        
        \item \textbf{Case 2:} Let  $ 2^{\varsigma-1}+1\leq \alpha \leq 2^\varsigma-1$. 
        \begin{enumerate}
            \item \textbf{Subcase i:} Let $1< \mathcal{W} \leq \beta \leq 2^{\varsigma-1}$. From Theorem \ref{thm3}, $d_H(\langle (x+1)^{\mathcal{W}} \rangle)=2$. Thus, $2\leq d_L(\mathcal{C}^{15}_7)\leq 4$. Following Theorem \ref{thm8}, we get that there exists no codeword of Lee weight 2 or 3. Hence $d_L(\mathcal{C}^{15}_7)=4$.
            
            \item \textbf{Subcase ii:} Let $2^{\varsigma-1}+1\leq \beta\leq 2^{\varsigma}-1 $. 
            \begin{itemize}
                \item  Let $1< \mathcal{W} \leq  2^{\varsigma-1}$. By Theorem \ref{thm3}, $d_H(\langle (x+1)^{\mathcal{W}}\rangle)=2$, Thus, $2\leq d_L(\mathcal{C}^{15}_7)\leq 4$. 
                Following Theorem \ref{thm9}, we can prove that $\mathcal{C}^{15}_7$ has no codeword of Lee weights 2 and 3. Hence $d_L(\mathcal{C}^{15}_7)=4$.
                    
                \item  Let $2^{\varsigma-1}+1\leq \mathcal{W} \leq 2^{\varsigma}-1 $
                By Theorem \ref{thm3},  $d_H(\langle (x+1)^{\mathcal{W}} \rangle)\geq 4$ and by Theorem \ref{thm17},  $d_L(\langle (x+1)^\alpha+u^2(x+1)^{\mathfrak{T}_2} z_2(x) \rangle)=4$ if $ \alpha \geq 2^{\varsigma-1}+\mathfrak{T}_2$.
                % Also, by Theorem \ref{thm8},  $d_L(\langle u(x+1)^{\beta}+u^2 (x+1)^{\mathfrak{T}_3}  z_3(x) \rangle)=4$ if $ \beta \geq 2^{\varsigma-1}+\mathfrak{T}_3$. 
                Thus, $d_L(\mathcal{C}^{15}_7)=4$.
                
                \item  Let $2^\varsigma-2^{\varsigma-\gamma}+1 \leq \mathcal{W}\leq \beta  \leq 2^\varsigma-2^{\varsigma-\gamma}+2^{\varsigma-\gamma-1}$, where $ 1\leq \gamma \leq \varsigma-1$.  From Theorem \ref{thm3}, $d_H(\langle (x+1)^{\mathcal{W}}\rangle)=2^{\gamma+1}$.  Then $d_L(\mathcal{C}^{15}_7)\geq 2^{\gamma+1}$. From Theorem \ref{thm17}, $d_L(\langle u(x+1)^{\alpha} +u^2(x+1)^t z(x) \rangle )=2^{\gamma+1}$. Then $d_L(\mathcal{C}^{15}_7)\leq 2^{\gamma+1}$ if $\alpha \leq 2^{\varsigma-1}-2^{\varsigma-\gamma-1}+2^{\varsigma-\gamma-2}+\frac{\mathfrak{T}_2}{2}$. Hence $d_L(\mathcal{C}^{15}_7)=2^{\gamma+1}$.
                \end{itemize}
        \end{enumerate}
    \end{enumerate}
\end{proof}

%\newpage
\subsection{If $z_1(x)\neq0$, $\mathfrak{T}_1=0$ $z_2(x)=0$ and $z_3(x)\neq0$, $\mathfrak{T}_3=0$ }
\begin{theorem}\label{thm50}
   Let $\mathcal{C}^{16}_7=\langle (x+1)^{\alpha}+u z_1(x), u(x+1)^{\beta}+u^2 z_3(x) \rangle$,
  where $0< \mathcal{W} \leq \beta <  \mathcal{U} \leq \alpha \leq 2^\varsigma-1$, $0< \beta $ and  $z_1(x)$ and $z_3(x)$ are units in $\mathcal{S}$. Then
  \begin{center}
	$d_L(\mathcal{C}^{16}_7)=$
		$\begin{cases}
			2  &\text{if}\quad  2\alpha+\beta \leq 2^{\varsigma-1}, \\
                3 &\text{if}\quad z_1(x)=1 \quad \text{and}\quad\alpha=2^{\varsigma-1}, \\
                 4& \text{otherwise}.
			\end{cases}$
		\end{center}
\end{theorem}

\begin{proof}
    Let $\mathcal{B}=\{ \zeta_1,\zeta_2,\ldots, \zeta_m\}$ be a TOB of $\mathbb{F}_{2^m}$ over $\mathbb{F}_2.$ Let $\mathcal{W}$ be the smallest integer such that $u^2(x+1)^{\mathcal{W}} \in \mathcal{C}^{16}_7$. By Theorem \ref{thm1}, $\mathcal{W}=min\{\beta, 2^\varsigma-\beta\}$. Then $1\leq \mathcal{W} \leq 2^{\varsigma-1}$. By Theorem \ref{thm34} and Theorem \ref{thm3}, $d_H(\mathcal{C}^{16}_7)=2$. Thus, $2\leq d_L(\mathcal{C}^{16}_7)\leq 4$.
    \begin{enumerate}
        \item \textbf{\textbf{Case 1:} }If $2^{\varsigma-1}\geq 2\alpha+\beta$ we have $\chi (x)=\zeta_1(x+1)^{2^{\varsigma-1}}=\zeta_1\Big[(x+1)^{\alpha}+u z_1(x)\Big]\Big[(x+1)^{2^{\varsigma-1}-\alpha}+u(x+1)^{2^{\varsigma-1}-2\alpha}z_1(x)\Big] +\Big[u(x+1)^{\beta}+u^2 z_3(x)\Big]\Big[u(x+1)^{2^{\varsigma-1}-2\alpha-\beta}z_1(x)z_1(x)\Big]\in \mathcal{C}^{16}_7$. Since $wt^{\mathcal{B}}_L(\chi (x))=2$, we have $d_L(\mathcal{C}^{16}_7)=2$.
        \item \textbf{\textbf{Case 2:}} Let $z_1(x)=1$ and $\alpha=2^{\varsigma-1}$.
        Following the same steps as in Theorem \ref{thm9}, we get $(1+ x)^{2^{\varsigma-1}}\in \mathcal{C}^{16}_7$. Then
        \begin{align*}
            (1+ x)^{2^{\varsigma-1}}=&\Big[(x+1)^{\alpha}+u z_1(x)\Big]\Big[\varphi_1(x)+u\varphi_2(x)+u^2\varphi_3(x)\Big]\\
            &+\Big[u(x+1)^{\beta}+u^2 z_3(x)\Big]\Big[\varkappa_1(x)+u\varkappa_2(x)\Big]\\
            =&(x+1)^{\alpha} \varphi_1(x)+u\Big[\varphi_1(x)z_1(x)+(x+1)^{\alpha} \varphi_2(x)+(x+1)^{\beta} \varkappa_1(x)\Big]\\
            &+u^2\Big[\varphi_2(x)z_1(x)+(x+1)^{\alpha} \varphi_3(x)+\varkappa_1(x)z_3(x)+(x+1)^{\beta} \varkappa_2(x)\Big]
        \end{align*}
        for some $\varphi_1(x),\varphi_2(x), \varphi_3(x),\varkappa_1(x),\varkappa_2(x) \in \frac{\mathbb{F}_{p^m}[x]}{\langle x^{2^\varsigma}-1 \rangle}$. Then $\varphi_1(x)=(x+1)^{2^{\varsigma-1}-\alpha}$, $\varphi_2(x)=(x+1)^{2^{\varsigma-1}-2\alpha}z_1(x)+(x+1)^{\beta-\alpha}\varkappa_1(x)$ and $\varkappa_2(x)=(x+1)^{2^{\varsigma-1}-2\alpha-\beta}z_1(x)z_1(x)+(x+1)^{-\alpha}\varkappa_1(x)z_1(x)+(x+1)^{\alpha-\beta}\varphi_3(x)+(x+1)^{-\beta}\varkappa_1(x)z_3(x)$. As $\alpha=2^{\varsigma-1}$ and $\beta>0$, we have $ \beta+2 \alpha >2^{\varsigma-1}$. Then we obtain a contradiction. Thus, there exists no codeword of Lee weight 2. Also, we have $\chi (x)=\zeta_1\big[x^{2^{\varsigma-1}}+1+u \big]=\zeta_1\big[(x+1)^{2^{\varsigma-1}}+u \big]\in \mathcal{C}^{16}_7$. Since $wt^{\mathcal{B}}_L(\chi (x))=3$, we have $d_L(\mathcal{C}^{16}_7)=3$.
        \item \textbf{\textbf{Case 3:}} Let $2\alpha+\beta>2^{\varsigma-1}$ and either $z_1(x)\neq 1$  or  $\alpha \neq 2^{\varsigma-1}$. Following as in the above case, there exists no codeword of Lee weight 2 as $2\alpha+\beta>2^{\varsigma-1}$. Also, following Theorem \ref{thm8}, $\mathcal{C}^{16}_7$ has no codeword of Lee weight 3. Hence $d_L(\mathcal{C}^{16}_7)=4$.
    \end{enumerate}
\end{proof}

%\newpage
\subsection{If $z_1(x)\neq0$, $\mathfrak{T}_1\neq0$ $z_2(x)=0$ and $z_3(x)\neq0$, $\mathfrak{T}_3=0$  }
\begin{theorem}\label{thm51} 
   Let $\mathcal{C}^{17}_7=\langle (x+1)^{\alpha}+u(x+1)^{\mathfrak{T}_1} z_1(x), u(x+1)^{\beta}+u^2   z_3(x) \rangle$,
  where $0<\mathcal{W} \leq \beta <  \mathcal{U} \leq \alpha \leq 2^\varsigma-1$, $0< \mathfrak{T}_1 < \beta $ and  $z_1(x)$ and $z_3(x)$ are units in $\mathcal{S}$. Then
 \begin{center}
	$d_L(\mathcal{C}^{17}_7)=$
		$\begin{cases}
			2  &\text{if}\quad 2^{\varsigma-1}\geq \alpha, 2^{\varsigma-1}+\mathfrak{T}_1\geq 2\alpha  \quad \text{and}\quad 2^{\varsigma-1}+2\mathfrak{T}_1\geq 2\alpha+\beta,\\
                 4& \text{otherwise}.
			\end{cases}$
		\end{center}
\end{theorem}
\begin{proof}
    Let $\mathcal{B}=\{ \zeta_1,\zeta_2,\ldots, \zeta_m\}$ be a TOB of $\mathbb{F}_{2^m}$ over $\mathbb{F}_2.$ Let $\mathcal{W}$ be the smallest integer such that $u^2(x+1)^{\mathcal{W}} \in \mathcal{C}^{17}_7$. By Theorem \ref{thm1}, $\mathcal{W}=min\{\beta, 2^\varsigma-\beta\}$. Then $1< \mathcal{W} \leq 2^{\varsigma-1}$. By Theorem \ref{thm3} and Theorem \ref{thm34}, $d_H(\mathcal{C}^{17}_7)=2$. Thus, $2\leq d_L(\mathcal{C}^{17}_7)\leq 4$.
    \begin{enumerate}
        \item \textbf{Case 1:} Let $2^{\varsigma-1}\geq \alpha$, $2^{\varsigma-1}+\mathfrak{T}_1\geq 2\alpha$  and $2^{\varsigma-1}+2\mathfrak{T}_1\geq 2\alpha+\beta$. We have $\chi (x)=\zeta_1(x^{2^{\varsigma-1}}+1)=\zeta_1(x+1)^{2^{\varsigma-1}}=\zeta_1\Big[(x+1)^{\alpha}+u(x+1)^{\mathfrak{T}_1} z_1(x)\Big]\Big[(x+1)^{2^{\varsigma-1}-\alpha}+u(x+1)^{2^{\varsigma-1}-2\alpha+\mathfrak{T}_1}z_1(x)\Big] \Big[u(x+1)^{\beta}+u^2   z_3(x)  \Big]\Big[u(x+1)^{2^{\varsigma-1}-2\alpha+2\mathfrak{T}_1-\beta}z_1(x)z_1(x) \Big]\in \mathcal{C}^{17}_7$. Since  $wt^{\mathcal{B}}_L(\chi (x))=2$, $d_L(\mathcal{C}^{17}_7)=2$.
        \item \textbf{Case 2:} Let either $2^{\varsigma-1}< \alpha$ or $2^{\varsigma-1}+\mathfrak{T}_1< 2\alpha$  or $2^{\varsigma-1}+2\mathfrak{T}_1<2\alpha+\beta$. Following Theorem \ref{thm9}, $(1+ x)^{2^{\varsigma-1}}\in \mathcal{C}^{17}_7$. Then
        %We get a contradiction as in Theorem \ref{thm8}.
            \begin{align*}
            (1+ x)^{2^{\varsigma-1}}=&\Big[(x+1)^{\alpha}+u (x+1)^{\mathfrak{T}_1}z_1(x)\Big]\Big[\varphi_1(x)+u\varphi_2(x)+u^2\varphi_3(x)\Big]\\
            &+\Big[u(x+1)^{\beta}+u^2 z_3(x)\Big]\Big[\varkappa_1(x)+u\varkappa_2(x)\Big]\\
            =&(x+1)^{\alpha} \varphi_1(x)+u\Big[(x+1)^{\mathfrak{T}_1}\varphi_1(x)z_1(x)+(x+1)^{\alpha} \varphi_2(x)+(x+1)^{\beta} \varkappa_1(x)\Big]\\
            &+u^2\Big[(x+1)^{\mathfrak{T}_1}\varphi_2(x)z_1(x)+(x+1)^{\alpha} \varphi_3(x)+\varkappa_1(x)z_3(x)+(x+1)^{\beta} \varkappa_2(x)\Big]
        \end{align*}
             for some $\varphi_1(x),\varphi_2(x), \varphi_3(x) \in \frac{\mathbb{F}_{p^m}[x]}{\langle x^{2^\varsigma}-1 \rangle}$. Then $\varphi_1(x)=(x+1)^{2^{\varsigma-1}-\alpha}$, $\varphi_2(x)=(x+1)^{2^{\varsigma-1}-2\alpha+\mathfrak{T}_1}z_1(x)+(x+1)^{\beta-\alpha}\varkappa_1(x)$ and $\varkappa_2(x)=(x+1)^{2^{\varsigma-1}-2\alpha+2\mathfrak{T}_1-\beta}z_1(x)z_1(x)+(x+1)^{\mathfrak{T}_1-\alpha}\varkappa_1(x)z_1(x)+(x+1)^{\alpha-\beta}\varphi_3(x)+(x+1)^{-\beta}\varkappa_1(x)z_3(x)$. Since  either $2^{\varsigma-1}< \alpha$ or $2^{\varsigma-1}+\mathfrak{T}_1< 2\alpha$  or $2^{\varsigma-1}+2\mathfrak{T}_1<2\alpha+\beta$, we get a contradiction. Thus, there exists no codeword of Lee weight 2. Also, following Theorem \ref{thm8}, $\mathcal{C}^{17}_7$ has no codeword of Lee weight 3. Hence $d_L(\mathcal{C}^{17}_7)=4$.
    \end{enumerate}
\end{proof}

%\newpage
\subsection{If $z_1(x)\neq0$, $\mathfrak{T}_1=0$ $z_2(x)=0$ and $z_3(x)\neq0$, $\mathfrak{T}_3\neq0$  }
\begin{theorem}\label{thm52} 
   Let $\mathcal{C}^{18}_7=\langle (x+1)^{\alpha}+u z_1(x), u(x+1)^{\beta}+u^2 (x+1)^{\mathfrak{T}_3}  z_3(x) \rangle$,
  where $1<\mathcal{W} \leq \beta <  \mathcal{U} \leq \alpha \leq 2^\varsigma-1$, $0< \beta $,  $0< \mathfrak{T}_3 < \mathcal{W} $ and  $z_1(x)$ and $z_3(x)$ are units in $\mathcal{S}$. Then
  \begin{equation*}
	d_L(\mathcal{C}^{18}_7)=
	\begin{cases}
            2&\text{if}\quad 1< \alpha \leq 2^{\varsigma-1} \quad\text{with}\quad 2\alpha \leq 2^{\varsigma-1}\quad\text{and}\quad  2\alpha+\beta \leq 2^{\varsigma-1},\\
            3 &\text{if}\quad 1< \alpha \leq 2^{\varsigma-1} \quad\text{with}\quad z_1(x)=1 \quad \text{and}\quad\alpha=2^{\varsigma-1}, \\
            4&\text{if}\quad 1< \alpha \leq 2^{\varsigma-1}\quad \text{either with}\quad2\alpha \leq 2^{\varsigma-1}\quad\text{or}\quad 2\alpha+\beta > 2^{\varsigma-1}\\
            &\qquad\text{and either}\quad z_1(x)\neq1 \quad \text{or}\quad\alpha\neq2^{\varsigma-1} ,\\
               
		4  &\text{if}\quad 2^{\varsigma-1}+1\leq \alpha \leq 
            2^{\varsigma}-1 \quad \text{with} \quad 1< \mathcal{W} \leq \beta \leq 2^{\varsigma-1},\\
            4  &\text{if}\quad 2^{\varsigma-1}+1\leq \beta <\alpha \leq 
            2^{\varsigma}-1 \quad \text{with} \quad 1< \mathcal{W} \leq 2^{\varsigma-1},\\
            4  &\text{if}\quad 2^{\varsigma-1}+1\leq \mathcal{W} \leq  \beta <\alpha \leq 
            2^{\varsigma}-1 \quad \text{with} \quad   \beta \geq 2^{\varsigma-1}+\mathfrak{T}_3.\\
            % 2^{\gamma+1} &\text{if}\quad 2^\varsigma-2^{\varsigma-\gamma}+1 \leq \mathcal{W}\leq \beta<\alpha  \leq 2^\varsigma-2^{\varsigma-\gamma}+2^{\varsigma-\gamma-1}  ,\\
            % &\qquad\text{with}\quad  \alpha \leq 2^{\varsigma-1}-2^{\varsigma-\gamma-1}+2^{\varsigma-\gamma-2}\quad \text{and} \quad 2\alpha+\beta \leq 2^{\varsigma}-2^{\varsigma-\gamma}+2^{\varsigma-\gamma-1}, \\
            % &\qquad\text{where}\quad 1\leq \gamma \leq \varsigma-1.
	\end{cases}
    \end{equation*}
\end{theorem}

\begin{proof}
    Let $\mathcal{B}=\{ \zeta_1,\zeta_2,\ldots, \zeta_m\}$ be a TOB of $\mathbb{F}_{2^m}$ over $\mathbb{F}_2.$
    \begin{enumerate}
        \item \textbf{Case 1:} Let$1<\mathcal{W}\leq\beta<\alpha \leq 2^{\varsigma-1}$.  By Theorem \ref{thm3}, $d_H(\langle (x+1)^{\mathcal{W}} \rangle)=2$. Thus, $2\leq d_L(\mathcal{C}^{18}_7)\leq 4$.
        \begin{enumerate}
            \item \textbf{Subcase i:} Let $2\alpha \leq 2^{\varsigma-1}$ and $2\alpha+\beta \leq 2^{\varsigma-1}$. We have $\chi (x)=\zeta_1(x^{2^{\varsigma-1}}+1)=\zeta_1(x+1)^{2^{\varsigma-1}}=\zeta_1\Big[(x+1)^{\alpha}+u z_1(x)\Big]\Big[(x+1)^{2^{\varsigma-1}-\alpha}+u(x+1)^{2^{\varsigma-1}-2\alpha}z_1(x)\Big] +\Big[u(x+1)^{\beta}+u^2(x+1)^{\mathfrak{T}_3}z_3(x)\Big]\Big[u(x+1)^{2^{\varsigma-1}-2\alpha-\beta}z_1(x)z_1(x)\Big]\in \mathcal{C}^{18}_7$. Since  $wt^{\mathcal{B}}_L(\chi (x))=2$, $d_L(\mathcal{C}^{18}_7)=2$.

            \item \textbf{Subcase ii:} Let $z_1(x)=1$ and $\alpha=2^{\varsigma-1}$. Following as in Theorem \ref{thm9}, we can prove $\mathcal{C}^{18}_7$ has no codeword of Lee weights 2 as $2\alpha+\beta > 2^{\varsigma-1}$. we have $\chi (x)=\zeta_1\big[x^{2^{\varsigma-1}}+1+u \big]=\zeta_1\big[(x+1)^{2^{\varsigma-1}}+u \big]\in \mathcal{C}^{18}_7$. Since $wt^{\mathcal{B}}_L(\chi (x))=3$, we have $d_L(\mathcal{C}^{18}_7)=3$.
            
            \item \textbf{Subcase iii:} Let either $2\alpha > 2^{\varsigma-1}$ or $2\alpha+\beta > 2^{\varsigma-1}$. Following Theorem \ref{thm9}, we can prove that $\mathcal{C}^{18}_7$ has no codeword of Lee weights 2 and 3. Hence $d_L(\mathcal{C}^{18}_7)=4$.
        \end{enumerate}
        
        \item \textbf{Case 3:} Let  $ 2^{\varsigma-1}+1\leq \alpha \leq 2^\varsigma-1$. 
        \begin{enumerate}
            \item \textbf{Subcase i:} Let $1<\mathcal{W} \leq \beta \leq 2^{\varsigma-1}$. From Theorem \ref{thm3}, $d_H(\langle (x+1)^{\mathcal{W}} \rangle)=2$. Thus, $2\leq d_L(\mathcal{C}^{18}_7)\leq 4$. Following Theorem \ref{thm9}, we get that there exists no codeword of Lee weight 2 or 3. Hence $d_L(\mathcal{C}^{18}_7)=4$.
            
            \item \textbf{Subcase ii:} Let $2^{\varsigma-1}+1\leq \beta\leq 2^{\varsigma}-1 $. 
            \begin{itemize}
                \item  Let $1< \mathcal{W} \leq  2^{\varsigma-1}$. By Theorem \ref{thm3}, $d_H(\langle (x+1)^{\mathcal{W}}\rangle)=2$, Thus, $2\leq d_L(\mathcal{C}^{18}_7)\leq 4$. 
                Following Theorem \ref{thm9}, we can prove that $\mathcal{C}^{18}_7$ has no codeword of Lee weights 2 and 3. Hence $d_L(\mathcal{C}^{18}_7)=4$.
                    
                \item  Let $2^{\varsigma-1}+1\leq \mathcal{W} \leq 2^{\varsigma}-1 $
                By Theorem \ref{thm3},  $d_H(\langle (x+1)^{\mathcal{W}} \rangle)\geq 4$ and by Theorem \ref{thm8},  $d_L(\langle u(x+1)^{\beta}+u^2 (x+1)^{\mathfrak{T}_3}  z_3(x) \rangle)=4$ if $ \beta \geq 2^{\varsigma-1}+\mathfrak{T}_3$. Thus, $d_L(\mathcal{C}^{18}_7)=4$.

        \end{itemize}
    \end{enumerate}
    \end{enumerate}
\end{proof}

%\newpage
\subsection{If $z_1(x)\neq0$, $\mathfrak{T}_1\neq0$ $z_2(x)=0$ and $z_3(x)\neq0$, $\mathfrak{T}_3\neq0$}
\begin{theorem}\label{thm53} 
   Let $\mathcal{C}^{19}_7=\langle (x+1)^{\alpha}+u(x+1)^{\mathfrak{T}_1} z_1(x), u(x+1)^{\beta}+u^2 (x+1)^{\mathfrak{T}_3}  z_3(x) \rangle$,
  where $1< \mathcal{W} \leq \beta <  \mathcal{U} \leq \alpha \leq 2^\varsigma-1$, $0<  \mathfrak{T}_1 < \beta $, $0<  \mathfrak{T}_3 < \mathcal{W} $ and  $z_1(x)$ and $z_3(x)$ are units in $\mathcal{S}$. Then
 \begin{equation*}
	d_L(\mathcal{C}^{19}_7)=
	\begin{cases}
            2&\text{if}\quad 1< \alpha \leq 2^{\varsigma-1} \quad\text{with}\quad 2\alpha \leq 2^{\varsigma-1}+\mathfrak{T}_1\quad\text{and}\quad  2\alpha+\beta \leq 2^{\varsigma-1}+\mathfrak{T}_1,\\
            4&\text{if}\quad 1< \alpha \leq 2^{\varsigma-1}\quad \text{either with}\quad2\alpha \leq 2^{\varsigma-1}+\mathfrak{T}_1\quad\text{or}\quad \quad 2\alpha+\beta > 2^{\varsigma-1}+\mathfrak{T}_1,\\
		4  &\text{if}\quad 2^{\varsigma-1}+1\leq \alpha \leq 
            2^{\varsigma}-1 \quad \text{with} \quad 1<\mathcal{W} \leq \beta \leq 2^{\varsigma-1},\\
            4  &\text{if}\quad 2^{\varsigma-1}+1\leq \beta <\alpha \leq 
            2^{\varsigma}-1 \quad \text{with} \quad 1< \mathcal{W} \leq 2^{\varsigma-1},\\
            4  &\text{if}\quad 2^{\varsigma-1}+1\leq \mathcal{W} \leq  \beta <\alpha \leq 
            2^{\varsigma}-1 \quad \text{with} \quad   \beta \geq 2^{\varsigma-1}+\mathfrak{T}_3 \quad \text{or} \quad \alpha \geq 2^{\varsigma-1}+\mathfrak{T}_1,\\
            2^{\gamma+1} &\text{if}\quad 2^\varsigma-2^{\varsigma-\gamma}+1 \leq \mathcal{W}\leq \beta<\alpha  \leq 2^\varsigma-2^{\varsigma-\gamma}+2^{\varsigma-\gamma-1}  ,\\
            &\qquad\text{with}\quad 
            %\alpha \leq 2^{\varsigma-1}-2^{\varsigma-\gamma-1}+2^{\varsigma-\gamma-2}+\mathfrak{T}_1\quad \text{and} \quad 2\alpha+\beta \leq 2^{\varsigma}-2^{\varsigma-\gamma}+2^{\varsigma-\gamma-1}+\mathfrak{T}_1, \\
            \alpha \leq 2^{\varsigma-1}-2^{\varsigma-\gamma-1}+2^{\varsigma-\gamma-2}+\frac{\mathfrak{T}_1}{2} ,\\
            &\qquad\text{and}\quad 3\alpha\leq 2^\varsigma-2^{\varsigma-\gamma}+2^{\varsigma-\gamma-1}+2\mathfrak{T}_1,\quad\text{where}\quad 1\leq \gamma \leq \varsigma-1.
	\end{cases}
    \end{equation*}
\end{theorem}

\begin{proof}
    Let $\mathcal{B}=\{ \zeta_1,\zeta_2,\ldots, \zeta_m\}$ be a TOB of $\mathbb{F}_{2^m}$ over $\mathbb{F}_2.$
    \begin{enumerate}
        \item \textbf{Case 1:} Let $1< \mathcal{W}\leq\beta<\alpha \leq 2^{\varsigma-1}$.  By Theorem \ref{thm3}, $d_H(\langle (x+1)^{\mathcal{W}} \rangle)=2$. Thus, $2\leq d_L(\mathcal{C}^{19}_7)\leq 4$.
        \begin{enumerate}
            \item \textbf{Subcase i:} Let $2\alpha \leq 2^{\varsigma-1}+\mathfrak{T}_1$ $2\alpha+\beta \leq 2^{\varsigma-1}+\mathfrak{T}_1$. We have $\chi (x)=\zeta_1(x^{2^{\varsigma-1}}+1)=\zeta_1(x+1)^{2^{\varsigma-1}}=\zeta_1\Big[(x+1)^{\alpha}+u(x+1)^{\mathfrak{T}_1} z_1(x)\Big]\Big[(x+1)^{2^{\varsigma-1}-\alpha}+u(x+1)^{2^{\varsigma-1}-2\alpha+\mathfrak{T}_1}z_1(x)\Big] +\Big[u(x+1)^{\beta}+u^2(x+1)^{\mathfrak{T}_3}z_3(x)\Big]\Big[u(x+1)^{2^{\varsigma-1}-2\alpha-\beta +2\mathfrak{T}_1}z_1(x)z_1(x)\Big]\in \mathcal{C}^{19}_7$. Since  $wt^{\mathcal{B}}_L(\chi (x))=2$, $d_L(\mathcal{C}^{19}_7)=2$.
            \item \textbf{Subcase ii:} Let either $2\alpha > 2^{\varsigma-1}+\mathfrak{T}_1$ or $2\alpha+\beta > 2^{\varsigma-1}+\mathfrak{T}_1$.  Following Theorem \ref{thm9}, we can prove that $\mathcal{C}^{19}_7$ has no codeword of Lee weights 2 and 3. Hence $d_L(\mathcal{C}^{19}_7)=4$.
            
        \end{enumerate}
        \item \textbf{Case 2:} Let  $ 2^{\varsigma-1}+1\leq \alpha \leq 2^\varsigma-1$. 
        \begin{enumerate}
            \item \textbf{Subcase i:} Let $1< \mathcal{W} \leq \beta \leq 2^{\varsigma-1}$. From Theorem \ref{thm3}, $d_H(\langle (x+1)^{\mathcal{W}} \rangle)=2$. Thus, $2\leq d_L(\mathcal{C}^{19}_7)\leq 4$. Following Theorem \ref{thm8}, we get that there exists no codeword of Lee weight 2 or 3. Hence $d_L(\mathcal{C}^{19}_7)=4$.
            
             \item \textbf{Subcase ii:} Let $2^{\varsigma-1}+1\leq \beta\leq 2^{\varsigma}-1 $. 
            \begin{itemize}
                \item   Let $1< \mathcal{W} \leq  2^{\varsigma-1}$. By Theorem \ref{thm3}, $d_H(\langle (x+1)^{\mathcal{W}}\rangle)=2$, Thus, $2\leq d_L(\mathcal{C}^{19}_7)\leq 4$. Following Theorem \ref{thm9}, we can prove that $\mathcal{C}^{19}_7$ has no codeword of Lee weights 2 and 3. Hence $d_L(\mathcal{C}^{19}_7)=4$.
                    
                \item  Let $2^{\varsigma-1}+1\leq \mathcal{W} \leq 2^{\varsigma}-1 $
                By Theorem \ref{thm3},  $d_H(\langle (x+1)^{\mathcal{W}} \rangle)\geq 4$ and by Theorem \ref{thm8},  $d_L(\langle u(x+1)^{\beta}+u^2 (x+1)^{\mathfrak{T}_3}  z_3(x) \rangle)=4$ if $ \beta \geq 2^{\varsigma-1}+\mathfrak{T}_3$. Also, by Theorem \ref{thm19},  $d_L(\langle (x+1)^{\alpha}+u(x+1)^{\mathfrak{T}_1} z_1(x)\rangle)=4$ if $ \alpha \geq 2^{\varsigma-1}+\mathfrak{T}_1$. Thus, $d_L(\mathcal{C}^{19}_7)=4$.

             \item  Let $2^\varsigma-2^{\varsigma-\gamma}+1 \leq \mathcal{W}\leq \beta  \leq 2^\varsigma-2^{\varsigma-\gamma}+2^{\varsigma-\gamma-1}$, where $ 1\leq \gamma \leq \varsigma-1$.  From Theorem \ref{thm3}, $d_H(\langle (x+1)^{\mathcal{W}}\rangle)=2^{\gamma+1}$.  Then $d_L(\mathcal{C}^{19}_7)\geq 2^{\gamma+1}$. From Theorem \ref{thm19}, $d_L(\langle (x+1)^{\alpha}+u(x+1)^{\mathfrak{T}_1} z_1(x) \rangle )=2^{\gamma+1}$. Then $d_L(\mathcal{C}^{19}_7)\leq 2^{\gamma+1}$ if $\alpha \leq 2^{\varsigma-1}-2^{\varsigma-\gamma-1}+2^{\varsigma-\gamma-2}+\frac{\mathfrak{T}_1}{2}$ and $3\alpha\leq 2^\varsigma-2^{\varsigma-\gamma}+2^{\varsigma-\gamma-1}+2\mathfrak{T}_1$. Hence $d_L(\mathcal{C}^{19}_7)=2^{\gamma+1}$.
                \end{itemize}
        \end{enumerate}
    \end{enumerate}
\end{proof}

%\newpage
\subsection{If $z_1(x)\neq0$, $\mathfrak{T}_1=0$ $z_2(x)\neq0$, $\mathfrak{T}_2=0$ and $z_3(x)\neq0$, $\mathfrak{T}_3=0$ }
\begin{theorem}\label{thm54} 
   Let $\mathcal{C}^{20}_7=\langle (x+1)^{\alpha}+u z_1(x)+u^2z_2(x), u(x+1)^{\beta}+u^2   z_3(x) \rangle$,
  where $0< \mathcal{W} \leq \beta <  \mathcal{U} \leq \alpha \leq 2^\varsigma-1$, $0< \beta $ and  $z_1(x)$, $z_2(x)$ and $z_3(x)$ are units in $\mathcal{S}$.  Then
\begin{center}
	$d_L(\mathcal{C}^{20}_7)=$
		$\begin{cases}
			2  &\text{if}\quad  2^{\varsigma-1}\geq 2\alpha,\quad \alpha+\beta\leq 2^{\varsigma-1} \quad \text{and}\quad 2\alpha+\beta\leq 2^{\varsigma-1}, \\
                3 &\text{if}\quad z_1(x)=z_2(x)=1 \quad \text{and}\quad\alpha=2^{\varsigma-1}, \\
                 4& \text{otherwise}.
			\end{cases}$
		\end{center}
\end{theorem}
\begin{proof}
    Let $\mathcal{B}=\{ \zeta_1,\zeta_2,\ldots, \zeta_m\}$ be a TOB of $\mathbb{F}_{2^m}$ over $\mathbb{F}_2.$ Let $\mathcal{W}$ be the smallest integer such that $u^2(x+1)^{\mathcal{W}} \in \mathcal{C}^{20}_7$. By Theorem \ref{thm1}, $\mathcal{W}=min\{\beta, 2^\varsigma-\beta\}$. Then $1\leq \mathcal{W} \leq 2^{\varsigma-1}$. By Theorem \ref{thm34} and Theorem \ref{thm3}, $d_H(\mathcal{C}^{20}_7)=2$. Thus, $2\leq d_L(\mathcal{C}^{20}_7)\leq 4$.
    \begin{enumerate}
        \item \textbf{\textbf{Case 1:} }If $2^{\varsigma-1}\geq 2\alpha$, $\alpha+\beta\leq 2^{\varsigma-1}$ and $2\alpha+\beta\leq 2^{\varsigma-1}$ we have $\chi (x)=\zeta_1(x+1)^{2^{\varsigma-1}}=\zeta_1\Big[(x+1)^{\alpha}+u z_1(x)+u^2z_2(x)\Big]\Big[(x+1)^{2^{\varsigma-1}-\alpha}+u(x+1)^{2^{\varsigma-1}-2\alpha}z_1(x)\Big] +\Big[ u(x+1)^{\beta}+u^2   z_3(x)\Big]\Big[u\Big( (x+1)^{2^{\varsigma-1}-\alpha-\beta}z_2(x)+(x+1)^{2^{\varsigma-1}-2\alpha-\beta}z_1(x)z_1(x)\Big)\Big]\in \mathcal{C}^{20}_7$. Since $wt^{\mathcal{B}}_L(\chi (x))=2$, we have $d_L(\mathcal{C}^{20}_7)=2$.
        \item \textbf{\textbf{Case 2:}} Let $z_1(x)=z_2(x)=1$ and $\alpha=2^{\varsigma-1}$.
        Following the same steps as in Theorem \ref{thm9}, we get $(1+ x)^{2^{\varsigma-1}}\in \mathcal{C}^{20}_7$. Then
        \begin{align*}
            (1+ x)^{2^{\varsigma-1}}=&\Big[(x+1)^{\alpha}+u z_1(x)+u^2z_2(x)\Big]\Big[\varphi_1(x)+u\varphi_2(x)+u^2\varphi_3(x)\Big]\\
            &+\Big[ u(x+1)^{\beta}+u^2   z_3(x)\Big]\Big[\varkappa_1(x)+u\varkappa_2(x)\Big]\\
            =&(x+1)^{\alpha} \varphi_1(x)+u\Big[\varphi_1(x)z_1(x)+(x+1)^{\alpha} \varphi_2(x)+(x+1)^{\beta} \varkappa_1(x)\Big]\\
            &+u^2\Big[\varphi_1(x)z_2(x)+\varphi_2(x)z_1(x)+(x+1)^{\alpha} \varphi_3(x)+(x+1)^{\beta} \varkappa_2(x)+\varkappa_1(x)z_3(x)\Big]
        \end{align*}
        for some $\varphi_1(x),\varphi_2(x), \varphi_3(x),\varkappa_1(x),\varkappa_2(x) \in \frac{\mathbb{F}_{p^m}[x]}{\langle x^{2^\varsigma}-1 \rangle}$. Then $\varphi_1(x)=(x+1)^{2^{\varsigma-1}-\alpha}$, $\varphi_2(x)=(x+1)^{2^{\varsigma-1}-2\alpha}z_1(x)+(x+1)^{\beta-\alpha}\varkappa_1(x)$ and $\varkappa_2(x)=(x+1)^{2^{\varsigma-1}-\alpha-\beta}z_2(x)+(x+1)^{2^{\varsigma-1}-2\alpha-\beta}z_1(x)z_1(x)+(x+1)^{-\alpha}\varkappa_1(x)z_1(x)+(x+1)^{\alpha-\beta}\varphi_3(x)+(x+1)^{-\beta}\varkappa_1(x)z_3(x)$. As $\alpha=2^{\varsigma-1}$ and $\beta>0$, we have $ \alpha+\beta >2^{\varsigma-1}$. Then we obtain a contradiction. Thus, there exists no codeword of Lee weight 2. Also, we have $\chi (x)=\zeta_1\big[x^{2^{\varsigma-1}}+1+u +u^2\big]=\zeta_1\big[(x+1)^{2^{\varsigma-1}}+u +u^2\big]\in \mathcal{C}^{20}_7$. Since $wt^{\mathcal{B}}_L(\chi (x))=3$, we have $d_L(\mathcal{C}^{20}_7)=3$.
        
        \item \textbf{\textbf{Case 3:}} Let $2^{\varsigma-1}< 2\alpha$ or $\alpha+\beta> 2^{\varsigma-1}$ or $2\alpha+\beta> 2^{\varsigma-1}$  and either $z_1(x)\neq 1$ or $z_2(x)\neq 1$ or  $\alpha \neq 2^{\varsigma-1}$.  Following as in the above case, there exists no codeword of Lee weight 2 as $\alpha + \beta>2^{\varsigma-1}$. Also, following Theorem \ref{thm8}, $\mathcal{C}^{20}_7$ has no codeword of Lee weight 3. Hence $d_L(\mathcal{C}^{20}_7)=4$.
    \end{enumerate}
\end{proof}

    %\newpage
\subsection{If $z_1(x)\neq0$, $\mathfrak{T}_1\neq0$ $z_2(x)\neq0$, $\mathfrak{T}_2=0$ and $z_3(x)\neq0$, $\mathfrak{T}_3=0$ }
\begin{theorem}\label{thm55}
   Let $\mathcal{C}^{21}_7=\langle (x+1)^{\alpha}+u(x+1)^{\mathfrak{T}_1} z_1(x)+u^2z_2(x), u(x+1)^{\beta}+u^2   z_3(x) \rangle$,
  where $0< \mathcal{W} \leq \beta <  \mathcal{U} \leq \alpha \leq 2^\varsigma-1$, $0<  \mathfrak{T}_1 < \beta $ and  $z_1(x)$, $z_2(x)$ and $z_3(x)$ are units in $\mathcal{S}$. Then
\begin{center}
	$d_L(\mathcal{C}^{21}_7)=$
		$\begin{cases}
			2  &\text{if}\quad  2^{\varsigma-1}\geq 2\alpha,\quad \alpha+\beta\leq 2^{\varsigma-1} \quad \text{and}\quad 2\alpha+\beta\leq 2^{\varsigma-1}, \\
               
                 4& \text{otherwise}.
			\end{cases}$
		\end{center}
\end{theorem}
\begin{proof}
    Let $\mathcal{B}=\{ \zeta_1,\zeta_2,\ldots, \zeta_m\}$ be a TOB of $\mathbb{F}_{2^m}$ over $\mathbb{F}_2.$ Let $\mathcal{W}$ be the smallest integer such that $u^2(x+1)^{\mathcal{W}} \in \mathcal{C}^{21}_7$. By Theorem \ref{thm1}, $\mathcal{W}=min\{\beta, 2^\varsigma-\beta\}$. Then $1< \mathcal{W} \leq 2^{\varsigma-1}$. By Theorem \ref{thm34} and Theorem \ref{thm3}, $d_H(\mathcal{C}^{21}_7)=2$. Thus, $2\leq d_L(\mathcal{C}^{21}_7)\leq 4$.
    \begin{enumerate}
        \item \textbf{\textbf{Case 1:} }If $2^{\varsigma-1}+\mathfrak{T}_1\geq 2\alpha$, $\alpha+\beta\leq 2^{\varsigma-1}$ and $2\alpha+\beta\leq 2^{\varsigma-1}+\mathfrak{T}_1$ we have $\chi (x)=\zeta_1(x+1)^{2^{\varsigma-1}}=\zeta_1\Big[(x+1)^{\alpha}+u(x+1)^{\mathfrak{T}_1} z_1(x)+u^2z_2(x)\Big]\Big[(x+1)^{2^{\varsigma-1}-\alpha}+u(x+1)^{2^{\varsigma-1}-2\alpha+\mathfrak{T}_1}z_1(x)+\Big] +\Big[ u(x+1)^{\beta}+u^2   z_3(x)\Big]\Big[u\Big( (x+1)^{2^{\varsigma-1}-\alpha-\beta}z_2(x)+(x+1)^{2^{\varsigma-1}-2\alpha-\beta+2\mathfrak{T}_1}z_1(x)z_1(x)\Big)\Big]\in \mathcal{C}^{21}_7$. Since $wt^{\mathcal{B}}_L(\chi (x))=2$, we have $d_L(\mathcal{C}^{21}_7)=2$.
        
        \item \textbf{\textbf{Case 2:}} Let either $2^{\varsigma-1}+\mathfrak{T}_1< 2\alpha$ or $\alpha+\beta> 2^{\varsigma-1}$ or $2\alpha+\beta> 2^{\varsigma-1}+\mathfrak{T}_1$. Following the same steps as in Theorem \ref{thm9}, we get $(1+ x)^{2^{\varsigma-1}}\in \mathcal{C}^{21}_7$. Then
        \begin{align*}
            (1+ x)^{2^{\varsigma-1}}=&\Big[(x+1)^{\alpha}+u(x+1)^{\mathfrak{T}_1} z_1(x)+u^2z_2(x)\Big]\Big[\varphi_1(x)+u\varphi_2(x)+u^2\varphi_3(x)\Big]\\
            &+\Big[ u(x+1)^{\beta}+u^2   z_3(x)\Big]\Big[\varkappa_1(x)+u\varkappa_2(x)\Big]\\
            =&(x+1)^{\alpha} \varphi_1(x)+u\Big[(x+1)^{\mathfrak{T}_1}\varphi_1(x)z_1(x)+(x+1)^{\alpha} \varphi_2(x)+(x+1)^{\beta} \varkappa_1(x)\Big]\\
            &+u^2\Big[\varphi_1(x)z_2(x)+(x+1)^{\mathfrak{T}_1}\varphi_2(x)z_1(x)+(x+1)^{\alpha} \varphi_3(x)+(x+1)^{\beta} \varkappa_2(x)+\varkappa_1(x)z_3(x)\Big]
        \end{align*}
        for some $\varphi_1(x),\varphi_2(x), \varphi_3(x),\varkappa_1(x),\varkappa_2(x) \in \frac{\mathbb{F}_{p^m}[x]}{\langle x^{2^\varsigma}-1 \rangle}$. Then $\varphi_1(x)=(x+1)^{2^{\varsigma-1}-\alpha}$, $\varphi_2(x)=(x+1)^{2^{\varsigma-1}-2\alpha+\mathfrak{T}_1}z_1(x)+(x+1)^{\beta-\alpha}\varkappa_1(x)$ and $\varkappa_2(x)=(x+1)^{2^{\varsigma-1}-\alpha-\beta}z_2(x)+(x+1)^{2^{\varsigma-1}-2\alpha-\beta+2\mathfrak{T}_1}z_1(x)z_1(x)+(x+1)^{\mathfrak{T}_1-\alpha}\varkappa_1(x)z_1(x)+(x+1)^{\alpha-\beta} \varphi_3(x)+(x+1)^{-\beta}\varkappa_1(x)z_3(x)$. Since  $2^{\varsigma-1}+\mathfrak{T}_1< 2\alpha$ or $\alpha+\beta> 2^{\varsigma-1}$ or $2\alpha+\beta> 2^{\varsigma-1}+\mathfrak{T}_1$, we obtain a contradiction. Thus, there exists no codeword of Lee weight 2. Also, following Theorem \ref{thm8}, $\mathcal{C}^{21}_7$ has no codeword of Lee weight 3. Hence $d_L(\mathcal{C}^{21}_7)=4$.
    \end{enumerate}
\end{proof}

%\newpage
\subsection{If $z_1(x)\neq0$, $\mathfrak{T}_1=0$ $z_2(x)\neq0$, $\mathfrak{T}_2\neq0$ and $z_3(x)\neq0$, $\mathfrak{T}_3=0$ }
\begin{theorem}\label{thm56} 
   Let $\mathcal{C}^{22}_7=\langle (x+1)^{\alpha}+u z_1(x)+u^2(x+1)^{\mathfrak{T}_2}z_2(x), u(x+1)^{\beta}+u^2 z_3(x) \rangle$,
  where $1< \mathcal{W} \leq \beta <  \mathcal{U} \leq \alpha \leq 2^\varsigma-1$, $0< \beta $, $0<  \mathfrak{T}_2 < \mathcal{W} $ and  $z_1(x)$, $z_2(x)$ and $z_3(x)$ are units in $\mathcal{S}$. Then
\begin{center}
	$d_L(\mathcal{C}^{22}_7)=$
		$\begin{cases}
			2  &\text{if}\quad  2^{\varsigma-1}\geq 2\alpha,\quad \alpha+\beta\leq 2^{\varsigma-1} +\mathfrak{T}_2\quad \text{and}\quad 2\alpha+\beta\leq 2^{\varsigma-1}, \\
                 4& \text{otherwise}.
			\end{cases}$
		\end{center}
\end{theorem}
\begin{proof}
    Let $\mathcal{B}=\{ \zeta_1,\zeta_2,\ldots, \zeta_m\}$ be a TOB of $\mathbb{F}_{2^m}$ over $\mathbb{F}_2.$ Let $\mathcal{W}$ be the smallest integer such that $u^2(x+1)^{\mathcal{W}} \in \mathcal{C}^{22}_7$. By Theorem \ref{thm1}, $\mathcal{W}=min\{\beta, 2^\varsigma-\beta\}$. Then $1< \mathcal{W} \leq 2^{\varsigma-1}$. By Theorem \ref{thm14} and Theorem \ref{thm3}, $d_H(\mathcal{C}^{22}_7)=2$. Thus, $2\leq d_L(\mathcal{C}^{22}_7)\leq 4$.
    \begin{enumerate}
        \item \textbf{\textbf{Case 1:} }If $2^{\varsigma-1}\geq 2\alpha$, $\alpha+\beta\leq 2^{\varsigma-1}+\mathfrak{T}_2$ and $2\alpha+\beta\leq 2^{\varsigma-1}$ we have $\chi(x)=\zeta_1(x+1)^{2^{\varsigma-1}}=\zeta_1\Big[(x+1)^{\alpha}+u z_1(x)+u^2(x+1)^{\mathfrak{T}_2}z_2(x)\Big]\Big[(x+1)^{2^{\varsigma-1}-\alpha}+u(x+1)^{2^{\varsigma-1}-2\alpha}z_1(x)\Big] +\Big[ u(x+1)^{\beta}+u^2   z_3(x)\Big]\Big[u\Big( (x+1)^{2^{\varsigma-1}-\alpha-\beta+\mathfrak{T}_2}z_2(x)+(x+1)^{2^{\varsigma-1}-2\alpha-\beta}z_1(x)z_1(x)\Big)\Big]\in \mathcal{C}^{22}_7$. Since $wt^{\mathcal{B}}_L(\chi (x))=2$, we have $d_L(\mathcal{C}^{22}_7)=2$.
    
        \item \textbf{\textbf{Case 2:}} Let either $2^{\varsigma-1}+\mathfrak{T}_2< 2\alpha$ or $\alpha+\beta> 2^{\varsigma-1}+\mathfrak{T}_2$ or $2\alpha+\beta> 2^{\varsigma-1}$. Following the same steps as in Theorem \ref{thm9}, we get $(1+ x)^{2^{\varsigma-1}}\in \mathcal{C}^{22}_7$. Then
        \begin{align*}
            (1+ x)^{2^{\varsigma-1}}=&\Big[(x+1)^{\alpha}+u z_1(x)+u^2(x+1)^{\mathfrak{T}_2}z_2(x)\Big]\Big[\varphi_1(x)+u\varphi_2(x)+u^2\varphi_3(x)\Big]\\
            &+\Big[ u(x+1)^{\beta}+u^2   z_3(x)\Big]\Big[\varkappa_1(x)+u\varkappa_2(x)\Big]\\
            =&(x+1)^{\alpha} \varphi_1(x)+u\Big[\varphi_1(x)z_1(x)+(x+1)^{\alpha} \varphi_2(x)+(x+1)^{\beta} \varkappa_1(x)\Big]\\
            &+u^2\Big[(x+1)^{\mathfrak{T}_2}\varphi_1(x)z_2(x)+\varphi_2(x)z_1(x)+(x+1)^{\alpha} \varphi_3(x)+(x+1)^{\beta} \varkappa_2(x)+\varkappa_1(x)z_3(x)\Big]
        \end{align*}
        for some $\varphi_1(x),\varphi_2(x), \varphi_3(x),\varkappa_1(x),\varkappa_2(x) \in \frac{\mathbb{F}_{p^m}[x]}{\langle x^{2^\varsigma}-1 \rangle}$. Then $\varphi_1(x)=(x+1)^{2^{\varsigma-1}-\alpha}$, $\varphi_2(x)=(x+1)^{2^{\varsigma-1}-2\alpha}z_1(x)+(x+1)^{\beta-\alpha}\varkappa_1(x)$ and $\varkappa_2(x)=(x+1)^{2^{\varsigma-1}-\alpha-\beta+\mathfrak{T}_2}z_2(x)+(x+1)^{2^{\varsigma-1}-2\alpha-\beta}z_1(x)z_1(x)+(x+1)^{-\alpha}\varkappa_1(x)z_1(x)+(x+1)^{\alpha-\beta}\varphi_3(x)+(x+1)^{-\beta}\varkappa_1(x)z_3(x)$. Since  $2^{\varsigma-1}< 2\alpha$ or $\alpha+\beta> 2^{\varsigma-1}$ or $2\alpha+\beta> 2^{\varsigma-1}$, we obtain a contradiction. Thus, there exists no codeword of Lee weight 2. Also, following Theorem \ref{thm8}, $\mathcal{C}^{22}_7$ has no codeword of Lee weight 3. Hence $d_L(\mathcal{C}^{22}_7)=4$.
    \end{enumerate}
\end{proof}

%\newpage
\subsection{If $z_1(x)\neq0$, $\mathfrak{T}_1=0$ $z_2(x)\neq0$, $\mathfrak{T}_2=0$ and $z_3(x)\neq0$, $\mathfrak{T}_3\neq0$ }
\begin{theorem}\label{thm57} 
   Let $\mathcal{C}^{23}_7=\langle (x+1)^{\alpha}+u z_1(x)+u^2z_2(x), u(x+1)^{\beta}+u^2 (x+1)^{\mathfrak{T}_3}  z_3(x) \rangle$,
  where $1< \mathcal{W} \leq \beta <  \mathcal{U} \leq \alpha \leq 2^\varsigma-1$, $0 < \beta $,  $0<  \mathfrak{T}_3 < \mathcal{W} $ and  $z_1(x)$, $z_2(x)$ and $z_3(x)$ are units in $\mathcal{S}$. Then
 \begin{equation*}
	d_L(\mathcal{C}^{23}_7)=
	\begin{cases}
            2&\text{if}\quad 1< \alpha \leq 2^{\varsigma-1} \quad\text{with}\quad 2\alpha \leq 2^{\varsigma-1}, \quad \alpha+\beta \leq 2^{\varsigma-1}\quad\text{and}\quad  2\alpha+\beta \leq 2^{\varsigma-1},\\
            3 &\text{if}\quad 1< \alpha \leq 2^{\varsigma-1} \quad\text{with}\quad \alpha=2^{\varsigma-1}\quad \text{and}\quad z_1(x)=z_2(x)=1 ,   \\  
            4&\text{if}\quad 1< \alpha \leq2^{\varsigma-1}\quad \text{either with}\quad2\alpha > 2^{\varsigma-1}\quad\text{or}\quad \alpha+\beta > 2^{\varsigma-1}\quad\text{or}\quad 2\alpha+\beta > 2^{\varsigma-1}\\
            &\qquad \text{and}\quad \text{either with} \quad z_1(x)\neq 1 \quad \text{or} \quad z_2(x)\neq 1 \quad \text{or} \quad  \alpha \neq 2^{\varsigma-1},\\
            % 3 &\text{if}\quad 1< \alpha \leq 2^{\varsigma-1} \quad z_1(x)=z_2(x)=1 \quad \text{and}\quad\alpha=2^{\varsigma-1},   \\   

		4  &\text{if}\quad 2^{\varsigma-1}+1\leq \alpha \leq 
            2^{\varsigma}-1 \quad \text{with} \quad 1< \mathcal{W} \leq \beta \leq 2^{\varsigma-1},\\
            4  &\text{if}\quad 2^{\varsigma-1}+1\leq \beta <\alpha \leq 
            2^{\varsigma}-1 \quad \text{with} \quad 1<\mathcal{W} \leq 2^{\varsigma-1},\\
            % 4  &\text{if}\quad 2^{\varsigma-1}+1\leq \mathcal{W} \leq  \beta <\alpha \leq 
            % 2^{\varsigma}-1 \quad \text{with} \quad   \beta \geq 2^{\varsigma-1}+\mathfrak{T}_3 ,\\
            % 2^{\gamma+1} &\text{if}\quad 2^\varsigma-2^{\varsigma-\gamma}+1 \leq \mathcal{W}\leq \beta<\alpha  \leq 2^\varsigma-2^{\varsigma-\gamma}+2^{\varsigma-\gamma-1}  ,\\
            % &\qquad\text{with}\quad  \alpha \leq 2^{\varsigma-1}-2^{\varsigma-\gamma-1}+2^{\varsigma-\gamma-2}, \quad \alpha+\beta \leq 2^{\varsigma}-2^{\varsigma-\gamma}+2^{\varsigma-\gamma-1},\\
            % &\qquad \text{and} \quad 2\alpha+\beta \leq 2^{\varsigma}-2^{\varsigma-\gamma}+2^{\varsigma-\gamma-1},
            % \text{where}\quad 1\leq \gamma \leq \varsigma-1.
            4 &\text{if}\quad 2^{\varsigma-1}+1\leq \mathcal{W}\leq \beta<\alpha \leq 2^{\varsigma}-1.
	\end{cases}
    \end{equation*}
\end{theorem}

\begin{proof}
    Let $\mathcal{B}=\{ \zeta_1,\zeta_2,\ldots, \zeta_m\}$ be a TOB of $\mathbb{F}_{2^m}$ over $\mathbb{F}_2.$
    \begin{enumerate}
        \item \textbf{Case 1:} Let $1< \mathcal{W}\leq\beta<\alpha \leq 2^{\varsigma-1}$.  By Theorem \ref{thm3}, $d_H(\langle (x+1)^{\mathcal{W}} \rangle)=2$. Thus, $2\leq d_L(\mathcal{C}^{23}_7)\leq 4$.
        \begin{enumerate}
            \item \textbf{Subcase i:} Let $2\alpha \leq 2^{\varsigma-1}$, $\alpha+\beta \leq 2^{\varsigma-1}$ and $2\alpha+\beta \leq 2^{\varsigma-1}$. We have $\chi (x)=\zeta_1(x^{2^{\varsigma-1}}+1)=\zeta_1(x+1)^{2^{\varsigma-1}}=\zeta_1\Big[(x+1)^{\alpha}+u z_1(x)+u^2z_2(x)\Big]\Big[(x+1)^{2^{\varsigma-1}-\alpha}+u(x+1)^{2^{\varsigma-1}-2\alpha}z_1(x)\Big] +\Big[u(x+1)^{\beta}+u^2(x+1)^{\mathfrak{T}_3}z_3(x)\Big]\Big[u\Big((x+1)^{2^{\varsigma-1}-\alpha-\beta}z_2(x)+(x+1)^{2^{\varsigma-1}-2\alpha-\beta}z_1(x)z_1(x)\Big)\Big]\in \mathcal{C}^{23}_7$. Since  $wt^{\mathcal{B}}_L(\chi (x))=2$, $d_L(\mathcal{C}^{23}_7)=2$.
            \item \textbf{Subcase ii:} Let $z_1(x)=z_2(x)=1$ and $\alpha=2^{\varsigma-1}$. Following as in Theorem \ref{thm9}, we can prove $\mathcal{C}^{23}_7$ has no codeword of Lee weights 2 as $\alpha+\beta > 2^{\varsigma-1}$. we have $\chi (x)=\zeta_1\big[x^{2^{\varsigma-1}}+1+u +u^2\big]=\zeta_1\big[(x+1)^{2^{\varsigma-1}}+u +u^2\big]\in \mathcal{C}^{23}_7$. Since $wt^{\mathcal{B}}_L(\chi (x))=3$, we have $d_L(\mathcal{C}^{23}_7)=3$.
            
            \item \textbf{Subcase iii:} Let either $2\alpha > 2^{\varsigma-1}$ or $\alpha+\beta >2^{\varsigma-1}$ or $2\alpha+\beta > 2^{\varsigma-1}$ and either $z_1(x)\neq 1$ or $z_2(x)\neq 1$ or  $\alpha \neq 2^{\varsigma-1}$. Following Theorem \ref{thm9}, we can prove that $\mathcal{C}^{23}_7$ has no codeword of Lee weights 2 and 3. Hence $d_L(\mathcal{C}^{23}_7)=4$.
        \end{enumerate}
        \item \textbf{Case 2:} Let  $ 2^{\varsigma-1}+1\leq \alpha \leq 2^\varsigma-1$. 
        \begin{enumerate}
            \item \textbf{Subcase i:} Let $1< \mathcal{W} \leq \beta \leq 2^{\varsigma-1}$. From Theorem \ref{thm3}, $d_H(\langle (x+1)^{\beta} \rangle)=2$. Thus, $2\leq d_L(\mathcal{C}^{23}_7)\leq 4$. Following as in Theorem \ref{thm8}, we get that there exists no codeword of Lee weight 2 or 3. Hence $d_L(\mathcal{C}^{23}_7)=4$.
            
            \item \textbf{Subcase ii:} Let $2^{\varsigma-1}+1\leq \beta\leq 2^{\varsigma}-1 $. 
            \begin{itemize}
                \item  Let $1< \mathcal{W} \leq  2^{\varsigma-1}$. By Theorem \ref{thm3}, $d_H(\langle (x+1)^{\mathcal{W}}\rangle)=2$, Thus, $2\leq d_L(\mathcal{C}^{23}_7)\leq 4$. Following Theorem \ref{thm9}, we can prove that $\mathcal{C}^{23}_7$ has no codeword of Lee weights 2 and 3.
                Hence $d_L(\mathcal{C}^{23}_7)=4$.

            \item 
            Let $2^{\varsigma-1}+1 \leq \mathcal{W}\leq \beta  \leq 2^\varsigma-1$. By Theorem \ref{thm3},  $d_H(\langle (x+1)^{\mathcal{W}} \rangle)\geq 4$. Then $d_L(\mathcal{C}^{23}_7)\geq 4$.
            Also by Theorem \ref{thm20}, $d_L(\langle (x+1)^{\alpha}+uz_1(x)+u^2 z_2(x)\rangle)=4 $. Thus, $d_L(\mathcal{C}^{23}_7)=4$.

                \end{itemize}
        \end{enumerate}
    \end{enumerate}
\end{proof}

%\newpage
\subsection{If $z_1(x)\neq0$, $\mathfrak{T}_1\neq0$ $z_2(x)\neq0$, $\mathfrak{T}_2\neq0$ and $z_3(x)\neq0$, $\mathfrak{T}_3=0$  }
\begin{theorem}\label{thm58} 
   Let $\mathcal{C}^{24}_7=\langle (x+1)^{\alpha}+u(x+1)^{\mathfrak{T}_1} z_1(x)+u^2(x+1)^{\mathfrak{T}_2}z_2(x), u(x+1)^{\beta}+u^2   z_3(x) \rangle$,
  where $1< \mathcal{W} \leq \beta <  \mathcal{U} \leq \alpha \leq 2^\varsigma-1$, $0<  \mathfrak{T}_1 < \beta $, $0<  \mathfrak{T}_2 < \mathcal{W} $ and  $z_1(x)$, $z_2(x)$ and $z_3(x)$ are units in $\mathcal{S}$. Then
\begin{center}
	$d_L(\mathcal{C}^{24}_7)=$
		$\begin{cases}
			2  &\text{if}\quad  2^{\varsigma-1}\geq 2\alpha,\quad \alpha+\beta\leq 2^{\varsigma-1} \quad \text{and}\quad 2\alpha+\beta\leq 2^{\varsigma-1}, \\
                 4& \text{otherwise}.
			\end{cases}$
		\end{center}
\end{theorem}
\begin{proof}
    Let $\mathcal{B}=\{ \zeta_1,\zeta_2,\ldots, \zeta_m\}$ be a TOB of $\mathbb{F}_{2^m}$ over $\mathbb{F}_2.$ Let $\mathcal{W}$ be the smallest integer such that $u^2(x+1)^{\mathcal{W}} \in \mathcal{C}^{24}_7$. By Theorem \ref{thm1}, $\mathcal{W}=min\{\beta, 2^\varsigma-\beta\}$. Then $1<\mathcal{W} \leq 2^{\varsigma-1}$. By Theorem \ref{thm34} and Theorem \ref{thm3}, $d_H(\mathcal{C}^{24}_7)=2$. Thus, $2\leq d_L(\mathcal{C}^{24}_7)\leq 4$.
    \begin{enumerate}
        \item \textbf{\textbf{Case 1:} }If $2^{\varsigma-1}+\mathfrak{T}_1\geq 2\alpha$, $\alpha+\beta\leq 2^{\varsigma-1}+\mathfrak{T}_2$ and $2\alpha+\beta\leq 2^{\varsigma-1}+\mathfrak{T}_1$ we have $\chi (x)=\zeta_1(x+1)^{2^{\varsigma-1}}=\zeta_1\Big[(x+1)^{\alpha}+u (x+1)^{\mathfrak{T}_1}z_1(x)+u^2(x+1)^{\mathfrak{T}_2}z_2(x)\Big]\Big[(x+1)^{2^{\varsigma-1}-\alpha}+u(x+1)^{2^{\varsigma-1}-2\alpha+\mathfrak{T}_1}z_1(x)\Big] +\Big[ u(x+1)^{\beta}+u^2   z_3(x)\Big]\Big[u\Big( (x+1)^{2^{\varsigma-1}-\alpha-\beta+\mathfrak{T}_2}z_2(x)+(x+1)^{2^{\varsigma-1}-2\alpha-\beta+2\mathfrak{T}_1}z_1(x)z_1(x)\Big)\Big]\in \mathcal{C}^{24}_7$. Since $wt^{\mathcal{B}}_L(\chi (x))=2$, we have $d_L(\mathcal{C}^{24}_7)=2$.
        
        \item \textbf{\textbf{Case 3:}} Let either $2^{\varsigma-1}+\mathfrak{T}_1< 2\alpha$ or $\alpha+\beta> 2^{\varsigma-1}+\mathfrak{T}_2$ or $2\alpha+\beta> 2^{\varsigma-1}+\mathfrak{T}_1$. Following the same steps as in Theorem \ref{thm9}, we get $(1+ x)^{2^{\varsigma-1}}\in \mathcal{C}^{24}_7$. Then
        \begin{align*}
            (1+ x)^{2^{\varsigma-1}}=&\Big[(x+1)^{\alpha}+u(x+1)^{\mathfrak{T}_1} z_1(x)+u^2(x+1)^{\mathfrak{T}_2}z_2(x)\Big]\Big[\varphi_1(x)+u\varphi_2(x)+u^2\varphi_3(x)\Big]\\
            &+\Big[ u(x+1)^{\beta}+u^2   z_3(x)\Big]\Big[\varkappa_1(x)+u\varkappa_2(x)\Big]\\
            =&(x+1)^{\alpha} \varphi_1(x)+u\Big[(x+1)^{\mathfrak{T}_1}\varphi_1(x)z_1(x)+(x+1)^{\alpha} \varphi_2(x)+(x+1)^{\beta} \varkappa_1(x)\Big]\\
            &+u^2\Big[(x+1)^{\mathfrak{T}_2}\varphi_1(x)z_2(x)+(x+1)^{\mathfrak{T}_1}\varphi_2(x)z_1(x)+(x+1)^{\alpha} \varphi_3(x)\\
            &+(x+1)^{\beta} \varkappa_2(x)+\varkappa_1(x)z_3(x)\Big]
        \end{align*}
        for some $\varphi_1(x),\varphi_2(x), \varphi_3(x),\varkappa_1(x),\varkappa_2(x) \in \frac{\mathbb{F}_{p^m}[x]}{\langle x^{2^\varsigma}-1 \rangle}$. Then $\varphi_1(x)=(x+1)^{2^{\varsigma-1}-\alpha}$, $\varphi_2(x)=(x+1)^{2^{\varsigma-1}-2\alpha+\mathfrak{T}_1}z_1(x)+(x+1)^{\beta-\alpha}\varkappa_1(x)$ and $\varkappa_2(x)=(x+1)^{2^{\varsigma-1}-\alpha-\beta+\mathfrak{T}_2}z_2(x)+(x+1)^{2^{\varsigma-1}-2\alpha-\beta+2\mathfrak{T}_1}z_1(x)z_1(x)+(x+1)^{\mathfrak{T}_1-\alpha}\varkappa_1(x)z_1(x)+(x+1)^{\alpha-\beta}\varphi_3(x)+(x+1)^{\beta}\varkappa_1(x)z_3(x)$. Since  $2^{\varsigma-1}< 2\alpha$ or $\alpha+\beta> 2^{\varsigma-1}$ or $2\alpha+\beta> 2^{\varsigma-1}$, we obtain a contradiction. Thus, there exists no codeword of Lee weight 2. Also, following Theorem \ref{thm8}, $\mathcal{C}^{24}_7$ has no codeword of Lee weight 3. Hence $d_L(\mathcal{C}^{24}_7)=4$.
    \end{enumerate}
\end{proof}

%\newpage
\subsection{If $z_1(x)\neq0$, $\mathfrak{T}_1=0$ $z_2(x)\neq0$, $\mathfrak{T}_2\neq0$ and $z_3(x)\neq0$, $\mathfrak{T}_3\neq0$ }
\begin{theorem}\label{thm59} 
   Let $\mathcal{C}^{25}_7=\langle (x+1)^{\alpha}+u z_1(x)+u^2(x+1)^{\mathfrak{T}_2}z_2(x), u(x+1)^{\beta}+u^2 (x+1)^{\mathfrak{T}_3}  z_3(x) \rangle$,
  where $1< \mathcal{W} \leq \beta <  \mathcal{U} \leq \alpha \leq 2^\varsigma-1$, $0 < \beta $, $0< \mathfrak{T}_2 < \mathcal{W} $, $0<  \mathfrak{T}_3 < \mathcal{W} $ and  $z_1(x)$, $z_2(x)$ and $z_3(x)$ are units in $\mathcal{S}$. Then
\begin{equation*}
	d_L(\mathcal{C}^{25}_7)=
	\begin{cases}
            2&\text{if}\quad 1< \alpha \leq 2^{\varsigma-1} \quad\text{with}\quad 2\alpha \leq 2^{\varsigma-1}, \quad \alpha+\beta \leq 2^{\varsigma-1}+\mathfrak{T}_2\quad\text{and}\quad  2\alpha+\beta \leq 2^{\varsigma-1},\\
            4&\text{if}\quad 1< \alpha \leq 2^{\varsigma-1}\quad \text{either with}\quad2\alpha > 2^{\varsigma-1}\quad\text{or}\quad \alpha+\beta > 2^{\varsigma-1}+\mathfrak{T}_2\quad\text{or}\quad 2\alpha+\beta > 2^{\varsigma-1},\\
            
		4  &\text{if}\quad 2^{\varsigma-1}+1\leq \alpha \leq 
            2^{\varsigma}-1 \quad \text{with} \quad 1< \mathcal{W} \leq \beta \leq 2^{\varsigma-1},\\
            4  &\text{if}\quad 2^{\varsigma-1}+1\leq \beta <\alpha \leq 
            2^{\varsigma}-1 \quad \text{with} \quad 1< \mathcal{W} \leq 2^{\varsigma-1},\\
            4  &\text{if}\quad 2^{\varsigma-1}+1\leq \mathcal{W} \leq  \beta <\alpha \leq 
            2^{\varsigma}-1.
            % 4  &\text{if}\quad 2^{\varsigma-1}+1\leq \mathcal{W} \leq  \beta <\alpha \leq 
            % 2^{\varsigma}-1 \quad \text{with} \quad   \beta \geq 2^{\varsigma-1}+\mathfrak{T}_3 ,\\
            % 2^{\gamma+1} &\text{if}\quad 2^\varsigma-2^{\varsigma-\gamma}+1 \leq \mathcal{W}\leq \beta<\alpha  \leq 2^\varsigma-2^{\varsigma-\gamma}+2^{\varsigma-\gamma-1}  ,\\
            % &\qquad\text{with}\quad  \alpha \leq 2^{\varsigma-1}-2^{\varsigma-\gamma-1}+2^{\varsigma-\gamma-2}, \quad \alpha+\beta \leq 2^{\varsigma}-2^{\varsigma-\gamma}+2^{\varsigma-\gamma-1}+\mathfrak{T}_2,\\
            % &\qquad \text{and} \quad 2\alpha+\beta \leq 2^{\varsigma}-2^{\varsigma-\gamma}+2^{\varsigma-\gamma-1},
            % \text{where}\quad 1\leq \gamma \leq \varsigma-1.
	\end{cases}
    \end{equation*}
\end{theorem}

\begin{proof}
    Let $\mathcal{B}=\{ \zeta_1,\zeta_2,\ldots, \zeta_m\}$ be a TOB of $\mathbb{F}_{2^m}$ over $\mathbb{F}_2.$
    \begin{enumerate}
        \item \textbf{Case 1:} Let $1< \mathcal{W}\leq\beta<\alpha \leq 2^{\varsigma-1}$.  By Theorem \ref{thm3}, $d_H(\langle (x+1)^{\mathcal{W}} \rangle)=2$. Thus, $2\leq d_L(\mathcal{C}^{25}_7)\leq 4$.
        \begin{enumerate}
            \item \textbf{Subcase i:} Let $2\alpha \leq 2^{\varsigma-1}$, $\alpha+\beta \leq 2^{\varsigma-1}+\mathfrak{T}_2$ and $2\alpha+\beta \leq 2^{\varsigma-1}$. We have $\chi (x)=\zeta_1(x^{2^{\varsigma-1}}+1)=\zeta_1(x+1)^{2^{\varsigma-1}}=\zeta_1\Big[(x+1)^{\alpha}+u z_1(x)+u^2(x+1)^{\mathfrak{T}_2} z_2(x)\Big]\Big[(x+1)^{2^{\varsigma-1}-\alpha}+u(x+1)^{2^{\varsigma-1}-2\alpha}z_1(x)\Big] +\Big[u(x+1)^{\beta}+u^2(x+1)^{\mathfrak{T}_3}z_3(x)\Big]\Big[u\Big((x+1)^{2^{\varsigma-1}-\alpha-\beta+\mathfrak{T}_2}z_2(x)+(x+1)^{2^{\varsigma-1}-2\alpha-\beta}z_1(x)z_1(x)\Big)\Big]\in \mathcal{C}^{25}_7$. Since  $wt^{\mathcal{B}}_L(\chi (x))=2$, $d_L(\mathcal{C}^{25}_7)=2$.
        
            \item \textbf{Subcase ii:} Let either $2\alpha > 2^{\varsigma-1}$ or $\alpha+\beta >2^{\varsigma-1}+\mathfrak{T}_2$ or $2\alpha+\beta > 2^{\varsigma-1}$ and either $z_1(x)\neq 1$ or $z_2(x)\neq 1$ or  $\alpha \neq 2^{\varsigma-1}$. Following Theorem \ref{thm9}, we can prove that $\mathcal{C}^{25}_7$ has no codeword of Lee weights 2 and 3. Hence $d_L(\mathcal{C}^{25}_7)=4$.
         \end{enumerate}
        \item \textbf{Case 2:} Let  $ 2^{\varsigma-1}+1\leq \alpha \leq 2^\varsigma-1$. 
        \begin{enumerate}
            \item \textbf{Subcase i:} Let $1< \mathcal{W} \leq \beta \leq 2^{\varsigma-1}$. From Theorem \ref{thm3}, $d_H(\langle (x+1)^{\mathcal{W}} \rangle)=2$. Thus, $2\leq d_L(\mathcal{C}^{25}_7)\leq 4$. Following Theorem \ref{thm8}, we get that there exists no codeword of Lee weight 2 or 3. Hence $d_L(\mathcal{C}^{25}_7)=4$.
            
            \item \textbf{Subcase ii:} Let $2^{\varsigma-1}+1\leq \beta\leq 2^{\varsigma}-1 $. 
            \begin{itemize}
                \item  Let $1< \mathcal{W} \leq  2^{\varsigma-1}$. By Theorem \ref{thm3}, $d_H(\langle (x+1)^{\mathcal{W}}\rangle)=2$, Thus, $2\leq d_L(\mathcal{C}^{25}_7)\leq 4$. 
                Following Theorem \ref{thm9}, we can prove that $\mathcal{C}^{25}_7$ has no codeword of Lee weights 2 and 3. Hence $d_L(\mathcal{C}^{25}_7)=4$.
                
                \item  Let $2^{\varsigma-1}+1 \leq \mathcal{W}\leq \beta  \leq 2^\varsigma-1$. By Theorem \ref{thm3},  $d_H(\langle (x+1)^{\mathcal{W}} \rangle)\geq 4$. Then $d_L(\mathcal{C}^{25}_7)\geq 4$.
                Also by Theorem \ref{thm22}, $d_L(\langle (x+1)^{\alpha}+uz_1(x)+u^2(x+1)^{\mathfrak{T}_2}z_2(x)\rangle)=4 $. Thus, $d_L(\mathcal{C}^{25}_7)=4$.

                \end{itemize}
        \end{enumerate}
    \end{enumerate}
\end{proof}

%\newpage
\subsection{If $z_1(x)\neq0$, $\mathfrak{T}_1\neq0$ $z_2(x)\neq0$, $\mathfrak{T}_2=0$ and $z_3(x)\neq0$, $\mathfrak{T}_3\neq0$ }
\begin{theorem}\label{thm60} 
   Let $\mathcal{C}^{26}_7=\langle (x+1)^{\alpha}+u(x+1)^{\mathfrak{T}_1} z_1(x)+u^2z_2(x), u(x+1)^{\beta}+u^2 (x+1)^{\mathfrak{T}_3}  z_3(x) \rangle$,
  where $1< \mathcal{W} \leq \beta <  \mathcal{U} \leq \alpha \leq 2^\varsigma-1$, $0<  \mathfrak{T}_1 < \beta $, $0<  \mathfrak{T}_3 < \mathcal{W} $ and  $z_1(x)$, $z_2(x)$ and $z_3(x)$ are units in $\mathcal{S}$. Then
\begin{equation*}
	d_L(\mathcal{C}^{26}_7)=
	\begin{cases}
            2&\text{if}\quad 1< \alpha \leq 2^{\varsigma-1} \quad\text{with}\quad 2\alpha \leq 2^{\varsigma-1}+\mathfrak{T}_1, \quad \alpha+\beta \leq 2^{\varsigma-1},\\
            &\quad\text{and}\quad  2\alpha+\beta \leq 2^{\varsigma-1}+\mathfrak{T}_1,\\
            4&\text{if}\quad 1< \alpha \leq 2^{\varsigma-1}\quad \text{either with}\quad2\alpha > 2^{\varsigma-1}+\mathfrak{T}_1\quad\text{or}\quad \alpha+\beta > 2^{\varsigma-1}\\&\quad\text{or}\quad 2\alpha+\beta > 2^{\varsigma-1}+\mathfrak{T}_1,\\
           
		4  &\text{if}\quad 2^{\varsigma-1}+1\leq \alpha \leq 
            2^{\varsigma}-1 \quad \text{with} \quad 1<\mathcal{W} \leq \beta \leq 2^{\varsigma-1},\\
            4  &\text{if}\quad 2^{\varsigma-1}+1\leq \beta <\alpha \leq 
            2^{\varsigma}-1 \quad \text{with} \quad 1< \mathcal{W} \leq 2^{\varsigma-1},\\
            4  &\text{if}\quad 2^{\varsigma-1}+1\leq \mathcal{W} \leq  \beta <\alpha \leq 
            2^{\varsigma}-1 \quad \text{with} \quad   \beta \geq 2^{\varsigma-1}+\mathfrak{T}_3 ,\\
            2^{\gamma+1} &\text{if}\quad 2^\varsigma-2^{\varsigma-\gamma}+1 \leq \mathcal{W}\leq \beta<\alpha  \leq 2^\varsigma-2^{\varsigma-\gamma}+2^{\varsigma-\gamma-1}  ,\\
            &\quad   \text{with} \quad 3\alpha \leq 2^{\varsigma}-2^{\varsigma-\gamma}+2^{\varsigma-\gamma-1}+2\mathfrak{T}_1 \quad\text{and}\\
            &\quad \alpha \leq 2^{\varsigma-1}+\frac{\mathfrak{T}_1}{2}
            \quad \text{where}\quad 1\leq \gamma \leq \varsigma-1.
	\end{cases}
    \end{equation*}
\end{theorem}

\begin{proof}
    Let $\mathcal{B}=\{ \zeta_1,\zeta_2,\ldots, \zeta_m\}$ be a TOB of $\mathbb{F}_{2^m}$ over $\mathbb{F}_2.$
    \begin{enumerate}
        \item \textbf{Case 1:} Let $1< \mathcal{W}\leq\beta<\alpha \leq 2^{\varsigma-1}$.  By Theorem \ref{thm3}, $d_H(\langle (x+1)^{\mathcal{W}} \rangle)=2$. Thus, $2\leq d_L(\mathcal{C}^{26}_7)\leq 4$.
        \begin{enumerate}
            \item \textbf{Subcase i:} Let $2\alpha \leq 2^{\varsigma-1}+\mathfrak{T}_1$, $\alpha+\beta \leq 2^{\varsigma-1}$ and $2\alpha+\beta \leq 2^{\varsigma-1}+\mathfrak{T}_1$. We have $\chi (x)=\zeta_1(x^{2^{\varsigma-1}}+1)=\zeta_1(x+1)^{2^{\varsigma-1}}=\zeta_1\Big[(x+1)^{\alpha}+u (x+1)^{\mathfrak{T}_1}z_1(x)+u^2z_2(x)\Big]\Big[(x+1)^{2^{\varsigma-1}-\alpha}+u(x+1)^{2^{\varsigma-1}-2\alpha+\mathfrak{T}_1}z_1(x)\Big] +\Big[u(x+1)^{\beta}+u^2(x+1)^{\mathfrak{T}_3}z_3(x)\Big]\Big[u\Big((x+1)^{2^{\varsigma-1}-\alpha-\beta}z_2(x)+(x+1)^{2^{\varsigma-1}-2\alpha-\beta+2\mathfrak{T}_1}z_1(x)z_1(x)\Big)\Big]\in \mathcal{C}^{26}_7$. Since  $wt^{\mathcal{B}}_L(\chi (x))=2$, $d_L(\mathcal{C}^{26}_7)=2$.

            \item \textbf{Subcase iii:} Let either $2\alpha > 2^{\varsigma-1}+\mathfrak{T}_1$ or $\alpha+\beta >2^{\varsigma-1}$ or $2\alpha+\beta > 2^{\varsigma-1}+\mathfrak{T}_1$   

 Following Theorem \ref{thm9}, we can prove that $\mathcal{C}^{26}_7$ has no codeword of Lee weights 2 and 3. Hence $d_L(\mathcal{C}^{26}_7)=4$.
            
        \end{enumerate}
        \item \textbf{Case 2:} Let  $ 2^{\varsigma-1}+1\leq \alpha \leq 2^\varsigma-1$. 
        \begin{enumerate}
            \item \textbf{Subcase i:} Let $1< \mathcal{W} \leq \beta \leq 2^{\varsigma-1}$. From Theorem \ref{thm3}, $d_H(\langle (x+1)^{\mathcal{W}} \rangle)=2$. Thus, $2\leq d_L(\mathcal{C}^{26}_7)\leq 4$. Following Theorem \ref{thm8}, we get that there exists no codeword of Lee weight 2 or 3. Hence $d_L(\mathcal{C}^{26}_7)=4$.
            
             \item \textbf{Subcase ii:} Let $2^{\varsigma-1}+1\leq \beta\leq 2^{\varsigma}-1 $. 
            \begin{itemize}
                \item  Let $1< \mathcal{W} \leq  2^{\varsigma-1}$. By Theorem \ref{thm3}, $d_H(\langle (x+1)^{\mathcal{W}}\rangle)=2$, Thus, $2\leq d_L(\mathcal{C}^{26}_7)\leq 4$. Following Theorem \ref{thm9}, we can prove that $\mathcal{C}^{26}_7$ has no codeword of Lee weights 2 and 3. Hence $d_L(\mathcal{C}^{26}_7)=4$.
                    
                \item  Let $2^{\varsigma-1}+1\leq \mathcal{W} \leq 2^{\varsigma}-1 $
                By Theorem \ref{thm3},  $d_H(\langle (x+1)^{\mathcal{W}} \rangle)\geq 4$ and by Theorem \ref{thm8},  $d_L(\langle u(x+1)^{\beta}+u^2 (x+1)^{\mathfrak{T}_3}  z_3(x) \rangle)=4$ if $ \beta \geq 2^{\varsigma-1}+\mathfrak{T}_3$.  Thus, $d_L(\mathcal{C}^{26}_7)=4$.

            \item 
            Let $2^\varsigma-2^{\varsigma-\gamma}+1 \leq \mathcal{W}\leq \beta  \leq 2^\varsigma-2^{\varsigma-\gamma}+2^{\varsigma-\gamma-1}$, where $ 1\leq \gamma \leq \varsigma-1$.  From Theorem \ref{thm3}, $d_H(\langle (x+1)^{\mathcal{W}}\rangle)=2^{\gamma+1}$.  Then $d_L(\mathcal{C}^{26}_7)\geq 2^{\gamma+1}$. From Theorem \ref{thm19}, $d_L(\langle (x+1)^{\alpha}+u(x+1)^{\mathfrak{T}_1} z_1(x)+u^2z_2(x)\rangle )=2^{\gamma+1}$. Then $d_L(\mathcal{C}^{26}_7)\leq 2^{\gamma+1}$ if $3\alpha \leq 2^{\varsigma}-2^{\varsigma-\gamma}+2^{\varsigma-\gamma-1}+2\mathfrak{T}_1 $ and $\alpha \leq 2^{\varsigma-1}+\frac{\mathfrak{T}_1}{2}$. Hence $d_L(\mathcal{C}^{26}_7)=2^{\gamma+1}$.
             
                \end{itemize}
        \end{enumerate}
    \end{enumerate}
\end{proof}

%\newpage
\subsection{If $z_1(x)\neq0$, $\mathfrak{T}_1\neq0$ $z_2(x)\neq0$, $\mathfrak{T}_2\neq0$ and $z_3(x)\neq0$, $\mathfrak{T}_3\neq0$  }
\begin{theorem}\label{thm61} 
   Let $\mathcal{C}^{27}_7=\langle (x+1)^{\alpha}+u(x+1)^{\mathfrak{T}_1} z_1(x)+u^2(x+1)^{\mathfrak{T}_2}z_2(x), u(x+1)^{\beta}+u^2 (x+1)^{\mathfrak{T}_3}  z_3(x) \rangle$,
  where $1< \mathcal{W} \leq \beta <  \mathcal{U} \leq \alpha \leq 2^\varsigma-1$, $0<  \mathfrak{T}_1 < \beta $, $0<  \mathfrak{T}_2 < \mathcal{W} $, $0<  \mathfrak{T}_3 < \mathcal{W} $ and  $z_1(x)$, $z_2(x)$ and $z_3(x)$ are units in $\mathcal{S}$. Then
 \begin{equation*}
	d_L(\mathcal{C}^{27}_7)=
	\begin{cases}
            2&\text{if}\quad 1< \alpha \leq 2^{\varsigma-1} \quad\text{with}\quad 2\alpha \leq 2^{\varsigma-1}+\mathfrak{T}_1, \quad \alpha+\beta \leq 2^{\varsigma-1}+\mathfrak{T}_2,\\
            &\qquad\text{and}\quad  2\alpha+\beta \leq 2^{\varsigma-1}+\mathfrak{T}_1,\\
            4&\text{if}\quad 1< \alpha \leq 2^{\varsigma-1}\quad \text{either with}\quad2\alpha > 2^{\varsigma-1}+\mathfrak{T}_1\quad\text{or}\quad \alpha+\beta > 2^{\varsigma-1}+\mathfrak{T}_2\\
            &\qquad\text{or}\quad 2\alpha+\beta > 2^{\varsigma-1}+\mathfrak{T}_1,\\
		4  &\text{if}\quad 2^{\varsigma-1}+1\leq \alpha \leq 
            2^{\varsigma}-1 \quad \text{with} \quad 1<\mathcal{W} \leq \beta \leq 2^{\varsigma-1},\\
            4  &\text{if}\quad 2^{\varsigma-1}+1\leq \beta <\alpha \leq 
            2^{\varsigma}-1 \quad \text{with} \quad 1< \mathcal{W} \leq 2^{\varsigma-1},\\
            4  &\text{if}\quad 2^{\varsigma-1}+1\leq \mathcal{W} \leq  \beta <\alpha \leq 
            2^{\varsigma}-1 \quad \text{with} \quad   \beta \geq 2^{\varsigma-1}+\mathfrak{T}_3 ,\\
            2^{\gamma+1} &\text{if}\quad 2^\varsigma-2^{\varsigma-\gamma}+1 \leq \mathcal{W}\leq \beta<\alpha  \leq 2^\varsigma-2^{\varsigma-\gamma}+2^{\varsigma-\gamma-1} \\
            % &\qquad\text{with}\quad  \alpha \leq 2^{\varsigma-1}-2^{\varsigma-\gamma-1}+2^{\varsigma-\gamma-2}+\mathfrak{T}_1, \quad \alpha+\beta \leq 2^{\varsigma}-2^{\varsigma-\gamma}+2^{\varsigma-\gamma-1}+\mathfrak{T}_2,\\
            % &\qquad \text{and} \quad 2\alpha+\beta \leq 2^{\varsigma}-2^{\varsigma-\gamma}+2^{\varsigma-\gamma-1}+\mathfrak{T}_1,
            % \text{where}\quad 1\leq \gamma \leq \varsigma-1.
            &\qquad   \text{with} \quad 3\alpha \leq 2^{\varsigma}-2^{\varsigma-\gamma}+2^{\varsigma-\gamma-1}+2\mathfrak{T}_1, \quad \alpha \leq 2^{\varsigma-1}-2^{\varsigma-\gamma-1}+2^{\varsigma-\gamma-2}+\frac{\mathfrak{T}_2}{2}\\
            &\qquad\text{and}\quad \alpha \leq 2^{\varsigma-1}+\frac{\mathfrak{T}_1}{2}
            \quad \text{where}\quad 1\leq \gamma \leq \varsigma-1.
	\end{cases}
    \end{equation*}
\end{theorem}

\begin{proof}
    Let $\mathcal{B}=\{ \zeta_1,\zeta_2,\ldots, \zeta_m\}$ be a TOB of $\mathbb{F}_{2^m}$ over $\mathbb{F}_2.$
    \begin{enumerate}
        \item \textbf{Case 1:} Let $1<\mathcal{W}\leq\beta<\alpha \leq 2^{\varsigma-1}$.  By Theorem \ref{thm3}, $d_H(\langle (x+1)^{\mathcal{W}} \rangle)=2$. Thus, $2\leq d_L(\mathcal{C}^{27}_7)\leq 4$.
        \begin{enumerate}
            \item \textbf{Subcase i:} Let $2\alpha \leq 2^{\varsigma-1}+\mathfrak{T}_1$, $\alpha+\beta \leq 2^{\varsigma-1}+\mathfrak{T}_2$ and $2\alpha+\beta \leq 2^{\varsigma-1}+\mathfrak{T}_1$. We have $\chi (x)=\zeta_1(x^{2^{\varsigma-1}}+1)=\zeta_1(x+1)^{2^{\varsigma-1}}=\zeta_1\Big[(x+1)^{\alpha}+u(x+1)^{\mathfrak{T}_1} z_1(x)+u^2(x+1)^{\mathfrak{T}_2}z_2(x)\Big]\Big[(x+1)^{2^{\varsigma-1}-\alpha}+u(x+1)^{2^{\varsigma-1}-2\alpha+\mathfrak{T}_1}z_1(x)\Big] +\Big[u(x+1)^{\beta}+u^2(x+1)^{\mathfrak{T}_3}z_3(x)\Big]\Big[u\Big((x+1)^{2^{\varsigma-1}-\alpha-\beta+\mathfrak{T}_2}z_2(x)+(x+1)^{2^{\varsigma-1}-2\alpha-\beta+2\mathfrak{T}_1}z_1(x)z_1(x)\Big)\Big]\in \mathcal{C}^{27}_7$. Since  $wt^{\mathcal{B}}_L(\chi (x))=2$, $d_L(\mathcal{C}^{27}_7)=2$.

            \item \textbf{Subcase iii:} Let either $2\alpha > 2^{\varsigma-1}+\mathfrak{T}_1$ or $\alpha+\beta >2^{\varsigma-1}+\mathfrak{T}_2$ or $2\alpha+\beta > 2^{\varsigma-1}+\mathfrak{T}_1$. Following Theorem \ref{thm9}, we can prove that $\mathcal{C}^{27}_7$ has no codeword of Lee weights 2 and 3. Hence $d_L(\mathcal{C}^{27}_7)=4$.
            
        \end{enumerate}
        \item \textbf{Case 2:} Let  $ 2^{\varsigma-1}+1\leq \alpha \leq 2^\varsigma-1$. 
        \begin{enumerate}
            \item \textbf{Subcase i:} Let $1< \mathcal{W} \leq \beta \leq 2^{\varsigma-1}$. From Theorem \ref{thm3}, $d_H(\langle (x+1)^{\mathcal{W}} \rangle)=2$. Thus, $2\leq d_L(\mathcal{C}^{27}_7)\leq 4$. Following Theorem \ref{thm8}, we get that there exists no codeword of Lee weight 2 or 3. Hence $d_L(\mathcal{C}^{27}_7)=4$.
            
            \item \textbf{Subcase ii:} Let $2^{\varsigma-1}+1\leq \beta\leq 2^{\varsigma}-1 $. 
            \begin{itemize}
                \item  Let $1<\mathcal{W} \leq  2^{\varsigma-1}$. By Theorem \ref{thm3}, $d_H(\langle (x+1)^{\mathcal{W}}\rangle)=2$, Thus, $2\leq d_L(\mathcal{C}^{27}_7)\leq 4$. 
                Following Theorem \ref{thm9}, we can prove that $\mathcal{C}^{27}_7$ has no codeword of Lee weights 2 and 3. Hence $d_L(\mathcal{C}^{27}_7)=4$.
                    
                \item  Let $2^{\varsigma-1}+1\leq \mathcal{W} \leq 2^{\varsigma}-1 $
                By Theorem \ref{thm3},  $d_H(\langle (x+1)^{\mathcal{W}} \rangle)\geq 4$ and by Theorem \ref{thm8},  $d_L(\langle u(x+1)^{\beta}+u^2 (x+1)^{\mathfrak{T}_3}  z_3(x) \rangle)=4$ if $ \beta \geq 2^{\varsigma-1}+\mathfrak{T}_3$.  Thus, $d_L(\mathcal{C}^{27}_7)=4$.

            \item 
            Let $2^\varsigma-2^{\varsigma-\gamma}+1 \leq \mathcal{W}\leq \beta  \leq 2^\varsigma-2^{\varsigma-\gamma}+2^{\varsigma-\gamma-1}$, where $ 1\leq \gamma \leq \varsigma-1$.  From Theorem \ref{thm3}, $d_H(\langle (x+1)^{\mathcal{W}}\rangle)=2^{\gamma+1}$.  Then $d_L(\mathcal{C}^{27}_7)\geq 2^{\gamma+1}$. From Theorem \ref{thm23}, $d_L(\langle (x+1)^{\alpha}+u(x+1)^{\mathfrak{T}_1} z_1(x)+u^2(x+1)^{\mathfrak{T}_2}z_2(x)\rangle )=2^{\gamma+1}$. Then $d_L(\mathcal{C}^{27}_7)\leq 2^{\gamma+1}$ if $3\alpha \leq 2^{\varsigma}-2^{\varsigma-\gamma}+2^{\varsigma-\gamma-1}+2\mathfrak{T}_1$, \quad $\alpha \leq 2^{\varsigma-1}-2^{\varsigma-\gamma-1}+2^{\varsigma-\gamma-2}+\frac{\mathfrak{T}_2}{2}$  and $\alpha \leq 2^{\varsigma-1}+\frac{\mathfrak{T}_1}{2}$. Hence $d_L(\mathcal{C}^{27}_7)=2^{\gamma+1}$.
        \end{itemize}
        \end{enumerate}
    \end{enumerate}
\end{proof}
%\newpage
\subsection{Type 8:}
\begin{theorem}\cite{dinh2021hamming}\label{thm62}Let $\mathcal{C}_8=\langle(x+1)^{\alpha}+u(x+1)^{\mathfrak{T}_1} z_1(x)+u^2(x+1)^{\mathfrak{T}_2}z_2(x), u(x+1)^{\beta}+u^2 (x+1)^{\mathfrak{T}_3}  z_3(x),u^2(x+1)^{\omega} \rangle $, where $0\leq \omega< \mathcal{W} \leq \mathcal{L}_1 \leq \beta <  \mathcal{U} \leq \alpha \leq 2^\varsigma-1$, $0\leq  \mathfrak{T}_1 < \beta $, $0\leq  \mathfrak{T}_2 < \omega $, $0\leq  \mathfrak{T}_3 < \omega $ and  $z_1(x)$, $z_2(x)$ and $z_3(x)$ are either 0 or a unit in $\mathcal{S}$. Then $d_H(\mathcal{C}_8)= d_H(\langle (x+1)^\omega \rangle)$.
\end{theorem}
%\newpage
\subsection{If $z_1(x)=0$, $z_2(x)=0$ and $z_3(x)=0$ }
\begin{theorem}
    $\mathcal{C}^1_8=\langle (x+1)^{\alpha}, u(x+1)^{\beta}, u^2(x+1)^{\omega} \rangle $, where $0\leq \omega<  \beta < \alpha \leq 2^\varsigma-1$. Then 
     \begin{equation*}
	d_L(\mathcal{C}^1_8)=
	\begin{cases}
            2&\text{if}\quad 1<\beta< \alpha<2^{\varsigma-1} \quad\text{with}\quad \omega=0,\\
		  2  &\text{if}\quad 2^{\varsigma-1}+1\leq \alpha \leq 
            2^{\varsigma}-1 \quad \text{with} \quad 1 < \beta \leq 2^{\varsigma-1}\quad \text{and} \quad \omega=0,\\
            4  &\text{if}\quad 2^{\varsigma-1}+1\leq \alpha \leq 
            2^{\varsigma}-1 \quad \text{with} \quad 1\leq \omega < \beta \leq 2^{\varsigma-1},\\
            2 &\text{if}\quad 2^{\varsigma-1}+1\leq \beta<\alpha \leq 
            2^{\varsigma}-1 \quad \text{with} \quad \omega=0,\\
            4  &\text{if}\quad 2^{\varsigma-1}+1\leq \beta<\alpha \leq 
            2^{\varsigma}-1 \quad \text{with} \quad 1\leq \omega \leq 2^{\varsigma-1},\\
            2^{\gamma+1} &\text{if}\quad 2^\varsigma-2^{\varsigma-\gamma}+1 \leq \omega<\beta < \alpha  \leq 2^\varsigma-2^{\varsigma-\gamma}+2^{\varsigma-\gamma-1},\\
            &\qquad \text{where}\quad 1\leq \gamma \leq \varsigma-1.\\
	\end{cases}
    \end{equation*}
    
\end{theorem}
\begin{proof}
    Let $\mathcal{B}=\{ \zeta_1,\zeta_2,\ldots, \zeta_m\}$ be a TOB of $\mathbb{F}_{2^m}$ over $\mathbb{F}_2.$ 
    \begin{enumerate}
        \item \textbf{Case 1:} Let $1\leq \alpha \leq 2^{\varsigma-1}$. From Theorem \ref{thm15}, $d_L(\mathcal{C}^1_8)\leq 2$.
         \begin{enumerate}
             \item If $\omega>0$, by Theorem \ref{thm3}, $d_H(\langle (x+1)^{\omega} \rangle)\geq 2$. Hence $d_L(\mathcal{C}^1_8)=2$.
             \item Let $\omega=0$. Suppose $\chi (x)=\lambda x^j \in \mathcal{C}^1_8$, $\lambda \in \mathcal{R}$ with $wt^{\mathcal{B}}_L(\chi (x))=1$
             \begin{enumerate}
                 \item if $\lambda$ is a unit in $\mathcal{R}$ then $\lambda x^j$ is a unit. This is not possible.
                 \item if $\lambda$ is non-unit in $\mathcal{R}$ then $\lambda \in \langle u \rangle$ and $wt^{\mathcal{B}}_L(\lambda)\geq 3 $. Again this is not possible.
             \end{enumerate}
             Hence $d_L(\mathcal{C}^1_8)=2$.
        \end{enumerate}
        \end{enumerate}
        
        \item \textbf{Case 2:} Let  $ 2^{\varsigma-1}+1\leq \alpha \leq 2^\varsigma-1$. 
        \begin{enumerate}
            \item \textbf{Subcase i:} Let $1 \leq \beta \leq 2^{\varsigma-1}$. 
            \begin{itemize}
                \item Let $\omega=0$. As in the above case, $\mathcal{C}^1_8$ has no codeword of Lee weights 1. we have $\chi (x)=\zeta_1u^2 \in \mathcal{C}^1_8$ with $wt^{\mathcal{B}}_L(\chi (x))=2$. Hence $d_L(\mathcal{C}^1_8)=2$.
                \item  Let $1\leq \omega \leq 2^{\varsigma-1}$. By Theorem \ref{thm3} and Theorem \ref{thm62}, $d_H(\mathcal{C}^1_8)=2$. Thus, $2\leq d_L(\mathcal{C}^1_8)$. Following Theorem \ref{thm9}, we can prove $\mathcal{C}^1_8$ has no codeword of Lee weights 2 and 3. A codeword $\wp(x)=u^2\zeta_1(x^{2^{\varsigma-1}}+1)=u^2\zeta_1(x+1)^{2^{\varsigma-1}}\in \langle u^2(x+1)^{\omega}\rangle \subseteq\mathcal{C}^1_8$ with $wt^{\mathcal{B}}_L(\wp(x))=4$. Thus, $d_L(\mathcal{C}^1_8)=4$.
            \end{itemize}
            
            \item \textbf{Subcase ii:} Let $2^{\varsigma-1}+1\leq \beta \leq 2^{\varsigma-1}-1$ 
            \begin{itemize}
                \item Let $\omega=0$. As in the above case, $\mathcal{C}^1_8$ has no codeword of Lee weights 1. we have $\chi (x)=\zeta_1u^2 \in \mathcal{C}^1_8$ with $wt^{\mathcal{B}}_L(\chi (x))=2$. Hence $d_L(\mathcal{C}^1_8)=2$.
                
                \item  Let $1\leq \omega \leq 2^{\varsigma-1}$. By Theorem \ref{thm3} and Theorem \ref{thm62}, $d_H(\mathcal{C}^1_8)=2$. Thus, $2\leq d_L(\mathcal{C}^1_8)$. Following Theorem \ref{thm9}, we can prove $\mathcal{C}^1_8$ has no codeword of Lee weights 2 and 3. A codeword $\wp(x)=u^2\zeta_1(x^{2^{\varsigma-1}}+1)=u^2\zeta_1(x+1)^{2^{\varsigma-1}}\in \langle u^2(x+1)^{\omega}\rangle \subseteq\mathcal{C}^1_8$ with $wt^{\mathcal{B}}_L(\wp(x))=4$. Thus, $d_L(\mathcal{C}^1_8)=4$.
                
                \item Let $2^\varsigma-2^{\varsigma-\gamma}+1 \leq \omega  \leq 2^\varsigma-2^{\varsigma-\gamma}+2^{\varsigma-\gamma-1}$, where $1\leq \gamma \leq \varsigma-1$. By Theorem \ref{thm3},  $d_H(\langle (x+1)^{\omega} \rangle)=2^{\gamma+1}$ and by Theorem \ref{thm15},  $d_L(\langle (x+1)^{\alpha} \rangle)=2^{\gamma+1}$. Thus, $d_L(\mathcal{C}^1_8)=2^{\gamma+1}$.
            \end{itemize}
    \end{enumerate}
\end{proof}

%\newpage
\subsection{If $z_1(x)\neq 0$, $\mathfrak{T}_1=0$ $z_2(x)=0$ and $z_3(x)=0$ }
\begin{theorem}
    $\mathcal{C}^2_8=\langle (x+1)^{\alpha}+u z_1(x), u(x+1)^{\beta},u^2(x+1)^{\omega} \rangle $, where $0\leq \omega <\beta <  \alpha \leq 2^\varsigma-1$, $0< \beta $ and  $z_1(x)$ is a unit in $\mathcal{S}$. Then
     \begin{equation*}
	d_L(\mathcal{C}^2_8)=
	\begin{cases}
            2&\text{if}\quad 1<\beta< \alpha<2^{\varsigma-1} \quad\text{with}\quad \omega=0,\\
            2&\text{if}\quad 1\leq\omega<\beta< \alpha < 2^{\varsigma-1} \quad\text{with}\quad\beta+\alpha \leq 2^{\varsigma-1},\\
            4&\text{if}\quad 1\leq\omega<\beta< \alpha < 2^{\varsigma-1} \quad\text{with}\quad\beta+\alpha >2^{\varsigma-1},\\
            2&\text{if}\quad \alpha=2^{\varsigma-1} \quad\text{and}\quad z_1(x)=1\quad\text{with}\quad\omega=0,\\
            4&\text{if}\quad \alpha=2^{\varsigma-1} \quad\text{and}\quad z_1(x)=1\quad\text{with}\quad\omega>0,\\
            
		  2  &\text{if}\quad 2^{\varsigma-1}+1\leq \alpha \leq 
            2^{\varsigma}-1 \quad \text{with} \quad 1 < \beta \leq 2^{\varsigma-1}\quad \text{and} \quad \omega=0,\\
            4  &\text{if}\quad 2^{\varsigma-1}+1\leq \alpha \leq 
            2^{\varsigma}-1 \quad \text{with} \quad 1\leq \omega < \beta \leq 2^{\varsigma-1},\\
            2  &\text{if}\quad 2^{\varsigma-1}+1\leq \beta<\alpha \leq 
            2^{\varsigma}-1 \quad \text{with} \quad \omega=0,\\
            4  &\text{if}\quad 2^{\varsigma-1}+1\leq \beta<\alpha \leq 
            2^{\varsigma}-1 \quad \text{with} \quad 1\leq \omega \leq 2^{\varsigma-1},\\
            4  &\text{if}\quad 2^{\varsigma-1}+1\leq \omega< \beta<\alpha \leq 
            2^{\varsigma}-1. \\
	\end{cases}
    \end{equation*}
    
\end{theorem}
\begin{proof}
    Let $\mathcal{B}=\{ \zeta_1,\zeta_2,\ldots, \zeta_m\}$ be a TOB of $\mathbb{F}_{2^m}$ over $\mathbb{F}_2.$
    \begin{enumerate}
    \item {\textbf{Case 1:}} Let $1\leq\beta< \alpha<2^{\varsigma-1}$. 
    % \hl{If $1\leq \alpha\leq2^{\varsigma-1}$ with $\alpha + \beta\leq 2^{\varsigma-1}$ and $z_1(x)\neq 1$}. 
     \begin{enumerate}
        \item \textbf{Subcase i:} Let $\omega=0$. 
        From Theorem \ref{thm3} and Theorem \ref{thm5}, $1\leq d_L(\mathcal{C}^2_8)\leq 2$. Suppose $\chi (x)=\lambda x^j \in \mathcal{C}^2_8$, $\lambda \in \mathcal{R}$ with $wt^{\mathcal{B}}_L(\chi (x))=1$. 
        If $\lambda$ is a unit in $\mathcal{R}$, then $\lambda x^j$ is a unit. This is not possible. If $\lambda$ is non-unit in $\mathcal{R}$ then $\lambda \in \langle u \rangle$ and $wt^{\mathcal{B}}_L(\lambda)\geq 3 $. Again, this is not possible.
        Hence $d_L(\mathcal{C}^2_8)=2$.
        
        \item \textbf{Subcase ii:} Let $1\leq \omega\leq2^{\varsigma-1}$. By Thoerem \ref{thm62} and Theorem \ref{thm3}, $d_H(\mathcal{C}^2_8)=2$. Hence $2\leq d_L(\mathcal{C}^2_8)$.
        \begin{itemize}
            \item Let $\alpha + \beta\leq2^{\varsigma-1}$.  We have $\chi (x)=\zeta_1(x^{2^{\varsigma-1}}+1)=\zeta_1(x+1)^{2^{\varsigma-1}}=\zeta_1[(x+1)^{\alpha}+uz_1(x)]+[(x+1)^{2^{\varsigma-1}-\alpha}]+u(x+1)^{\beta}[(x+1)^{2^{\varsigma-1}-\alpha-\beta}z_1(x)] \in \mathcal{C}^2_8$. Since  $wt^{\mathcal{B}}_L(\chi (x))=2$, $d_L(\mathcal{C}^2_8)=2$.
    
            \item  Let $\alpha + \beta>2^{\varsigma-1}$. Following Theorem \ref{thm9}, we can prove $\mathcal{C}^2_8$ has no codeword of Lee weights 2 and 3. A codeword $\wp(x)=u^2\zeta_1(x^{2^{\varsigma-1}}+1)=u^2\zeta_1(x+1)^{2^{\varsigma-1}}\in \langle u^2(x+1)^{\omega}\rangle \subseteq\mathcal{C}^2_8$ with $wt^{\mathcal{B}}_L(\wp(x))=4$. Thus, $d_L(\mathcal{C}^2_8)=4$.
         \end{itemize}
     \end{enumerate}
     
    \item {\textbf{Case 2:}} Let $z_1(x)=1$ and $\alpha=2^{\varsigma-1}$.
    \begin{enumerate}
        \item \textbf{Subcase i:} Let $\omega=0$. As in the above case, $\mathcal{C}^2_8$ has no codeword of Lee weights 1. we have $\chi (x)=\zeta_1u^2 \in \mathcal{C}^2_8$ with $wt^{\mathcal{B}}_L(\chi (x))=2$. Hence $d_L(\mathcal{C}^2_8)=2$.
        
        \item \textbf{Subcase ii:} Let $\omega >0$. By Theorem \ref{thm3} and Theorem \ref{thm62}, $d_H(\mathcal{C}^2_8)=2$. Thus, $2\leq d_L(\mathcal{C}^2_8)$. Since $\alpha=2^{\varsigma-1}$ and $\beta >0$, we have $\alpha + \beta>2^{\varsigma-1}$. From the above case, $\mathcal{C}^2_8$ has no codeword of Lee weights 2.  We have  $\chi (x)=\zeta_1((x+1)^{2^{\varsigma-1}}+u)=\zeta_1(x^{2^{\varsigma-1}}+1+u)\in \mathcal{C}^2_8$. Since $wt^{\mathcal{B}}_L(\chi (x))=3$, we have $d_L(\mathcal{C}^2_8)=3$.
    \end{enumerate}
    
    \item \textbf{Case 3:} Let $2^{\varsigma-1}+1\leq \alpha \leq 2^\varsigma-1$.
    \begin{enumerate}
        \item \textbf{Subcase i:} Let $1\leq \beta \leq 2^{\varsigma-1}$.
        \begin{itemize}
            \item Let $\omega=0$. As in the above case, $\mathcal{C}^2_8$ has no codeword of Lee weights 1. we have $\chi (x)=\zeta_1u^2 \in \mathcal{C}^2_8$ with $wt^{\mathcal{B}}_L(\chi (x))=2$. Hence $d_L(\mathcal{C}^2_8)=2$.
            
            \item Let $1\leq \omega \leq 2^{\varsigma-1}$. By Theorem \ref{thm3} and Theorem \ref{thm62}, $d_H(\mathcal{C}^2_8)=2$. Thus, $2\leq d_L(\mathcal{C}^2_8)$. Following Theorem \ref{thm9}, we can prove $\mathcal{C}^2_8$ has no codeword of Lee weights 2 and 3. A codeword $\wp(x)=u^2\zeta_1(x^{2^{\varsigma-1}}+1)=u^2\zeta_1(x+1)^{2^{\varsigma-1}}\in \langle u^2(x+1)^{\omega}\rangle \subseteq\mathcal{C}^2_8$ with $wt^{\mathcal{B}}_L(\wp(x))=4$. Thus, $d_L(\mathcal{C}^2_8)=4$.
        \end{itemize}

        \item \textbf{Subcase ii:} Let $ 2^{\varsigma-1}+1 \leq \beta \leq 2^\varsigma-1$.
        \begin{itemize}
            \item Let $\omega=0$. As in the above case, $\mathcal{C}^2_8$ has no codeword of Lee weights 1. we have $\chi (x)=\zeta_1u^2 \in \mathcal{C}^2_8$ with $wt^{\mathcal{B}}_L(\chi (x))=2$. Hence $d_L(\mathcal{C}^2_8)=2$.
            
            \item Let $1\leq \omega \leq 2^{\varsigma-1}$. By Theorem \ref{thm3} and Theorem \ref{thm62}, $d_H(\mathcal{C}^2_8)=2$. Thus, $2\leq d_L(\mathcal{C}^2_8)$. Following Theorem \ref{thm9}, we can prove $\mathcal{C}^2_8$ has no codeword of Lee weights 2 and 3. A codeword $\wp(x)=u^2\zeta_1(x^{2^{\varsigma-1}}+1)=u^2\zeta_1(x+1)^{2^{\varsigma-1}}\in \langle u^2(x+1)^{\omega}\rangle \subseteq\mathcal{C}^2_8$ with $wt^{\mathcal{B}}_L(\wp(x))=4$. Thus, $d_L(\mathcal{C}^2_8)=4$.
        
            \item Let $ 2^{\varsigma-1}+1 \leq \omega \leq 2^\varsigma-1$.By Theorem \ref{thm3}, $d_L(\mathcal{C}^2_8)\geq 4$.  From Theorem \ref{thm18}, $d_L(\langle (x+1)^{\alpha}+u z_1(x) \rangle )=4$. Then $d_L(\mathcal{C}^2_8)\leq 4$. Hence $d_L(\mathcal{C}^2_8)=4$.
        \end{itemize}
    \end{enumerate}
\end{enumerate}
\end{proof}

%\newpage
\subsection{If $z_1(x)\neq 0$, $\mathfrak{T}_1\neq0$, $z_2(x)=0$ and $z_3(x)=0$ }
\begin{theorem}
    $\mathcal{C}^3_8=\langle (x+1)^{\alpha}+u(x+1)^{\mathfrak{T}_1} z_1(x), u(x+1)^{\beta},u^2(x+1)^{\omega} \rangle $, where $0\leq \omega< \mathcal{W} \leq \mathcal{L}_1 \leq \beta <  \mathcal{U} \leq \alpha \leq 2^\varsigma-1$, $0\leq  \mathfrak{T}_1 < \beta $, $0\leq  \mathfrak{T}_2 < \omega $, $0\leq  \mathfrak{T}_3 < \omega $ and  $z_1(x)$ is a unit in $\mathcal{S}$. Then 
    \begin{equation*}
	d_L(\mathcal{C}^3_8)=
	\begin{cases}
            2&\text{if}\quad 1<\beta< \alpha \leq 2^{\varsigma-1} \quad\text{with}\quad \omega=0,\\
            2&\text{if}\quad 1\leq\omega<\beta< \alpha \leq 2^{\varsigma-1} \quad\text{with}\quad\beta+\alpha \leq 2^{\varsigma-1}+\mathfrak{T}_1,\\
            4&\text{if}\quad 1\leq\omega<\beta< \alpha \leq 2^{\varsigma-1} \quad\text{with}\quad\beta+\alpha >2^{\varsigma-1}+\mathfrak{T}_1,\\
		2  &\text{if}\quad 2^{\varsigma-1}+1\leq \alpha \leq 
            2^{\varsigma}-1 \quad \text{with} \quad 1 < \beta \leq 2^{\varsigma-1}\quad \text{and} \quad \omega=0,\\
            4  &\text{if}\quad 2^{\varsigma-1}+1\leq \alpha \leq 
            2^{\varsigma}-1 \quad \text{with} \quad 1\leq \omega < \beta \leq 2^{\varsigma-1},\\
            2  &\text{if}\quad 2^{\varsigma-1}+1\leq \beta<\alpha \leq 
            2^{\varsigma}-1 \quad \text{with} \quad \omega=0,\\
            4  &\text{if}\quad 2^{\varsigma-1}+1\leq \beta<\alpha \leq 
            2^{\varsigma}-1 \quad \text{with} \quad 1\leq \omega \leq 2^{\varsigma-1},\\
            2^{\gamma+1} &\text{if}\quad 2^\varsigma-2^{\varsigma-\gamma}+1 \leq \omega<\beta < \alpha  \leq 2^\varsigma-2^{\varsigma-\gamma}+2^{\varsigma-\gamma-1},\\
            &\qquad   \text{with} \quad \alpha \leq 2^{\varsigma-1}-2^{\varsigma-\gamma-1}+2^{\varsigma-\gamma-2}+\frac{\mathfrak{T}_1}{2}\\
            &\qquad \text{and} \quad 3\alpha\leq 2^\varsigma-2^{\varsigma-\gamma}+2^{\varsigma-\gamma-1}+2\mathfrak{T}_1,
            \quad \text{where}\quad 1\leq \gamma \leq \varsigma-1.
	\end{cases}
    \end{equation*}
\end{theorem}
\begin{proof}
    Let $\mathcal{B}=\{ \zeta_1,\zeta_2,\ldots, \zeta_m\}$ be a TOB of $\mathbb{F}_{2^m}$ over $\mathbb{F}_2.$ 
    \begin{enumerate}
        \item \textbf{Case 1:} Let $1<\beta<\alpha \leq 2^{\varsigma-1}$. 
        \begin{enumerate}
            \item \textbf{Subcase i:} Let $\omega=0$. 
             From Theorem \ref{thm3} and Theorem \ref{thm5}, $1\leq d_L(\mathcal{C}^3_8)\leq 2$. Suppose $\chi (x)=\lambda x^j \in \mathcal{C}^3_8$, $\lambda \in \mathcal{R}$ with $wt^{\mathcal{B}}_L(\chi (x))=1$. 
             If $\lambda$ is a unit in $\mathcal{R}$, then $\lambda x^j$ is a unit. This is not possible. If $\lambda$ is non-unit in $\mathcal{R}$ then $\lambda \in \langle u \rangle$ and $wt^{\mathcal{B}}_L(\lambda)\geq 3 $. Again, this is not possible.
             Hence $d_L(\mathcal{C}^3_8)=2$.
             
            \item \textbf{Subcase ii:} Let $1\leq \omega \leq 2^{\varsigma-1}$.
            \begin{itemize}
                \item  Let $\beta+\alpha \leq 2^{\varsigma-1}+\mathfrak{T}_1$. We have $\chi (x)=\zeta_1(x^{2^{\varsigma-1}}+1)=\zeta_1(x+1)^{2^{\varsigma-1}}=\zeta_1[(x+1)^{\alpha}+u(x+1)^{\mathfrak{T}_1} z_1(x)][(x+1)^{2^{\varsigma-1}-\alpha}] +[u (x+1)^{\beta}][(x+1)^{2^{\varsigma-1}-\alpha+\mathfrak{T}_1-\beta}z_1(x)]\in \mathcal{C}^3_8$. Since  $wt^{\mathcal{B}}_L(\chi (x))=2$, $d_L(\mathcal{C}^3_8)=2$.
            
                \item Let $\beta+\alpha > 2^{\varsigma-1}+\mathfrak{T}_1$. Following Theorem \ref{thm9}, we can prove $\mathcal{C}^3_8$ has no codeword of Lee weights 2 and 3. Hence $d_L(\mathcal{C}^3_8)=4$. 
            \end{itemize}
        \end{enumerate}
        \item \textbf{Case 2:} Let  $ 2^{\varsigma-1}+1\leq \alpha \leq 2^\varsigma-1$. 
        \begin{enumerate}
            \item \textbf{Subcase i:} Let $1 \leq \beta \leq 2^{\varsigma-1}$. 
            \begin{itemize}
                \item Let $\omega=0$. As in the above case, $\mathcal{C}^3_8$ has no codewo
                rd of Lee weights 1. we have $\chi (x)=\zeta_1u^2 \in \mathcal{C}^3_8$ with $wt^{\mathcal{B}}_L(\chi (x))=2$. Hence $d_L(\mathcal{C}^3_8)=2$.
                
                \item  Let $1\leq \omega \leq 2^{\varsigma-1}$. By Theorem \ref{thm3} and Theorem \ref{thm62}, $d_H(\mathcal{C}^3_8)=2$. Thus, $2\leq d_L(\mathcal{C}^3_8)$. Following Theorem \ref{thm9}, we can prove $\mathcal{C}^3_8$ has no codeword of Lee weights 2 and 3. A codeword $\wp(x)=u^2\zeta_1(x^{2^{\varsigma-1}}+1)=u^2\zeta_1(x+1)^{2^{\varsigma-1}}\in \langle u^2(x+1)^{\omega}\rangle \subseteq\mathcal{C}^3_8$ with $wt^{\mathcal{B}}_L(\wp(x))=4$. Thus, $d_L(\mathcal{C}^3_8)=4$.
            \end{itemize}
            
        \item \textbf{Subcase ii:} Let $2^{\varsigma-1}+1\leq \beta \leq 2^{\varsigma-1}-1$ 
        \begin{itemize}
            \item Let $\omega=0$. As in the above case, $\mathcal{C}^3_8$ has no codeword of Lee weights 1. we have $\chi (x)=\zeta_1u^2 \in \mathcal{C}^3_8$ with $wt^{\mathcal{B}}_L(\chi (x))=2$. Hence $d_L(\mathcal{C}^3_8)=2$.
                
            \item  Let $1\leq \omega \leq 2^{\varsigma-1}$. By Theorem \ref{thm3} and Theorem \ref{thm62}, $d_H(\mathcal{C}^3_8)=2$. Thus, $2\leq d_L(\mathcal{C}^3_8)$. Following Theorem \ref{thm9}, we can prove $\mathcal{C}^3_8$ has no codeword of Lee weights 2 and 3. A codeword $\wp(x)=u^2\zeta_1(x^{2^{\varsigma-1}}+1)=u^2\zeta_1(x+1)^{2^{\varsigma-1}}\in \langle u^2(x+1)^{\omega}\rangle \subseteq\mathcal{C}^3_8$ with $wt^{\mathcal{B}}_L(\wp(x))=4$. Thus, $d_L(\mathcal{C}^3_8)=4$.
                
            \item Let $2^{\varsigma-1}+1\leq \omega \leq 2^{\varsigma-1}-1$. By Theorem \ref{thm3},  $d_H(\langle (x+1)^{\omega} \rangle)\geq 4$ and by Theorem \ref{thm19},  $d_L(\langle (x+1)^{\alpha}+u(x+1)^{\mathfrak{T}_1} z_1(x) \rangle)=4$ if $\alpha \geq 2^{\varsigma-1}+\mathfrak{T}_1$. Thus, $d_L(\mathcal{C}^3_8)=4$.
            
            \item Let $2^\varsigma-2^{\varsigma-\gamma}+1 \leq \omega  \leq 2^\varsigma-2^{\varsigma-\gamma}+2^{\varsigma-\gamma-1}$, where $1\leq \gamma \leq \varsigma-1$. By Theorem \ref{thm3},  $d_H(\langle (x+1)^{\omega} \rangle)=2^{\gamma+1}$ and by Theorem \ref{thm19},  $d_L(\langle (x+1)^{\alpha}+u(x+1)^{\mathfrak{T}_1} z_1(x) \rangle)=2^{\gamma+1}$ if $\alpha \leq 2^{\varsigma-1}-2^{\varsigma-\gamma-1}+2^{\varsigma-\gamma-2}+\frac{\mathfrak{T}_1}{2}$ and $3\alpha\leq 2^\varsigma-2^{\varsigma-\gamma}+2^{\varsigma-\gamma-1}+2\mathfrak{T}_1$. Thus, $d_L(\mathcal{C}^3_8)=2^{\gamma+1}$.
        \end{itemize}
    \end{enumerate}
\end{enumerate}
\end{proof}

Using Theorem \ref{thm62} and by considering the cases on the variables $\alpha$, $\beta$ and $\omega$ as in the previous theorems, we can determine the Lee distances of the remaining cases of Type 8 cyclic codes of length $2^\varsigma$ over $\mathcal{R}$.

% \section{Conclusion}\label{sec4}

%References imported from bib file
%%%%%%%%%%%%%%%%%%%%%%%%%%%%%%%%%%%%%%%%%%%%%%%
 % \bibliographystyle{ieeetr}
	% \bibliography{Reference}

\begin{thebibliography}{00}
%[1]
\bibitem{berman1967semisimple} S. Berman, “Semisimple cyclic and abelian codes. II,” Cybernetics, vol. 3, no. 3, pp. 17–23, 1967.

%[2]
\bibitem{massey2003polynomial}J. L. Massey, D. J. Costello, and J. Justesen, “Polynomial weights and code constructions,”
IEEE Transactions on Information Theory, vol. 19, no. 1, pp. 101–110, 2003.

%[3]
\bibitem{falkner1979existence}G. Falkner, W. Heise, B. Kowol, and E. Zehendner, “On the existence of cyclic optimal codes,” 1979.

%[4]
\bibitem{roth2003cyclic}R. Roth and G. Seroussi, “On cyclic MDS codes of length q over GF (q) (Corresp.),” IEEE transactions on information theory, vol. 32, no. 2, pp. 284–285, 2003.

%[5]
\bibitem{van1991repeated}J. H. van Lint, “Repeated-root cyclic codes,” IEEE Transactions on Information Theory, vol. 37, no. 2, pp. 343–345, 1991.

%[6]
\bibitem{castagnoli1991repeated}
G. Castagnoli, J. L. Massey, P. A. Schoeller, and N. Von Seemann, “On repeated-root cyclic codes,” IEEE Transactions on Information Theory, vol. 37, no. 2, pp. 337–342, 1991.

%[7]
\bibitem{sualuagean2006repeated}A. Sălăgean, “Repeated-root cyclic and negacyclic codes over a finite chain ring,” Discrete applied mathematics, vol. 154, no. 2, pp. 413–419, 2006.

%[8] 
\bibitem{dinh2008linear}H. Q. Dinh, “On the linear ordering of some classes of negacyclic and cyclic codes and their distance distributions,” Finite Fields and Their Applications, vol. 14, no. 1, pp. 22–40, 2008.

%[9]
\bibitem{cao2015repeated} Y. Cao and Y. Gao, “Repeated root cyclic $\mathbb{F}_q$-linear codes over $\mathbb{F}_{q^l}$,” Finite Fields and Their Applications, vol. 31, pp. 202–227, 2015.

%[10]
\bibitem{batoul2016some}A. Batoul, K. Guenda, and T. A. Gulliver, “Some constacyclic codes over finite chain rings,” Advances in mathematics of communications, vol. 10, no. 4, 2016.

%[11]
\bibitem{zhao2018all}W. Zhao, X. Tang, and Z. Gu, “All $\alpha+ u\beta$-constacyclic codes of length $np^s$ over $\mathbb{F}_{p^m}+u\mathbb{F}_{p^m}$,” Finite Fields and Their Applications, vol. 50, pp. 1–16, 2018.

%[12]
\bibitem{cao2018constacyclic}Y. Cao, Y. Cao, H. Q. Dinh, F.-W. Fu, J. Gao, and S. Sriboonchitta, “Constacyclic codes of length $np^s$ over $\mathbb{F}_{p^m}+u\mathbb{F}_{p^m}$,” Adv. Math. Commun, vol. 12, no. 2, pp. 231–262, 2018.

%[13]
\bibitem{sidana2020repeated}T. Sidana and A. Sharma, “Repeated-Root Constacyclic Codes Over the Chain Ring $\frac{\mathbb{F}_{p^m}[u]}{\langle u^3 \rangle}$,”
IEEE Access, vol. 8, pp. 101320–101337, 2020.

%[14]
\bibitem{li2022unique}H. Li, P. Yu, J. Liang, and F. Zhao, “Unique Generators for Cyclic Codes of Arbitrary Length over $\frac{\mathbb{F}_{p^m}[u]}{\langle u^3 \rangle}$ and Their Applications,” Journal of Mathematics, vol. 2022, no. 1, p. 6108863, 2022.

\bibitem{ankur2020type} Ankur, “Type I and Type II codes over the ring $\mathbb{F}_2+ s \mathbb{F}_2+ s^2 \mathbb{F}_2$,” Arabian Journal of Mathematics, vol. 9, no. 1, pp. 1–7, 2020.

%[15]
\bibitem{dinh2021hamming}
H. Q. Dinh, J. Laaouine, M. E. Charkani, and W. Chinnakum, “Hamming Distance of Constacyclic Codes of Length $p^s$ Over $\mathbb{F}_{p^m}+ u\mathbb{F}_{p^m}+u^2\mathbb{F}_{p^m}$,” IEEE Access, vol. 9, pp. 141064–141078,
2021.

%[16]
\bibitem{dinh2018hamming}H. Q. Dinh, B. T. Nguyen, A. K. Singh, and S. Sriboonchitta, “Hamming Distance of Constacyclic Codes of Length $p^s$ Over $\mathbb{F}_{p^m}+u\mathbb{F}_{p^m}+u^2\mathbb{F}_{p^m}$,” IEEE Communications Letters, vol. 22, no. 12, pp. 2400–2403, 2018.

%[17]
\bibitem{dinh2020hamming}H. Q. Dinh, A. Gaur, I. Gupta, A. K. Singh, M. K. Singh, and R. Tansuchat, “Hamming distance of repeated-root constacyclic codes of length $2p^s$ over $\mathbb{F}_{p^m}+u\mathbb{F}_{p^m}$,” Applicable Algebra in Engineering, Communication and Computing, vol. 31, no. 3, pp. 291–305, 2020.

%[18]
\bibitem{lee1958some}
C. Lee, “Some properties of nonbinary error-correcting codes,” IRE Transactions on Information Theory, vol. 4, no. 2, pp. 77–82, 1958.

%[19]
\bibitem{hammons1994z}
A. R. Hammons, P. V. Kumar, A. R. Calderbank, N. J. Sloane, and P. Solé, “The $\mathbb{Z}_4$-linearity of Kerdock, Preparata, Goethals, and related codes,” IEEE Transactions on Information Theory,
vol. 40, no. 2, pp. 301–319, 1994.

%[20]
\bibitem{dinh2007complete} H. Dinh, “Complete distances of all negacyclic codes of length $2^s$ over $\mathbb{Z}_{2^a}$,” IEEE Trans Inform
Theory, vol. 53, no. 1, pp. 4252–4262, 2007.

%[21]
\bibitem{kai2010distances} X. Kai and S. Zhu, “On the distances of cyclic codes of length $2^e$ over $\mathbb{Z}_4$,” Discrete mathematics,
vol. 310, no. 1, pp. 12–20, 2010.

%[22]
\bibitem{kim2017lee} B. Kim and Y. Lee, “Lee weights of cyclic self-dual codes over Galois rings of characteristic $p^2$,”
Finite Fields and Their Applications, vol. 45, pp. 107–130, 2017.

%[23]
\bibitem{dinh2021lee} H. Q. Dinh, P. K. Kewat, and N. K. Mondal, “Lee Distance of $(4z-1)$-Constacyclic Codes of Length $2^s$ Over the Galois Ring $GR (2^a, m)$,” IEEE Communications Letters, vol. 25, no. 7, pp. 2114–2117, 2021.

%[24]
\bibitem{dinh2022lee} H. Q. Dinh, P. K. Kewat, and N. K. Mondal, “Lee distance distribution of repeated-root constacyclic codes over $GR(2^e,m)$ and related MDS codes,” Journal of Applied Mathematics and Computing, vol. 68, no. 6, pp. 3861–3872, 2022.

%[25]
\bibitem{betsumiya2004type}K. Betsumiya, S. Ling, and F. R. Nemenzo, “Type II codes over $\mathbb{F}_{2^m}+u\mathbb{F}_{2^m}$,” Discrete Mathematics, vol. 275, no. 1-3, pp. 43–65, 2004.

%[26]
\bibitem{dinh2021lee1}H. Q. Dinh, P. K. Kewat, and N. K. Mondal, “Lee distance of cyclic and $(1+ u\gamma)$-constacyclic codes of length $2^s$ over $\mathbb{F}_{2^m}+ u\mathbb{F}_{2^m}$,” Discrete Mathematics, vol. 344, no. 11, p. 112551, 2021.

%[27]
\bibitem{macwilliams1977theory} F. J. MacWilliams and N. J. A. Sloane, The theory of error-correcting codes, vol. 16. Elsevier,
1977.

%[28]
\bibitem{huffman2010fundamentals} W. C. Huffman and V. Pless, Fundamentals of error-correcting codes. Cambridge University Press, 2010.


%[29] 
\bibitem{lidl1997finite} R. Lidl and H. Niederreiter, Finite fields. No. 20, Cambridge University Press, 1997.

%[30]
\bibitem{lempel1975matrix} A. Lempel, “Matrix Factorization over $GF(2)$ and Trace-Orthogonal Bases of $GF(2^n)$,” SIAM Journal on Computing, vol. 4, no. 2, pp. 175–186, 1975.

%[31] 
\bibitem{laaouine2021complete} J. Laaouine, M. E. Charkani, and L. Wang, “Complete classification of repeated-root $\sigma$-constacyclic codes of prime power length over $\frac{\mathbb{F}_{p^m}[u]}{\langle u^3 \rangle}$,” Discrete Mathematics, vol. 344, no. 6,
p. 112325, 2021.
\end{thebibliography}
%%%%%%%%%%%%%%%%%%%%%%%%%%%%%%%%%%%%%%%%%%%%%%%%%
    
%Reference written
%%%%%%%%%%%%%%%%%%%%%%%%%%%%%%%%%%%%%%%%%%%%%%%%%%%%
 
\end{document}